\documentclass[10pt]{amsart}
\usepackage{latexsym,exscale,enumerate}
\usepackage{amsmath,amsthm,amsfonts,amssymb}
\usepackage{amscd, stmaryrd}
\usepackage{ulem}
\usepackage[all]{xy}
\SelectTips{cm}{}
\usepackage{graphicx}
\usepackage{tikz}
\usetikzlibrary{decorations.markings}
\usetikzlibrary{decorations.pathreplacing}
\usetikzlibrary{arrows,shapes,positioning}
\tikzstyle directed=[postaction={decorate,decoration={markings,
    mark=at position #1 with {\arrow{>}}}}]
\tikzstyle rdirected=[postaction={decorate,decoration={markings,
    mark=at position #1 with {\arrow{<}}}}]
\usepackage{bbm}
\usepackage[bookmarks=false]{hyperref}

\usepackage[all]{xy}
\SelectTips{cm}{}

\usepackage{graphicx}

\usepackage{tikz}
\usetikzlibrary{decorations.markings}
\usetikzlibrary{decorations.pathreplacing}
\usetikzlibrary{arrows,shapes,positioning}
\tikzstyle directed=[postaction={decorate,decoration={markings,
    mark=at position #1 with {\arrow{>}}}}]
\tikzstyle rdirected=[postaction={decorate,decoration={markings,
    mark=at position #1 with {\arrow{<}}}}]

\newcommand{\onenn}[1]{{\mathbf 1}_{#1}}

\newcommand{\onel}{{\mathbf 1}_{\lambda}}

\newcommand{\onea}{{\mathbf 1}_{{\bf a}}}

\newcommand\sE{{\cal{E}}}
\newcommand\sF{{\cal{F}}}
\def\cal#1{\mathcal{#1}}%

\newcommand{\Sym}{\mathrm{Sym}}
\newcommand{\MF}{\mathrm{MF}}
\newcommand{\stab}{\mathrm{stab}}

\newcommand{\HMF}{\textbf{HMF}}

\newcommand{\Foam}[1]{#1\textbf{Foam}}

\newcommand{\bV}{\textstyle{\bigwedge}}

\def\1{\mathbbm{1}}%

\newcommand\T{{\sf{T}}}
\def\l{\lambda}

\newcommand{\U}{\dot{{\bf U}}}
\newcommand{\Ucat}{\cal{U}}
\newcommand{\Ucatc}{\check{\cal{U}}}

\newcommand{\UcatD}{\dot{\cal{U}}}

\newcommand{\sln}{\mf{sl}_n}
\newcommand{\slm}{\mf{sl}_m}
\newcommand{\slnn}[1]{\mf{sl}_{#1}}

\newcommand{\glm}{\mf{gl}_m}
\newcommand{\glnn}[1]{\mf{gl}_{#1}}

\newcommand{\foam}[1]{#1\cat{Foam}}

\newcommand{\refequal}[1]{\xy {\ar@{=}^{#1}
(-1,0)*{};(1,0)*{}};
\endxy}

\newcommand{\cat}[1]{\ensuremath{\mbox{\bfseries {\upshape {#1}}}}}

\newcommand{\Hom}{{\rm Hom}}

\renewcommand{\to}{\rightarrow}

\newcommand{\id}{{\rm id}}
\newcommand{\bigb}[1]{
\begin{tikzpicture}[scale=.8]
\node[draw, thick, rectangle, fill=white]  at (0,0) {\small $#1$} ;
\end{tikzpicture}}

\newcommand{\End}{{\rm End}}

\newcommand{\scs}{\scriptstyle}

\def\mf{\mathfrak}
\def\shuffle{\,\raise 1pt\hbox{$\scriptscriptstyle\cup{\mskip
               -4mu}\cup$}\,}

\numberwithin{equation}{section}

\def\emph#1{{\sl #1\/}}

\let\hat=\widehat
\let\tilde=\widetilde

\let\phi=\varphi
\let\theta=\vartheta
\let\epsilon=\varepsilon

\usepackage{bbm}
\def\C{{\mathbbm C}}
\def\N{{\mathbbm N}}

\def\Z{{\mathbbm Z}}

\def\P{{\mathbbm P}}

\def\P{{\mathbbm P}}

\def\V{{\mathbbm V}}
\def\W{{\mathbbm W}}
\def\X{{\mathbbm X}}
\def\Y{{\mathbbm Y}}

\def\L{{\mathbbm L}}
\def\M{{\mathbbm M}}
\def\A{{\mathbbm A}}
\def\S{{\mathbbm S}}
\def\T{{\mathbbm T}}

\newcommand{\Bb}{\mathbb{B}}

\def\cal#1{\mathcal{#1}}%
\def\1{\mathbbm{1}}%
\def\nn{\notag}

\def\la{\langle}
\def\ra{\rangle}

\newcommand{\lowrru}[1]{\xybox{%
  (-8,0)*{};
  (8,0)*{};
  (-6,-18)*{};(6,-9)*{} **\crv{(-6,-13) & (6,-15)} ?(1)*\dir{>};
  (6,-9)*{};(6,0)*{}  **\dir{-} ?(.3)*\dir{ }+(2,0)*{\scs {\bf j}};
}}

\newcommand{\lowllu}[1]{\xybox{%
  (-8,0)*{};
  (8,0)*{};
  (6,-18)*{};(-6,-9)*{} **\crv{(6,-13) & (-6,-15)} ?(1)*\dir{>};
  (-6,-9)*{};(-6,0)*{}  **\dir{-} ?(.3)*\dir{ }+(-2,0)*{\scs {\bf j}};
}}

\newcommand{\cbub}[1]{
\xybox{%
 (-6,0)*{};
  (6,0)*{};
  (-4,0)*{}="t1";
  (4,0)*{}="t2";
  "t2";"t1" **\crv{(4,6) & (-4,6)};
   ?(0)*\dir{<} ?(.95)*\dir{<};
  "t2";"t1" **\crv{(4,-6) & (-4,-6)};
   ?(.75)*\dir{}+(0,0)*{\bullet}+(0,-3)*{\scs {#1}};
}}

\newcommand{\crossingEEfoam}[1][1]{
\xy
(0,0)*{
\begin{tikzpicture} [scale=.3,fill opacity=0.2]
	\path[fill=red] (3.5,0) to (-2.5,0) to (-2.5,7.5) to (3.5,7.5);
	\path[fill=red] (2.5,1) to (-3.5,1) to (-3.5,8.5) to (2.5,8.5);
	\path[fill=blue] (.5,2.75) to (-.5,3.75) to (-.5,5.75) to (.5,4.75);
	\path[fill=blue] (-.75,7.5) to [out=270,in=130] (-.125,4.91) to (-1.125,5.91) to [out=130,in=270] (-1.75,8.5);
	\path[fill=blue] (-.125,4.91) to [out=320,in=180] (.5,4.75) to (-.5,5.75) to [out=180,in=320] (-1.125,5.91);
	\path[fill=blue] (.5,4.75) to [out=0,in=270] (1.75,7.5) to (.75,8.5) to [out=270,in=0] (-.5,5.75);
	\path[fill=blue] (-.75,0) to [out=90,in=180] (.5,2.75) to (-.5,3.75) to [out=180,in=90] (-1.75,1);
	\path[fill=blue] (.5,2.75) to [out=0,in=140] (1.125,2.59) to (.125,3.59) to [out=140,in=0] (-.5,3.75);
	\path[fill=blue] (1.125,2.59) to [out=320,in=270] (1.75,0) to (.75,1) to [out=270,in=320] (.125,3.59);
	\draw[very thick] (3.5,0) to (-2.5,0);
	\draw[very thick] (2.5,1) to (-3.5,1);
	\draw[very thick, directed=.6] (1.75,0) to (.75,1);	
	\draw[very thick, directed=.6] (-.75,0) to (-1.75,1);		
	\draw[very thick, red] (-.75,0) to [out=90,in=180] (.5,2.75);
	\draw[very thick, red] (1.75,0) to [out=90,in=0] (.5,2.75);
	\draw[very thick, red] (.5,2.75) to [out=90,in=270] (.5,4.75);
	\draw[very thick, red] (.5,4.75) to [out=0,in=270] (1.75,7.5);
	\draw[very thick, red] (.5,4.75) to [out=180,in=270] (-.75,7.5);
	\draw[very thick, red] (.5,2.75) to (-.5,3.75);
	\draw[very thick, red] (.5,4.75) to (-.5,5.75);
	\draw[very thick, red] (-1.75,1) to [out=90,in=180] (-.5,3.75);
	\draw[very thick, red] (.75,1) to [out=90,in=0] (-.5,3.75);
	\draw[very thick, red] (-.5,3.75) to [out=90,in=270] (-.5,5.75);
	\draw[very thick, red] (-.5,5.75) to [out=0,in=270] (.75,8.5);
	\draw[very thick, red] (-.5,5.75) to [out=180,in=270] (-1.75,8.5);
	\draw[very thick] (3.5,0) to (3.5,7.5);	
	\draw[very thick] (2.5,1) to (2.5,8.5);
	\draw[very thick] (-2.5,0) to (-2.5,7.5);	
	\draw[very thick] (-3.5,1) to (-3.5,8.5);
	\draw[very thick] (3.5,7.5) to (-2.5,7.5);
	\draw[very thick] (2.5,8.5) to (-3.5,8.5);
	\draw[very thick, directed=.6] (1.75,7.5) to (.75,8.5);	
	\draw[very thick, directed=.6] (-.75,7.5) to (-1.75,8.5);
	\end{tikzpicture}
};
\endxy
}

\newcommand{\cupEFfoam}[3][.5]{
\xy
(0,0)*{
\begin{tikzpicture} [scale=#1,fill opacity=0.2]
\path[fill=red] (2,-2) to (-2,-2) to (-2,1) to (2,1);
\path[fill=blue] (1.5,2) to [out=270,in=0] (0,-.5) to [out=180,in=270] (-1.5,2) to
	(-.5,1) to [out=270,in=180] (0,.25) to [out=0,in=270] (.5,1);
\path[fill=red] (2.5,-1) to (-2.5,-1) to (-2.5,2) to (2.5,2);
\draw[very thick, directed=.5] (2.5,-1) to (-2.5,-1);
\draw[very thick] (2.5,-1) to (2.5,2);
\draw[very thick] (-2.5,-1) to (-2.5,2);
\draw[very thick, directed=.5] (2.5,2) to (-2.5,2);
\draw[very thick, red, rdirected=.65] (1.5,2) to [out=270,in=0] (0,-.5)
	to [out=180,in=270] (-1.5,2);
\node[red, opacity=1] at (2,1.5) {$_{#2}$};
\draw[very thick, directed=.5] (2,-2) to (-2,-2);
\draw[very thick] (2,-2) to (2,1);
\draw[very thick] (-2,-2) to (-2,1);
\draw[very thick, directed=.5] (2,1) to (-2,1);
\draw[very thick, red, rdirected=.75] (-.5,1) to [out=270,in=180] (0,.25)
	to [out=0,in=270] (.5,1);
\node[red, opacity=1] at (1.25,-1.5) {$_{#3}$};
\draw[very thick, directed=.5] (1.5,2) to (.5,1);
\draw[very thick, directed=.5] (-.5,1) to (-1.5,2);
\end{tikzpicture}};
\endxy
}

\newcommand{\cupFEfoam}[3][.5]{
\xy
(0,0)*{
\begin{tikzpicture} [scale=#1,fill opacity=0.2]
\path[fill=red] (2,-2) to (-2,-2) to (-2,1) to (2,1);
\path[fill=blue] (1.5,1) to [out=270,in=0] (0,-.25) to [out=180,in=270] (-1.5,1) to
	(-.5,2) to [out=270,in=180] (0,1.25) to [out=0,in=270] (.5,2);
\path[fill=red] (2.5,-1) to (-2.5,-1) to (-2.5,2) to (2.5,2);
\draw[very thick, directed=.5] (2.5,-1) to (-2.5,-1);
\draw[very thick] (2.5,-1) to (2.5,2);
\draw[very thick] (-2.5,-1) to (-2.5,2);
\draw[very thick, directed=.5] (2.5,2) to (-2.5,2);
\draw[very thick, red, rdirected=.5] (-.5,2) to [out=270,in=180] (0,1.25)
	to [out=0,in=270] (.5,2);
\node[red, opacity=1] at (2,1.5) {$_{#2}$};
\draw[very thick, directed=.5] (2,-2) to (-2,-2);
\draw[very thick] (2,-2) to (2,1);
\draw[very thick] (-2,-2) to (-2,1);
\draw[very thick, directed=.5] (2,1) to (-2,1);
\draw[very thick, red, rdirected=.5] (1.5,1) to [out=270,in=0] (0,-.25)
	to [out=180,in=270] (-1.5,1);
\node[red, opacity=1] at (1.25,-1.5) {$_{#3}$};
\draw[very thick, directed=.5] (1.5,1) to (.5,2);
\draw[very thick, directed=.5] (-.5,2) to (-1.5,1);
\end{tikzpicture}
}
\endxy
}

\newcommand{\capEFfoam}[3][.5]{
\xy
(0,0)*{
\begin{tikzpicture} [scale=#1,fill opacity=0.2]
\path[fill=red] (2.5,2) to (-2.5,2) to (-2.5,-1) to (2.5,-1);
\path[fill=blue] (1.5,-1) to [out=90,in=0] (0,.25) to [out=180,in=90] (-1.5,-1) to
	(-.5,-2) to [out=90,in=180] (0,-1.25) to [out=0,in=90] (.5,-2);
\path[fill=red] (2,1) to (-2,1) to (-2,-2) to (2,-2);
\draw[very thick, directed=.5] (2.5,2) to (-2.5,2);
\draw[very thick] (2.5,2) to (2.5,-1);
\draw[very thick] (-2.5,2) to (-2.5,-1);
\draw[very thick, directed=.5] (2.5,-1) to (-2.5,-1);
\draw[very thick, red, directed=.5] (1.5,-1) to [out=90,in=0] (0,.25)
	to [out=180,in=90] (-1.5,-1);
\node[red, opacity=1] at (2,1.5) {$_{#2}$};
\draw[very thick, directed=.5] (1.5,-1) to (.5,-2);
\draw[very thick, directed=.5] (-.5,-2) to (-1.5,-1);
\draw[very thick, directed=.5] (2,1) to (-2,1);
\draw[very thick] (2,1) to (2,-2);
\draw[very thick] (-2,1) to (-2,-2);
\draw[very thick, directed=.5] (2,-2) to (-2,-2);
\draw[very thick, red, directed=.5] (-.5,-2) to [out=90,in=180] (0,-1.25)
	to [out=0,in=90] (.5,-2);
\node[red, opacity=1] at (1.25,.5) {$_{#3}$};
\end{tikzpicture}};
\endxy
}

\newcommand{\capFEfoam}[3][.5]{
\xy
(0,0)*{
\begin{tikzpicture} [scale=#1,fill opacity=0.2]
\path[fill=red] (2.5,2) to (-2.5,2) to (-2.5,-1) to (2.5,-1);
\path[fill=blue] (1.5,-2) to [out=90,in=0] (0,.5) to [out=180,in=90] (-1.5,-2) to
	(-.5,-1) to [out=90,in=180] (0,-.25) to [out=0,in=90] (.5,-1);
\path[fill=red] (2,1) to (-2,1) to (-2,-2) to (2,-2);
\draw[very thick, directed=.5] (2.5,2) to (-2.5,2);
\draw[very thick] (2.5,2) to (2.5,-1);
\draw[very thick] (-2.5,2) to (-2.5,-1);
\draw[very thick, directed=.5] (2.5,-1) to (-2.5,-1);
\draw[very thick, red, directed=.75] (-.5,-1) to [out=90,in=180] (0,-.25)
	to [out=0,in=90] (.5,-1);
\node[red, opacity=1] at (1.75,1.5) {$_{#2}$};
\draw[very thick, directed=.5] (1.5,-2) to (.5,-1);
\draw[very thick, directed=.5] (-.5,-1) to (-1.5,-2);
\draw[very thick, directed=.5] (2,1) to (-2,1);
\draw[very thick] (2,1) to (2,-2);
\draw[very thick] (-2,1) to (-2,-2);
\draw[very thick, directed=.5] (2,-2) to (-2,-2);
\draw[very thick, red, directed=.65] (1.5,-2) to [out=90,in=0] (0,.5)
	to [out=180,in=90] (-1.5,-2);
\node[red, opacity=1] at (1.25,.5) {$_{#3}$};
\end{tikzpicture}};
\endxy
}

\usepackage{wrapfig}

\theoremstyle{plain}
\newtheorem{thm}{Theorem}[]
\newtheorem{prop}[thm]{Proposition}
\newtheorem{cor}[thm]{Corollary}
\newtheorem{lem}[thm]{Lemma}

\newtheorem{conv}[thm]{Convention}
\theoremstyle{definition}
\newtheorem{defi}[thm]{Definition}
\theoremstyle{remark}
\newtheorem{rem}[thm]{Remark}

\newtheorem{exa}[thm]{Example}
\newtheorem*{exa-nono}{Example}

\begin{document}

\newcommand{\comm}[1]{}
\newcommand{\BNC}{\cal{N}\cal{H}}
\newcommand{\web}[3][M]{#2\cat{Web}_{#3}(#1)}
\newcommand{\ASym}{\mathrm{Antisym}}

\newcommand{\thinsphere}[1][1]{\xy
(0,0)*{
\begin{tikzpicture} [fill opacity=0.2, scale=#1]
	\path [fill=green] (0,0) circle (1);
	\draw (-1,0) .. controls (-1,-.4) and (1,-.4) .. (1,0);
	\draw[dashed] (-1,0) .. controls (-1,.4) and (1,.4) .. (1,0);
	\draw[black] (0,0) circle (1);
\end{tikzpicture}};
\endxy}
\newcommand{\thindottedsphere}[2][1]{
\xy
(0,0)*{
\begin{tikzpicture} [fill opacity=0.2, scale=#1]
	\path [fill=green] (0,0) circle (1);
	\draw (-1,0) .. controls (-1,-.4) and (1,-.4) .. (1,0);
	\draw[dashed] (-1,0) .. controls (-1,.4) and (1,.4) .. (1,0);
	\draw[black] (0,0) circle (1);
	\node [opacity=1]  at (0,0.6) {$\bullet$};
	\node [opacity=1] at (-0.3,0.6) {$#2$};
\end{tikzpicture}};
\endxy}

\newcommand{\thincylinder}[1][1]{
\xy
(0,0)*{
\begin{tikzpicture} [fill opacity=0.2,  decoration={markings, mark=at position 0.6 with {\arrow{>}};  }, scale=#1]
	\path [fill=green] (0,4) ellipse (1 and 0.5);
	\path [fill=green] (0,0) ellipse (1 and 0.5);
	\path[fill=green, opacity=.3] (-1,4) .. controls (-1,3.34) and (1,3.34) .. (1,4) --
		(1,0) .. controls (1,.66) and (-1,.66) .. (-1,0) -- cycle;
	\draw [black] (0,4) ellipse (1 and 0.5);
	\draw [black] (0,0) ellipse (1 and 0.5);
	\draw[black] (1,4) -- (1,0);
	\draw[black] (-1,4) -- (-1,0);
\end{tikzpicture}};
\endxy
}

\newcommand{\grbambooRHSone}[1][1]{
\xy
(0,0)*{
\begin{tikzpicture} [fill opacity=0.2,  decoration={markings, mark=at position 0.6 with {\arrow{>}};  }, scale=#1]
	\draw [fill=green] (0,4) ellipse (1 and 0.5);
	\draw (-1,4) .. controls (-1,2) and (1,2) .. (1,4);
	\path [fill=green, opacity=.3] (1,4) .. controls (1,3.34) and (-1,3.34) .. (-1,4) --
		(-1,4) .. controls (-1,2) and (1,2) .. (1,4);
	\node[opacity=1] at (0,3) {$\bullet$};
	\path [fill=green] (0,0) ellipse (1 and 0.5);
	\draw (1,0) .. controls (1,-.66) and (-1,-.66) .. (-1,0);
	\draw[dashed] (1,0) .. controls (1,.66) and (-1,.66) .. (-1,0);
	\draw (-1,0) .. controls (-1,2) and (1,2) .. (1,0);
	\path [fill=green, opacity=.3] (1,0) .. controls (1,.66) and (-1,.66) .. (-1,0) --
		(-1,0) .. controls (-1,2) and (1,2) .. (1,0);
\end{tikzpicture}};
\endxy
}

\newcommand{\grbambooRHStwo}[1][1]{
\xy
(0,0)*{
\begin{tikzpicture} [fill opacity=0.2,  decoration={markings, mark=at position 0.6 with {\arrow{>}};  }, scale=#1]
	\draw [fill=green] (0,4) ellipse (1 and 0.5);
	\draw (-1,4) .. controls (-1,2) and (1,2) .. (1,4);
	\path [fill=green, opacity=.3] (1,4) .. controls (1,3.34) and (-1,3.34) .. (-1,4) --
		(-1,4) .. controls (-1,2) and (1,2) .. (1,4);
	\path [fill=green] (0,0) ellipse (1 and 0.5);
	\draw (1,0) .. controls (1,-.66) and (-1,-.66) .. (-1,0);
	\draw[dashed] (1,0) .. controls (1,.66) and (-1,.66) .. (-1,0);
	\draw (-1,0) .. controls (-1,2) and (1,2) .. (1,0);
	\path [fill=green, opacity=.3] (1,0) .. controls (1,.66) and (-1,.66) .. (-1,0) --
		(-1,0) .. controls (-1,2) and (1,2) .. (1,0);
	\node[opacity=1] at (0,1) {$\bullet$};
\end{tikzpicture}};
\endxy
}
\newcommand{\twodotsheet}[2][1]{
\xy
(0,0)*{
\begin{tikzpicture} [fill opacity=0.2,decoration={markings, mark=at position 0.6 with {\arrow{>}}; }, scale=#1]
	\draw [fill=green] (1,1) -- (-1,2) -- (-1,-1) -- (1,-2) -- cycle;
	\node [opacity=1] at (0,-0.25) {$\bullet$};
	\node [opacity=1] at (0,0.25) {$\bullet$};
\end{tikzpicture}};
\endxy}

\newcommand{\grsheet}[1][1]{
\xy
(0,0)*{
\begin{tikzpicture} [fill opacity=0.2,decoration={markings, mark=at position 0.6 with {\arrow{>}}; }, scale=#1]
	\draw [fill=green] (1,1) -- (-1,2) -- (-1,-1) -- (1,-2) -- cycle;
\end{tikzpicture}};
\endxy}

\newcommand{\dotredsheet}[3][1]{
\xy
(0,0)*{
\begin{tikzpicture} [fill opacity=0.2,decoration={markings, mark=at position 0.6 with {\arrow{>}}; }, scale=#1]
	\draw [fill=red] (1,1) -- (-1,2) -- (-1,-1) -- (1,-2) -- cycle;
	\node [opacity=1] at (0,0) {$\bullet^#2$};
	\node[red, opacity=1] at (-.75,1.5) {\tiny{$#3$}};
\end{tikzpicture}};
\endxy}

\newcommand{\unzipzip}[5][1]{
\xy
(0,0)*{
\begin{tikzpicture} [scale=.5,fill opacity=0.2]
	\path[fill=red] (2.5,0) to [out=180,in=315] (1,0.5) to [out=90,in=0]  (0.25,1.25) to [out=180, in=90] (-0.5,0.5) to [out=225, in=0] (-1.5,0) to (-1.5,4) to [out=0, in=225] (-0.5,4.5) to [out=270, in=180] (0.25,3.75) to [out=0, in=270] (1,4.5) to [out=315, in=180] (2.5,4) to (2.5,0);
	\path[fill=blue] (2,1) to [out=180,in=45] (1,0.5) to [out=90,in=0]  (0.25,1.25) to [out=180, in=90] (-0.5,0.5) to [out=135, in=0] (-2,1) to (-2,5) to [out=0, in=135] (-0.5,4.5) to [out=270, in=180] (0.25,3.75) to [out=0, in=270] (1,4.5) to [out=45, in=180] (2,5) to (2,1);
	\path[fill=green]  (1,0.5) to [out=90,in=0]  (0.25,1.25) to [out=180, in=90] (-0.5,0.5) to (1,0.5);
	\path[fill=green] (-0.5,4.5) to [out=270, in=180] (0.25,3.75) to [out=0, in=270] (1,4.5) to (-0.5,4.5);
	\draw[very thick, directed=.55] (2,1) to [out=180,in=45] (1,.5);
	\draw[very thick,directed=.55] (2.5,0) to [out=180,in=315] (1,.5);
	\draw[very thick, directed=.55] (1,.5) to (-.5,.5);
	\draw[very thick, directed=.55] (-.5,.5) to [out=225,in=0] (-1.5,0);
	\draw[very thick, directed=.55] (-.5,.5) to [out=135,in=0] (-2,1);
	\draw[very thick, red, directed=.65] (-.5,.5) to [out=90,in=180] (.25,1.25) to [out=0, in=90] (1,.5);
	\draw[very thick] (2,1) to (2,5);
	\draw[very thick] (2.5,0) to (2.5,4);
	\draw[very thick] (-1.5,0) to (-1.5,4);
	\draw[very thick] (-2,1) to (-2,5);
	\draw[blue, dashed] (2,3) to (-2,3);
	\draw[red, dashed] (2.5,2) to  (-1.5,2);
	\draw[very thick, red, directed=.65] (1,4.5) to [out=270, in=0](0.25,3.75) to [out=180,in=270]  (-.5,4.5);
	\draw[very thick, directed=.55] (2,5) to [out=180,in=45] (1,4.5);
	\draw[very thick,directed=.55] (2.5,4) to [out=180,in=315] (1,4.5);
	\draw[very thick, directed=.55] (1,4.5) to (-.5,4.5);
	\draw[very thick, directed=.55] (-.5,4.5) to [out=225,in=0] (-1.5,4);
	\draw[very thick, directed=.55] (-.5,4.5) to [out=135,in=0] (-2,5);
	\node[opacity=1] at (.25,.75) {\tiny $_{a+b}$};
	\node[blue,opacity=1] at (1.75,4.65) {\tiny $b$};
	\node[red,opacity=1] at (2.25,3.75) {\tiny $a$};
	\node[opacity=1] at (.25,4.25) {\tiny $_{a+b}$};
	\node[opacity=1] at (2,.5) {$_{#2}$};
	\node[opacity=1] at (-1.25,4.5) {$_{#3}$};
	\node[opacity=1] at (1.45,4.4) {$_{#4}$};
	\node[opacity=1] at (-1.2,.4) {$_{#5}$};
\end{tikzpicture}
};
\endxy}

\newcommand{\zipthrough}[5][1]{
\xy
(0,0)*{
\begin{tikzpicture} [scale=.5,fill opacity=0.2]
	\path[fill=blue] (2,5) to [out=180,in=45] (1,4.5) to (1,.5) to [out=45,in=180] (2,1);
	\path[fill=red] (2.5,4) to [out=180,in=315] (1,4.5) to (1,.5) to [out=315,in=180] (2.5,0);
	\path[fill=green] (1,4.5) to (-.5,4.5) to (-.5,.5) to (1,.5);
	\path[fill=red] (-1.5,0) to [out=0,in=225] (-.5,.5) to (-.5,4.5) to [out=225,in=0] (-1.5,4);
	\path[fill=blue] (-2,1) to [out=0,in=135] (-.5,.5) to (-.5,4.5) to [out=135,in=0] (-2,5);
	\draw[very thick, directed=.55] (2,1) to [out=180,in=45] (1,.5);
	\draw[very thick,directed=.55] (2.5,0) to [out=180,in=315] (1,.5);
	\draw[very thick, directed=.55] (1,.5) to (-.5,.5);
	\draw[very thick, directed=.55] (-.5,.5) to [out=225,in=0] (-1.5,0);
	\draw[very thick, directed=.55] (-.5,.5) to [out=135,in=0] (-2,1);
	\draw[very thick, red, directed=.35] (1,4.5) to (1,.5);
	\draw[very thick, red, directed=.72] (-.5,.5) to (-.5,4.5);
	\draw[very thick] (2,1) to (2,5);
	\draw[very thick] (2.5,0) to (2.5,4);
	\draw[very thick] (-1.5,0) to (-1.5,4);
	\draw[very thick] (-2,1) to (-2,5);	
	\draw[dashed, blue] (2,3) to [out=180,in=45] (1,2.5);
	\draw[red, dashed] (2.5,2) to [out=180,in=315] (1,2.5);
	\draw[green, dashed] (1,2.5) to (-.5,2.5);
	\draw[red, dashed] (-.5,2.5) to [out=225,in=0] (-1.5,2);
	\draw[dashed, blue] (-.5,2.5) to [out=135,in=0] (-2,3);
	\draw[very thick, directed=.55] (2,5) to [out=180,in=45] (1,4.5);
	\draw[very thick,directed=.55] (2.5,4) to [out=180,in=315] (1,4.5);
	\draw[very thick, directed=.55] (1,4.5) to (-.5,4.5);
	\draw[very thick, directed=.55] (-.5,4.5) to [out=225,in=0] (-1.5,4);
	\draw[very thick, directed=.65] (-.5,4.5) to [out=135,in=0] (-2,5);
	\node[blue,opacity=1] at (1.8,4.65) {\tiny $b$};
	\node[red,opacity=1] at (2.25,3.75) {\tiny $a$};
	\node[blue,opacity=1] at (-1.75,4.65) {\tiny $b$};
	\node[opacity=1] at (.25,4.15) {\tiny $_{a+b}$};
	\node[red,opacity=1] at (-1,3.75) {\tiny $_{a}$};
	\node[opacity=1] at (2,.5) {$_{#2}$};
	\node[opacity=1] at (-1.25,4.5) {$_{#3}$};
	\node[opacity=1] at (1.5,4.45) {$_{#4}$};
	\node[opacity=1] at (-1.2,.4) {$_{#5}$};
\end{tikzpicture}
};
\endxy
}

\newcommand{\zipthroughflipcol}[5][1]{
\xy
(0,0)*{
\begin{tikzpicture} [scale=.5,fill opacity=0.2]
	\path[fill=blue] (2,5) to [out=180,in=45] (1,4.5) to (1,.5) to [out=45,in=180] (2,1);
	\path[fill=red] (2.5,4) to [out=180,in=315] (1,4.5) to (1,.5) to [out=315,in=180] (2.5,0);
	\path[fill=green] (1,4.5) to (-.5,4.5) to (-.5,.5) to (1,.5);
	\path[fill=blue] (-1.5,0) to [out=0,in=225] (-.5,.5) to (-.5,4.5) to [out=225,in=0] (-1.5,4);
	\path[fill=red] (-2,1) to [out=0,in=135] (-.5,.5) to (-.5,4.5) to [out=135,in=0] (-2,5);
	\draw[very thick, directed=.55] (2,1) to [out=180,in=45] (1,.5);
	\draw[very thick,directed=.55] (2.5,0) to [out=180,in=315] (1,.5);
	\draw[very thick, directed=.55] (1,.5) to (-.5,.5);
	\draw[very thick, directed=.55] (-.5,.5) to [out=225,in=0] (-1.5,0);
	\draw[very thick, directed=.55] (-.5,.5) to [out=135,in=0] (-2,1);
	\draw[very thick, red, directed=.35] (1,4.5) to (1,.5);
	\draw[very thick, red, directed=.72] (-.5,.5) to (-.5,4.5);
	\draw[very thick] (2,1) to (2,5);
	\draw[very thick] (2.5,0) to (2.5,4);
	\draw[very thick] (-1.5,0) to (-1.5,4);
	\draw[very thick] (-2,1) to (-2,5);	
	\draw[dashed, blue] (2,3) to [out=180,in=45] (1,2.5);
	\draw[red, dashed] (2.5,2) to [out=180,in=315] (1,2.5);
	\draw[green, dashed] (1,2.5) to (-.5,2.5);
	\draw[blue, dashed] (-.5,2.5) to [out=225,in=0] (-1.5,2);
	\draw[dashed, red] (-.5,2.5) to [out=135,in=0] (-2,3);
	\draw[very thick, directed=.55] (2,5) to [out=180,in=45] (1,4.5);
	\draw[very thick,directed=.55] (2.5,4) to [out=180,in=315] (1,4.5);
	\draw[very thick, directed=.55] (1,4.5) to (-.5,4.5);
	\draw[very thick, directed=.55] (-.5,4.5) to [out=225,in=0] (-1.5,4);
	\draw[very thick, directed=.65] (-.5,4.5) to [out=135,in=0] (-2,5);
	\node[blue,opacity=1] at (1.8,4.65) {\tiny $b$};
	\node[red,opacity=1] at (2.25,3.75) {\tiny $a$};
	\node[red,opacity=1] at (-1.75,4.65) {\tiny $a$};
	\node[opacity=1] at (.25,4.15) {\tiny $_{a+b}$};
	\node[blue,opacity=1] at (-1,3.75) {\tiny $_{b}$};
	\node[opacity=1] at (2,.5) {$_{#2}$};
	\node[opacity=1] at (-1.25,4.5) {$_{#3}$};
	\node[opacity=1] at (1.5,4.45) {$_{#4}$};
	\node[opacity=1] at (-1.2,.4) {$_{#5}$};
\end{tikzpicture}
};
\endxy
}

\newcommand{\zipthroughspecA}[7][1]{
\xy
(0,0)*{
\begin{tikzpicture} [scale=.5,fill opacity=0.2]
	\path[fill=blue] (4,5) to [out=180,in=45] (3,4.5) to (3,.5) to [out=45,in=180] (4,1);
	\path[fill=red] (4.5,4) to [out=180,in=315] (3,4.5) to (3,.5) to [out=315,in=180] (4.5,0);
	\path[fill=blue] (-1.5,0) to [out=0,in=225] (-.5,.5) to (-.5,4.5) to [out=225,in=0] (-1.5,4);
	\path[fill=red] (-2,1) to [out=0,in=135] (-.5,.5) to (-.5,4.5) to [out=135,in=0] (-2,5);	
	\path[fill=blue]  (2.25,2.5) to [out=90,in=0] (1.25,4) to [out=180, in=90] (.25,2.5) to [out=270,in=180] (1.25,1) to [out=0,in=270] (2.25,2.5) ;
	\path[fill=red]  (2.25,2.5) to [out=90,in=0] (1.25,4) to [out=180, in=90] (.25,2.5) to [out=270,in=180] (1.25,1) to [out=0,in=270] (2.25,2.5) ;
\path[fill=green] (3,4.5) to (3,2.5) to (2.25,2.5) to [out=90,in=0] (1.25,4) to [out=180, in=90] (0.25,2.5) to [out=270,in=180] (1.25,1) to [out=0,in=270] (2.25,2.5) to (3, 2.5) to (3,0.5) to (-0.5 , 0.5) to (-0.5, 4.5) to (3,4.5);
	\draw[very thick, directed=.55] (4,1) to [out=180,in=45] (3,.5);
	\draw[very thick,directed=.55] (4.5,0) to [out=180,in=315] (3,.5);
	\draw[very thick, directed=.55] (3,.5) to (-.5,.5);
	\draw[very thick, directed=.55] (-.5,.5) to [out=225,in=0] (-1.5,0);
	\draw[very thick, directed=.55] (-.5,.5) to [out=135,in=0] (-2,1);
	\draw[very thick, red, directed=.35] (3,4.5) to (3,.5);
	\draw[very thick, red, directed=.72] (-.5,.5) to (-.5,4.5);
	\draw[very thick, red, directed=.65] (2.25,2.5) to [out=90,in=0] (1.25,4) to [out=180,in=90] (0.25,2.5);
	\draw[very thick, red, directed=.65] (0.25,2.5) to [out=270,in=180] (1.25,1) to [out=0,in=270] (2.25,2.5);
	\draw[very thick] (4,1) to (4,5);
	\draw[very thick] (4.5,0) to (4.5,4);
	\draw[very thick] (-1.5,0) to (-1.5,4);
	\draw[very thick] (-2,1) to (-2,5);	
	\draw[dashed, blue] (4,3) to [out=180,in=45] (3,2.5);
	\draw[red, dashed] (4.5,2) to [out=180,in=315] (3,2.5);
	\draw[green, dashed] (3,2.5) to (2.25,2.5);
	\draw[green, dashed] (0.25,2.5) to (-.5,2.5);
\draw[blue, dashed] (-.5,2.5) to [out=225,in=0] (-1.5,2);
	\draw[dashed, red] (-.5,2.5) to [out=135,in=0] (-2,3);
	\draw[blue, dashed] (2.25,2.5) to [out=135,in=0] (1.25,3) to [out=180, in=45] (0.25,2.5);
	\draw[red , dashed] (2.25,2.5) to [out=225,in=0] (1.25,2) to [out=180,in=315] (0.25,2.5);
	\draw[very thick, directed=.55] (4,5) to [out=180,in=45] (3,4.5);
	\draw[very thick,directed=.55] (4.5,4) to [out=180,in=315] (3,4.5);
	\draw[very thick, directed=.55] (3,4.5) to (-.5,4.5);
	\draw[very thick, directed=.55] (-.5,4.5) to [out=225,in=0] (-1.5,4);
	\draw[very thick, directed=.65] (-.5,4.5) to [out=135,in=0] (-2,5);
	\node[blue,opacity=1] at (3.8,4.65) {\tiny $_{b}$};
	\node[red,opacity=1] at (4.25,3.75) {\tiny $_{a}$};
	\node[red,opacity=1] at (-1.75,4.65) {\tiny $_{a}$};
	\node[opacity=1] at (2.25,4.15) {\tiny $_{a+b}$};
	\node[blue,opacity=1] at (-1,3.75) {\tiny $_{b}$};
	\node[red,opacity=1] at (0.8,1.45) {\tiny $a$};
	\node[blue,opacity=1] at (1.65,3.45) {\tiny $b$};
	\node[opacity=1] at (1.45,1.5) {$_{#3}$};
	\node[opacity=1] at (1.05, 3.3) {$_{#2}$};
 	\node[opacity=1] at (-1.2,.4) {$_{#4}$};
	\node[opacity=1] at (-1.3,4.5) {$_{#5}$};
	\node[opacity=1] at (0.8, 1.85) {$_{#6}$};
	\node[opacity=1] at (1.85, 3) {$_{#7}$};
\end{tikzpicture}
};
\endxy
}

\newcommand{\zipthroughspecB}[7][1]{
\xy
(0,0)*{
\begin{tikzpicture} [scale=.5,fill opacity=0.2]
	\path[fill=blue] (4,5) to [out=180,in=45] (3,4.5) to (3,.5) to [out=45,in=180] (4,1);
	\path[fill=red] (4.5,4) to [out=180,in=315] (3,4.5) to (3,.5) to [out=315,in=180] (4.5,0);
	\path[fill=blue] (-1.5,0) to [out=0,in=225] (-.5,.5) to (-.5,4.5) to [out=225,in=0] (-1.5,4);
	\path[fill=red] (-2,1) to [out=0,in=135] (-.5,.5) to (-.5,4.5) to [out=135,in=0] (-2,5);	
	\path[fill=blue]  (2.25,2.5) to [out=90,in=0] (1.25,4) to [out=180, in=90] (.25,2.5) to [out=270,in=180] (1.25,1) to [out=0,in=270] (2.25,2.5) ;
	\path[fill=red]  (2.25,2.5) to [out=90,in=0] (1.25,4) to [out=180, in=90] (.25,2.5) to [out=270,in=180] (1.25,1) to [out=0,in=270] (2.25,2.5) ;
\path[fill=green] (3,4.5) to (3,2.5) to (2.25,2.5) to [out=90,in=0] (1.25,4) to [out=180, in=90] (0.25,2.5) to [out=270,in=180] (1.25,1) to [out=0,in=270] (2.25,2.5) to (3, 2.5) to (3,0.5) to (-0.5 , 0.5) to (-0.5, 4.5) to (3,4.5);
	\draw[very thick, directed=.55] (4,1) to [out=180,in=45] (3,.5);
	\draw[very thick,directed=.55] (4.5,0) to [out=180,in=315] (3,.5);
	\draw[very thick, directed=.55] (3,.5) to (-.5,.5);
	\draw[very thick, directed=.55] (-.5,.5) to [out=225,in=0] (-1.5,0);
	\draw[very thick, directed=.55] (-.5,.5) to [out=135,in=0] (-2,1);
	\draw[very thick, red, directed=.35] (3,4.5) to (3,.5);
	\draw[very thick, red, directed=.72] (-.5,.5) to (-.5,4.5);
	\draw[very thick, red, directed=.65] (2.25,2.5) to [out=90,in=0] (1.25,4) to [out=180,in=90] (0.25,2.5);
	\draw[very thick, red, directed=.65] (0.25,2.5) to [out=270,in=180] (1.25,1) to [out=0,in=270] (2.25,2.5);
	\draw[very thick] (4,1) to (4,5);
	\draw[very thick] (4.5,0) to (4.5,4);
	\draw[very thick] (-1.5,0) to (-1.5,4);
	\draw[very thick] (-2,1) to (-2,5);	
	\draw[dashed, blue] (4,3) to [out=180,in=45] (3,2.5);
	\draw[red, dashed] (4.5,2) to [out=180,in=315] (3,2.5);
	\draw[green, dashed] (3,2.5) to (2.25,2.5);
	\draw[green, dashed] (0.25,2.5) to (-.5,2.5);
\draw[blue, dashed] (-.5,2.5) to [out=225,in=0] (-1.5,2);
	\draw[dashed, red] (-.5,2.5) to [out=135,in=0] (-2,3);
	\draw[red, dashed] (2.25,2.5) to [out=135,in=0] (1.25,3) to [out=180, in=45] (0.25,2.5);
	\draw[blue , dashed] (2.25,2.5) to [out=225,in=0] (1.25,2) to [out=180,in=315] (0.25,2.5);
	\draw[very thick, directed=.55] (4,5) to [out=180,in=45] (3,4.5);
	\draw[very thick,directed=.55] (4.5,4) to [out=180,in=315] (3,4.5);
	\draw[very thick, directed=.55] (3,4.5) to (-.5,4.5);
	\draw[very thick, directed=.55] (-.5,4.5) to [out=225,in=0] (-1.5,4);
	\draw[very thick, directed=.65] (-.5,4.5) to [out=135,in=0] (-2,5);
	\node[blue,opacity=1] at (3.8,4.65) {\tiny $_{b}$};
	\node[red,opacity=1] at (4.25,3.75) {\tiny $_{a}$};
	\node[red,opacity=1] at (-1.75,4.65) {\tiny $_{a}$};
	\node[opacity=1] at (2.25,4.15) {\tiny $_{a+b}$};
	\node[blue,opacity=1] at (-1,3.75) {\tiny $_{b}$};
	\node[blue,opacity=1] at (0.8,1.45) {\tiny $b$};
	\node[red,opacity=1] at (1.65,3.45) {\tiny $a$};
	\node[opacity=1] at (1.45,1.5) {$_{#2}$};
	\node[opacity=1] at (1.05, 3.3) {$_{#3}$};
 	\node[opacity=1] at (-1.2,.4) {$_{#4}$};
	\node[opacity=1] at (-1.3,4.5) {$_{#5}$};
	\node[opacity=1] at (0.8, 1.85) {$_{#6}$};
	\node[opacity=1] at (1.85, 3) {$_{#7}$};
\end{tikzpicture}
};
\endxy
}

\newcommand{\splitterbubble}[4][1]{
\xy
(0,0)*{
\begin{tikzpicture} [scale=.5,fill opacity=0.2]
\path[fill=blue]  (-3.5,1) to [out=0,in=135] (-2,0.5) to (-2,4.5) to [out=135,in=0] (-3.5,5) to (-3.5,1) ;
\path[fill=red]  (-3,0) to [out=0,in=225] (-2,0.5) to (-2,4.5) to [out=225,in=0] (-3,4) to (-3,0) ;
\path[fill=blue]  (1,2.5) to [out=90,in=0] (0,4) to [out=180, in=90] (-1,2.5) to [out=270,in=180] (0,1) to [out=0,in=270] (1,2.5) ;
\path[fill=red]  (1,2.5) to [out=90,in=0] (0,4) to [out=180, in=90] (-1,2.5) to [out=270,in=180] (0,1) to [out=0,in=270] (1,2.5) ;
\path[fill=green] (2,4.5) to (2,2.5) to (1,2.5) to [out=90,in=0] (0,4) to [out=180, in=90] (-1,2.5) to [out=270,in=180] (0,1) to [out=0,in=270] (1,2.5) to (2, 2.5) to (2,0.5) to (-2 , 0.5) to (-2, 4.5) to (2,4.5);
	\draw[very thick, directed=.55] (2,0.5) to (-2,0.5);
	\draw[very thick, directed=.55] (-2,0.5) to [out=135,in=0] (-3.5,1);
	\draw[very thick, directed=.55] (-2,0.5) to [out=225,in=0] (-3,0);
	\draw[very thick, red, directed=.65] (1,2.5) to [out=90,in=0] (0,4) to [out=180,in=90] (-1,2.5);
	\draw[very thick, red, directed=.65] (-1,2.5) to [out=270,in=180] (0,1) to [out=0,in=270] (1,2.5);
	\draw[very thick, red, directed=.72] (-2,0.5) to (-2,4.5);
	\draw[very thick] (2,0.5) to (2,4.5);
	\draw[very thick] (-3,0) to (-3,4);
	\draw[very thick] (-3.5,1) to (-3.5,5);
\draw[green, dashed] (2,2.5) to (1,2.5);
	\draw[green, dashed] (-1,2.5) to  (-2,2.5);
	\draw[blue, dashed] (1,2.5) to [out=135,in=0] (0,3) to [out=180, in=45] (-1,2.5);
	\draw[blue, dashed] (-2,2.5) to [out=135,in=0] (-3.5,3);
	\draw[red, dashed] (-2,2.5) to [out=225,in=0] (-3,2);
	\draw[red , dashed] (1,2.5) to [out=225,in=0] (0,2) to [out=180,in=315] (-1,2.5);
	\draw[very thick, directed=.55] (2,4.5) to (-2,4.5);
	\draw[very thick, directed=.55] (-2,4.5) to [out=135,in=0] (-3.5,5);
	\draw[very thick, directed=.55] (-2,4.5) to [out=225,in=0] (-3,4);
	\node[opacity=1] at (1.5,4.1) {\tiny $_{a+b}$};
	\node[opacity=1] at (1.5,4.1) {\tiny $_{a+b}$};
		\node[blue,opacity=1] at (-3.25,4.65) {\tiny $b$};
		\node[red,opacity=1] at (-2.5,3.75) {\tiny $_{a}$};
		\node[red,opacity=1] at (-0.45,1.45) {\tiny $a$};
		\node[blue,opacity=1] at (0.45,3.45) {\tiny $b$};
			\node[opacity=1] at (0.2,1.5) {$_{#2}$};
	\node[opacity=1] at (-0.2,3.3) {$_{#3}$};
	\node[opacity=1] at (-2.55,.45) {$_{#4}$};
\end{tikzpicture}
};
\endxy
}

\newcommand{\fatsplitter}[1][1]{
\xy
(0,0)*{
\begin{tikzpicture} [scale=.5,fill opacity=0.2]
\path[fill=blue]  (-3.5,1) to [out=0,in=135] (-2,0.5) to [out=90, in=180] (0,1) to [out=0,in=270] (1,2.5) to [out=90,in=0] (0,4) to [out=180, in=270] (-2,4.5) to [out=135,in=0] (-3.5,5) to (-3.5,1) ;
\path[fill=red]  (-3,0) to [out=0,in=225] (-2,0.5) to [out=90, in=180] (0,1) to [out=0,in=270] (1,2.5) to [out=90,in=0] (0,4) to [out=180, in=270] (-2,4.5) to [out=225,in=0] (-3,4) to (-3,0) ;
\path[fill=green] (2,4.5) to (2,.5) to (-2,0.5) to [out=90, in=180] (0,1) to [out=0,in=270] (1,2.5) to [out=90,in=0] (0,4) to [out=180, in=270] (-2,4.5) to (2,4.5);
	\draw[very thick, directed=.55] (2,0.5) to (-2,0.5);
	\draw[very thick, directed=.55] (-2,0.5) to [out=135,in=0] (-3.5,1);
	\draw[very thick, directed=.55] (-2,0.5) to [out=225,in=0] (-3,0);
\draw[very thick, red, directed=.65] (-2,0.5) to [out=90, in=180] (0,1) to [out=0,in=270] (1,2.5) to [out=90,in=0] (0,4) to [out=180, in=270] (-2,4.5);
	\draw[very thick] (2,0.5) to (2,4.5);
	\draw[very thick] (-3,0) to (-3,4);
	\draw[very thick] (-3.5,1) to (-3.5,5);
	\draw[green, dashed] (2,2.5) to (1,2.5);
	\draw[blue, dashed] (1,2.5) to [out=135,in=0] (0,3) to (-3.5,3);
	\draw[red , dashed] (1,2.5) to [out=225,in=0] (0,2) to (-3,2);
	\draw[very thick, directed=.55] (2,4.5) to (-2,4.5);
	\draw[very thick, directed=.55] (-2,4.5) to [out=135,in=0] (-3.5,5);
	\draw[very thick, directed=.55] (-2,4.5) to [out=225,in=0] (-3,4);
	\node[opacity=1] at (1.5,4.1) {\tiny $_{a+b}$};
		\node[blue,opacity=1] at (-3.25,4.65) {\tiny $_{b}$};
		\node[red,opacity=1] at (-2.5,3.75) {\tiny $_{a}$};
\end{tikzpicture}
};
\endxy
}

\newcommand{\splitter}[1][1]{
\xy
(0,0)*{
\begin{tikzpicture} [scale=.5,fill opacity=0.2]
\path[fill=blue]  (-3.5,1) to [out=0,in=135] (-2,0.5) to (-2,4.5) to [out=135,in=0] (-3.5,5) to (-3.5,1) ;
\path[fill=red]  (-3,0) to [out=0,in=225] (-2,0.5) to (-2,4.5) to [out=225,in=0] (-3,4) to (-3,0) ;
\path[fill=green] (2,4.5) to (2,.5) to (-2,0.5) to (-2,4.5) to (2,4.5);
	\draw[very thick, directed=.55] (2,0.5) to (-2,0.5);
	\draw[very thick, directed=.55] (-2,0.5) to [out=135,in=0] (-3.5,1);
	\draw[very thick, directed=.55] (-2,0.5) to [out=225,in=0] (-3,0);
\draw[very thick, red, directed=.72] (-2,0.5) to (-2,4.5);
	\draw[very thick] (2,0.5) to (2,4.5);
	\draw[very thick] (-3,0) to (-3,4);
	\draw[very thick] (-3.5,1) to (-3.5,5);
	\draw[green, dashed] (2,2.5) to (-2,2.5);
	\draw[blue, dashed] (-2,2.5) to [out=135,in=0] (-3.5,3);
	\draw[red , dashed] (-2,2.5) to [out=225,in=0] (-3,2);
	\draw[very thick, directed=.55] (2,4.5) to (-2,4.5);
	\draw[very thick, directed=.55] (-2,4.5) to [out=135,in=0] (-3.5,5);
	\draw[very thick, directed=.55] (-2,4.5) to [out=225,in=0] (-3,4);
	\node[opacity=1] at (1.5,4.1) {\tiny $_{a+b}$};
		\node[blue,opacity=1] at (-3.25,4.65) {\tiny $_{b}$};
		\node[red,opacity=1] at (-2.5,3.75) {\tiny $_{a}$};
\end{tikzpicture}
};
\endxy
}

\newcommand{\parallelsheets}[3][1]{
\xy
(0,0)*{
\begin{tikzpicture} [scale=.5,fill opacity=0.2]
\path[fill=blue]  (2,5) to [out=90,in=270] (2,1) to (-2,1) to (-2,5) to (2,5);
\path[fill=red]  (2.5,4) to [out=90,in=270] (2.5,0)  to (-1.5,0) to (-1.5,4) to (2.5,4);
	\draw[very thick, directed=.65] (2,1) to [out=180,in=0] (-2,1);
	\draw[very thick, directed=.55] (2.5,0) to [out=180,in=0] (-1.5,0);
	\draw[very thick] (2,1) to (2,5);
	\draw[very thick] (2.5,0) to (2.5,4);
	\draw[very thick] (-1.5,0) to (-1.5,4);
	\draw[very thick] (-2,1) to (-2,5);	
	\draw[dashed, blue] (2,3) to (-2,3);
	\draw[red, dashed] (2.5,2) to (-1.5,2);
	\draw[very thick, directed=.65] (2,5) to [out=180,in=0] (-2,5);
	\draw[very thick, directed=.55] (2.5,4) to [out=180,in=0] (-1.5,4);
	\node[blue,opacity=1] at (1.75,4.75) {\tiny $b$};
	\node[red,opacity=1] at (2.25,3.75) {\tiny $a$};
	\node[opacity=1] at (1.75,.5) {$_{#2}$};
	\node[opacity=1] at (-1,4.5) {$_{#3}$};
\end{tikzpicture}
};
\endxy
}

\newcommand{\zipunzip}[3][1]{
\xy
(0,0)*{
\begin{tikzpicture} [scale=.5,fill opacity=0.2]
\path[fill=blue]  (2,5) to [out=90,in=270] (2,1) to [out=180,in=0] (0.25,1) to  (0.25,1.75) to [out=0, in=270] (1,2.5) to [out=90, in=0] (0.25,3.25) to [out=180, in=90] (-0.5,2.5) to [out=270,in=180] (0.25,1.75) to (0.25,1)  to (-2,1) to (-2,5) to (2,5);
\path[fill=red]  (2.5,4) to [out=90,in=270] (2.5,0) to [out=180,in=0] (0.25,0) to  (0.25,1.75) to [out=0, in=270] (1,2.5) to [out=90, in=0] (0.25,3.25) to [out=180, in=90] (-0.5,2.5) to [out=270,in=180] (0.25,1.75) to (0.25,0)  to (-1.5,0) to (-1.5,4) to (2.5,4);
\path[fill=green] (0.25,1.75) to [out=0, in=270] (1,2.5) to [out=90, in=0] (0.25,3.25) to [out=180, in=90] (-0.5,2.5) to [out=270,in=180] (0.25,1.75);
	\draw[very thick, directed=.65] (2,1) to [out=180,in=0] (-2,1);
	\draw[very thick, directed=.55] (2.5,0) to [out=180,in=0] (-1.5,0);
	\draw[very thick] (2,1) to (2,5);
	\draw[very thick] (2.5,0) to (2.5,4);
	\draw[very thick] (-1.5,0) to (-1.5,4);
	\draw[very thick] (-2,1) to (-2,5);	
	\draw[dashed, blue] (2,3) to [out=180,in=45] (1,2.5);
	\draw[red, dashed] (2.5,2) to [out=180,in=315] (1,2.5);
	\draw[green, dashed] (1,2.5) to (-.5,2.5);
	\draw[red, dashed] (-.5,2.5) to [out=225,in=0] (-1.5,2);
	\draw[dashed, blue] (-.5,2.5) to [out=135,in=0] (-2,3);
	\draw[very thick, red, directed=.65] (1,2.5) to [out=270,in=0]  (0.25,1.75) to [out=180, in = 270] (-0.5,2.5) to [out=90, in = 180] (0.25,3.25) to [out=0, in = 90] (1,2.5);
	\draw[very thick, directed=.65] (2,5) to [out=180,in=0] (-2,5);
	\draw[very thick, directed=.55] (2.5,4) to [out=180,in=0] (-1.5,4);
	\node[blue,opacity=1] at (1.75,4.75) {\tiny $b$};
	\node[red,opacity=1] at (2.25,3.75) {\tiny $a$};
	\node[opacity=1] at (.25,2.75) {\tiny $_{a+b}$};
	\node[opacity=1] at (1.75,.5) {$_{#2}$};
	\node[opacity=1] at (-1,4.5) {$_{#3}$};
\end{tikzpicture}
};
\endxy
}

\newcommand{\zip}[3][1]{
\xy
(0,0)*{
\begin{tikzpicture} [scale=.5,fill opacity=0.2]
\path[fill=blue]  (2,3) to [out=180,in=45] (1,2.5) to [out=270,in=0] (0.25,1.75) to [out=180,in=270] (-0.5,2.5) to [out=135,in=0] (-2,3) to (-2,1) to (2,1) to (2,3);
\path[fill=red]  (2.5,2) to [out=180, in=315] (1,2.5)to [out=270,in=0] (0.25,1.75) to [out=180,in=270] (-0.5,2.5) to [out=225, in=0] (-1.5,2) to (-1.5,0) to (2.5,0) to (2.5,2);
\path[fill=green] (1,2.5) to [out=270, in=0] (0.25,1.75) to [out=180, in=270] (-0.5,2.5);
	\draw[very thick, directed=.55] (2,1) to [out=180,in=0] (-2,1);
	\draw[very thick, directed=.55] (2.5,0) to [out=180,in=0] (-1.5,0);
	\draw[very thick] (2,1) to (2,3);
	\draw[very thick] (2.5,0) to (2.5,2);
	\draw[very thick] (-1.5,0) to (-1.5,2);
	\draw[very thick] (-2,1) to (-2,3);	
	\draw[very thick, directed=.55] (2,3) to [out=180,in=45] (1,2.5);
	\draw[very thick, directed=.55] (2.5,2) to [out=180,in=315] (1,2.5);
	\draw[very thick, directed=.55] (1,2.5) to (-.5,2.5);
	\draw[very thick, directed=.55] (-.5,2.5) to [out=225,in=0] (-1.5,2);
	\draw[very thick, directed=.55] (-.5,2.5) to [out=135,in=0] (-2,3);
	\draw[very thick, red, directed=.65] (1,2.5) to [out=270,in=0]  (0.25,1.75) to [out=180, in = 270] (-0.5,2.5);
	\node[blue,opacity=1] at (1.75,2.55) {\tiny $b$};
	\node[red,opacity=1] at (2.25,1.75) {\tiny $a$};
	\node[opacity=1] at (.25,2.2) {\tiny $_{a+b}$};
		\node[opacity=1] at (2,.5) {$_{#2}$};
	\node[opacity=1] at (-1.5,2.5) {$_{#3}$};
\end{tikzpicture}
};
\endxy
}

\newcommand{\unzip}[3][1]{
\xy
(0,0)*{
\begin{tikzpicture} [scale=.5,fill opacity=0.2]
	\path[fill=red] (2.5,0) to [out=180,in=315] (1,0.5) to [out=90,in=0]  (0.25,1.25) to [out=180, in=90] (-0.5,0.5) to [out=225, in=0] (-1.5,0) to (-1.5,2) to (2.5,2) to (2.5,0);
	\path[fill=blue] (2,1) to [out=180,in=45] (1,0.5) to [out=90,in=0]  (0.25,1.25) to [out=180, in=90] (-0.5,0.5) to [out=135, in=0] (-2,1) to (-2,3) to (2,3) to (2,1);
	\path[fill=green]  (1,0.5) to [out=90,in=0]  (0.25,1.25) to [out=180, in=90] (-0.5,0.5) to (1,0.5);
	\draw[very thick, directed=.55] (2,1) to [out=180,in=45] (1,.5);
	\draw[very thick,directed=.55] (2.5,0) to [out=180,in=315] (1,.5);
	\draw[very thick, directed=.55] (1,.5) to (-.5,.5);
	\draw[very thick, directed=.55] (-.5,.5) to [out=225,in=0] (-1.5,0);
	\draw[very thick, directed=.55] (-.5,.5) to [out=135,in=0] (-2,1);
	\draw[very thick, red, directed=.65] (-.5,.5) to [out=90,in=180] (.25,1.25) to [out=0, in=90] (1,.5);
	\draw[very thick] (2,1) to (2,3);
	\draw[very thick] (2.5,0) to (2.5,2);
	\draw[very thick] (-1.5,0) to (-1.5,2);
	\draw[very thick] (-2,1) to (-2,3);
	\draw[very thick, directed=.55] (2,3) to (-2,3);
	\draw[very thick, directed=.55] (2.5,2) to  (-1.5,2);
	\node[opacity=1] at (.25,.75) {\tiny $_{a+b}$};
	\node[blue,opacity=1] at (1.75,2.65) {\tiny $b$};
	\node[red,opacity=1] at (2.25,1.75) {\tiny $a$};
			\node[opacity=1] at (2,.5) {$_{#2}$};
	\node[opacity=1] at (-1.5,2.5) {$_{#3}$};
\end{tikzpicture}
};
\endxy
}

\newcommand{\tube}[3][1]{
\xy
(0,0)*{
\begin{tikzpicture} [scale=.5,fill opacity=0.2]
\path[fill=blue]  (1,4.5) to [out=135,in=0] (0,5) to [out=180, in=45] (-1,4.5) to (-1,0.5) to [out=45,in=180] (0,1) to [out=0, in=135] (1,0.5) to (1,4.5) ;
\path[fill=red]  (1,4.5) to [out=225,in=0] (0,4) to [out=180, in=315] (-1,4.5) to (-1,0.5) to [out=315,in=180] (0,0) to [out=0, in=225] (1,0.5) to (1,4.5) ;
	\path[fill=green]  (1,0.5) to (1,4.5) to (2,4.5) to (2, 0.5) to (1,0.5);
	\path[fill=green] (-1,0.5) to (-1,4.5) to (-2,4.5) to (-2, 0.5) to (-1,0.5);
	\draw[very thick, directed=.55] (2,0.5) to (1,0.5);
	\draw[very thick,directed=.55] (-1,0.5) to  (-2,0.5);
	\draw[very thick, directed=.55] (1,0.5) to [out=135,in=0] (0,1) to [out=180, in=45] (-1,0.5);
	\draw[very thick, directed=.55] (1,0.5) to [out=225,in=0] (0,0) to [out=180,in=315] (-1,0.5);
	\draw[very thick, red, directed=.65] (1,.5) to (1,4.5);
	\draw[very thick, red, directed=.65] (-1,4.5) to (-1,0.5);
	\draw[very thick] (2,0.5) to (2,4.5);
	\draw[very thick] (-2,0.5) to (-2,4.5);
		\draw[green, dashed] (2,2.5) to (1,2.5);
	\draw[green, dashed] (-1,2.5) to  (-2,2.5);
	\draw[blue, dashed] (1,2.5) to [out=135,in=0] (0,3) to [out=180, in=45] (-1,2.5);
	\draw[red , dashed] (1,2.5) to [out=225,in=0] (0,2) to [out=180,in=315] (-1,2.5);
	\draw[very thick, directed=.55] (2,4.5) to (1,4.5);
	\draw[very thick,directed=.55] (-1,4.5) to  (-2,4.5);
	\draw[very thick, directed=.55] (1,4.5) to [out=135,in=0] (0,5) to [out=180, in=45] (-1,4.5);
	\draw[very thick, directed=.55] (1,4.5) to [out=225,in=0] (0,4) to [out=180,in=315] (-1,4.5);
	\node[blue,opacity=1] at (0.45,4.65) {\tiny $b$};
	\node[opacity=1] at (1.5,4.1) {\tiny $_{a+b}$};
	\node[opacity=1] at (-1.5,4.1) {\tiny $_{a+b}$};
	\node[red,opacity=1] at (-0.45,0.35) {\tiny $a$};
	\node[opacity=1] at (0.5,.5) {$_{#2}$};
	\node[opacity=1] at (-0.3,4.5) {$_{#3}$};
\end{tikzpicture}
};
\endxy
}
\newcommand{\blistercloseopen}[4][1]{
\xy
(0,0)*{
\begin{tikzpicture} [scale=.5,fill opacity=0.2]
\path[fill=blue]  (1,4.5) to [out=135,in=0] (0,5) to [out=180, in=45] (-1,4.5) to [out=270,in=180] (0,3) to [out=0,in=270] (1,4.5) ;
\path[fill=blue]  (-1,.5) to [out=90,in=180] (0,2) to [out=0, in=90] (1,.5) to [out=135,in=0] (0,1) to [out=180,in=45] (-1,.5) ;
\path[fill=red]  (1,4.5) to [out=225,in=0] (0,4) to [out=180, in=315] (-1,4.5) to [out=270,in=180] (0,3) to [out=0,in=270] (1,4.5);
\path[fill=red]  (1,0.5) to [out=225,in=0] (0,0) to [out=180, in=315] (-1,.5) to [out=90,in=180] (0,2) to [out=0,in=90] (1,.5);
	\path[fill=green]  (2,0.5) to (2,4.5) to (1,4.5) to [out=270,in=0] (0, 3) to [out=180,in=270] (-1,4.5) to (-2, 4.5) to (-2,0.5) to (-1,0.5) to [out=90,in=180] (0,2) to [out=0,in=90] (1,0.5) to (2,0.5);
	\draw[very thick, directed=.55] (2,0.5) to (1,0.5);
	\draw[very thick,directed=.55] (-1,0.5) to  (-2,0.5);
	\draw[very thick, directed=.55] (1,0.5) to [out=135,in=0] (0,1) to [out=180, in=45] (-1,0.5);
	\draw[very thick, directed=.55] (1,0.5) to [out=225,in=0] (0,0) to [out=180,in=315] (-1,0.5);
	\draw[very thick, red, directed=.65] (1,.5) to [out=90,in=0] (0,2) to [out=180,in=90] (-1,.5);
	\draw[very thick, red, directed=.65] (-1,4.5) to [out=270,in=180] (0,3) to [out=0,in=270] (1,4.5);
	\draw[very thick] (2,0.5) to (2,4.5);
	\draw[very thick] (-2,0.5) to (-2,4.5);
		\draw[green, dashed] (2,2.5) to (-2,2.5);
	\draw[very thick, directed=.55] (2,4.5) to (1,4.5);
	\draw[very thick,directed=.55] (-1,4.5) to  (-2,4.5);
	\draw[very thick, directed=.55] (1,4.5) to [out=135,in=0] (0,5) to [out=180, in=45] (-1,4.5);
	\draw[very thick, directed=.55] (1,4.5) to [out=225,in=0] (0,4) to [out=180,in=315] (-1,4.5);
	\node[blue,opacity=1] at (0.45,4.65) {\tiny $b$};
	\node[opacity=1] at (1.5,4.1) {\tiny $_{a+b}$};
	\node[red,opacity=1] at (-0.45,0.35) {\tiny $a$};
	\node[opacity=1] at (0.5,.5) {$_{#2}$};
	\node[opacity=1] at (-0.3,4.5) {$_{#3}$};
	\node[opacity=1] at (0.3,3.5) {$_{#4}$};
\end{tikzpicture}
};
\endxy
}
\newcommand{\flatsheet}[1][1]{
\xy
(0,0)*{
\begin{tikzpicture} [scale=.5,fill opacity=0.2]
	\path[fill=green]  (-2,0.5) to (-2,4.5) to (2,4.5) to (2, 0.5) to (-2,0.5);
	\draw[very thick, directed=.55] (2,0.5) to (-2,0.5);
	\draw[very thick] (2,0.5) to (2,4.5);
	\draw[very thick] (-2,0.5) to (-2,4.5);
	\draw[green, dashed] (-2,2.5) to  (2,2.5);
	\draw[very thick, directed=.55] (2,4.5) to (-2,4.5);
	\node[opacity=1] at (1.5,4.2) {\tiny $_{a+b}$};
\end{tikzpicture}
};
\endxy 
}
\newcommand{\greenbubble}[3][1]{
\xy
(0,0)*{
\begin{tikzpicture} [scale=.5,fill opacity=0.2]
\path[fill=blue]  (1,2.5) to [out=90,in=0] (0,4) to [out=180, in=90] (-1,2.5) to [out=270,in=180] (0,1) to [out=0,in=270] (1,2.5) ;
\path[fill=red]  (1,2.5) to [out=90,in=0] (0,4) to [out=180, in=90] (-1,2.5) to [out=270,in=180] (0,1) to [out=0,in=270] (1,2.5) ;
\path[fill=green] (2,4.5) to (2,2.5) to (1,2.5) to [out=90,in=0] (0,4) to [out=180, in=90] (-1,2.5) to [out=270,in=180] (0,1) to [out=0,in=270] (1,2.5) to (2, 2.5) to (2,0.5) to (-2 , 0.5) to (-2, 4.5) to (2,4.5);
	\draw[very thick, directed=.55] (2,0.5) to (-2,0.5);
	\draw[very thick, red, directed=.65] (1,2.5) to [out=90,in=0] (0,4) to [out=180,in=90] (-1,2.5);
	\draw[very thick, red, directed=.65] (-1,2.5) to [out=270,in=180] (0,1) to [out=0,in=270] (1,2.5);
	\draw[very thick] (2,0.5) to (2,4.5);
	\draw[very thick] (-2,0.5) to (-2,4.5);
\draw[green, dashed] (2,2.5) to (1,2.5);
	\draw[green, dashed] (-1,2.5) to  (-2,2.5);
	\draw[blue, dashed] (1,2.5) to [out=135,in=0] (0,3) to [out=180, in=45] (-1,2.5);
	\draw[red , dashed] (1,2.5) to [out=225,in=0] (0,2) to [out=180,in=315] (-1,2.5);
	\draw[very thick, directed=.55] (2,4.5) to (-2,4.5);
	\node[opacity=1] at (1.5,4.1) {\tiny $_{a+b}$};
	\node[blue,opacity=1] at (0.45,3.45) {\tiny $b$};
	\node[opacity=1] at (1.5,4.1) {\tiny $_{a+b}$};
	\node[red,opacity=1] at (-0.45,1.45) {\tiny $a$};
	\node[opacity=1] at (0.2,1.5) {$_{\pi_{\alpha}}$};
	\node[opacity=1] at (-0.2,3.3) {$_{\pi_{\hat{\alpha}}}$};
\end{tikzpicture}
};
\endxy
}

\newcommand{\blisterclose}[3][1]{
\xy
(0,0)*{
\begin{tikzpicture} [scale=.5,fill opacity=0.2]
\path[fill=blue]  (-1,.5) to [out=90,in=180] (0,2) to [out=0, in=90] (1,.5) to [out=135,in=0] (0,1) to [out=180,in=45] (-1,.5) ;
\path[fill=red]  (1,0.5) to [out=225,in=0] (0,0) to [out=180, in=315] (-1,.5) to [out=90,in=180] (0,2) to [out=0,in=90] (1,.5);
	\path[fill=green]  (2,0.5) to (2,2.5) to (-2, 2.5) to (-2,0.5) to (-1,0.5) to [out=90,in=180] (0,2) to [out=0,in=90] (1,0.5) to (2,0.5);
	\draw[very thick, directed=.55] (2,0.5) to (1,0.5);
	\draw[very thick,directed=.55] (-1,0.5) to  (-2,0.5);
	\draw[very thick, directed=.55] (1,0.5) to [out=135,in=0] (0,1) to [out=180, in=45] (-1,0.5);
	\draw[very thick, directed=.55] (1,0.5) to [out=225,in=0] (0,0) to [out=180,in=315] (-1,0.5);
	\draw[very thick, red, directed=.65] (1,.5) to [out=90,in=0] (0,2) to [out=180,in=90] (-1,.5);
	\draw[very thick] (2,0.5) to (2,2.5);
	\draw[very thick] (-2,0.5) to (-2,2.5);
		\draw[very thick, directed=.55] (2,2.5) to (-2,2.5);
	\node[blue,opacity=1] at (0.45,1.45) {\tiny $b$};
	\node[opacity=1] at (1.5,2.2) {\tiny $_{a+b}$};
	\node[red,opacity=1] at (-0.45,0.35) {\tiny $a$};
	\node[opacity=1] at (0.3,.5) {$_{#2}$};
	\node[opacity=1] at (-0.2,1.4) {$_{#3}$};
\end{tikzpicture}
};
\endxy
}

\newcommand{\blisteropen}[3][1]{
\xy
(0,0)*{
\begin{tikzpicture} [scale=.5,fill opacity=0.2]
\path[fill=blue]  (1,2.5) to [out=135,in=0] (0,3) to [out=180, in=45] (-1,2.5) to [out=270,in=180] (0,1) to [out=0,in=270] (1,2.5) ;
\path[fill=red]  (1,2.5) to [out=225,in=0] (0,2) to [out=180, in=315] (-1,2.5) to [out=270,in=180] (0,1) to [out=0,in=270] (1,2.5) ;
\path[fill=green]   (-1,2.5) to [out=270,in=180] (0,1) to [out=0,in=270] (1,2.5) to (2, 2.5) to (2,0.5) to (-2 , 0.5) to (-2, 2.5) to (-1,2.5);
	\draw[very thick, directed=.55] (2,0.5) to (-2,0.5);
	\draw[very thick, red, directed=.65] (-1,2.5) to [out=270,in=180] (0,1) to [out=0,in=270] (1,2.5);
	\draw[very thick] (2,0.5) to (2,2.5);
	\draw[very thick] (-2,0.5) to (-2,2.5);
\draw[very thick, directed=.65] (2,2.5) to (1,2.5);
	\draw[very thick, directed=.65] (-1,2.5) to  (-2,2.5);
	\draw[very thick, directed=.65] (1,2.5) to [out=135,in=0] (0,3) to [out=180, in=45] (-1,2.5);
	\draw[very thick, directed=.65] (1,2.5) to [out=225,in=0] (0,2) to [out=180,in=315] (-1,2.5);
	\node[blue,opacity=1] at (0.45,2.45) {\tiny $b$};
	\node[opacity=1] at (1.5,2.1) {\tiny $_{a+b}$};
	\node[red,opacity=1] at (-0.45,1.45) {\tiny $a$};
		\node[opacity=1] at (0.3,1.5) {$_{#2}$};
	\node[opacity=1] at (-0.3,2.5) {$_{#3}$};
\end{tikzpicture}
};
\endxy
}

\newcommand{\leftsplit}[1][1]{
\xy
(0,0)*{
\begin{tikzpicture} [scale=.5,fill opacity=0.2]
	\path[fill=green] (1,4.5) to (-.5,4.5) to (-.5,.5) to (1,.5);
	\path[fill=red] (-1.5,0) to [out=0,in=225] (-.5,.5) to (-.5,4.5) to [out=225,in=0] (-1.5,4);
	\path[fill=blue] (-2,1) to [out=0,in=135] (-.5,.5) to (-.5,4.5) to [out=135,in=0] (-2,5);
	\draw[very thick, directed=.55] (1,.5) to (-.5,.5);
	\draw[very thick, directed=.55] (-.5,.5) to [out=225,in=0] (-1.5,0);
	\draw[very thick, directed=.55] (-.5,.5) to [out=135,in=0] (-2,1);
	\draw[very thick, red, directed=.55] (-.5,.5) to (-.5,4.5);
	\draw[very thick] (-1.5,0) to (-1.5,4);
	\draw[very thick] (1,0.5) to (1,4.5);
	\draw[very thick] (-2,1) to (-2,5);	
	\draw[green, dashed] (1,2.5) to (-.5,2.5);
	\draw[red, dashed] (-.5,2.5) to [out=225,in=0] (-1.5,2);
	\draw[dashed, blue] (-.5,2.5) to [out=135,in=0] (-2,3);
	\draw[very thick, directed=.55] (1,4.5) to (-.5,4.5);
	\draw[very thick, directed=.55] (-.5,4.5) to [out=225,in=0] (-1.5,4);
	\draw[very thick, directed=.65] (-.5,4.5) to [out=135,in=0] (-2,5);
	\node[blue,opacity=1] at (-1.75,4.65) {\tiny $b$};
	\node[opacity=1] at (.25,4.15) {\tiny $_{a+b}$};
	\node[red,opacity=1] at (-1,3.75) {\tiny $_{a}$};
\end{tikzpicture}
};
\endxy
}

\newcommand{\rightsplit}[1][1]{
\xy
(0,0)*{
\begin{tikzpicture} [scale=.5,fill opacity=0.2]
	\path[fill=blue] (2,5) to [out=180,in=45] (1,4.5) to (1,.5) to [out=45,in=180] (2,1);
	\path[fill=red] (2.5,4) to [out=180,in=315] (1,4.5) to (1,.5) to [out=315,in=180] (2.5,0);
	\path[fill=green] (1,4.5) to (-.5,4.5) to (-.5,.5) to (1,.5);
	\draw[very thick, directed=.55] (2,1) to [out=180,in=45] (1,.5);
	\draw[very thick,directed=.55] (2.5,0) to [out=180,in=315] (1,.5);
	\draw[very thick, directed=.55] (1,.5) to (-.5,.5);
	\draw[very thick, red, directed=.55] (1,4.5) to (1,.5);
	\draw[very thick] (-.5,.5) to (-.5,4.5);
	\draw[very thick] (2,1) to (2,5);
	\draw[very thick] (2.5,0) to (2.5,4);
	\draw[dashed, blue] (2,3) to [out=180,in=45] (1,2.5);
	\draw[red, dashed] (2.5,2) to [out=180,in=315] (1,2.5);
	\draw[green, dashed] (1,2.5) to (-.5,2.5);
	\draw[very thick, directed=.55] (2,5) to [out=180,in=45] (1,4.5);
	\draw[very thick,directed=.55] (2.5,4) to [out=180,in=315] (1,4.5);
	\draw[very thick, directed=.55] (1,4.5) to (-.5,4.5);
	\node[blue,opacity=1] at (1.75,4.65) {\tiny $b$};
	\node[red,opacity=1] at (2.25,3.75) {\tiny $a$};
	\node[opacity=1] at (.25,4.15) {\tiny $_{a+b}$};
\end{tikzpicture}
};
\endxy
}

 \newcommand{\tikzcrossing}[7][1]{
\begin{tikzpicture} [scale=.5,fill opacity=0.2]
	\draw[very thick]  (2,0) to (1,0) to [out=180,in=315] (0.2,0.3);
	\draw[very thick, directed=.65] (-0.2,0.7) to [out=135,in=0] (-1,1) to (-2,1);
	\draw[very thick, directed=.85] (2,1) to (1,1) to [out=180,in=0](-1,0) to (-2,0) ;
	\node[opacity=1] at (1.8,1.5) {$_{#4}$};
	\node[opacity=1] at (1.8,-0.5) {$_{#5}$};
	\node[opacity=1] at (-1.8,1.5) {$_{#6}$};
	\node[opacity=1] at (-1.8,-0.5) {$_{#7}$};
\end{tikzpicture}
 }
 
  \newcommand{\tikzcrossingp}[7][1]{
\begin{tikzpicture} [scale=.5,fill opacity=0.2]
	\draw[very thick, directed=.85]  (2,0) to (1,0) to [out=180,in=0] (-1,1) to (-2,1);
	\draw[very thick, directed=.65] (-0.2,0.3) to [out=225,in=0] (-1,0) to (-2,0);
	\draw[very thick] (2,1) to (1,1)  to [out=180,in=45] (0.2,0.7) ;
	\node[opacity=1] at (1.8,1.5) {$_{#4}$};
	\node[opacity=1] at (1.8,-0.5) {$_{#5}$};
	\node[opacity=1] at (-1.8,1.5) {$_{#6}$};
	\node[opacity=1] at (-1.8,-0.5) {$_{#7}$};
\end{tikzpicture}
 }

 \newcommand{\crossing}[7][1]{
 \xy
(0,0)*{\tikzcrossing[#1]{#2}{#3}{#4}{#5}{#6}{#7}
};
\endxy
 }
 
  \newcommand{\crossingw}[7][1]{
 \xy
(0,0)*{\tikzcrossing[#1]{#2}{#3}{#4}{#5}{#6}{#7}
};
(12,-3)*{_{#3}};(12,3)*{_{#2}};
\endxy
 }
 
 \newcommand{\squareweb}[8][1]{
 \xy
(0,0)*{
\begin{tikzpicture} [scale=.5,fill opacity=0.2]
	\draw[very thick, directed=.55]  (2,0) to (1.25,0);
	\draw[very thick]  (1.25,0) to (0.5,0);
	\draw[very thick, directed=.55]  (-1.25,0) to (-2,0);
	\draw[very thick]  (-0.5,0) to (-1.25,0);
	\draw[very thick, directed=.55]  (2,1) to (1.25,1);
	\draw[very thick, directed=.55]  (-1.25,1) to (-2,1);
	\draw[very thick] (1.25,1) to (-1.25,1);
	\draw[very thick, directed=.55]  (0.5,1) to (-0.5,1);
	\draw[very thick, directed=.55]  (0.5,0) to (-0.5,0);
	\draw[very thick, directed=.55] (1.25,1) to [out=180,in=0](0.5,0) ;
	\draw[very thick, directed=.55] (-0.5,0) to [out=180,in=0](-1.25,1) ;
	\node[opacity=1] at (2.25,1) {$_{#2}$};
	\node[opacity=1] at (2.25,0) {$_{#3}$};
	\node[opacity=1] at (-2,0.45) {\tiny $#4$};
		\node[opacity=1] at (1.8,1.5) {$_{#5}$};
	\node[opacity=1] at (1.8,-0.5) {$_{#6}$};
	\node[opacity=1] at (-1.8,1.5) {$_{#7}$};
	\node[opacity=1] at (-1.8,-0.5) {$_{#8}$};
\end{tikzpicture}
};
\endxy }

 \newcommand{\squarewebmirror}[8][1]{
 \xy
(0,0)*{
\begin{tikzpicture} [scale=.5,fill opacity=0.2]
	\draw[very thick, directed=.55]  (2,0) to (1.25,0);
	\draw[very thick]  (1.25,0) to (0.5,0);
	\draw[very thick, directed=.55]  (-1.25,0) to (-2,0);
	\draw[very thick]  (-0.5,0) to (-1.25,0);
	\draw[very thick, directed=.55]  (2,1) to (1.25,1);
	\draw[very thick, directed=.55]  (-1.25,1) to (-2,1);
	\draw[very thick] (1.25,1) to (-1.25,1);
	\draw[very thick, directed=.55]  (0.5,1) to (-0.5,1);
	\draw[very thick, directed=.55]  (0.5,0) to (-0.5,0);
	\draw[very thick, directed=.55] (1.25,0) to [out=180,in=0](0.5,1) ;
	\draw[very thick, directed=.55] (-0.5,1) to [out=180,in=0](-1.25,0) ;
	\node[opacity=1] at (2.25,1) {$_{#2}$};
	\node[opacity=1] at (2.25,0) {$_{#3}$};
	\node[opacity=1] at (-2,0.45) {\tiny $#4$};
		\node[opacity=1] at (1.8,1.5) {$_{#5}$};
	\node[opacity=1] at (1.8,-0.5) {$_{#6}$};
	\node[opacity=1] at (-1.8,1.5) {$_{#7}$};
	\node[opacity=1] at (1.8,-0.5) {$_{#8}$};
\end{tikzpicture}
};
\endxy }
  
 \newcommand{\squarewebAA}[8][1]{
 \xy
(0,0)*{
\begin{tikzpicture} [scale=.5,fill opacity=0.2]
	\draw[very thick, directed=.55]  (2,0) to (1.25,0);
	\draw[very thick]  (1.25,0) to (0.5,0);
	\draw[very thick, directed=.55]  (-1.25,0) to (-2,0);
	\draw[very thick]  (-0.5,0) to (-1.25,0);
	\draw[very thick, directed=.55]  (2,1) to (1.25,1);
	\draw[very thick, directed=.55]  (-1.25,1) to (-2,1);
	\draw[very thick] (1.25,1) to (-1.25,1);
	\draw[very thick, directed=.55]  (0.5,1) to (-0.5,1);
	\draw[very thick, directed=.55]  (0.5,0) to (-0.5,0);
	\draw[very thick, directed=.55] (1.25,1) to [out=180,in=0](0.5,0) ;
	\draw[very thick, directed=.55] (-0.5,0) to [out=180,in=0](-1.25,1) ;
	\node[opacity=1] at (2,0.4) {\tiny $#4$};
		\node[opacity=1] at (1.8,1.5) {$_{#5}$};
	\node[opacity=1] at (1.8,-0.5) {$_{#6}$};
	\node[opacity=1] at (-1.8,1.5) {$_{#7}$};
	\node[opacity=1] at (-1.8,-0.5) {$_{#8}$};
\end{tikzpicture}
};
\endxy }

\newcommand{\eqnA}{
\xy
(0,0)*{
\begin{tikzpicture} [scale=.5,fill opacity=0.2]
	\draw[blue, very thick, directed=.55] (-2,3) to [out=180, in=45] (-3,2.5);
	\draw[red, very thick, directed=.55]  (-1.5,2) to [out=180, in=315] (-3,2.5);
	\draw[green, very thick, directed=.55] (-3,2.5) to (-4.5,2.5);
	\draw[blue, very thick, directed=.55] (-4.5,2.5) to [out=135,in=0] (-6,3);
	\draw[red, very thick, directed=.55] (-4.5,2.5) to [out=225,in=0] (-5.5,2);	
\end{tikzpicture}
};
\endxy ~ \cong ~ 
\xy
(0,0)*{
\begin{tikzpicture} [scale=.5,fill opacity=0.2]
	\draw[blue, very thick, directed=.55] (-2,3) to (-6,3);
	\draw[red, very thick, directed=.55] (-1.5,2) to  (-5.5,2);	
\end{tikzpicture}
};
\endxy, \quad
\xy
(0,0)*{
\begin{tikzpicture} [scale=.5,fill opacity=0.2]
	\draw[green,very thick, directed=.55] (1,2.5) to (-.5,2.5);
	\draw[red, very thick, directed=.55] (-.5,2.5) to [out=135,in=0] (-2,3) to [out=180, in=45] (-3,2.5);
	\draw[blue, very thick, directed=.55] (-.5,2.5) to [out=225,in=0] (-1.5,2) to [out=180, in=315] (-3,2.5);
	\draw[green, very thick, directed=.55] (-3,2.5) to (-4.5,2.5);
\end{tikzpicture}
};
\endxy ~ \cong ~ 
\xy
(0,0)*{
\begin{tikzpicture} [scale=.5,fill opacity=0.2]
	\draw[green, very thick, directed=.55] (1,2.5) to (-4.5,2.5);
\end{tikzpicture}
};
\endxy
}

\newcommand{\eqnB}{ 
\xy
(0,0)*{
\begin{tikzpicture} [scale=.5,fill opacity=0.2]
	\draw[blue, very thick, directed=.55] (0,3) to [out=180,in=45] (-3,2.25);
	\draw[red, very thick, directed=.55] (0.75,1.5) to (-2,1.5);
	\draw[red, very thick, directed=.55](-2,1.5) to [out=135,in=315] (-3,2.25);
	\draw[green, very thick, directed=.55] (-3,2.25) to (-4,2.25);
	\draw[blue, very thick, directed=.55] (-4,2.25) to [out=225,in=45] (-5,1.5);
	\draw[red, very thick, directed=.55] (-2,1.5) to [out=225,in=315] (-5,1.5);
	\draw[green, very thick, directed=.55] (-5,1.5) to (-6,1.5);
	\draw[blue, very thick, directed=.55] (-6,1.5) to [out=225,in=0] (-7,1);
	\draw[red, very thick, directed=.55] (-6,1.5) to [out=135,in=0] (-7.5,2);
	\draw[red, very thick, directed=.55] (-4,2.25) to [out=135,in=0] (-8,3);
\end{tikzpicture}
};
\endxy ~ \cong ~ 
\xy
(0,0)*{
\begin{tikzpicture} [scale=.5,fill opacity=0.2]
	\draw[blue, very thick, directed=.55] (0,3) to [out=180,in=45] (-2,2);
	\draw[red, very thick, directed=.55] (0.75,1.5) to [out=180,in=315](-2,2);
	\draw[green, very thick, directed=.55] (-2,2) to (-3,2);
	\draw[blue, very thick, directed=.55] (-3,2) to [out=225,in=0] (-7,1);
	\draw[red, very thick, directed=.55] (-3,2) to [out=135,in=0] (-5,2.5);
	\draw[red, very thick, directed=.55] (-5,2.5) to [out=225,in=0] (-7.5,2);
	\draw[red, very thick, directed=.55] (-5,2.5) to [out=135,in=0] (-8,3);
\end{tikzpicture}
};
\endxy
}

\newcommand{\eqnC}{
\xy
(0,0)*{
\begin{tikzpicture} [scale=.5,fill opacity=0.2]
	\draw[blue, very thick, directed=.55] (3,2.25) to [out=135,in=0] (-0.75,3);
	\draw[red, very thick, directed=.55] (2,1.5) to (0,1.5);
	\draw[red, very thick, directed=.55](3,2.25) to (2,1.5)[out=225,in=45] ;
	\draw[green, very thick, directed=.55] (4,2.25) to (3,2.25);
	\draw[blue, very thick, directed=.55] (5,1.5) to [out=135,in=315](4,2.25) ;
	\draw[red, very thick, directed=.55] (5,1.5) to [out=225,in=315] (2,1.5);
	\draw[green, very thick, directed=.55] (6,1.5) to (5,1.5);	
	\draw[blue, very thick, directed=.55] (8,1) to [out=180,in=315](6,1.5) ;
	\draw[red, very thick, directed=.55] (7.5,2) to [out=180,in=45](6,1.5) ;
	\draw[red, very thick, directed=.55] (7,3) to [out=180,in=45] (4,2.25) ;
\end{tikzpicture}
};
\endxy ~ \cong ~ 
\xy
(0,0)*{
\begin{tikzpicture} [scale=.5,fill opacity=0.2]
	\draw[blue, very thick, directed=.55] (2,2) to [out=135,in=0] (-0.75,3);
	\draw[red, very thick, directed=.55] (2,2)  to [out=225,in=0](0,1.5);
	\draw[green, very thick, directed=.55] (3,2) to (2,2);
	\draw[blue, very thick, directed=.55] (7,1) to [out=180,in=315] (3,2) ;
	\draw[red, very thick, directed=.55] (5,2.5) to [out=180,in=45] (3,2);
	\draw[red, very thick, directed=.55] (7.5,2) to [out=180,in=315] (5,2.5);
	\draw[red, very thick, directed=.55] (7,3) to [out=180,in=45] (5,2.5) ;
\end{tikzpicture}
};
\endxy
}

\newcommand{\eqnD}{
 LHS ~\overset{\eqref{webiso1}}{\cong} ~
\xy
(0,0)*{
\begin{tikzpicture} [scale=.5,fill opacity=0.2]
	\draw[blue, very thick, directed=.55] (1,3) to [out=180,in=45] (-0,1.5);
	\draw[red, very thick, directed=.55] (0.75,1.5) to (0,1.5);
	\draw[green, very thick, directed=.55] (0,1.5) to (-1,1.5);
	\draw[red, very thick, directed=.55] (-1,1.5) to (-2,1.5);
	\draw[blue, very thick, directed=.55] (-1,1.5) to [out=135, in=45](-3,2.25);
	\draw[red, very thick, directed=.55](-2,1.5) to [out=135,in=315] (-3,2.25);
	\draw[green, very thick, directed=.55] (-3,2.25) to (-4,2.25);
	\draw[blue, very thick, directed=.55] (-4,2.25) to [out=225,in=45] (-5,1.5);
	\draw[red, very thick, directed=.55] (-2,1.5) to [out=225,in=315] (-5,1.5);
	\draw[green, very thick, directed=.55] (-5,1.5) to (-6,1.5);
	\draw[blue, very thick, directed=.55] (-6,1.5) to [out=225,in=0] (-7,1);
	\draw[red, very thick, directed=.55] (-6,1.5) to [out=135,in=0] (-7.5,2);
	\draw[red, very thick, directed=.55] (-4,2.25) to [out=135,in=0] (-8,3);
\end{tikzpicture}
};
\endxy 
~ \overset{\eqref{FoamRel1}}{\cong} ~
\xy
(0,0)*{
\begin{tikzpicture} [scale=.5,fill opacity=0.2]
	\draw[blue, very thick, directed=.55] (1,3) to [out=180,in=45] (-0,1.5);
	\draw[red, very thick, directed=.55] (0.75,1.5) to (0,1.5);
	\draw[green, very thick, directed=.55] (0,1.5) to (-2,1.5);
	\draw[blue, very thick, directed=.55] (-2.5,2.25) to [out=135, in=45](-3.5,2.25);
	\draw[red, very thick, directed=.55] (-2.5,2.25) to [out=225, in=315](-3.5,2.25);
	\draw[green, very thick, directed=.55](-2,1.5) to [out=135,in=315] (-2.5,2.25);
	\draw[green, very thick, directed=.55] (-3.5,2.25) to [out=135,in=0] (-4,2.25);
	\draw[blue, very thick, directed=.55] (-4,2.25) to [out=225,in=45] (-5,1.5);
	\draw[red, very thick, directed=.55] (-2,1.5) to [out=225,in=315] (-5,1.5);
	\draw[green, very thick, directed=.55] (-5,1.5) to (-6,1.5);
	\draw[blue, very thick, directed=.55] (-6,1.5) to [out=225,in=0] (-7,1);
	\draw[red, very thick, directed=.55] (-6,1.5) to [out=135,in=0] (-7.5,2);
	\draw[red, very thick, directed=.55] (-4,2.25) to [out=135,in=0] (-8,3);
\end{tikzpicture}
};
\endxy \\
~ \overset{\eqref{webiso1}}{\cong} ~
\xy
(0,0)*{
\begin{tikzpicture} [scale=.5,fill opacity=0.2]
	\draw[blue, very thick, directed=.55] (1,3) to [out=180,in=45] (-0,1.5);
	\draw[red, very thick, directed=.55] (0.75,1.5) to (0,1.5);
	\draw[green, very thick, directed=.55] (0,1.5) to (-2,1.5);
	\draw[green, very thick, directed=.55](-2,1.5) to [out=135,in=0] (-4,2.25);
	\draw[blue, very thick, directed=.55] (-4,2.25) to [out=225,in=45] (-5,1.5);
	\draw[red, very thick, directed=.55] (-2,1.5) to [out=225,in=315] (-5,1.5);
	\draw[green, very thick, directed=.55] (-5,1.5) to (-6,1.5);
	\draw[blue, very thick, directed=.55] (-6,1.5) to [out=225,in=0] (-7,1);
	\draw[red, very thick, directed=.55] (-6,1.5) to [out=135,in=0] (-7.5,2);
	\draw[red, very thick, directed=.55] (-4,2.25) to [out=135,in=0] (-8,3);
\end{tikzpicture}
};
\endxy ~ \overset{\eqref{FoamRel1}}{\cong} ~
\xy
(0,0)*{
\begin{tikzpicture} [scale=.5,fill opacity=0.2]
	\draw[blue, very thick, directed=.55] (1,3) to [out=180,in=45] (-0,1.5);
	\draw[red, very thick, directed=.55] (0.75,1.5) to (0,1.5);
	\draw[green, very thick, directed=.55] (0,1.5) to (-1,1.5) to [out=225,in=0] (-3,1.5);
	\draw[blue, very thick, directed=.55] (-3,1.5) to [out=135,in=45] (-5,1.5);
	\draw[red, very thick, directed=.55] (-3,1.5) to [out=225,in=315] (-5,1.5);
	\draw[green, very thick, directed=.55] (-5,1.5) to (-6,1.5);
	\draw[blue, very thick, directed=.55] (-6,1.5) to [out=225,in=0] (-7,1);
	\draw[red, very thick, directed=.55] (-6,1.5) to [out=135,in=0] (-7.5,2);
	\draw[red, very thick, directed=.55] (-1,1.5) to [out=135,in=0] (-8,3);
\end{tikzpicture}
};
\endxy
}

\newcommand{\eqnE}{
\overset{\eqref{webiso1}}{\cong} ~
\xy
(0,0)*{
\begin{tikzpicture} [scale=.5,fill opacity=0.2]
	\draw[blue, very thick, directed=.55] (1,3) to [out=180,in=45] (-0,1.5);
	\draw[red, very thick, directed=.55] (0.75,1.5) to (0,1.5);
	\draw[green, very thick, directed=.55] (0,1.5) to (-1,1.5) to [out=225,in=0] (-6,1.5);
	\draw[blue, very thick, directed=.55] (-6,1.5) to [out=225,in=0] (-7,1);
	\draw[red, very thick, directed=.55] (-6,1.5) to [out=135,in=0] (-7.5,2);
	\draw[red, very thick, directed=.55] (-1,1.5) to [out=135,in=0] (-8,3);
\end{tikzpicture}
};
\endxy
 ~ \overset{\eqref{FoamRel1}}{\cong} ~
 RHS 
}

\newcommand{\eqnF}{
\xy
(0,0)*{
\begin{tikzpicture} [scale=.5,fill opacity=0.2]
	\draw[blue, very thick, directed=.55] (-2,3) to [out=180, in=45] (-3,2.5);
	\draw[red, very thick, directed=.55]  (-1.5,2) to [out=180, in=315] (-3,2.5);
	\draw[green, very thick, directed=.55] (-3,2.5) to (-4.5,2.5);
	\draw[red, very thick, directed=.55] (-4.5,2.5) to [out=135,in=0] (-6,3);
	\draw[blue, very thick, directed=.55] (-4.5,2.5) to [out=225,in=0] (-5.5,2);	
\end{tikzpicture}
};
\endxy ~ =: ~ 
\xy
(0,0)*{
\begin{tikzpicture} [scale=.5,fill opacity=0.2]
	\draw[red, very thick, directed=.85] (-1.5,2) to  (-2.5,2) to  [out=180, in=0](-5,3)to (-6,3);
	\draw[blue, very thick, directed=.85] (-2,3) to (-3,3) to [out=180, in=0] (-4.5,2) to  (-5.5,2);	
\end{tikzpicture}
};
\endxy
}

\newcommand{\eqnG}{
\xy
(0,0)*{
\begin{tikzpicture} [scale=.5,fill opacity=0.2]
	\draw[blue, very thick, directed=.55] (0,3) to [out=180,in=45] (-3,2.25) to [out=225,in=0] (-7,1) ;
	\draw[red, very thick, directed=.55] (0.75,1.5) to (-2,1.5);
	\draw[red, very thick, directed=.85](-2,1.5) to [out=225,in=0] (-7.5,2);
	\draw[red, very thick, directed=.85] (-2,1.5) to [out=135,in=0] (-8,3) ;
\end{tikzpicture}
};
\endxy ~ \cong ~ 
\xy
(0,0)*{
\begin{tikzpicture} [scale=.5,fill opacity=0.2]
	\draw[blue, very thick, directed=.55] (0,3) to [out=180,in=45] (-3,2.25) to [out=225,in=0] (-7,1) ;
	\draw[red, very thick, directed=.35] (0.75,1.5) to [out=180,in=0](-5,2.5);
	\draw[red, very thick, directed=.55] (-5,2.5) to [out=225,in=0] (-7.5,2);
	\draw[red, very thick, directed=.55] (-5,2.5) to [out=135,in=0] (-8,3);
\end{tikzpicture}
};
\endxy
}

\newcommand{\eqnH}{
\xy
(0,0)*{
\begin{tikzpicture} [scale=.5,fill opacity=0.2]
	\draw[red, very thick, directed=.55] (2,1.5) to (0,1.5);
	\draw[blue, very thick, directed=.75] (8,1)  to [out=180,in=0] (-0.75,3) ;
	\draw[red, very thick, directed=.25] (7.5,2)  to [out=180,in=315] (2,1.5) ;
	\draw[red, very thick, directed=.25] (7,3) to [out=180,in=45] (2,1.5) ;
\end{tikzpicture}
};
\endxy ~ \cong ~ 
\xy
(0,0)*{
\begin{tikzpicture} [scale=.5,fill opacity=0.2]
	\draw[blue, very thick, directed=.75] (8,1)  to [out=180,in=0] (-0.75,3) ;
	\draw[red, very thick, directed=.75] (5,2.5) to [out=180,in=0] (0,1.5);
	\draw[red, very thick, directed=.55] (7.5,2) to [out=180,in=315] (5,2.5);
	\draw[red, very thick, directed=.55] (7,3) to [out=180,in=45] (5,2.5) ;
\end{tikzpicture}
};
\endxy
}

\newcommand{\eqnI}{
\xy
(0,0)*{
\begin{tikzpicture} [scale=.5,fill opacity=0.2]
	\draw[blue, very thick, directed=.55] (1,1.5) to [out=135,in=0](-1,2);
	\draw[red, very thick, directed=.55] (1,1.5) to [out=225,in=0](-0.5,1);
	\draw[green, very thick, directed=.55] (3,1.5) to [out=180,in=0](1,1.5);
	\draw[green, very thick, directed=.55] (4.5,2.5) to [out=180,in=45](3,1.5);
	\draw[green, very thick, directed=.55] (5,0.5) to [out=180,in=315](3,1.5);
	\draw[blue, very thick, directed=.55] (5.8,3)  to [out=180,in=45] (4.5,2.5) ;
	\draw[blue, very thick, directed=.55] (6.6,1)  to [out=180,in=45] (5,0.5) ;
	\draw[red, very thick, directed=.55] (6.2,2)  to [out=180,in=315] (4.5,2.5) ;
	\draw[red, very thick, directed=.55] (7,0)  to [out=180,in=315] (5,0.5) ;
\end{tikzpicture}
};
\endxy
\cong \xy
(0,0)*{
\begin{tikzpicture} [scale=.5,fill opacity=0.2]
	\draw[blue, very thick, directed=.55] (0.5,1.5) to [out=135,in=0](-1,2);
	\draw[red, very thick, directed=.55] (0.5,1.5) to [out=225,in=0](-0.5,1);
	\draw[green, very thick, directed=.55] (5,1.5) to (2,1.5) to [out=180,in=0](0.5,1.5);
	\draw[blue, very thick, directed=.55] (5.8,3)  to [out=180,in=45] (1,1.5) ;
	\draw[blue, very thick, directed=.55] (6.6,1)  to [out=180,in=315] (5,1.5) ;
	\draw[red, very thick, directed=.55] (6.2,2)  to [out=180,in=45] (5,1.5) ;
	\draw[red, very thick, directed=.55] (7,0)  to [out=180,in=315] (3,1.5) ;
\end{tikzpicture}
};
\endxy
\cong \xy
(0,0)*{
\begin{tikzpicture} [scale=.5,fill opacity=0.2]
	\draw[blue, very thick, directed=.55] (0.5,1.5) to [out=135,in=0](-1,2);
	\draw[red, very thick, directed=.55] (0.5,1.5) to [out=225,in=0](-0.5,1);
	\draw[green, very thick] (5,1.5) to (4.5,1.5);
	\draw[green, very thick] (3.5,1.5) to (2.5,1.5);
	\draw[green, very thick] (1.5,1.5) to (0.5,1.5);
\draw[blue, very thick, directed=.55] (2.5,1.5)  to [out=135,in=0] (2,1.8) to [out=180,in=45] (1.5,1.5) ;
\draw[red, very thick, directed=.55] (2.5,1.5)  to [out=225,in=0] (2,1.2) to [out=180,in=315] (1.5,1.5) ;
\draw[blue, very thick, directed=.55] (4.5,1.5)  to [out=135,in=0] (4,1.8) to [out=180,in=45] (3.5,1.5) ;
\draw[red, very thick, directed=.55] (4.5,1.5)  to [out=225,in=0] (4,1.2) to [out=180,in=315] (3.5,1.5) ;	
	\draw[blue, very thick, directed=.55] (5.8,3)  to [out=180,in=45] (1,1.5) ;
	\draw[blue, very thick, directed=.55] (6.6,1)  to [out=180,in=315] (5,1.5) ;
	\draw[red, very thick, directed=.55] (6.2,2)  to [out=180,in=45] (5,1.5) ;
	\draw[red, very thick, directed=.55] (7,0)  to [out=180,in=315] (3,1.5) ;
\end{tikzpicture}
};
\endxy
}

\newcommand{\eqnJ}{
\xy
(0,0)*{
\begin{tikzpicture} [scale=.5,fill opacity=0.2]
	\draw[blue, very thick, directed=.55] (0.5,1.5) to [out=135,in=0](-1,2);
	\draw[red, very thick, directed=.55] (0.5,1.5) to [out=225,in=0](-0.5,1);
	\draw[green, very thick] (5,1.5) to (4.5,1.5);
	\draw[green, very thick, directed=.55] (3.5,1.5) to (2.5,1.5);
	\draw[green, very thick, directed=.55] (1.5,1.5) to (0.5,1.5);
\draw[blue, very thick] (2.5,1.5)  to [out=135,in=0] (2,1.8) to [out=180,in=45] (1.5,1.5) ;
\draw[red, very thick, directed=.55] (2.5,1.5)  to [out=225,in=0] (2,1.2) to [out=180,in=315] (1.5,1.5) ;
\draw[blue, very thick, directed=.55] (4.5,1.5)  to [out=135,in=0] (4,1.8) to [out=180,in=45] (3.5,1.5) ;
\draw[red, very thick] (4.5,1.5)  to [out=225,in=0] (4,1.2) to [out=180,in=315] (3.5,1.5) ;	
	\draw[blue, very thick, directed=.55] (5.8,3)  to [out=180,in=45] (2,1.8) ;
	\draw[blue, very thick, directed=.55] (6.6,1)  to [out=180,in=315] (5,1.5) ;
	\draw[red, very thick, directed=.55] (6.2,2)  to [out=180,in=45] (5,1.5) ;
	\draw[red, very thick, directed=.55] (7,0)  to [out=180,in=315] (4,1.2) ;
\end{tikzpicture}
};
\endxy
\cong \xy
(0,0)*{
\begin{tikzpicture} [scale=.5,fill opacity=0.2]
	\draw[blue, very thick, directed=.75] (4,1.5) to [out=135,in=0] (2,2) to (-1,2);
	\draw[red, very thick, directed=.75] (4,1.5)  to [out=225,in=0](2,1) to (-0.5,1);
	\draw[green, very thick, directed=.55] (5,1.5) to (4,1.5);
	\draw[blue, very thick, directed=.55] (5.8,3)  to [out=180,in=45] (2,2) ;
	\draw[blue, very thick, directed=.55] (6.6,1)  to [out=180,in=315] (5,1.5) ;
	\draw[red, very thick, directed=.55] (6.2,2)  to [out=180,in=45] (5,1.5) ;
	\draw[red, very thick, directed=.55] (7,0)  to [out=180,in=315] (2,1) ;
\end{tikzpicture}
};
\endxy
}

\newcommand{\eqnK}{
\xy
(0,0)*{
\begin{tikzpicture} [scale=.5,fill opacity=0.2]
	\draw[orange, very thick, directed=.55] (2,1.5) to [out=225,in=0](0,0);
	\draw[green, very thick, directed=.55] (2,1.5) to [out=135,in=0](1,2);
	\draw[red, very thick, directed=.55] (1,2) to [out=225,in=0](-0.5,1.5);
	\draw[blue, very thick, directed=.55] (1,2) to [out=135,in=0](-1,2.5);
	\draw[black, very thick, directed=.55] (4,1.5) to [out=180,in=0](2,1.5);
	\draw[orange, very thick, directed=.55] (6,2.5) to [out=180,in=45](4,1.5);
	\draw[green, very thick, directed=.55] (5,0.5) to [out=180,in=315](4,1.5);
	\draw[red, very thick, directed=.55] (7,0) to [out=180,in=315](5,0.5);
	\draw[blue, very thick, directed=.55] (6.5,1) to [out=180,in=45](5,0.5);
\end{tikzpicture}
};
\endxy
\cong 
\xy
(0,0)*{
\begin{tikzpicture} [scale=.5,fill opacity=0.2]
	\draw[orange, very thick, directed=.55] (1,1) to [out=225,in=0](0,0);
	\draw[magenta, very thick, directed=.55] (2,1.5) to [out=225,in=0](1,1);
	\draw[red, very thick, directed=.55] (1,1) to [out=135,in=0](-0.5,1.5);
	\draw[blue, very thick, directed=.55] (2,1.5) to [out=135,in=0](-1,2.5);
	\draw[black, very thick, directed=.55] (3,1.5) to [out=180,in=0](2,1.5);
	\draw[orange, very thick, directed=.55] (6,2.5) to [out=180,in=45](5,2);
	\draw[cyan, very thick, directed=.55] (5,2) to [out=180,in=45](3,1.5);
	\draw[red, very thick, directed=.55] (7,0) to [out=180,in=315](3,1.5);
	\draw[blue, very thick, directed=.55] (6.5,1) to [out=180,in=315](5,2);
\end{tikzpicture}
};
\endxy
\cong 
\xy
(0,0)*{
\begin{tikzpicture} [scale=.5,fill opacity=0.2]
	\draw[orange, very thick, directed=.55] (1,1) to [out=225,in=0](0,0);
	\draw[magenta, very thick, directed=.55] (2,1.5) to [out=225,in=0](1,1);
	\draw[red, very thick, directed=.55] (1,1) to [out=135,in=0](-0.5,1.5);
	\draw[blue, very thick, directed=.55] (2,1.5) to [out=135,in=0](-1,2.5);
	\draw[black, very thick, directed=.75] (2.5,1.5) to [out=180,in=0](2,1.5);
	\draw[orange, very thick, directed=.55] (6,2.5) to [out=180,in=45](5,2);
\draw[orange, very thick, directed=.55] (4.5,2) to [out=225,in=0] (4,1.8) to [out=180, in=315](3.5,2);
	\draw[blue, very thick, directed=.55] (4.5,2) to [out=135,in=0](4,2.2)to [out=180,in=45] (3.5,2);
	\draw[cyan, very thick, directed=.75]  (5,2) to (4.5,2);
	\draw[cyan, very thick, directed=.55]  (3.5,2) to [out=180,in=45](2.5,1.5);
	\draw[red, very thick, directed=.55] (7,0) to [out=180,in=315](2.5,1.5);
	\draw[blue, very thick, directed=.55] (6.5,1) to [out=180,in=315](5,2);
\end{tikzpicture}
};
\endxy
}

\newcommand{\eqnL}{
\cong 
\xy
(0,0)*{
\begin{tikzpicture} [scale=.5,fill opacity=0.2]
	\draw[orange, very thick, directed=.55] (1,1) to [out=225,in=0](0,0);
	\draw[magenta, very thick, directed=.55] (2,1.5) to [out=225,in=0](1,1);
	\draw[red, very thick, directed=.55] (1,1) to [out=135,in=0](-0.5,1.5);
	\draw[blue, very thick, directed=.55] (2,1.5) to [out=135,in=0](-1,2.5);
	\draw[orange, very thick, directed=.55] (6,2.5) to [out=180,in=45](5,2);
	\draw[orange, very thick] (4,2) to [out=225,in=0] (3.5,1.8);
	\draw[magenta, very thick](3.5,1.8) to [out=180, in=315](3,2);
	\draw[blue, very thick, directed=.55] (4,2) to [out=135,in=0](3.5,2.2)to [out=180,in=45] (3,2);
	\draw[cyan, very thick, directed=.55]  (5,2) to (4,2);
	\draw[black, very thick, directed=.55]  (3,2) to [out=180,in=0](2,1.5);
	\draw[red, very thick, directed=.55] (7,0) to [out=180,in=315](3.5,1.8);
	\draw[blue, very thick, directed=.55] (6.5,1) to [out=180,in=315](5,2);
\end{tikzpicture}
};
\endxy
\cong 
\xy
(0,0)*{
\begin{tikzpicture} [scale=.5,fill opacity=0.2]
	\draw[orange, very thick, directed=.55] (1,1) to [out=225,in=0](0,0);
	\draw[magenta, very thick, directed=.55] (2,1) to (1,1);
	\draw[red, very thick, directed=.55] (1,1) to [out=135,in=0](-0.5,1.5);
	\draw[blue, very thick, directed=.55] (4,2) to [out=135,in=0] (3,2.5) to (-1,2.5);
	\draw[orange, very thick, directed=.55] (6,2.5) to [out=180,in=45](5,2);
	\draw[cyan, very thick, directed=.55]  (5,2) to (4,2);
	\draw[orange, very thick, directed=.55]  (4,2) to [out=225,in=45](2,1);
	\draw[red, very thick, directed=.55] (7,0) to [out=180,in=315](2,1);
	\draw[blue, very thick, directed=.55] (6.5,1) to [out=180,in=315](5,2);
\end{tikzpicture}
};
\endxy
}

\newcommand{\eqnM}{
\left \llbracket
\xy (0,0)*{
\tikzcrossing[1]{a}{b}{}{}{}{}
};
(12,-3)*{_b};(12,3)*{_a};
\endxy \right \rrbracket^\Sigma \!\! = \uwave{\xy
(0,0)*{
\begin{tikzpicture} [scale=.5,fill opacity=0.2]
	\draw[very thick, directed=.55]  (2,0) to (1.25,0);
	\draw[very thick]  (1.25,0) to (0.5,0);
	\draw[very thick, directed=.55]  (-1.25,0) to (-2,0);
	\draw[very thick]  (-0.5,0) to (-1.25,0);
	\draw[very thick, directed=.55]  (2,1) to (1.25,1);
	\draw[very thick, directed=.55]  (-1.25,1) to (-2,1);
	\draw[very thick, directed=.55]  (0.5,0) to (-0.5,0);
	\draw[very thick, directed=.55] (1.25,1) to [out=180,in=0](0.5,0) ;
	\draw[very thick, directed=.55] (-0.5,0) to [out=180,in=0](-1.25,1) ;
	\node[opacity=1] at (2.25,1) {${_a}$};
	\node[opacity=1] at (2.25,0) {${_b}$};
	\node[opacity=1] at (-1.35,0.45) {\tiny ${b}$};
\end{tikzpicture}
};
\endxy} \overset{d_b}{\to}\!\!\!\!\! 
\squareweb[1]{a}{b}{\quad b-1}{}{}{}{}\overset{d_{b\!-\! 1}}{\to} \cdots \overset{d_2}{\to}\!\!\!\!\!\!\! \squareweb[1]{a}{b}{\quad\quad 1}{}{}{}{}
 \overset{d_1}{\to} \xy
(0,0)*{
\begin{tikzpicture} [scale=.5,fill opacity=0.2]
	\draw[very thick, directed=.55]  (2,0) to (1.25,0);
	\draw[very thick]  (1.25,0) to (0.5,0);
	\draw[very thick, directed=.55]  (-1.25,0) to (-2,0);
	\draw[very thick]  (-0.5,0) to (-1.25,0);
	\draw[very thick, directed=.55]  (2,1) to (1.25,1);
	\draw[very thick, directed=.55]  (-1.25,1) to (-2,1);
	\draw[very thick] (1.25,1) to (-1.25,1);
	\draw[very thick, directed=.55]  (0.5,1) to (-0.5,1);
	\draw[very thick, directed=.55]  (0.5,0) to (-0.5,0);
	\draw[very thick, directed=.55] (1.25,1) to [out=180,in=0](0.5,0) ;\node[opacity=1] at (2.25,1) {${_a}$};
	\node[opacity=1] at (2.25,0) {${_b}$};
\end{tikzpicture}
};
\endxy
}

\newcommand{\eqnMtwo}{
\left \llbracket
\xy (0,0)*{
\tikzcrossing[1]{j}{i}{}{}{}{}
};
(12,-3)*{_i};(12,3)*{_j};
\endxy \right \rrbracket^\Sigma \!\! := \uwave{\xy
(0,0)*{
\begin{tikzpicture} [scale=.5,fill opacity=0.2]
	\draw[very thick, directed=.55]  (2,1) to (1.25,1);
	\draw[very thick]  (1.25,1) to (0.5,1);
	\draw[very thick, directed=.55]  (-1.25,1) to (-2,1);
	\draw[very thick]  (-0.5,1) to (-1.25,1);
	\draw[very thick, directed=.55]  (2,0) to (1.25,0);
	\draw[very thick, directed=.55]  (-1.25,0) to (-2,0);
	\draw[very thick, directed=.55]  (0.5,1) to (-0.5,1);
	\draw[very thick, directed=.55] (1.25,0) to [out=180,in=0](0.5,1) ;
	\draw[very thick, directed=.55] (-0.5,1) to [out=180,in=0](-1.25,0) ;
	\node[opacity=1] at (2.25,1) {${_j}$};
	\node[opacity=1] at (2.25,0) {${_i}$};
	\node[opacity=1] at (1.35,0.4) {\tiny ${j}$};
\end{tikzpicture}
};
\endxy} \overset{d'_j}{\to}\!\!\!\!\!
\squarewebmirror[1]{j}{i}{\quad j-1}{}{}{}{} \overset{d'_{j\!-\! 1}}{\to} \cdots \overset{d'_2}{\to}\!\!\!\!\! \squarewebmirror[1]{j}{i}{\qquad\; 1}{}{}{}{}
 \overset{d'_1}{\to} \xy
(0,0)*{
\begin{tikzpicture} [scale=.5,fill opacity=0.2]
	\draw[very thick, directed=.55]  (2,0) to (1.25,0);
	\draw[very thick]  (1.25,0) to (0.5,0);
	\draw[very thick, directed=.55]  (-1.25,0) to (-2,0);
	\draw[very thick]  (-0.5,0) to (-1.25,0);
	\draw[very thick, directed=.55]  (2,1) to (1.25,1);
	\draw[very thick, directed=.55]  (-1.25,1) to (-2,1);
	\draw[very thick] (1.25,1) to (-1.25,1);
	\draw[very thick]  (0.5,1) to (-0.5,1);
	\draw[very thick]  (0.5,0) to (-0.5,0);
	\draw[very thick, directed=.55] (1.25,0) to [out=180,in=0](0.5,1) ;\node[opacity=1] at (2.25,1) {${_j}$};
	\node[opacity=1] at (2.25,0) {${_i}$};
\end{tikzpicture}
};
\endxy
}

\newcommand{\eqnMthree}{
\left \llbracket
\xy (0,0)*{
\tikzcrossing[1]{j}{i}{}{}{}{}
};
(12,-3)*{_i};(12,3)*{_j};
\endxy \right \rrbracket^\Sigma \!\! &=& \Phi_\Sigma(\tilde{\cal{T}}^{-1}_i) \cong \Phi_\Sigma(\tilde{\cal{T'}}^{-1}_i)\\ &= &  \uwave{\xy
(0,0)*{
\begin{tikzpicture} [scale=.5,fill opacity=0.2]
	\draw[very thick, directed=.55]  (2,0) to (1.25,0);
	\draw[very thick]  (1.25,0) to (0.5,0);
	\draw[very thick, directed=.55]  (-1.25,0) to (-2,0);
	\draw[very thick]  (-0.5,0) to (-1.25,0);
	\draw[very thick, directed=.55]  (2,1) to (1.25,1);
	\draw[very thick, directed=.55]  (-1.25,1) to (-2,1);
	\draw[very thick, directed=.55]  (0.5,0) to (-0.5,0);
	\draw[very thick, directed=.55] (1.25,1) to [out=180,in=0](0.5,0) ;
	\draw[very thick, directed=.55] (-0.5,0) to [out=180,in=0](-1.25,1) ;
	\node[opacity=1] at (2.25,1) {${_j}$};
	\node[opacity=1] at (2.25,0) {${_i}$};
	\node[opacity=1] at (1.35,0.4) {\tiny ${j}$};
\end{tikzpicture}
};
\endxy} \overset{d_j}{\to}
\squareweb[1]{j}{i}{j-1\quad}{}{}{}{}\!\!\!\!\! \overset{d_{j\!-\! 1}}{\to} \cdots \overset{d_2}{\to} \squareweb[1]{j}{i}{1\qquad\;}{}{}{}{}\!\!\!\!\!
 \overset{d_1}{\to} \xy
(0,0)*{
\begin{tikzpicture} [scale=.5,fill opacity=0.2]
	\draw[very thick, directed=.55]  (2,0) to (1.25,0);
	\draw[very thick]  (1.25,0) to (0.5,0);
	\draw[very thick, directed=.55]  (-1.25,0) to (-2,0);
	\draw[very thick]  (-0.5,0) to (-1.25,0);
	\draw[very thick, directed=.55]  (2,1) to (1.25,1);
	\draw[very thick, directed=.55]  (-1.25,1) to (-2,1);
	\draw[very thick] (1.25,1) to (-1.25,1);
	\draw[very thick, directed=.55]  (0.5,1) to (-0.5,1);
	\draw[very thick, directed=.55]  (0.5,0) to (-0.5,0);
	\draw[very thick, directed=.55] (-0.5,0) to [out=180,in=0](-1.25,1) ;\node[opacity=1] at (2.25,1) {${_j}$};
	\node[opacity=1] at (2.25,0) {${_i}$};
\end{tikzpicture}
};
\endxy
}

\newcommand{\diff}{
\xy
(0,0)*{
\begin{tikzpicture} [scale=.5,fill opacity=0.1]
\path[fill=green, opacity=0.3] (1,1.25) to [out=90,in=0] (0.5,1.75) to [out=180,in=90] (0,1.25) to (-2,2.75) to [out=90,in=180] (0,4.25) to [out=0,in=90] (2,2.75) to (1,1.25) ;
\path[fill=green] (0,0) to (0,5) to (-2,6) to (-2,1) to (0,0);
\path[fill=green] (1,0) to (1,5) to (2,6) to (2,1) to (1,0);
\path[fill=green] (-3,1) to (3,1) to (3,6) to (-3,6);
\path[fill=green] (-2.5,0) to (3.5,0) to (3.5,5) to (-2.5,5);
	\draw[very thick, red, rdirected=.55] (1,1.25) to [out=90,in=0] (0.5,1.75) to [out=180,in=90] (0,1.25);
	\draw[very thick, red, rdirected=.55] (-2,2.75) to [out=90,in=180] (0,4.25) to [out=0,in=90] (2,2.75);
\draw[very thick, red, directed=.90] (1,5) to (1,0);
\draw[very thick, red, directed=.90] (-2,6) to (-2,1);
\draw[very thick, red, directed=.15] (0,0) to (0,5);
\draw[very thick, red, directed=.15] (2,1) to (2,6);
\draw[very thick, red, rdirected=.55] (2,2.75) to  (1,1.25);
\draw[very thick, red, rdirected=.55] (0,1.25) to  (-2,2.75);
	\draw[very thick, directed=.44] (3,1) to (-3,1);
	\draw[very thick, directed=.52] (3.5,0) to (-2.5,0);
	\draw[very thick, directed=.55] (2,1) to  (1,0);
	\draw[very thick, directed=.55] (0,0) to  (-2,1);
	\draw[very thick, directed=.44] (3,6) to (-3,6);
	\draw[very thick, directed=.52] (3.5,5) to (-2.5,5);
	\draw[very thick, directed=.55] (2,6) to  (1,5);
	\draw[very thick, directed=.55] (0,5) to  (-2,6);
\draw[very thick] (-3,6) to (-3,1);
\draw[very thick] (3,6) to (3,1);
\draw[very thick] (-2.5,5) to (-2.5,0);
\draw[very thick] (3.5,5) to (3.5,0);
\node[opacity=1] at (-.5,.65) {${_k}$};
	\end{tikzpicture}
};
\endxy
}

\newcommand{\eqnO}{
\xy
(0,0)*{
\begin{tikzpicture} [scale=.5,fill opacity=0.2]
\path[fill=green] (0,0) to (0,2) to (-2,3) to (-2,1) to (0,0);
\path[fill=green] (1,0) to (1,2) to (2,3) to (2,1) to (1,0);
\path[fill=green] (-3,3) to (-3,1) to (3,1) to (3,3) to (-3,3);
\path[fill=green] (-2.5,2) to (-2.5,0) to (3.5,0) to (3.5,2) to (-2.5,2);
\draw[very thick, red, directed=.35] (1,2) to (1,0);
\draw[very thick, red, directed=.35] (-2,3) to (-2,1);
\draw[very thick, red, directed=.80] (0,0) to (0,2);
\draw[very thick, red, directed=.80] (2,1) to (2,3);
	\draw[very thick, directed=.44] (3,1) to (-3,1);
	\draw[very thick, directed=.52] (3.5,0) to (-2.5,0);
	\draw[very thick, directed=.55] (2,1) to  (1,0);
	\draw[very thick, directed=.55] (0,0) to  (-2,1);
	\draw[very thick, directed=.44] (3,3) to (-3,3);
	\draw[very thick, directed=.52] (3.5,2) to (-2.5,2);
	\draw[very thick, directed=.55] (2,3) to  (1,2);
	\draw[very thick, directed=.55] (0,2) to  (-2,3);
	\draw[very thick] (-3,3) to (-3,1);
	\draw[very thick] (3,3) to (3,1);
	\draw[very thick] (-2.5,2) to (-2.5,0);
	\draw[very thick] (3.5,2) to (3.5,0);
	\node[opacity=1] at (2.6,2.6) {${r}$};
	\node[opacity=1] at (3.1,.4) {${s}$};
	\node[opacity=1] at (-2.7,2.6) {${p}$};
	\node[opacity=1] at (-2.1,.4) {${q}$};
	\end{tikzpicture}
};
\endxy
}

\newcommand{\eqnQ}{
\xy
(0,0)*{
\begin{tikzpicture} [scale=.5,fill opacity=0.1]
\path[fill=green, opacity=0.3] (1,4.25) to [out=270,in=0] (0.5,3.75) to [out=180,in=270] (0,4.25)  to (-2,4.75) to [out=270,in=180] (0,2.75) to [out=0,in=270] (2,4.75) ;
\path[fill=green] (0,1) to (0,5) to (-2,6) to (-2,2) to (0,1);
\path[fill=green] (1,1) to (1,5) to (2,6) to (2,2) to (1,1);
\path[fill=green] (-3,2) to (3,2) to (3,6) to (-3,6);
\path[fill=green] (-2.5,1) to (3.5,1) to (3.5,5) to (-2.5,5);
	\draw[very thick, red] (1,4.25) to [out=270,in=0] (0.5,3.75) to [out=180,in=270] (0,4.25);
	\draw[very thick, red, directed=.60] (-2,4.75) to [out=270,in=180] (0,2.75) to [out=0,in=270] (2,4.75);
\draw[very thick, red, directed=.90] (1,5) to (1,1);
\draw[very thick, red, directed=.90] (-2,6) to (-2,2);
\draw[very thick, red, directed=.15] (0,1) to (0,5);
\draw[very thick, red, directed=.15] (2,2) to (2,6);
\draw[very thick, red, directed=.55] (2,4.75) to  (1,4.25);
\draw[very thick, red, directed=.55] (0,4.25) to  (-2,4.75);
	\draw[very thick, directed=.44] (3,2) to (-3,2);
	\draw[very thick, directed=.52] (3.5,1) to (-2.5,1);
	\draw[very thick, directed=.55] (2,2) to  (1,1);
	\draw[very thick, directed=.55] (0,1) to  (-2,2);
	\draw[very thick, directed=.44] (3,6) to (-3,6);
	\draw[very thick, directed=.52] (3.5,5) to (-2.5,5);
	\draw[very thick, directed=.55] (2,6) to  (1,5);
	\draw[very thick, directed=.55] (0,5) to  (-2,6);
\draw[very thick] (-3,6) to (-3,2);
\draw[very thick] (3,6) to (3,2);
\draw[very thick] (-2.5,5) to (-2.5,1);
\draw[very thick] (3.5,5) to (3.5,1);
\node[opacity=1] at (-0.7,3.4) {\tiny $_{i-1}$};
\node [opacity=1]  at (-0.8,3.75) {$\bullet$};
	\end{tikzpicture}
};
\endxy
}

\newcommand{\eqnR}{
=&  \sum_{ \substack{ p+q+r+s \\= i}} 
 \xy
(0,0)*{
\begin{tikzpicture} [scale=.5,fill opacity=0.2]
	\draw[very thick, directed=.55]  (2,0) to (1.25,0);
	\draw[very thick]  (1.25,0) to (0.5,0);
	\draw[very thick, directed=.55]  (-1.25,0) to (-2,0);
	\draw[very thick]  (-0.5,0) to (-1.25,0);
	\draw[very thick, directed=.55]  (2,1) to (1.25,1);
	\draw[very thick, directed=.55]  (-1.25,1) to (-2,1);
	\draw[very thick] (1.25,1) to (-1.25,1);
	\draw[very thick, directed=.55]  (0.5,1) to (-0.5,1);
	\draw[very thick, directed=.55]  (0.5,0) to (-0.5,0);
	\draw[very thick, directed=.55] (1.25,1) to [out=180,in=0](0.5,0) ;
	\draw[very thick, directed=.55] (-0.5,0) to [out=180,in=0](-1.25,1) ;
	\node[opacity=1] at (0,1.5) {$_{h_s}$};
	\node[opacity=1] at (1.5,0.45) {$_{h_q}$};
	\node[opacity=1] at (-1.5,0.45) {$_{h_p}$};
	\node[opacity=1] at (0,-0.5) {$_{(-1)^r e_r}$};
\end{tikzpicture}
};
\endxy
 =  \sum_{ \substack{ p+q+r \\= i}} 
 \xy
(0,0)*{
\begin{tikzpicture} [scale=.5,fill opacity=0.2]
	\draw[very thick, directed=.55]  (2,0) to (1.25,0);
	\draw[very thick]  (1.25,0) to (0.5,0);
	\draw[very thick, directed=.55]  (-1.25,0) to (-2,0);
	\draw[very thick]  (-0.5,0) to (-1.25,0);
	\draw[very thick, directed=.55]  (2,1) to (1.25,1);
	\draw[very thick, directed=.55]  (-1.25,1) to (-2,1);
	\draw[very thick] (1.25,1) to (-1.25,1);
	\draw[very thick, directed=.55]  (0.5,1) to (-0.5,1);
	\draw[very thick, directed=.55]  (0.5,0) to (-0.5,0);
	\draw[very thick, directed=.55] (1.25,1) to [out=180,in=0](0.5,0) ;
	\draw[very thick, directed=.55] (-0.5,0) to [out=180,in=0](-1.25,1) ;
	\node[opacity=1] at (-1.8,1.5) {$_{h_p}$};
	\node[opacity=1] at (1.5,0.45) {$_{h_q}$};
	\node[opacity=1] at (0,-0.5) {$_{(-1)^r e_r}$};
\end{tikzpicture}
};
\endxy
 =   \xy
(0,0)*{
\begin{tikzpicture} [scale=.5,fill opacity=0.2]
	\draw[very thick, directed=.55]  (2,0) to (1.25,0);
	\draw[very thick]  (1.25,0) to (0.5,0);
	\draw[very thick, directed=.55]  (-1.25,0) to (-2,0);
	\draw[very thick]  (-0.5,0) to (-1.25,0);
	\draw[very thick, directed=.55]  (2,1) to (1.25,1);
	\draw[very thick, directed=.55]  (-1.25,1) to (-2,1);
	\draw[very thick] (1.25,1) to (-1.25,1);
	\draw[very thick, directed=.55]  (0.5,1) to (-0.5,1);
	\draw[very thick, directed=.55]  (0.5,0) to (-0.5,0);
	\draw[very thick, directed=.55] (1.25,1) to [out=180,in=0](0.5,0) ;
	\draw[very thick, directed=.55] (-0.5,0) to [out=180,in=0](-1.25,1) ;
	\node[opacity=1] at (-1.8,1.5) {$_{h_i}$};
	\node[opacity=1] at (0,-0.5) {$_{\textcolor[rgb]{1,1,1}{1}}$};
\end{tikzpicture}};
\endxy +  \sum_{ \substack{ p+q+r \\= i\\p<i}} 
 \xy
(0,0)*{
\begin{tikzpicture} [scale=.5,fill opacity=0.2]
	\draw[very thick, directed=.55]  (2,0) to (1.25,0);
	\draw[very thick]  (1.25,0) to (0.5,0);
	\draw[very thick, directed=.55]  (-1.25,0) to (-2,0);
	\draw[very thick]  (-0.5,0) to (-1.25,0);
	\draw[very thick, directed=.55]  (2,1) to (1.25,1);
	\draw[very thick, directed=.55]  (-1.25,1) to (-2,1);
	\draw[very thick] (1.25,1) to (-1.25,1);
	\draw[very thick, directed=.55]  (0.5,1) to (-0.5,1);
	\draw[very thick, directed=.55]  (0.5,0) to (-0.5,0);
	\draw[very thick, directed=.55] (1.25,1) to [out=180,in=0](0.5,0) ;
	\draw[very thick, directed=.55] (-0.5,0) to [out=180,in=0](-1.25,1) ;
	\node[opacity=1] at (-1.8,1.5) {$_{h_p}$};
	\node[opacity=1] at (1.5,0.45) {$_{h_q}$};
	\node[opacity=1] at (0,-0.5) {$_{(-1)^r e_r}$};
\end{tikzpicture}
};
\endxy
\\
\sim &   \xy
(0,0)*{
\begin{tikzpicture} [scale=.5,fill opacity=0.2]
	\draw[very thick, directed=.55]  (2,0) to (1.25,0);
	\draw[very thick]  (1.25,0) to (0.5,0);
	\draw[very thick, directed=.55]  (-1.25,0) to (-2,0);
	\draw[very thick]  (-0.5,0) to (-1.25,0);
	\draw[very thick, directed=.55]  (2,1) to (1.25,1);
	\draw[very thick, directed=.55]  (-1.25,1) to (-2,1);
	\draw[very thick] (1.25,1) to (-1.25,1);
	\draw[very thick, directed=.55]  (0.5,1) to (-0.5,1);
	\draw[very thick, directed=.55]  (0.5,0) to (-0.5,0);
	\draw[very thick, directed=.55] (1.25,1) to [out=180,in=0](0.5,0) ;
	\draw[very thick, directed=.55] (-0.5,0) to [out=180,in=0](-1.25,1) ;
	\node[opacity=1] at (-1.8,1.5) {$_{h_i}$};
	\node[opacity=1] at (0,-0.6) {${\textcolor[rgb]{1,1,1}{(-1)^r}}$};
\end{tikzpicture}};
\endxy  + \sum_{ \substack{ p+q+r \\= i\\p<i}} 
 \xy
(0,0)*{
\begin{tikzpicture} [scale=.5,fill opacity=0.2]
	\draw[very thick, directed=.55]  (2,0) to (1.25,0);
	\draw[very thick]  (1.25,0) to (0.5,0);
	\draw[very thick, directed=.55]  (-1.25,0) to (-2,0);
	\draw[very thick]  (-0.5,0) to (-1.25,0);
	\draw[very thick, directed=.55]  (2,1) to (1.25,1);
	\draw[very thick, directed=.55]  (-1.25,1) to (-2,1);
	\draw[very thick] (1.25,1) to (-1.25,1);
	\draw[very thick, directed=.55]  (0.5,1) to (-0.5,1);
	\draw[very thick, directed=.55]  (0.5,0) to (-0.5,0);
	\draw[very thick, directed=.55] (1.25,1) to [out=180,in=0](0.5,0) ;
	\draw[very thick, directed=.55] (-0.5,0) to [out=180,in=0](-1.25,1) ;
	\node[opacity=1] at (1,-0.5) {$_{\textcolor[rgb]{1,1,1}{(-1)^r}}$};
	\node[opacity=1] at (1,1.5) {$_{\textcolor[rgb]{1,1,1}{(-1)^r}}$};
	\node[opacity=1] at (1.8,-0.5) {$_{h_p}$};
	\node[opacity=1] at (1.5,0.45) {$_{h_q}$};
	\node[opacity=1] at (0,-0.5) {$_{(-1)^r e_r}$};
\end{tikzpicture}
};
\endxy 
= \underbrace{\xy
(0,0)*{
\begin{tikzpicture} [scale=.5,fill opacity=0.2]
	\draw[very thick, directed=.55]  (2,0) to (1.25,0);
	\draw[very thick]  (1.25,0) to (0.5,0);
	\draw[very thick, directed=.55]  (-1.25,0) to (-2,0);
	\draw[very thick]  (-0.5,0) to (-1.25,0);
	\draw[very thick, directed=.55]  (2,1) to (1.25,1);
	\draw[very thick, directed=.55]  (-1.25,1) to (-2,1);
	\draw[very thick] (1.25,1) to (-1.25,1);
	\draw[very thick, directed=.55]  (0.5,1) to (-0.5,1);
	\draw[very thick, directed=.55]  (0.5,0) to (-0.5,0);
	\draw[very thick, directed=.55] (1.25,1) to [out=180,in=0](0.5,0) ;
	\draw[very thick, directed=.55] (-0.5,0) to [out=180,in=0](-1.25,1) ;
	\node[opacity=1] at (-1.8,1.5) {$_{h_i}$};
	\node[opacity=1] at (1,-0.5) {$_{\textcolor[rgb]{1,1,1}{(-1)^r}}$};
	\node[opacity=1] at (1,1.5) {$_{\textcolor[rgb]{1,1,1}{(-1)^r}}$};
\end{tikzpicture}};
\endxy}_{=\; h_i^l} -
 \underbrace{\xy
(0,0)*{
\begin{tikzpicture} [scale=.5,fill opacity=0.2]
	\draw[very thick, directed=.55]  (2,0) to (1.25,0);
	\draw[very thick]  (1.25,0) to (0.5,0);
	\draw[very thick, directed=.55]  (-1.25,0) to (-2,0);
	\draw[very thick]  (-0.5,0) to (-1.25,0);
	\draw[very thick, directed=.55]  (2,1) to (1.25,1);
	\draw[very thick, directed=.55]  (-1.25,1) to (-2,1);
	\draw[very thick] (1.25,1) to (-1.25,1);
	\draw[very thick, directed=.55]  (0.5,1) to (-0.5,1);
	\draw[very thick, directed=.55]  (0.5,0) to (-0.5,0);
	\draw[very thick, directed=.55] (1.25,1) to [out=180,in=0](0.5,0) ;
	\draw[very thick, directed=.55] (-0.5,0) to [out=180,in=0](-1.25,1) ;
	\node[opacity=1] at (1.8,-0.5) {$_{h_i}$};
	\node[opacity=1] at (1,-0.5) {$_{\textcolor[rgb]{1,1,1}{(-1)^r}}$};
	\node[opacity=1] at (1,1.5) {$_{\textcolor[rgb]{1,1,1}{(-1)^r}}$};
\end{tikzpicture}};
\endxy}_{=\;h_i^r}
+ \underbrace{\sum_{r=0}^i \xy
(0,0)*{
\begin{tikzpicture} [scale=.5,fill opacity=0.2]
	\draw[very thick, directed=.55]  (2,0) to (1.25,0);
	\draw[very thick]  (1.25,0) to (0.5,0);
	\draw[very thick, directed=.55]  (-1.25,0) to (-2,0);
	\draw[very thick]  (-0.5,0) to (-1.25,0);
	\draw[very thick, directed=.55]  (2,1) to (1.25,1);
	\draw[very thick, directed=.55]  (-1.25,1) to (-2,1);
	\draw[very thick] (1.25,1) to (-1.25,1);
	\draw[very thick, directed=.55]  (0.5,1) to (-0.5,1);
	\draw[very thick, directed=.55]  (0.5,0) to (-0.5,0);
	\draw[very thick, directed=.55] (1.25,1) to [out=180,in=0](0.5,0) ;
	\draw[very thick, directed=.55] (-0.5,0) to [out=180,in=0](-1.25,1) ;
	\node[opacity=1] at (1,-0.5) {$_{\textcolor[rgb]{1,1,1}{(-1)^r}}$};
	\node[opacity=1] at (1,1.5) {$_{\textcolor[rgb]{1,1,1}{(-1)^r}}$};
	\node[opacity=1] at (0,-0.5) {$_{(-1)^r e_r h_{i-r}}$};
\end{tikzpicture}
};
\endxy}_{=\;0}
}

\newcommand{\eqnS}{
\sum_{q+r=i-p}~ \xy
(0,-3)*{
\begin{tikzpicture} [scale=.5,fill opacity=0.2]
	\draw[very thick, directed=.55]  (2,0) to (1.25,0);
	\draw[very thick]  (1.25,0) to (0.5,0);
	\draw[very thick, directed=.55]  (-1.25,0) to (-2,0);
	\draw[very thick]  (-0.5,0) to (-1.25,0);
	\draw[very thick, directed=.55]  (2,1) to (1.25,1);
	\draw[very thick, directed=.55]  (-1.25,1) to (-2,1);
	\draw[very thick] (1.25,1) to (-1.25,1);
	\draw[very thick, directed=.55]  (0.5,1) to (-0.5,1);
	\draw[very thick, directed=.55]  (0.5,0) to (-0.5,0);
	\draw[very thick, directed=.55] (1.25,1) to [out=180,in=0](0.5,0) ;
	\draw[very thick, directed=.55] (-0.5,0) to [out=180,in=0](-1.25,1) ;
	\node[opacity=1] at (1.5,0.45) {$_{h_q}$};	
	\node[opacity=1] at (0,-0.5) {$_{(-1)^r e_r}$};
\end{tikzpicture}
};
\endxy
}

\newcommand{\eqnT}{
\sum_{r+q=c} (-1)^{r}~
\xy
(0,0)*{
\begin{tikzpicture} [scale=.5,fill opacity=0.2]
\draw[very thick, red, directed=.55] (1,5) to (1,0);
\draw[very thick, red, directed=.55] (-2,6) to (-2,1);
\draw[very thick, red, directed=.55] (0,0) to (0,5);
\draw[very thick, red, directed=.55] (2,1) to (2,6);
\draw[very thick, red, rdirected=.55] (2,2.25) to  (1,1.25);
\draw[very thick, red, rdirected=.55] (0,1.25) to  (-2,2.25);
	\draw[very thick, directed=.55] (3,1) to (-3,1);
	\draw[very thick, directed=.55] (3.5,0) to (-1.5,0);
	\draw[very thick, directed=.55] (2,1) to  (1,0);
	\draw[very thick, directed=.55] (0,0) to  (-2,1);
	\draw[very thick, directed=.55] (3,6) to (-3,6);
	\draw[very thick, directed=.55] (3.5,5) to (-1.5,5);
	\draw[very thick, directed=.55] (2,6) to  (1,5);
	\draw[very thick, directed=.55] (0,5) to  (-2,6);
	\draw[very thick, red, rdirected=.55] (1,1.25) to [out=90,in=0] (0.5,1.75) to [out=180,in=90] (0,1.25);
	\draw[very thick, red, rdirected=.55] (-2,2.25) to [out=90,in=180] (0,4.25) to [out=0,in=90] (2,2.25);
\draw[very thick] (-3,6) to (-3,1);
\draw[very thick] (3,6) to (3,1);
\draw[very thick] (-1.5,5) to (-1.5,0);
\draw[very thick] (3.5,5) to (3.5,0);
\node[opacity=1] at (0.5,0.5) {${e_r}$};
\node[opacity=1] at (1.4,1) {${h_q}$};
	\end{tikzpicture}
};
\endxy  \quad 
= \sum_{r_1+r_2+q_1+q_2=c} (-1)^{r_1+r_2}~
\xy
(0,0)*{
\begin{tikzpicture} [scale=.5,fill opacity=0.2]
\draw[very thick, red, directed=.55] (1,5) to (1,0);
\draw[very thick, red, directed=.55] (-2,6) to (-2,1);
\draw[very thick, red, directed=.55] (0,0) to (0,5);
\draw[very thick, red, directed=.55] (2,1) to (2,6);
\draw[very thick, red, rdirected=.55] (2,2.25) to  (1,1.25);
\draw[very thick, red, rdirected=.55] (0,1.25) to  (-2,2.25);
	\draw[very thick, directed=.55] (3,1) to (-3,1);
	\draw[very thick, directed=.55] (3.5,0) to (-1.5,0);
	\draw[very thick, directed=.55] (2,1) to  (1,0);
	\draw[very thick, directed=.55] (0,0) to  (-2,1);
	\draw[very thick, directed=.55] (3,6) to (-3,6);
	\draw[very thick, directed=.55] (3.5,5) to (-1.5,5);
	\draw[very thick, directed=.55] (2,6) to  (1,5);
	\draw[very thick, directed=.55] (0,5) to  (-2,6);
	\draw[very thick, red, rdirected=.55] (1,1.25) to [out=90,in=0] (0.5,1.75) to [out=180,in=90] (0,1.25);
	\draw[very thick, red, rdirected=.55] (-2,2.25) to [out=90,in=180] (0,4.25) to [out=0,in=90] (2,2.25);
\draw[very thick] (-3,6) to (-3,1);
\draw[very thick] (3,6) to (3,1);
\draw[very thick] (-1.5,5) to (-1.5,0);
\draw[very thick] (3.5,5) to (3.5,0);
\node[opacity=1] at (0.5,4.5) {${e_{r_2}}$};
\node[opacity=1] at (0.5,2.5) {${e_{r_1}}$};
\node[opacity=1] at (0.5,3.5) {${h_{q_1}}$};
\node[opacity=1] at (1.4,4.5) {${h_{q_2}}$};
	\end{tikzpicture}
};
\endxy
}

\newcommand{\eqnU}{
= \sum_{a=0}^c ~\sum_{r_2+q_2=c-a} (-1)^{r_2} \sum_{r_1+q_1=a} (-1)^{r_1}~
\xy
(0,0)*{
\begin{tikzpicture} [scale=.5,fill opacity=0.2]
\draw[very thick, red, directed=.55] (1,5) to (1,0);
\draw[very thick, red, directed=.55] (-2,6) to (-2,1);
\draw[very thick, red, directed=.55] (0,0) to (0,5);
\draw[very thick, red, directed=.55] (2,1) to (2,6);
\draw[very thick, red, rdirected=.55] (2,2.25) to  (1,1.25);
\draw[very thick, red, rdirected=.55] (0,1.25) to  (-2,2.25);
	\draw[very thick, directed=.55] (3,1) to (-3,1);
	\draw[very thick, directed=.55] (3.5,0) to (-1.5,0);
	\draw[very thick, directed=.55] (2,1) to  (1,0);
	\draw[very thick, directed=.55] (0,0) to  (-2,1);
	\draw[very thick, directed=.55] (3,6) to (-3,6);
	\draw[very thick, directed=.55] (3.5,5) to (-1.5,5);
	\draw[very thick, directed=.55] (2,6) to  (1,5);
	\draw[very thick, directed=.55] (0,5) to  (-2,6);
	\draw[very thick, red, rdirected=.55] (1,1.25) to [out=90,in=0] (0.5,1.75) to [out=180,in=90] (0,1.25);
	\draw[very thick, red, rdirected=.55] (-2,2.25) to [out=90,in=180] (0,4.25) to [out=0,in=90] (2,2.25);
\draw[very thick] (-3,6) to (-3,1);
\draw[very thick] (3,6) to (3,1);
\draw[very thick] (-1.5,5) to (-1.5,0);
\draw[very thick] (3.5,5) to (3.5,0);
\node[opacity=1] at (0.5,4.5) {${e_{r_2}}$};
\node[opacity=1] at (0.5,2.5) {${e_{r_1}}$};
\node[opacity=1] at (0.5,3.5) {${h_{q_1}}$};
\node[opacity=1] at (1.4,4.5) {${h_{q_2}}$};
	\end{tikzpicture}
};
\endxy
\quad =  ~\sum_{r_2+q_2=c} (-1)^{r_2}~
\xy
(0,0)*{
\begin{tikzpicture} [scale=.5,fill opacity=0.2]
\draw[very thick, red, directed=.55] (1,5) to (1,0);
\draw[very thick, red, directed=.55] (-2,6) to (-2,1);
\draw[very thick, red, directed=.55] (0,0) to (0,5);
\draw[very thick, red, directed=.55] (2,1) to (2,6);
\draw[very thick, red, rdirected=.55] (2,2.25) to  (1,1.25);
\draw[very thick, red, rdirected=.55] (0,1.25) to  (-2,2.25);
	\draw[very thick, directed=.55] (3,1) to (-3,1);
	\draw[very thick, directed=.55] (3.5,0) to (-1.5,0);
	\draw[very thick, directed=.55] (2,1) to  (1,0);
	\draw[very thick, directed=.55] (0,0) to  (-2,1);
	\draw[very thick, directed=.55] (3,6) to (-3,6);
	\draw[very thick, directed=.55] (3.5,5) to (-1.5,5);
	\draw[very thick, directed=.55] (2,6) to  (1,5);
	\draw[very thick, directed=.55] (0,5) to  (-2,6);
	\draw[very thick, red, rdirected=.55] (1,1.25) to [out=90,in=0] (0.5,1.75) to [out=180,in=90] (0,1.25);
	\draw[very thick, red, rdirected=.55] (-2,2.25) to [out=90,in=180] (0,4.25) to [out=0,in=90] (2,2.25);
\draw[very thick] (-3,6) to (-3,1);
\draw[very thick] (3,6) to (3,1);
\draw[very thick] (-1.5,5) to (-1.5,0);
\draw[very thick] (3.5,5) to (3.5,0);
\node[opacity=1] at (0.5,4.5) {${e_{r_2}}$};
\node[opacity=1] at (1.4,4.5) {${h_{q_2}}$};
	\end{tikzpicture}
};
\endxy
}

\newcommand{\eqnAA}{
W_k:= \squareweb[1]{}{}{\quad k}{A}{B}{B}{A}
}

\newcommand{\squarewebA}{
\xy
(0,0)*{
\begin{tikzpicture} [scale=.5,fill opacity=0.2]
	\draw[very thick, directed=.55]  (2,0) to (1.25,0);
	\draw[very thick]  (1.25,0) to (0.5,0);
	\draw[very thick, directed=.55]  (-1.25,0) to (-2,0);
	\draw[very thick]  (-0.5,0) to (-1.25,0);
	\draw[very thick, directed=.55]  (2,1) to (1.25,1);
	\draw[very thick, directed=.55]  (-1.25,1) to (-2,1);
	\draw[very thick] (1.25,1) to (-1.25,1);
	\draw[very thick, directed=.55]  (0.5,1) to (-0.5,1);
	\draw[very thick, directed=.55]  (0.5,0) to (-0.5,0);
	\draw[very thick, directed=.55] (1.25,1) to [out=180,in=0](0.5,0) ;
	\draw[very thick, directed=.55] (-0.5,0) to [out=180,in=0](-1.25,1) ;
	\node[opacity=1] at (1.8,1.5) {${_A}$};
	\node[opacity=1] at (0,1.5) {$_{A\cap B}$};
	\node[opacity=1] at (0,-0.5) {$_{A \uplus B}$};
	\node[opacity=1] at (1.8,-0.5) {${_B}$};
	\node[opacity=1] at (-1.8,1.5) {${_B}$};
	\node[opacity=1] at (-1.8,-0.5) {${_A}$};
	\node[opacity=1] at (1.75,0.45) {$_{A\setminus B}$};
	\node[opacity=1] at (-1.75,0.45) {$_{B\setminus A}$};
\end{tikzpicture}
};
\endxy
}

\newcommand{\eqnAB}{
W_{k_{max}}=\squarewebD{} =
\squarewebD{A \uplus B}  \quad, \quad W_{k_{min}}=\squareweb[1]{}{}{k_{min}}{A}{B}{B}{A} =
\squarewebA 
}

\newcommand{\squarewebD}[1]{
\xy
(0,0)*{
\begin{tikzpicture} [scale=.5,fill opacity=0.2]
	\draw[very thick, directed=.55]  (2,0) to (1.25,0);
	\draw[very thick]  (1.25,0) to (0.5,0);
	\draw[very thick, directed=.55]  (-1.25,0) to (-2,0);
	\draw[very thick]  (-0.5,0) to (-1.25,0);
	\draw[very thick, directed=.55]  (2,1) to (1.25,1);
	\draw[very thick, directed=.55]  (-1.25,1) to (-2,1);
	\draw[very thick, directed=.55]  (0.5,0) to (-0.5,0);
	\draw[very thick] (1.25,1) to [out=180,in=0](0.5,0) ;
	\draw[very thick] (-0.5,0) to [out=180,in=0](-1.25,1)  ;
	\node[opacity=1] at (1.8,1.5) {${_A}$};
	\node[opacity=1] at (0,-0.5) {$_{#1}$};
	\node[opacity=1] at (1.8,-0.5) {${_B}$};
	\node[opacity=1] at (-1.8,1.5) {${_B}$};
	\node[opacity=1] at (-1.8,-0.5) {${_A}$};
\end{tikzpicture}
};
\endxy
}

\newcommand{\squarewebB}{
\xy
(0,0)*{
\begin{tikzpicture} [scale=.5,fill opacity=0.2]
	\draw[very thick, directed=.55]  (2,0) to (1.25,0);
	\draw[very thick]  (1.25,0) to (0.5,0);
	\draw[very thick, directed=.55]  (-1.25,0) to (-2,0);
	\draw[very thick]  (-0.5,0) to (-1.25,0);
	\draw[very thick, directed=.55]  (2,1) to (1.25,1);
	\draw[very thick, directed=.55]  (-1.25,1) to (-2,1);
	\draw[very thick] (1.25,1) to (-1.25,1);
	\draw[very thick, directed=.55]  (0.5,1) to (-0.5,1);
	\draw[very thick, directed=.55]  (0.5,0) to (-0.5,0);
	\draw[very thick, directed=.55] (1.25,1) to [out=180,in=0](0.5,0) ;
	\draw[very thick, directed=.55] (-0.5,0) to [out=180,in=0](-1.25,1) ;
	\node[opacity=1] at (0,1.5) {$_{\{\lambda\}}$};
	\node[opacity=1] at (0,-0.5) {$_{ (A\uplus B) \setminus\{\lambda\}}$};
	\node[opacity=1] at (2.1,0.45) {$_{A\setminus \{\lambda\}}$};
	\node[opacity=1] at (-2.1,0.45) {$_{B\setminus \{\lambda\}}$};
\end{tikzpicture}
};
\endxy
}

\newcommand{\eqnAC}{
W_{k_{max}-1}=\squareweb[1]{}{}{ k_{max}\!-\! 1\quad}{A}{B}{B}{A} =~ \bigoplus_{\lambda}
\squarewebB
}

\newcommand{\squarewebC}{W_C:=
\xy
(0,0)*{
\begin{tikzpicture} [scale=.5,fill opacity=0.2]
	\draw[very thick, directed=.55]  (2,0) to (1.25,0);
	\draw[very thick]  (1.25,0) to (0.5,0);
	\draw[very thick, directed=.55]  (-1.25,0) to (-2,0);
	\draw[very thick]  (-0.5,0) to (-1.25,0);
	\draw[very thick, directed=.55]  (2,1) to (1.25,1);
	\draw[very thick, directed=.55]  (-1.25,1) to (-2,1);
	\draw[very thick] (1.25,1) to (-1.25,1);
	\draw[very thick, directed=.55]  (0.5,1) to (-0.5,1);
	\draw[very thick, directed=.55]  (0.5,0) to (-0.5,0);
	\draw[very thick, directed=.55] (1.25,1) to [out=180,in=0](0.5,0) ;
	\draw[very thick, directed=.55] (-0.5,0) to [out=180,in=0](-1.25,1) ;
	\node[opacity=1] at (0,1.5) {$_{C}$};
	\node[opacity=1] at (0,-0.5) {$_{(A \uplus B) \setminus C}$};
	\node[opacity=1] at (2.1,0.45) {$_{A\setminus C}$};
	\node[opacity=1] at (-2.1,0.45) {$_{B\setminus C}$};
\end{tikzpicture}
};
\endxy
}

\newcommand{\eqnAE}{
\xy
(0,0)*{
\begin{tikzpicture} [scale=.5,fill opacity=0.2]
	\path[fill=green] (-2.5,4) to [out=0,in=135] (.75,3.25) to [out=270,in=90] (-.75,.5) to(-.75,-2.5)
		 to [out=1-35,in=0] (-2.5,-2) to (-2.5,1);
	\path[fill=green] (-.75,2.5) to [out=270,in=125] (-.35,1.5) to [out=45,in=270] (.75,3.25) 
		to [out=225,in=0] (-.75,2.5);
	\path[fill=green] (-.75,-2.5) to (-.75,.5) to [out=90,in=235] (-.35,1.5) to [out=315,in=90] (.75,-.25)  to (.75,-3.25)
		to [out=135,in=0] (-.75,-2.5);	
	\path[fill=green] (-2,3) to [out=0,in=135] (-.75,2.5) to [out=270,in=125] (-.35,1.5) 
		to [out=235,in=90] (-.75,.5) to (-0.75,-2.5)to [out=225,in=0] (-2,-3);
	\path[fill=green] (-1.5,2) to [out=0,in=225] (-.75,2.5) to [out=270,in=90] (.75,-.25) to (0.75,-3.25)
		to [out=225,in=0] (-1.5,-4);
	\path[fill=green] (2,3) to [out=180,in=0] (.75,3.25) to [out=270,in=45] (-.35,1.5) 
		to [out=315,in=90] (.75,-.25) to (.75,-3.25) to [out=0,in=180] (2,-3);				
	\path[fill=red]	(-1,1) to [out=0,in=180] (2.5,2) to (2.5,-0.75) to [out=270,in=45] (2,-1.5) to [out=180,in=45] (.75,-2) to [out=225,in=0] (-1.5,-3.5) to [out=45,in=270] (-1,-2.75);
	\draw[very thick, directed=.55] (2,-3) to [out=180,in=0] (.75,-3.25);
	\draw[very thick, directed=.35] (.75,-3.25) to [out=135,in=0] (-.75,-2.5);
	\draw[very thick, directed=.55] (.75,-3.25) to [out=225,in=0] (-1.5,-4);
	\draw[very thick, directed=.55]  (-.75,-2.5) to [out=225,in=0] (-2,-3);
	\draw[very thick, directed=.4]  (-.75,-2.5) to [out=135,in=0] (-2.5,-2);	
	\draw[very thick, red, rdirected=.65] (-.75,2.5) to [out=270,in=90] (.75,-.25) to (.75,-3.25);
	\draw[very thick, red, rdirected=.75] (.75,3.25) to [out=270,in=90] (-.75,.5) to (-.75,-2.5);
	\draw[very thick] (-1.5,-4) -- (-1.5,2);	
	\draw[very thick] (-2,-3) -- (-2,3);
	\draw[very thick] (-2.5,-2) -- (-2.5,4);	
	\draw[very thick] (2,3) -- (2,-3);
	\draw[very thick] (-1.5,-3.5) to [out=45,in=270](-1,-2.75) -- (-1,1);
	\draw[very thick] (2,-1.5) to [out=45,in=270](2.5,-.75) -- (2.5,2);
	\draw[very thick, directed=.55] (2,3) to [out=180,in=0] (.75,3.25);
	\draw[very thick, directed=.55] (.75,3.25) to [out=225,in=0] (-.75,2.5);
	\draw[very thick, directed=.55] (.75,3.25) to [out=135,in=0] (-2.5,4);
	\draw[very thick, directed=.55] (-.75,2.5) to [out=135,in=0] (-2,3);
	\draw[very thick, directed=.65] (-.75,2.5) to [out=225,in=0] (-1.5,2);
	\draw[very thick,black , directed=.55] (2.5,2) to [out=180,in=0] (-1,1);
	\draw[very thick, red] (2,-1.5) to [out=180,in=45] (.75,-2);
	\draw[very thick, red , rdirected=.75] (.75,-2) to [out=225,in=0] (-1.5,-3.5);
\end{tikzpicture}
};
\endxy
 \quad =\quad 
 \xy
(0,0)*{
\begin{tikzpicture} [scale=.5,fill opacity=0.2]
	\path[fill=green] (-2.5,4) to [out=0,in=135] (.75,3.25) to (0.75,0.25) to [out=270,in=90] (-.75,-2.5)
		 to [out=135,in=0] (-2.5,-2) to (-2.5,1);
	\path[fill=green] (-.75,2.5) to (-0.75,-0.5) to [out=270,in=125] (-.35,-1.5) to [out=45,in=270](.75,0.25) to  (.75,3.25) 
		to [out=225,in=0] (-.75,2.5);
	\path[fill=green] (-.75,-2.5)  to [out=90,in=235] (-.35,-1.5) to [out=315,in=90] (.75,-3.25) to [out=135,in=0] (-.75,-2.5);	
	\path[fill=green] (-2,3) to [out=0,in=135] (-.75,2.5) to  (-.75,-0.5)  to [out=270,in=125] (-.35,-1.5) to [out=235,in=90] (-.75,-2.5) to [out=225,in=0] (-2,-3);
	\path[fill=green] (-1.5,2) to [out=0,in=225] (-.75,2.5) to (-0.75,-0.5) to [out=270,in=90] (.75,-3.25) to [out=225,in=0] (-1.5,-4) to (-1.5,2);
	\path[fill=green] (2,3) to [out=180,in=0] (.75,3.25) to (.75,0.25) to [out=270,in=45] (-.35,-1.5) 
		to [out=315,in=90] (.75,-3.25)  to [out=0,in=180] (2,-3);				
	\path[fill=red]	(-1,1) to [out=0,in=180] (2.5,2)  to [out=270,in=45] (2,1.25) to [out=180,in=45] (.75,0.75) to [out=225,in=0] (-1.5,-.75) to [out=45,in=270] (-1,0);
	\draw[very thick, directed=.55] (2,-3) to [out=180,in=0] (.75,-3.25);
	\draw[very thick, directed=.55] (.75,-3.25) to [out=135,in=0] (-.75,-2.5);
	\draw[very thick, directed=.55] (.75,-3.25) to [out=225,in=0] (-1.5,-4);
	\draw[very thick, directed=.45]  (-.75,-2.5) to [out=225,in=0] (-2,-3);
	\draw[very thick, directed=.35]  (-.75,-2.5) to [out=135,in=0] (-2.5,-2);	
	\draw[very thick, red, rdirected=.57] (-.75,2.5) to (-0.75,-0.5)to [out=270,in=90] (.75,-3.25);
	\draw[very thick, red, rdirected=.65] (.75,3.25) to (0.75,0.25) to [out=270,in=90] (-.75,-2.5);
	\draw[very thick, red] (2,1.25) to [out=180,in=45] (.75,.75);
	\draw[very thick, red , rdirected=.45] (.75,.75) to [out=225,in=0] (-1.5,-.75);
	\draw[very thick] (-1.5,-4) -- (-1.5,2);	
	\draw[very thick] (-2,-3) -- (-2,3);
	\draw[very thick] (-2.5,-2) -- (-2.5,4);	
	\draw[very thick] (2,3) -- (2,-3);
	\draw[very thick] (-1.5,-.75) to [out=45,in=270](-1,0) -- (-1,1);
	\draw[very thick] (2,1.25) to [out=45,in=270](2.5,2);
	\draw[very thick, directed=.55] (2,3) to [out=180,in=0] (.75,3.25);
	\draw[very thick, directed=.55] (.75,3.25) to [out=225,in=0] (-.75,2.5);
	\draw[very thick, directed=.55] (.75,3.25) to [out=135,in=0] (-2.5,4);
	\draw[very thick, directed=.55] (-.75,2.5) to [out=135,in=0] (-2,3);
	\draw[very thick, directed=.65] (-.75,2.5) to [out=225,in=0] (-1.5,2);
	\draw[very thick,black , directed=.45] (2.5,2) to [out=180,in=0] (-1,1);
\end{tikzpicture}
};
\endxy\quad, \quad 
\xy
(0,0)*{
\begin{tikzpicture} [scale=.5,fill opacity=0.2]
	\path[fill=green] (-2.5,4) to [out=0,in=135] (.75,3.25) to [out=270,in=90] (-.75,.5) to(-.75,-2.5)
		 to [out=135,in=0] (-2.5,-2) to (-2.5,1);
	\path[fill=green] (-.75,2.5) to [out=270,in=125] (-.35,1.5) to [out=45,in=270] (.75,3.25) 
		to [out=225,in=0] (-.75,2.5);
	\path[fill=green] (-.75,-2.5) to (-.75,.5) to [out=90,in=235] (-.35,1.5) to [out=315,in=90] (.75,-.25)  to (.75,-3.25)
		to [out=135,in=0] (-.75,-2.5);	
	\path[fill=green] (-2,3) to [out=0,in=135] (-.75,2.5) to [out=270,in=125] (-.35,1.5) 
		to [out=235,in=90] (-.75,.5) to (-0.75,-2.5)to [out=225,in=0] (-2,-3);
	\path[fill=green] (-1.5,2) to [out=0,in=225] (-.75,2.5) to [out=270,in=90] (.75,-.25) to (0.75,-3.25)
		to [out=225,in=0] (-1.5,-4);
	\path[fill=green] (2,3) to [out=180,in=0] (.75,3.25) to [out=270,in=45] (-.35,1.5) 
		to [out=315,in=90] (.75,-.25) to (.75,-3.25) to [out=0,in=180] (2,-3);				
	;\path[fill=red]	(-1,-5) to (-1,-1.5) to [out=90,in=315]  (-1.5,-.75) to [out=0,in=225] (.75,.75) to [out=45,in=180]  (2,1.25)    to [out=315,in=90](2.5,0.5) -- (2.5,-4) to [out=180,in=0] (-1,-5);
	\draw[very thick, directed=.55] (2,-3) to [out=180,in=0] (.75,-3.25);
	\draw[very thick, directed=.55] (.75,-3.25) to [out=135,in=0] (-.75,-2.5);
	\draw[very thick, directed=.55] (.75,-3.25) to [out=225,in=0] (-1.5,-4);
	\draw[very thick, directed=.55]  (-.75,-2.5) to [out=225,in=0] (-2,-3);
	\draw[very thick, directed=.4]  (-.75,-2.5) to [out=135,in=0] (-2.5,-2);	
	\draw[very thick, red, rdirected=.65] (-.75,2.5) to [out=270,in=90] (.75,-.25) to (.75,-3.25);
	\draw[very thick, red, rdirected=.75] (.75,3.25) to [out=270,in=90] (-.75,.5) to (-.75,-0.45);
	\draw[very thick, red, rdirected=.75]  (-.75,-0.75) to (-.75,-2.5);
	\draw[very thick] (-1.5,-4) -- (-1.5,2);	
	\draw[very thick] (-2,-3) -- (-2,3);
	\draw[very thick] (-2.5,-2) -- (-2.5,4);	
	\draw[very thick] (2,3) -- (2,-3);
\draw[very thick] (-1.5,-.75) to [out=315,in=90](-1,-1.5) -- (-1,-5);
	\draw[very thick] (2,1.25) to [out=315,in=90](2.5,0.5) -- (2.5,-4);
	\draw[very thick, directed=.55] (2,3) to [out=180,in=0] (.75,3.25);
	\draw[very thick, directed=.55] (.75,3.25) to [out=225,in=0] (-.75,2.5);
	\draw[very thick, directed=.55] (.75,3.25) to [out=135,in=0] (-2.5,4);
	\draw[very thick, directed=.55] (-.75,2.5) to [out=135,in=0] (-2,3);
	\draw[very thick, directed=.65] (-.75,2.5) to [out=225,in=0] (-1.5,2);
	\draw[very thick,black , directed=.55] (2.5,-4) to [out=180,in=0] (-1,-5);
	\draw[very thick, red] (2,1.25) to [out=180,in=45] (.75,.75);
	\draw[very thick, red , directed=.45] (.75,.75) to [out=225,in=0] (-1.5,-.75);
\end{tikzpicture}
};
\endxy
 \quad =\quad 
 \xy
(0,0)*{
\begin{tikzpicture} [scale=.5,fill opacity=0.2]
	\path[fill=green] (-2.5,4) to [out=0,in=135] (.75,3.25) to (0.75,0.25) to [out=270,in=90] (-.75,-2.5)
		 to [out=135,in=0] (-2.5,-2) to (-2.5,1);
	\path[fill=green] (-.75,2.5) to (-0.75,-0.5) to [out=270,in=125] (-.35,-1.5) to [out=45,in=270](.75,0.25) to  (.75,3.25) 
		to [out=225,in=0] (-.75,2.5);
	\path[fill=green] (-.75,-2.5)  to [out=90,in=235] (-.35,-1.5) to [out=315,in=90] (.75,-3.25) to [out=135,in=0] (-.75,-2.5);	
	\path[fill=green] (-2,3) to [out=0,in=135] (-.75,2.5) to  (-.75,-0.5)  to [out=270,in=125] (-.35,-1.5) to [out=235,in=90] (-.75,-2.5) to [out=225,in=0] (-2,-3);
	\path[fill=green] (-1.5,2) to [out=0,in=225] (-.75,2.5) to (-0.75,-0.5) to [out=270,in=90] (.75,-3.25) to [out=225,in=0] (-1.5,-4) to (-1.5,2);
	\path[fill=green] (2,3) to [out=180,in=0] (.75,3.25) to (.75,0.25) to [out=270,in=45] (-.35,-1.5) 
		to [out=315,in=90] (.75,-3.25)  to [out=0,in=180] (2,-3);				
	\path[fill=red]	(-1,-5) to (-1,-1.5) to [out=90,in=315]  (-1.5,-.75) to [out=0,in=225] (.75,.75) to [out=45,in=180]  (2,1.25)    to [out=315,in=90](2.5,0.5) -- (2.5,-4) to [out=180,in=0] (-1,-5);
	\draw[very thick, directed=.55] (2,-3) to [out=180,in=0] (.75,-3.25);
	\draw[very thick, directed=.55] (.75,-3.25) to [out=135,in=0] (-.75,-2.5);
	\draw[very thick, directed=.55] (.75,-3.25) to [out=225,in=0] (-1.5,-4);
	\draw[very thick, directed=.55]  (-.75,-2.5) to [out=225,in=0] (-2,-3);
	\draw[very thick, directed=.4]  (-.75,-2.5) to [out=135,in=0] (-2.5,-2);	
	\draw[very thick, red, rdirected=.57] (-.75,2.5) to (-0.75,-0.5)to [out=270,in=90] (.75,-3.25);
	\draw[very thick, red, rdirected=.65] (.75,3.25) to (0.75,0.25) to [out=270,in=90] (-.75,-2.5);
	\draw[very thick, red] (2,1.25) to [out=180,in=45] (.75,.75);
	\draw[very thick, red , directed=.45] (.75,.75) to [out=225,in=0] (-1.5,-.75);
	\draw[very thick] (-1.5,-4) -- (-1.5,2);	
	\draw[very thick] (-2,-3) -- (-2,3);
	\draw[very thick] (-2.5,-2) -- (-2.5,4);	
	\draw[very thick] (2,3) -- (2,-3);
	\draw[very thick] (-1.5,-.75) to [out=315,in=90](-1,-1.5) -- (-1,-5);
	\draw[very thick] (2,1.25) to [out=315,in=90](2.5,0.5) -- (2.5,-4);
	\draw[very thick, directed=.55] (2,3) to [out=180,in=0] (.75,3.25);
	\draw[very thick, directed=.55] (.75,3.25) to [out=225,in=0] (-.75,2.5);
	\draw[very thick, directed=.55] (.75,3.25) to [out=135,in=0] (-2.5,4);
	\draw[very thick, directed=.55] (-.75,2.5) to [out=135,in=0] (-2,3);
	\draw[very thick, directed=.65] (-.75,2.5) to [out=225,in=0] (-1.5,2);
	\draw[very thick,black , directed=.55] (2.5,-4) to [out=180,in=0] (-1,-5);
\end{tikzpicture}
};
\endxy
}

\newcommand{\eqnAF}{
\xy
(0,0)*{
\begin{tikzpicture} [scale=0.5,fill opacity=0.2]
\path[fill=blue]  (1,2.5) to [out=135,in=0] (0,3) to [out=180, in=45] (-1,2.5) to [out=270,in=180] (0,1) to [out=0,in=270] (1,2.5) ;
\path[fill=red]  (1,2.5) to [out=225,in=0] (0,2) to [out=180, in=315] (-1,2.5) to [out=270,in=180] (0,1) to [out=0,in=270] (1,2.5) ;
\path[fill=red]  (-1.5,1.5) to (2.5,1.5) to (2.5,0.75) to [out=270,in=45] (2.25,0) to (-1.75,-2) to [out=45,in=270] (-1.5,-1.25) to (-1.5,1.5) ;
\path[fill=green]   (-1,2.5) to [out=270,in=180] (0,1) to [out=0,in=270] (1,2.5) to (2, 2.5) to (2,0.75) to [out=270,in=135] (2.25,0) to (2.25,-2.5) to (-1.75 , -2.5) to (-1.75,-2) to [out=135,in=270] (-2,-1.25) to (-2, 2.5) to (-1,2.5);
	\draw[very thick, directed=.55] (2.25,-2.5) to (-1.75,-2.5);
	\draw[very thick, red, directed=.65] (-1,2.5) to [out=270,in=180] (0,1) to [out=0,in=270] (1,2.5);
	\draw[very thick, red, directed=.65] (-1.75,-2) to (2.25,0);
	\draw[very thick] (2.25,-2.5) to (2.25,0);
	\draw[very thick] (2.25,0) to [out=45,in=270](2.5,0.75) to (2.5,1.5);
	\draw[very thick] (2.25,0) to [out=135,in=270](2,0.75) to (2,2.5);
	\draw[very thick] (-1.75,-2.5) to (-1.75,-2);
	\draw[very thick] (-1.75,-2) to [out=45,in=270](-1.5,-1.25) to (-1.5,1.5);
	\draw[very thick] (-1.75,-2) to [out=135,in=270](-2,-1.25) to (-2,2.5);
\draw[very thick, directed=.65] (2,2.5) to (1,2.5);
	\draw[very thick, directed=.65] (-1,2.5) to  (-2,2.5);
	\draw[very thick, directed=.65] (1,2.5) to [out=135,in=0] (0,3) to [out=180, in=45] (-1,2.5);
	\draw[very thick, directed=.65] (1,2.5) to [out=225,in=0] (0,2) to [out=180,in=315] (-1,2.5);
	\draw[very thick, directed=.65] (2.5,1.5) to  (-1.5,1.5);
\end{tikzpicture}
};
\endxy \quad = \quad \xy
(0,0)*{
\begin{tikzpicture} [scale=0.5,fill opacity=0.2]
\path[fill=blue]  (1,2.5) to [out=135,in=0] (0,3) to [out=180, in=45]  (-1,2.5) to (-1,0) to [out=270,in=135] (-0.75,-0.75)to  [out=270,in=180] (0.25,-1.75) to [out=0,in=270](1.25,-0.75)to (1.25,0.25) to  [out=135,in=270] (1,1) to  (1,2.5) ;
\path[fill=red]  (1,2.5) to [out=225,in=0] (0,2) to [out=180, in=315]  (-1,2.5) to (-1,0) to [out=270,in=135] (-0.75,-0.75)to  [out=270,in=180] (0.25,-1.75) to [out=0,in=270](1.25,-0.75)to (1.25,0.25) to  [out=135,in=270] (1,1) to  (1,2.5) ;
\path[fill=red]  (-1.5,1.5) to (2.5,1.5) to (2.5,1.5) to [out=270,in=45] (2.25,0.75) to (1.25,0.25) to [out=247.5,in=0]  (-0.75,-0.75) to (-1.75,-1.25) to [out=45,in=270] (-1.5,-0.5) to (-1.5,1.5) ;
\path[fill=green]   (-1,2.5) to (-1,0) to [out=270,in=135] (-0.75,-0.75)to  [out=270,in=180] (0.25,-1.75) to [out=0,in=270](1.25,-0.75)to (1.25,0.25) to  [out=135,in=270] (1,1) to  (1,2.5) to (2, 2.5) to (2,1.5) to [out=270,in=135] (2.25,0.75) to (2.25,-2.5) to (-1.75 , -2.5) to (-1.75,-1.25) to [out=135,in=270] (-2,-0.5) to (-2, 2.5) to (-1,2.5);
	\draw[very thick, directed=.55] (2.25,-2.5) to (-1.75,-2.5);
	\draw[very thick, red, directed=.65] (-1,2.5) to (-1,0) to [out=270,in=135] (-0.75,-0.75);
	\draw[very thick, red, directed=.65](-0.75,-0.75) to  [out=270,in=180] (0.25,-1.75) to [out=0,in=270](1.25,-0.75);
	\draw[very thick, red, directed=.65](1.25,-0.75) to (1.25,0.25) to  [out=135,in=270] (1,1) to  (1,2.5);
	\draw[very thick, red, directed=.65](-1.75,-1.25) to (-0.75,-0.75);
	\draw[very thick, red, directed=.65](1.25,0.25) to (2.25,0.75);
	\draw[very thick, red, directed=.65]  (-0.75,-0.75) to [out=0,in=247.5] (1.25,0.25);
	\draw[very thick] (2.25,-2.5) to (2.25,0.75);
	\draw[very thick] (2.25,0.75) to [out=45,in=270](2.5,1.5);
	\draw[very thick] (2.25,0.75) to [out=135,in=270](2,1.5) to (2,2.5);
	\draw[very thick] (-1.75,-2.5) to (-1.75,-1.25);
	\draw[very thick] (-1.75,-1.25) to [out=45,in=270](-1.5,-0.5) to (-1.5,1.5);
	\draw[very thick] (-1.75,-1.25) to [out=135,in=270](-2,-0.5) to (-2,2.5);
\draw[very thick, directed=.65] (2,2.5) to (1,2.5);
	\draw[very thick, directed=.65] (-1,2.5) to  (-2,2.5);
	\draw[very thick, directed=.65] (1,2.5) to [out=135,in=0] (0,3) to [out=180, in=45] (-1,2.5);
	\draw[very thick, directed=.65] (1,2.5) to [out=225,in=0] (0,2) to [out=180,in=315] (-1,2.5);
	\draw[very thick, directed=.65] (2.5,1.5) to  (-1.5,1.5);
\end{tikzpicture}
};
\endxy\quad, \quad \xy
(0,0)*{
\begin{tikzpicture} [scale=0.5,fill opacity=0.2]
\path[fill=blue]  (1,2.5) to [out=135,in=0] (0,3) to [out=180, in=45] (-1,2.5) to [out=270,in=180] (0,1) to [out=0,in=270] (1,2.5) ;
\path[fill=red]  (1,2.5) to [out=225,in=0] (0,2) to [out=180, in=315] (-1,2.5) to [out=270,in=180] (0,1) to [out=0,in=270] (1,2.5) ;
\path[fill=red]  (-2,0.5)  to (2,-1.5)to [out=315,in=90] (2.25,-2.25)to (2.25,-3) to (-1.75,-3) to (-1.75,-0.25) to [out=90,in=315] (-2,0.5);
\path[fill=green]  (-1,2.5) to [out=270,in=180] (0,1) to [out=0,in=270] (1,2.5) to (2, 2.5) to (2,-1.5) to [out=225,in=90] (1.75,-2.25) to (1.75,-2.5) to (-2.25 , -2.5) to (-2.25,-0.25) to [out=90,in=225] (-2,0.5) to (-2, 2.5) to (-1,2.5);
	\draw[very thick, red, directed=.65] (-1,2.5) to [out=270,in=180] (0,1) to [out=0,in=270] (1,2.5);
	\draw[very thick, red, rdirected=.65](-2,0.5) to (2,-1.5);
	\draw[very thick] (1.75,-2.5) to (1.75,-2.25);
	\draw[very thick] (2.25,-3) to (2.25,-2.25) to [out=90,in=315](2,-1.5);
	\draw[very thick] (1.75,-2.25) to [out=90,in=225](2,-1.5) to (2,2.5);
	\draw[very thick] (-2.25,-2.5) to (-2.25,-0.25);
	\draw[very thick] (-2.25,-0.25) to [out=90,in=225](-2,0.5);
	\draw[very thick] (-1.75,-3) to (-1.75,-0.25) to [out=90,in=315](-2,0.5) to (-2,2.5);
\draw[very thick, directed=.65] (2,2.5) to (1,2.5);
	\draw[very thick, directed=.65] (-1,2.5) to  (-2,2.5);
	\draw[very thick, directed=.65] (1,2.5) to [out=135,in=0] (0,3) to [out=180, in=45] (-1,2.5);
	\draw[very thick, directed=.65] (1,2.5) to [out=225,in=0] (0,2) to [out=180,in=315] (-1,2.5);
	\draw[very thick, rdirected=.65] (-2.25,-2.5) to  (1.75,-2.5);
	\draw[very thick, directed=.55] (2.25,-3) to (-1.75,-3);
\end{tikzpicture}
};
\endxy \quad = \quad \xy
(0,0)*{
\begin{tikzpicture} [scale=0.5,fill opacity=0.2]
\path[fill=blue]  (-1,2.5) to (-1,1) to [out=225,in=90] (-1.25,0.25)to (-1.25,-0.75) to  [out=270,in=180] (-0.25,-1.75) to [out=0,in=270](0.75,-0.75)to  [out=90,in=225] (1,0) to  (1,2.5) to [out=135,in=0] (0,3) to [out=180,in=45] (-1,2.5);
\path[fill=red]  (-1,2.5) to (-1,1) to [out=225,in=90] (-1.25,0.25)to (-1.25,-0.75) to  [out=270,in=180] (-0.25,-1.75) to [out=0,in=270](0.75,-0.75)to  [out=90,in=225] (1,0) to  (1,2.5) to [out=225,in=0] (0,2) to [out=180,in=315] (-1,2.5) ;
\path[fill=red]  (-2,1.5) to (-1,1) to [out=292.5,in=180] (1,0) to (2,-0.5)to [out=315,in=90] (2.25,-1.25)to (2.25,-3) to (-1.75,-3) to (-1.75,0.75) to [out=90,in=315] (-2,1.5);
\path[fill=green]   (-1,2.5) to (-1,1) to [out=225,in=90] (-1.25,0.25)to (-1.25,-0.75) to  [out=270,in=180] (-0.25,-1.75) to [out=0,in=270](0.75,-0.75)to  [out=90,in=225] (1,0) to  (1,2.5) to (2, 2.5) to (2,-0.5) to [out=225,in=90] (1.75,-1.25) to (1.75,-2.5) to (-2.25 , -2.5) to (-2.25,0.75) to [out=90,in=225] (-2,1.5) to (-2, 2.5) to (-1,2.5);
	\draw[very thick, red, directed=.45] (-1,2.5) to (-1,1) to [out=225,in=90] (-1.25,0.25);
	\draw[very thick, red, directed=.65](-1.25,0.25) to (-1.25,-0.75) to  [out=270,in=180] (-0.25,-1.75) to [out=0,in=270](0.75,-0.75);
	\draw[very thick, red, directed=.65](0.75,-0.75)  to  [out=90,in=225] (1,0) to  (1,2.5);
	\draw[very thick, red, rdirected=.65](-2,1.5) to (-1,1);
	\draw[very thick, red, rdirected=.65](1,0) to (2,-0.5);
	\draw[very thick, red, rdirected=.65]  (-1,1) to [out=292.5,in=180] (1,0);
	\draw[very thick] (1.75,-2.5) to (1.75,-1.25);
	\draw[very thick] (2.25,-3) to (2.25,-1.25) to [out=90,in=315](2,-0.5);
	\draw[very thick] (1.75,-1.25) to [out=90,in=225](2,-0.5) to (2,2.5);
	\draw[very thick] (-2.25,-2.5) to (-2.25,0.75);
	\draw[very thick] (-2.25,0.75) to [out=90,in=225](-2,1.5);
	\draw[very thick] (-1.75,-3) to (-1.75,0.75) to [out=90,in=315](-2,1.5) to (-2,2.5);
\draw[very thick, directed=.65] (2,2.5) to (1,2.5);
	\draw[very thick, directed=.65] (-1,2.5) to  (-2,2.5);
	\draw[very thick, directed=.65] (1,2.5) to [out=135,in=0] (0,3) to [out=180, in=45] (-1,2.5);
	\draw[very thick, directed=.65] (1,2.5) to [out=225,in=0] (0,2) to [out=180,in=315] (-1,2.5);
	\draw[very thick, rdirected=.65] (-2.25,-2.5) to  (1.75,-2.5);
	\draw[very thick, directed=.55] (2.25,-3) to (-1.75,-3);
	\node[opacity=1] at (0,2.5) {$*_2$};
	\node[opacity=1] at (0,1.5) {$*_1$};
	\node[opacity=1] at (-0.25,-1.25) {$\circ_1$};
	\node[opacity=1] at (-0.25,-.25) {$\circ_2$};
\end{tikzpicture}
};
\endxy 
}

\newcommand{\eqnAG}{
\xy
(0,0)*{
\begin{tikzpicture} [scale=.5,fill opacity=0.2]
\path[fill=blue] (-2,-2.5) to [out=0,in=135] (-0.5,-3) to (-0.5,-2.6) to [out=135, in=270] (-0.75,-1.85) to (-0.75,-1.15) to [out=90,in=180](0,-0.4) to [out=0,in=90] (0.75,-1.15) to [out=270,in=135](1,-1.85) to(1,-3)  to [out=45,in=180](2,-2.5) to (2,3) to (-2,3);
\path[fill=red] (-1.25,-3.5) to [out=0,in=225] (-0.5,-3) to (-0.5,-2.6) to [out=135, in=270] (-0.75,-1.85) to (-0.75,-1.15) to [out=90,in=180](0,-0.4) to [out=0,in=90] (0.75,-1.15) to [out=270,in=135](1,-1.85) to(1,-3)  to [out=315,in=180](2.75,-3.5) to (2.75,-1)to [out=135, in=270](2.5,-.25) to (2.5,2) to (-1.5,2) to (-1.5,-2.25)to [out=270,in=135] (-1.25,-3);
\path[fill=red]  (3,1) to (-1,1) to (-1,-2.25) to [out=270,in=45] (-1.25,-3) to (2.75,-1) to  [out=45,in=270](3,-0.25) to  (3,1);
\path[fill=green] (-0.5,-3) to (-0.5,-2.6) to [out=135, in=270] (-0.75,-1.85) to (-0.75,-1.15) to [out=90,in=180](0,-0.4) to [out=0,in=90] (0.75,-1.15) to [out=270,in=135](1,-1.85) to(1,-3) ;
	\draw[very thick, red, directed=.5] (-0.5,-3) to (-0.5,-2.6) to [out=135, in=270] (-0.75,-1.85) to (-0.75,-1.15) to [out=90,in=180](0,-0.4) to [out=0,in=90] (0.75,-1.15) to [out=270,in=135](1,-1.85) to(1,-3)  ;
	\draw[very thick, red, directed=.5] (-1.25,-3) to (2.75,-1);
	\draw[very thick, directed=.55] (2,3) to [out=180,in=0] (-2,3);
	\draw[very thick, directed=.55] (2.5,2) to [out=180,in=0] (-1.5,2);
	\draw[very thick] (2,-2.5) to (2,3);
	\draw[very thick] (2.75,-3.5) to (2.75,-1) to  [out=135,in=270](2.5,-.25) to  (2.5,2);
	\draw[very thick] (-1.25,-3.5) to (-1.25,-3) to [out=135,in=270](-1.5,-2.25) to (-1.5,2);
	\draw[very thick]  (2.75,-1) to[out=45,in=270]  (3,-0.25) to (3,1);
	\draw[very thick] (-1.25,-3) to [out=45,in=270](-1,-2.25) to (-1,1);
	\draw[very thick] (-2,-2.5) to (-2,3);		
	\draw[very thick, directed=.55] (2,-2.5) to [out=180,in=45] (1,-3);
	\draw[very thick, directed=.55] (2.75,-3.5) to [out=180,in=315] (1,-3);
	\draw[very thick, directed=.55] (1,-3) to (-.5,-3);
	\draw[very thick, directed=.55] (-.5,-3) to [out=225,in=0] (-1.25,-3.5);
	\draw[very thick, directed=.9] (-.5,-3) to [out=135,in=0] (-2,-2.5);
	\draw[very thick, directed=.55] (3,1) to (-1,1);
\end{tikzpicture}
};
\endxy 
 \quad = \quad 
\xy
(0,0)*{
\begin{tikzpicture} [scale=.5,fill opacity=0.2]
\path[fill=blue] (-2,-2.5) to [out=0,in=135] (-0.5,-3) to [out=90,in=180]  (0.25,-2.25) to [out=0,in=90] (1,-3) to [out=45,in=180](2,-2.5) to (2,3) to (-2,3);
\path[fill=red] (-1.25,-3.5) to [out=0,in=225] (-0.5,-3) to [out=90,in=180]  (0.25,-2.25) to [out=0,in=90] (1,-3) to [out=315,in=180](2.75,-3.5) to (2.75,0.25)to [out=135, in=270] (2.5,2) to (-1.5,2) to (-1.5,-1)to [out=270,in=135] (-1.25,-1.75);
\path[fill=red]  (3,1) to (-1,1) to (-1,-1) to [out=270,in=45] (-1.25,-1.75) to (2.75,0.25) to  [out=45,in=270](3,1) to  (3,1);
\path[fill=green] (-0.5,-3) to [out=90,in=180]  (0.25,-2.25) to [out=0, in = 90] (1,-3);
	\draw[very thick, red, directed=.5] (-0.5,-3) to [out=90,in=180]  (0.25,-2.25) to [out=0, in = 90] (1,-3);
	\draw[very thick, red, directed=.5] (-1.25,-1.75) to (2.75,0.25);
	\draw[very thick, directed=.55] (2,3) to [out=180,in=0] (-2,3);
	\draw[very thick, directed=.55] (2.5,2) to [out=180,in=0] (-1.5,2);
	\draw[very thick] (2,-2.5) to (2,3);
	\draw[very thick] (2.75,-3.5) to (2.75,0.25) to  [out=135,in=270](2.5,1) to  (2.5,2);
	\draw[very thick] (-1.25,-3.5) to (-1.25,-1.75) to [out=135,in=270](-1.5,-1) to (-1.5,2);
	\draw[very thick] (2.75,-0.5) to  (2.75,0.25) to[out=45,in=270]  (3,1);
	\draw[very thick] (-1.25,-2.5)to(-1.25,-1.75)  to [out=45,in=270](-1,-1) to (-1,1);
	\draw[very thick] (-2,-2.5) to (-2,3);		
	\draw[very thick, directed=.55] (2,-2.5) to [out=180,in=45] (1,-3);
	\draw[very thick, directed=.55] (2.75,-3.5) to [out=180,in=315] (1,-3);
	\draw[very thick, directed=.55] (1,-3) to (-.5,-3);
	\draw[very thick, directed=.55] (-.5,-3) to [out=225,in=0] (-1.25,-3.5);
	\draw[very thick, directed=.9] (-.5,-3) to [out=135,in=0] (-2,-2.5);
	\draw[very thick, directed=.55] (3,1) to (-1,1);
\end{tikzpicture}
};
\endxy\quad , \quad 
\xy
(0,0)*{
\begin{tikzpicture} [scale=.5,fill opacity=0.2]
\path[fill=blue] (-2,-2.5) to [out=0,in=135] (-0.5,-3)  to  (-0.5,-1.5) to [out=90, in=225] (-0.25,-0.75) to (-0.25,0) to [out=90,in=180](0.5,0.75) to [out=0,in=90] (1.25,0) to (1.25,-1.5) to [out=225,in=90] (1,-2.25) to (1,-3)  to [out=45,in=180](2,-2.5) to (2,3) to (-2,3);
\path[fill=red] (-1.5,-3.5) to [out=0,in=225] (-0.5,-3)  to  (-0.5,-1.5) to [out=90, in=225] (-0.25,-0.75) to (-0.25,0) to [out=90,in=180](0.5,0.75) to [out=0,in=90] (1.25,0) to (1.25,-1.5) to [out=225,in=90] (1,-2.25) to (1,-3)  to [out=315,in=180](2.5,-3.5) to (2.5,-3)to [out=90, in=225](2.75,-2.25) to (2.75,1.5) to (-1.25,1.5) to (-1.25,-.25)to [out=225,in=90] (-1.5,-1);
\path[fill=red]  (-1.25,-.25) to (2.75,-2.25) to [out=315,in=90] (3,-3) to (3,-4.5) to (-1,-4.5) to (-1,-1) to [out=90, in=315](-1.25,-.25);
\path[fill=green] (-0.5,-3)  to  (-0.5,-1.5) to [out=90, in=225] (-0.25,-0.75) to (-0.25,0) to [out=90,in=180](0.5,0.75) to [out=0,in=90] (1.25,0) to (1.25,-1.5) to [out=225,in=90] (1,-2.25) to (1,-3);
	\draw[very thick, red, directed=.5] (-0.5,-3)  to  (-0.5,-1.5) to [out=90, in=225] (-0.25,-0.75) to (-0.25,0) to [out=90,in=180](0.5,0.75) to [out=0,in=90] (1.25,0) to (1.25,-1.5) to [out=225,in=90] (1,-2.25) to (1,-3)  ;
	\draw[very thick, red, rdirected=.5] (-1.25,-.25) to (2.75,-2.25);
	\draw[very thick, directed=.55] (2,3) to [out=180,in=0] (-2,3);
	\draw[very thick, directed=.55] (2.75,1.5) to [out=180,in=0] (-1.25,1.5);
	\draw[very thick] (2,-2.5) to (2,3);
	\draw[very thick] (2.5,-3.5) to (2.5,-3) ;
	\draw[very thick] (2.75,-2.25) to[out=315,in=90] (3,-3) to (3,-4.5);
	\draw[very thick] (-1.25,-.25) to[out=315,in=90] (-1,-1) to (-1,-4.5);
	\draw[very thick]  (2.5,-3) to[out=90,in=225]  (2.75,-2.25) to (2.75,1.5);
	\draw[very thick] (-1.5,-3.5) to (-1.5,-1) to [out=90,in=225](-1.25,-.25) to (-1.25,1.5);
	\draw[very thick] (-2,-2.5) to (-2,3);		
	\draw[very thick, directed=.55] (2,-2.5) to [out=180,in=45] (1,-3);
	\draw[very thick, directed=.55] (2.5,-3.5) to [out=180,in=315] (1,-3);
	\draw[very thick, directed=.55] (1,-3) to (-.5,-3);
	\draw[very thick, directed=.85] (-.5,-3) to [out=225,in=0] (-1.5,-3.5);
	\draw[very thick, directed=.9] (-.5,-3) to [out=135,in=0] (-2,-2.5);
	\draw[very thick, directed=.55] (3,-4.5) to (-1,-4.5);
\end{tikzpicture}
};
\endxy 
 \quad = \quad 
\sum_{*,\circ}~ \xy
(0,0)*{
\begin{tikzpicture} [scale=.5,fill opacity=0.2]
\path[fill=blue] (-2,-2.5) to [out=0,in=135] (-0.5,-3)  to [out=90,in=180](0.25,-2.25) to [out=0,in=90] (1,-3)  to [out=45,in=180](2,-2.5) to (2,3) to (-2,3);
\path[fill=red] (-1.5,-3.5) to [out=0,in=225] (-0.5,-3)  to [out=90,in=180](0.25,-2.25) to [out=0,in=90] (1,-3)  to [out=315,in=180](2.5,-3.5) to (2.5,-3)to [out=90, in=225](2.75,-2.25) to (2.75,1.5) to (-1.25,1.5) to (-1.25,-.25)to [out=225,in=90] (-1.5,-1);
\path[fill=red]  (-1.25,-.25) to (2.75,-2.25) to [out=315,in=90] (3,-3) to (3,-4.5) to (-1,-4.5) to (-1,-1) to [out=90, in=315](-1.25,-.25);
\path[fill=green] (-0.5,-3)  to [out=90,in=180](0.25,-2.25) to [out=0,in=90] (1,-3);
	\draw[very thick, red, directed=.5] (-0.5,-3)  to [out=90,in=180](0.25,-2.25) to [out=0,in=90] (1,-3) ;
	\draw[very thick, red, rdirected=.5] (-1.25,-.25) to (2.75,-2.25);
	\draw[very thick, directed=.55] (2,3) to [out=180,in=0] (-2,3);
	\draw[very thick, directed=.55] (2.75,1.5) to [out=180,in=0] (-1.25,1.5);
	\draw[very thick] (2,-2.5) to (2,3);
	\draw[very thick] (2.5,-3.5) to (2.5,-3) ;
	\draw[very thick] (2.75,-2.25) to[out=315,in=90] (3,-3) to (3,-4.5);
	\draw[very thick] (-1.25,-.25) to[out=315,in=90] (-1,-1) to (-1,-4.5);
	\draw[very thick]  (2.5,-3) to[out=90,in=225]  (2.75,-2.25) to (2.75,1.5);
	\draw[very thick] (-1.5,-3.5) to (-1.5,-1) to [out=90,in=225](-1.25,-.25) to (-1.25,1.5);
	\draw[very thick] (-2,-2.5) to (-2,3);		
	\draw[very thick, directed=.55] (2,-2.5) to [out=180,in=45] (1,-3);
	\draw[very thick, directed=.55] (2.5,-3.5) to [out=180,in=315] (1,-3);
	\draw[very thick, directed=.55] (1,-3) to (-.5,-3);
	\draw[very thick, directed=.85] (-.5,-3) to [out=225,in=0] (-1.5,-3.5);
	\draw[very thick, directed=.9] (-.5,-3) to [out=135,in=0] (-2,-2.5);
	\draw[very thick, directed=.55] (3,-4.5) to (-1,-4.5);
	\node[opacity=1] at (-1.65,-.1) {$*_2$};
	\node[opacity=1] at (2,-3) {$*_1$};
	\node[opacity=1] at (-0.75,1) {$\circ_1$};
	\node[opacity=1] at (-1.5,2.5) {$\circ_2$};
\end{tikzpicture}
};
\endxy
}

\newcommand{\eqnAH}{
\xy
(0,0)*{
\begin{tikzpicture} [scale=.5,fill opacity=0.2]
	\draw[very thick,green, directed=.25]  (3,0) to (1.25,0);
	\draw[very thick,green]  (1.25,0) to (0.5,0);
	\draw[very thick,green, directed=.75]  (-1.25,0) to (-3,0);
	\draw[very thick,green]  (-0.5,0) to (-1.25,0);
	\draw[very thick,green, directed=.25]  (3,1) to (1.25,1);
	\draw[very thick,green, directed=.75]  (-1.25,1) to (-3,1);
	\draw[very thick,green] (1.25,1) to (-1.25,1);
	\draw[very thick,green, directed=.55]  (0.5,1) to (-0.5,1);
	\draw[very thick,green, directed=.55]  (0.5,0) to (-0.5,0);
	\draw[very thick,green, directed=.55] (1.25,1) to [out=180,in=0](0.5,0) ;
	\draw[very thick,green, directed=.55] (-0.5,0) to [out=180,in=0](-1.25,1) ;
\end{tikzpicture}
};
\endxy 
\to
\xy
(0,0)*{
\begin{tikzpicture} [scale=.5,fill opacity=0.2]
	\draw[very thick,green, directed=.25]  (3,0) to (1.25,0);
	\draw[very thick,green]  (1.25,0) to (0.5,0);
	\draw[very thick,green, directed=.75]  (-1.25,0) to (-3,0);
	\draw[very thick,green]  (-0.5,0) to (-1.25,0);
	\draw[very thick,green, directed=.25]  (3,1) to (1.25,1);
	\draw[very thick,green, directed=.75]  (-1.25,1) to (-3,1);
	\draw[very thick,green] (1.25,1) to (-1.25,1);
	\draw[very thick,green, directed=.55]  (0.5,1) to (-0.5,1);
	\draw[very thick,green, directed=.55]  (0.5,0) to (-0.5,0);
	\draw[very thick,red, directed=.55] (1.25,1) to [out=180,in=0](0.5,0) ;
	\draw[very thick,blue, directed=.55] (2,1) to [out=180,in=0](1.25,0);
	\draw[very thick,blue, directed=.55] (-0.5,0) to [out=180,in=0](-1.25,1) ;
	\draw[very thick,red, directed=.55] (-1.25,0) to [out=180,in=0](-2,1) ;
\end{tikzpicture}
};
\endxy
\to
\xy
(0,0)*{
\begin{tikzpicture} [scale=.5,fill opacity=0.2]
	\draw[very thick,green, directed=.25]  (3,0) to (1.25,0);
	\draw[very thick,green]  (1.25,0) to (0.5,0);
	\draw[very thick,green, directed=.75]  (-1.25,0) to (-3,0);
	\draw[very thick,green]  (-0.5,0) to (-1.25,0);
	\draw[very thick,green, directed=.25]  (3,1) to (1.25,1);
	\draw[very thick,green, directed=.75]  (-1.25,1) to (-3,1);
	\draw[very thick,green] (1.25,1) to (-1.25,1);
	\draw[very thick,green, directed=.55]  (0.5,1) to (-0.5,1);
	\draw[very thick,green, directed=.55]  (0.5,0) to (-0.5,0);
	\draw[very thick,red, directed=.55] (-0.25,1) to [out=180,in=0](-1,0) ;
	\draw[very thick,blue, directed=.55] (2.25,1) to [out=180,in=0](1.5,0);
	\draw[very thick,blue, directed=.55] (1,0) to [out=180,in=0](0.25,1) ;
	\draw[very thick,red, directed=.55] (-1.5,0) to [out=180,in=0](-2.25,1) ;
\end{tikzpicture}
};
\endxy
}

\newcommand{\eqnAI}{
\Phi_\Sigma \left (
\xy
(0,0)*{
\begin{tikzpicture} [scale=.5,fill opacity=0.2]
	\draw[very thick,green, rdirected=.25]  (3,0) to (3,1);
	\draw[very thick,green, directed=.25]  (0,0) to (0,1);
	\draw[very thick,red, directed=.65]  (0,1) to [out=135,in=270](-1.5,2) to (-1.5,3.5);
	\draw[very thick,blue, directed=.55]  (0,1) to [out=45,in=270](1.5,2);
	\draw[very thick,blue, rdirected=.55]  (3,1) to [out=45,in=270](3.5,2) to (3.5,3.5);
	\draw[very thick,blue]  (1.5,2) to [out=90,in=270](0.5,3)to(0.5,3.5);
	\draw[very thick,red, directed=.15]  (-0.5,3.5) to (-0.5,3) to [out=270,in=90](-0.5,2) to [out=270,in=180] (0,1.5) to [out=0,in=270](0.5,2) to [out=90,in=270](1.5,3) to [out=90,in=180] (2,3.5) to [out=0,in=90](2.5,3) to (2.5,2) to [out=270,in=135] (3,1);
\end{tikzpicture}
};
\endxy \right )\;
= \;
\Phi_\Sigma \left (
\xy
(0,0)*{
\begin{tikzpicture} [scale=.5,fill opacity=0.2]
	\draw[very thick,green, rdirected=.25]  (1,0) to (1,1);
	\draw[very thick,green, directed=.25]  (-1,0) to (-1,1);
	\draw[very thick,red, directed=.55]  (-1,1) to [out=135,in=270](-1.5,2) to (-1.5,3.5);
	\draw[very thick,blue, directed=.55]  (-1,1) to [out=45,in=270](-0.5,2);
	\draw[very thick,blue]  (-0.5,2) to [out=90,in=270](0.5,3)to(0.5,3.5);
	\draw[very thick,red, rdirected=.55]  (1,1) to [out=135,in=270](0.5,2);
	\draw[very thick,blue, rdirected=.55]  (1,1) to [out=45,in=270](1.5,2) to (1.5,3.5);
	\draw[very thick,red]  (0.5,2) to [out=90,in=270](-0.5,3) to (-0.5, 3.5);
\end{tikzpicture}
};
\endxy \right )\;
= \;
\Phi_\Sigma \left (
\xy
(0,0)*{
\begin{tikzpicture} [scale=.5,fill opacity=0.2]
	\draw[very thick,green, rdirected=.25]  (2,0) to (2,1);
	\draw[very thick,green, directed=.25]  (-1,0) to (-1,1);
	\draw[very thick,red, directed=.55]  (-1,1) to [out=135,in=270](-1.5,2) to (-1.5,3.5);
	\draw[very thick,blue, directed=.55]  (-1,1) to [out=45,in=225](-0,2.5);
	\draw[very thick,green,directed=.55]  (0,2.5) to(1,2.5);
	\draw[very thick,red, rdirected=.55]  (2,1) to [out=135,in=315](1,2.5);
	\draw[very thick,blue, rdirected=.55]  (2,1) to [out=45,in=270](2.5,2) to (2.5,3.5);
	\draw[very thick,red, directed=.55]  (-0.5,3.5) to [out=270,in=135](0,2.5);
	\draw[very thick,blue, directed=.55]  (1,2.5) to [out=45,in=270](1.5,3.5);
\end{tikzpicture}
};
\endxy \right )
}

\newcommand{\eqnAJ}{
\Phi_\Sigma \left (
\xy
(0,0)*{
\begin{tikzpicture} [scale=.4,fill opacity=0.2]
	\draw[very thick,green, rdirected=.25]  (2,-1) to (2,1);
	\draw[very thick,green, directed=.25]  (-1,-1) to (-1,1);
	\draw[very thick,red, directed=.55]  (-1,1) to [out=135,in=270](-1.5,2) to (-1.5,4.5);
	\draw[very thick,blue, directed=.55]  (-1,1) to [out=45,in=225](-0,2.5);
	\draw[very thick,green,directed=.55]  (0,2.5) to(1,2.5);
	\draw[very thick,red, rdirected=.55]  (2,1) to [out=135,in=315](1,2.5);
	\draw[very thick,blue, rdirected=.55]  (2,1) to [out=45,in=270](2.5,2) to (2.5,4.5);
	\draw[very thick,red, directed=.55]  (-0.5,4.5) to (-0.5,3.5) to [out=270,in=135](0,2.5);
	\draw[very thick,blue, directed=.55]  (1,2.5) to [out=45,in=270](1.5,3.5) to (1.5,4.5);
	\draw[blue, directed=.55]  (-1,0) to [out=45,in=135] (2,0);
\end{tikzpicture}
};
\endxy \right ) ~
= 
\Phi_\Sigma \left (
\xy
(0,0)*{
\begin{tikzpicture} [scale=.4,fill opacity=0.2]
	\draw[very thick,green, rdirected=.25]  (2,-1) to (2,0);
	\draw[very thick,green, directed=.25]  (-1,-1) to (-1,0);
	\draw[very thick,red, directed=.55]  (-1,0) to [out=135,in=270](-1.5,2) to (-1.5,4.5);
	\draw[very thick,blue, directed=.35]  (-1,0) to (-0,3.5);
	\draw[very thick,green,directed=.55]  (0,3.5) to(1,3.5);
	\draw[very thick,red, rdirected=.55]  (1.5,1.75) to (1,3.5);
	\draw[very thick,green, rdirected=.55]  (2,0) to (1.5,1.75);
	\draw[very thick,blue, rdirected=.55]  (2,0) to [out=45,in=270](2.5,2) to (2.5,4.5);
	\draw[very thick,red, directed=.6]  (-0.5,4.5) to [out=270,in=135](0,3.5);
	\draw[very thick,blue, directed=.55]  (1,3.5) to [out=45,in=270] (1.5,4.5);
	\draw[blue, directed=.55]  (-0.5,1.75) to  [out=45,in=135](1.5,1.75);
\end{tikzpicture}
};
\endxy \right ) ~
= 
\Phi_\Sigma \left (
\xy
(0,0)*{
\begin{tikzpicture} [scale=.4,fill opacity=0.2]
	\draw[very thick,green, rdirected=.25]  (3,-1) to (3,0);
	\draw[very thick,green, directed=.25]  (-1,-1) to (-1,0);
	\draw[very thick,red, directed=.55]  (-1,0) to [out=135,in=270](-1.5,2) to (-1.5,4.5);
	\draw[very thick,blue, directed=.35]  (-1,0) to [out=45,in=315](-0.5,1);
	\draw[very thick,green,directed=.55]  (1,3.5) to(2,3.5);
	\draw[very thick,red, rdirected=.55]  (2.5,1.75) to [out=90,in=315](2,3.5);
	\draw[very thick,green, rdirected=.55]  (3,0) to (2.5,1.75);
	\draw[very thick,blue, rdirected=.55]  (3,0) to [out=45,in=270](3.5,2) to (3.5,4.5);
	\draw[very thick,red, directed=.55]  (-0.5,4.5) to [out=270,in=90] (-1,3) to (-1,1)to [out=270,in=225](-0.5,1);
	\draw[very thick,blue, directed=.55]  (2,3.5) to [out=45,in=270] (2.5,4.5);
	\draw[very thick,green,directed=.55]  (-0.5,1) to(-0.5,2);
	\draw[very thick,red,directed=.55]  (-0.5,2) to  [out=90,in=135] (1,3.5);
	\draw[very thick,blue, directed=.35]  (-0.5,2) to [out=0,in=225] (1,3.5);
	\draw[blue, directed=.55]  (0.35,2.5) to  [out=0,in=135](2.5,1.75);
	\node[opacity=1] at (-0.9,.6) {$_{\circ_1}$};
	\node[opacity=1] at (0,.6) {$_{\circ_2}$};
\end{tikzpicture}
};
\endxy \right )
  = 
  \Phi_\Sigma \left (
\xy
(0,0)*{
\begin{tikzpicture} [scale=.4,fill opacity=0.2]
	\draw[very thick,green, rdirected=.25]  (3,-1) to (3,0);
	\draw[very thick,green, directed=.25]  (-1,-1) to (-1,0);
	\draw[very thick,red, directed=.55]  (-1,0) to [out=135,in=270](-1.5,2) to (-1.5,4.5);
	\draw[very thick,blue, directed=.35]  (-1,0) to [out=45,in=315](-0.5,1);
	\draw[very thick,green,directed=.55]  (1,3.5) to(2,3.5);
	\draw[very thick,red, rdirected=.55]  (2.5,1.75) to [out=90,in=315](2,3.5);
	\draw[very thick,green, rdirected=.55]  (3,0) to (2.5,1.75);
	\draw[very thick,blue, rdirected=.55]  (3,0) to [out=45,in=270](3.5,2) to (3.5,4.5);
	\draw[very thick,red, directed=.55]  (-0.5,4.5) to [out=270,in=90] (-1,3) to (-1,1)to [out=270,in=225](-0.5,1);
	\draw[very thick,blue, directed=.55]  (2,3.5) to [out=45,in=270] (2.5,4.5);
	\draw[very thick,green,directed=.55]  (-0.5,1) to(-0.5,2.75);
	\draw[very thick,red,directed=.55]  (-0.5,2.75) to  [out=90,in=135] (1,3.5);
	\draw[very thick,blue, directed=.35]  (-0.5,2.75) to [out=0,in=225] (1,3.5);
	\draw[blue, directed=.55]  (-0.5,2) to  [out=45,in=135](2.5,1.75);
	\node[opacity=1] at (-0.9,.6) {$_{\circ_1}$};
	\node[opacity=1] at (0,.6) {$_{\circ_2}$};
\end{tikzpicture}
};
\endxy \right ) \\
=
\Phi_\Sigma \left (
\xy
(0,0)*{
\begin{tikzpicture} [scale=.4,fill opacity=0.2]
	\draw[very thick,green, rdirected=.25]  (3,-1) to (3,0);
	\draw[very thick,green, directed=.25]  (-1,-1) to (-1,0);
	\draw[very thick,red, directed=.55]  (-1,0) to [out=135,in=270](-1.5,2) to (-1.5,4.5);
	\draw[very thick,blue, directed=.35]  (-1,0) to [out=45,in=315](-0.5,1);
	\draw[very thick,red, rdirected=.55]  (2.5,1.75) to [out=90,in=315](2,3.5);
	\draw[very thick,green, rdirected=.55]  (3,0) to (2.5,1.75);
	\draw[very thick,blue, rdirected=.55]  (3,0) to [out=45,in=270](3.5,2) to (3.5,4.5);
	\draw[very thick,red, directed=.55]  (-0.5,4.5) to [out=270,in=90] (-1,3) to (-1,1)to [out=270,in=225](-0.5,1);
	\draw[very thick,blue, directed=.85]  (2,3.5) to [out=45,in=270] (2.5,4.5);
	\draw[very thick,green,directed=.75]  (-0.5,1) to(-0.5,2.75) to [out=90,in=180] (2,3.5);
	\draw[blue, directed=.55]  (-0.5,2) to  [out=45,in=135](2.5,1.75);
	\node[opacity=1] at (-0.9,.6) {$_{\circ_1}$};
	\node[opacity=1] at (0,.6) {$_{\circ_2}$};
	\node[opacity=1] at (2.7,3.7) {$_{*_2}$};
	\node[opacity=1] at (2.7,3.2) {$_{*_1}$};
\end{tikzpicture}
};
\endxy  \right )
= 
\Phi_\Sigma \left (
\xy
(0,0)*{
\begin{tikzpicture} [scale=.4,fill opacity=0.2]
	\draw[very thick,green, rdirected=.25]  (3,-1) to (3,0);
	\draw[very thick,green, directed=.25]  (-1,-1) to (-1,0);
	\draw[very thick,red, directed=.55]  (-1,0) to [out=135,in=270](-1.5,2) to (-1.5,4.5);
	\draw[very thick,blue, directed=.35]  (-1,0) to [out=45,in=315](-0.5,1);
	\draw[very thick,red, rdirected=.55]  (2.5,1.75) to [out=90,in=45](1,3);
	\draw[very thick,green, rdirected=.55]  (3,0) to (2.5,1.75);
	\draw[very thick,blue, rdirected=.55]  (3,0) to [out=45,in=270](3.5,2) to (3.5,4.5);
	\draw[very thick,red, directed=.55]  (-0.5,4.5) to [out=270,in=90] (-1,3) to (-1,1)to [out=270,in=225](-0.5,1);
	\draw[very thick,blue, directed=.55]  (0,3.5) to [out=45,in=270] (2.5,4.5);
	\draw[very thick,green,directed=.55]  (-0.5,1) to(-0.5,2.75) to [out=90,in=180] (0,3.5) to [out=315,in=180](1,3);
	\draw[blue, directed=.55]  (1,3) to  [out=315,in=135](2.5,1.75);
	\node[opacity=1] at (-0.9,.6) {$_{\circ_1}$};
	\node[opacity=1] at (0,.6) {$_{\circ_2}$};
	\node[opacity=1] at (2.6,3.7) {$_{*_2}$};
	\node[opacity=1] at (2.6,3) {$_{*_1}$};
\end{tikzpicture}
};
\endxy  \right )
 = (-1)^r \Phi_\Sigma \left ( \xy
(0,0)*{
\begin{tikzpicture} [scale=.4,fill opacity=0.2]
	\draw[very thick,green, rdirected=.25]  (2,-1) to (2,1);
	\draw[very thick,green, directed=.25]  (-1,-1) to (-1,1);
	\draw[very thick,red, directed=.55]  (-1,1) to [out=135,in=270](-1.5,2) to (-1.5,4.5);
	\draw[very thick,blue, directed=.55]  (-1,1) to [out=45,in=225](-0,2.5);
	\draw[very thick,green,directed=.55]  (0,2.5) to(1,2.5);
	\draw[very thick,red, rdirected=.55]  (2,1) to [out=135,in=315](1,2.5);
	\draw[very thick,blue, rdirected=.55]  (2,1) to [out=45,in=270](2.5,2) to (2.5,4.5);
	\draw[very thick,red, directed=.55]  (-0.5,4.5) to (-0.5,3.5) to [out=270,in=135](0,2.5);
	\draw[very thick,blue, directed=.35]  (1,2.5) to [out=45,in=270](1.5,3.5) to (1.5,4.5);
	\draw[blue, directed=.55]  (1.5,3.5) to [out=45,in=135] (2.5,3.5);	
	\node[opacity=1] at (-0.7,2.8) {$_{\circ_1}$};
	\node[opacity=1] at (-0.7,2.2) {$_{\circ_2}$};
	\node[opacity=1] at (1.1,4) {$_{*_2}$};
	\node[opacity=1] at (1.1,1.6) {$_{*_1}$};
\end{tikzpicture}
};
\endxy \right )
 =(-1)^r \Phi_\Sigma \left (\xy
(0,0)*{
\begin{tikzpicture} [scale=.4,fill opacity=0.2]
	\draw[very thick,green, rdirected=.25]  (2,-1) to (2,1);
	\draw[very thick,green, directed=.25]  (-1,-1) to (-1,1);
	\draw[very thick,red, directed=.55]  (-1,1) to [out=135,in=270](-1.5,2) to (-1.5,4.5);
	\draw[very thick,blue, directed=.55]  (-1,1) to [out=45,in=225](-0,2.5);
	\draw[very thick,green,directed=.55]  (0,2.5) to(1,2.5);
	\draw[very thick,red, rdirected=.55]  (2,1) to [out=135,in=315](1,2.5);
	\draw[very thick,blue, rdirected=.55]  (2,1) to [out=45,in=270](2.5,2) to (2.5,4.5);
	\draw[very thick,red, directed=.55]  (-0.5,4.5) to (-0.5,3.5) to [out=270,in=135](0,2.5);
	\draw[very thick,blue]  (1,2.5) to [out=45,in=270](1.5,3.5) to (1.5,4.5);
	\draw[blue, directed=.55]  (1.5,3.5) to [out=45,in=135] (2.5,3.5);	
	\node[opacity=1] at (-0.7,2.8) {$_{\circ_1}$};
	\node[opacity=1] at (1.8,2.8) {$_{\circ_2}$};
	\node[opacity=1] at (1.1,4) {$_{*_2}$};
	\node[opacity=1] at (-0.1,4) {$_{*_1}$};
\end{tikzpicture}
};
\endxy \right )
}

\newcommand{\eqnAK}{
 \Phi_\Sigma \left (\xy
(0,0)*{
\begin{tikzpicture} [scale=.4,fill opacity=0.2]
	\draw[very thick,green, rdirected=.25]  (2,-1) to (2,1);
	\draw[very thick,green, directed=.25]  (-1,-1) to (-1,1);
	\draw[very thick,red, directed=.55]  (-1,1) to [out=135,in=270](-1.5,2) to (-1.5,4.5);
	\draw[very thick,blue, directed=.55]  (-1,1) to [out=45,in=225](-0,2.5);
	\draw[very thick,green,directed=.55]  (0,2.5) to(1,2.5);
	\draw[very thick,red, rdirected=.55]  (2,1) to [out=135,in=315](1,2.5);
	\draw[very thick,blue, rdirected=.55]  (2,1) to [out=45,in=270](2.5,2) to (2.5,4.5);
	\draw[very thick,red, directed=.55]  (-0.5,4.5) to (-0.5,3.5) to [out=270,in=135](0,2.5);
	\draw[very thick,blue, directed=.55]  (1,2.5) to [out=45,in=270](1.5,3.5) to (1.5,4.5);
	\draw[red, directed=.55]  (-1,0) to [out=45,in=135] (2,0);
\end{tikzpicture}
};
\endxy  \right )
= ~(-1)^b  \Phi_\Sigma \left (\xy
(0,0)*{
\begin{tikzpicture} [scale=.4,fill opacity=0.2]
	\draw[very thick,green, rdirected=.25]  (2,-1) to (2,1);
	\draw[very thick,green, directed=.25]  (-1,-1) to (-1,1);
	\draw[very thick,red, directed=.55]  (-1,1) to [out=135,in=270](-1.5,2) to (-1.5,4.5);
	\draw[very thick,blue, directed=.55]  (-1,1) to [out=45,in=225](-0,2.5);
	\draw[very thick,green,directed=.55]  (0,2.5) to(1,2.5);
	\draw[very thick,red, rdirected=.55]  (2,1) to [out=135,in=315](1,2.5);
	\draw[very thick,blue, rdirected=.55]  (2,1) to [out=45,in=270](2.5,2) to (2.5,4.5);
	\draw[very thick,red]  (-0.5,4.5) to (-0.5,3.5) to [out=270,in=135](0,2.5);
	\draw[very thick,blue, directed=.55]  (1,2.5) to [out=45,in=270](1.5,3.5) to (1.5,4.5);
	\draw[red, directed=.55]  (-1.5,3.5) to [out=45,in=135] (-0.5,3.5);	
	\node[opacity=1] at (-0.7,2.8) {$_{\circ_2}$};
	\node[opacity=1] at (1.8,2.8) {$_{\circ_1}$};
	\node[opacity=1] at (1.1,4) {$_{*_1}$};
	\node[opacity=1] at (-0.1,4) {$_{*_2}$};
\end{tikzpicture}
};
\endxy  \right )
=  \Phi_\Sigma \left (
\xy
(0,0)*{
\begin{tikzpicture} [scale=.4,fill opacity=0.2]
	\draw[very thick,green, rdirected=.25]  (2,-1) to (2,1);
	\draw[very thick,green, directed=.25]  (-1,-1) to (-1,1);
	\draw[very thick,red, directed=.55]  (-1,1) to [out=135,in=270](-1.5,2) to (-1.5,4.5);
	\draw[very thick,blue, directed=.55]  (-1,1) to [out=45,in=225](-0,2.5);
	\draw[very thick,green,directed=.55]  (0,2.5) to(1,2.5);
	\draw[very thick,red, rdirected=.55]  (2,1) to [out=135,in=315](1,2.5);
	\draw[very thick,blue, rdirected=.55]  (2,1) to [out=45,in=270](2.5,2) to (2.5,4.5);
	\draw[very thick,red]  (-0.5,4.5) to (-0.5,3.5) to [out=270,in=135](0,2.5);
	\draw[very thick,blue, directed=.55]  (1,2.5) to [out=45,in=270](1.5,3.5) to (1.5,4.5);
	\draw[red, directed=.55]  (-1.5,3.5) to [out=45,in=135] (-0.5,3.5);	
	\node[opacity=1] at (-0.7,2.8) {$_{\circ_1}$};
	\node[opacity=1] at (1.8,2.8) {$_{\circ_2}$};
	\node[opacity=1] at (1.1,4) {$_{*_2}$};
	\node[opacity=1] at (-0.1,4) {$_{*_1}$};
\end{tikzpicture}
};
\endxy \right )
}

\newcommand{\eqnAL}{
\xy
(0,0)*{
\begin{tikzpicture} [scale=.5,fill opacity=0.2]
	\draw[very thick,green, directed=.25]  (5,0) to (2.25,0);
	\draw[very thick,green]  (2.25,0) to (1.5,0);
	\draw[very thick,green, directed=.75]  (-2.25,0) to (-5,0);
	\draw[very thick,green]  (-1.5,0) to (-2.25,0);
	\draw[very thick,green, directed=.25]  (5,1) to (2.25,1);
	\draw[very thick,green, directed=.75]  (-2.25,1) to (-5,1);
	\draw[very thick,green] (2.25,1) to (-2.25,1);
	\draw[very thick,green, directed=.55]  (1.5,1) to (-1.5,1);
	\draw[very thick,green, directed=.55]  (1.5,0) to (-1.5,0);
	\draw[very thick,red, directed=.55] (-1.25,1) to [out=180,in=0](-2,0) ;
	\draw[very thick,blue, directed=.55] (3.25,1) to [out=180,in=0](2.5,0);
	\draw[very thick,blue, directed=.55] (2,0) to [out=180,in=0](1.25,1) ;
	\draw[very thick,red, directed=.55] (-2.5,0) to [out=180,in=0](-3.25,1) ;
\end{tikzpicture}
};
\endxy \to 
\xy
(0,0)*{
\begin{tikzpicture} [scale=.5,fill opacity=0.2]
	\draw[very thick,red, directed=.55]  (4,0) to [out=225,in=0](3,-1) to (1.5,-1) to [out=180,in=315](0.5,0);
	\draw[very thick,red, directed=.55]  (4,1) to [out=135,in=0](3,2) to (1.5,2) to [out=180,in=45](0.5,1);
	\draw[very thick,blue, directed=.55]  (-0.5,0) to [out=225,in=0] (-1.5,-1) to (-3,-1)to [out=135,in=0](-4,0);
	\draw[very thick,blue, directed=.55]  (-0.5,1) to [out=135,in=0] (-1.5,2) to (-3,2) to [out=225,in=0](-4,1);
	\draw[very thick,green, directed=.55]  (5,0) to (4,0);
	\draw[very thick,green, directed=.55]  (5,1) to (4,1);
	\draw[very thick,green, directed=.55]  (-4,0) to (-5,0);
	\draw[very thick,green, directed=.55]  (-4,1) to (-5,1);
	\draw[very thick,blue, directed=.55]  (4,0) to (0.5,0);
	\draw[very thick,red, directed=.55]  (-0.5,0) to (-4,0);
	\draw[very thick,blue, directed=.55]  (4,1) to (0.5,1);
	\draw[very thick,red, directed=.55]  (-0.5,1) to (-4,1);
	\draw[very thick,green, directed=.55] (0.5,1) to (-0.5,1);
	\draw[very thick,green, directed=.55]  (0.5,0) to (-0.5,0);
	\draw[very thick,red, directed=.55] (-1.25,1) to [out=180,in=0](-2,0) ;
	\draw[very thick,blue, directed=.55] (3.25,1) to [out=180,in=0](2.5,0);
	\draw[very thick,blue, directed=.55] (2,0) to [out=180,in=0](1.25,1) ;
	\draw[very thick,red, directed=.55] (-2.5,0) to [out=180,in=0](-3.25,1) ;
\end{tikzpicture}
};
\endxy
}

\newcommand{\eqnAM}{
\xy
(0,0)*{
\begin{tikzpicture} [scale=.5,fill opacity=0.2]
	\draw[very thick,red, directed=.55]  (3,0) to (2.25,0);
	\draw[very thick,red]  (2.25,0) to (0.5,0);
	\draw[very thick, red, directed=.85]  (-2.25,0) to (-3,0);
	\draw[very thick,red]  (-0.5,0) to (-2.25,0);
	\draw[very thick,red, directed=.55]  (3,1) to (2.25,1);
	\draw[very thick,red, directed=.85]  (-2.25,1) to (-3,1);
	\draw[very thick,red] (2.25,1) to (-2.25,1);
	\draw[very thick,red, directed=.55]  (0.5,1) to (-0.5,1);
	\draw[very thick,red, directed=.55]  (0.5,0) to (-0.5,0);
	\draw[very thick,red, directed=.55] (1.25,1) to [out=180,in=0](0.5,0) ;
	\draw[very thick,red, directed=.55] (-0.5,0) to [out=180,in=0](-1.25,1) ;
	\draw[very thick,blue, directed=.55]  (3,-1) to (2.25,-1);
	\draw[very thick,blue]  (2.25,-1) to (1.5,-1);
	\draw[very thick,blue, directed=.85]  (-2.25,-1) to (-3,-1);
	\draw[very thick,blue]  (-1.5,-1) to (-2.25,-1);
	\draw[very thick,blue, directed=.55]  (1.5,-1) to (-1.5,-1);
	\draw[very thick,blue, directed=.55] (3,2) to [out=180,in=0](1.5,-1) ;
	\draw[very thick,blue, directed=.55] (-1.5,-1) to [out=180,in=0](-3,2) ;
\end{tikzpicture}
};
(20,0)*{ \to};
(40,-4)*{
\begin{tikzpicture} [scale=.5,fill opacity=0.2]
	\draw[very thick,red, directed=.40]  (3,0) to (2,0);
	\draw[very thick,red, directed=.40]  (3,1) to (2,1);
	\draw[very thick,red, directed=.85]  (-2.25,1) to (-3,1);
	\draw[very thick, red, directed=.85]  (-2.25,0) to (-3,0);
	\draw[very thick,red]  (-2,0) to (-2.25,0);
	\draw[very thick,red]  (-2,1) to (-2.25,1);
	\draw[very thick,red, directed=.84] (2,1) to [out=180,in=90](1,-1) to[out=270,in=0](0.25,-2.5) ;
	\draw[very thick,red, directed=.30] (0.25,-2.5) to (-0.25,-2.5) to [out=180,in=270](-1,-1) to [out=90,in=0](-2,1) ;
	\draw[very thick,red] (2,0) to [out=180,in=0] (1,-2.5) to (-1,-2.5) to [out=180,in=0](-2,0) ;
	\draw[very thick,red, directed=.55]  (0.25,-1.5) to (-0.25,-1.5);
	\draw[very thick,red, directed=.50]  (0.25,-2.5) to (-0.25,-2.5);
	\draw[very thick,red](0.98,-1.25) to [out=270,in=0](0.25,-1.5) to  (0,-1.5) to (-0.25,-1.5) to [out=180,in=270](-0.98,-1.25) ;
	\draw[very thick,blue, directed=.55]  (3,-1) to (2.25,-1);
	\draw[very thick,blue]  (2.25,-1) to (1.5,-1);
	\draw[very thick,blue, directed=.85]  (-2.25,-1) to (-3,-1);
	\draw[very thick,blue]  (-1.5,-1) to (-2.25,-1);
	\draw[very thick,blue, directed=.55]  (1.5,-1) to (-1.5,-1);
	\draw[very thick,blue, directed=.55] (3,2) to [out=180,in=0](1.5,-1) ;
	\draw[very thick,blue, directed=.55] (-1.5,-1) to [out=180,in=0](-3,2) ;
\end{tikzpicture}
};
\endxy
}

\newcommand{\eqnAN}{
\xy
(0,-2)*{
\begin{tikzpicture} [scale=.4,fill opacity=0.2]
	\draw[very thick,red, directed=.55]  (3,0) to (2.25,0);
	\draw[very thick,red]  (2.25,0) to (0.5,0);
	\draw[very thick, red, directed=.85]  (-2.25,0) to (-3,0);
	\draw[very thick,red]  (-0.5,0) to (-2.25,0);
	\draw[very thick,red, directed=.55]  (3,1) to (2.25,1);
	\draw[very thick,red, directed=.85]  (-2.25,1) to (-3,1);
	\draw[very thick,red] (2.25,1) to (-2.25,1);
	\draw[very thick,red, directed=.55]  (0.5,1) to (-0.5,1);
	\draw[very thick,red, directed=.55]  (0.5,0) to (-0.5,0);
	\draw[very thick,red, directed=.55] (1.25,1) to [out=180,in=0](0.5,0) ;
	\draw[very thick,red, directed=.55] (-0.5,0) to [out=180,in=0](-1.25,1) ;
	\draw[very thick,blue, directed=.55]  (1.5,-1) to (-1.5,-1);
	\draw[very thick,blue, directed=.55] (3,2) to [out=180,in=0](1.5,-1) ;
	\draw[very thick,blue, directed=.55] (-1.5,-1) to [out=180,in=0](-3,2) ;
\end{tikzpicture}
};
(15,0)*{\cong};
(30,0)*{
\begin{tikzpicture} [scale=.4,fill opacity=0.2]
	\draw[very thick,red, directed=.55]  (3,0) to (2.25,0);
	\draw[very thick,red]  (2.25,0) to (0.5,0);
	\draw[very thick, red, directed=.85]  (-2.25,0) to (-3,0);
	\draw[very thick,red]  (-0.5,0) to (-2.25,0);
	\draw[very thick,red, directed=.55]  (3,1) to (2.25,1);
	\draw[very thick,red, directed=.85]  (-2.25,1) to (-3,1);
	\draw[very thick,red] (2.25,1) to (-2.25,1);
	\draw[very thick,red, directed=.55]  (0.5,1) to (-0.5,1);
	\draw[very thick,red, directed=.55] (1.25,1) to [out=180,in=0](0.5,0) ;
	\draw[very thick,red, directed=.55] (-0.5,0) to [out=180,in=0](-1.25,1) ;
	\draw[very thick,green, directed=.55]  (1.5,0) to (-1.5,0);
	\draw[very thick,blue, directed=.4] (3,2) to [out=180,in=0](1.5,0) ;
	\draw[very thick,blue, directed=.7] (-1.5,0) to [out=180,in=0](-3,2) ;
\end{tikzpicture}
};
\endxy  \cong
\xy
(0,0)*{
\begin{tikzpicture} [scale=.4,fill opacity=0.2]
	\draw[very thick,red, directed=.55]  (3,0) to (2.25,0);
	\draw[very thick,red]  (2.25,0) to (0.5,0);
	\draw[very thick, red, directed=.85]  (-2.25,0) to (-3,0);
	\draw[very thick,red]  (-0.5,0) to (-2.25,0);
	\draw[very thick,red, directed=.55]  (3,1) to (2.25,1);
	\draw[very thick,red, directed=.85]  (-2.25,1) to (-3,1);
	\draw[very thick,red] (2.25,1) to (-2.25,1);
	\draw[very thick,red, directed=.55]  (0.5,1) to (-0.5,1);
	\draw[very thick,green, directed=.75] (2,1) to (1.25,1) to [out=180,in=0](0.5,0) ;
	\draw[very thick,green, directed=.35] (-0.5,0) to [out=180,in=0](-1.25,1) to (-2,1) ;
	\draw[very thick,green, directed=.55]  (0.5,0) to (-0.5,0);
	\draw[very thick,blue, directed=.4] (3,2) to [out=180,in=0](2,1) ;
	\draw[very thick,blue, directed=.7] (-2,1) to [out=180,in=0](-3,2) ;
\end{tikzpicture}
};
\endxy  \cong
\xy
(0,0)*{
\begin{tikzpicture} [scale=.4,fill opacity=0.2]
	\draw[very thick,red, directed=.55]  (3,0) to (2.25,0);
	\draw[very thick,red]  (2.25,0) to (0.5,0);
	\draw[very thick, red, directed=.85]  (-2.25,0) to (-3,0);
	\draw[very thick,red]  (-0.5,0) to (-2.25,0);
	\draw[very thick,red, directed=.55]  (3,1) to (2.25,1);
	\draw[very thick,red, directed=.85]  (-2.25,1) to (-3,1);
	\draw[very thick,red] (2.25,1) to (-2.25,1);
	\draw[very thick,red, directed=.55]  (1.25,1) to [out=180,in=0](0.5,0) ;
	\draw[very thick,red, directed=.55] (-0.5,0) to [out=180,in=0](-1.25,1) ;
	\draw[very thick,green, directed=.55]  (2,1) to (-2,1);
	\draw[very thick,red, directed=.55]  (0.5,0) to (-0.5,0);
	\draw[very thick,blue, directed=.4] (3,2) to [out=180,in=0](2,1) ;
	\draw[very thick,blue, directed=.7] (-2,1) to [out=180,in=0](-3,2) ;
\end{tikzpicture}
};
\endxy \cong
\xy
(0,0)*{
\begin{tikzpicture} [scale=.4,fill opacity=0.2]
	\draw[very thick,red, directed=.55]  (3,0) to (2.25,0);
	\draw[very thick,red]  (2.25,0) to (0.5,0);
	\draw[very thick, red, directed=.85]  (-2.25,0) to (-3,0);
	\draw[very thick,red]  (-0.5,0) to (-2.25,0);
	\draw[very thick,red, directed=.55]  (3,1) to (2.25,1);
	\draw[very thick,red, directed=.85]  (-2.25,1) to (-3,1);
	\draw[very thick,red] (2.25,1) to (-2.25,1);
	\draw[very thick,red, directed=.55]  (1.25,1) to [out=180,in=0](0.5,0) ;
	\draw[very thick,red, directed=.55] (-0.5,0) to [out=180,in=0](-1.25,1) ;
	\draw[very thick,red, directed=.55]  (2,1) to (-2,1);
	\draw[very thick,red, directed=.55]  (0.5,0) to (-0.5,0);
	\draw[very thick,blue, directed=.55] (3,2) to [out=180,in=0](-3,2) ;
\end{tikzpicture}
};
\endxy
}

\newcommand{\eqnAO}{
\xy
(0,0)*{
\begin{tikzpicture} [scale=.4,fill opacity=0.2]
	\draw[very thick,red, directed=.55]  (3,0) to (2.25,0);
	\draw[very thick,red]  (2.25,0) to (0.5,0);
	\draw[very thick, red, directed=.85]  (-2.25,0) to (-3,0);
	\draw[very thick,red]  (-0.5,0) to (-2.25,0);
	\draw[very thick,red, directed=.55]  (3,1) to (2.25,1);
	\draw[very thick,red, directed=.85]  (-2.25,1) to (-3,1);
	\draw[very thick,red] (2.25,1) to (-2.25,1);
	\draw[very thick,red, directed=.55]  (0.5,1) to (-0.5,1);
	\draw[very thick,red, directed=.55]  (0.5,0) to (-0.5,0);
	\draw[very thick,red, directed=.55] (1.25,1) to [out=180,in=0](0.5,0) ;
	\draw[very thick,red, directed=.55] (-0.5,0) to [out=180,in=0](-1.25,1) ;
	\draw[very thick,blue, directed=.55]  (1.5,2) to (-1.5,2);
	\draw[very thick,blue, directed=.55] (3,-1) to [out=180,in=0](1.5,2) ;
	\draw[very thick,blue, directed=.55] (-1.5,2) to [out=180,in=0](-3,-1) ;
\end{tikzpicture}
};
\endxy   \cong
\xy
(0,0)*{
\begin{tikzpicture} [scale=.4,fill opacity=0.2]
	\draw[very thick,red, directed=.55]  (3,0) to (2.25,0);
	\draw[very thick,red]  (2.25,0) to (0.5,0);
	\draw[very thick, red, directed=.85]  (-2.25,0) to (-3,0);
	\draw[very thick,red]  (-0.5,0) to (-2.25,0);
	\draw[very thick,red, directed=.55]  (3,1) to (2.25,1);
	\draw[very thick,red, directed=.85]  (-2.25,1) to (-3,1);
	\draw[very thick,red] (2.25,1) to (-2.25,1);
	\draw[very thick,red, directed=.55]  (0.5,1) to (-0.5,1);
	\draw[very thick,red, directed=.55]  (0.5,0) to (-0.5,0);
	\draw[very thick,red, directed=.55] (1.25,1) to [out=180,in=0](0.5,0) ;
	\draw[very thick,red, directed=.55] (-0.5,0) to [out=180,in=0](-1.25,1) ;
	\draw[very thick,blue, directed=.55] (3,-1) to [out=180,in=0](2,1.5) to  [out=180,in=0] (1,-1) to (-1,-1) to [out=180,in=0](-2,1.5) to [out=180,in=0](-3,-1);
\end{tikzpicture}
};
\endxy   \cong
\xy
(0,0)*{
\begin{tikzpicture} [scale=.4,fill opacity=0.2]
	\draw[very thick,red, directed=.55]  (3,0) to (2.25,0);
	\draw[very thick,red]  (2.25,0) to (0.5,0);
	\draw[very thick, red, directed=.85]  (-2.25,0) to (-3,0);
	\draw[very thick,red]  (-0.5,0) to (-2.25,0);
	\draw[very thick,red, directed=.55]  (3,1) to (2.25,1);
	\draw[very thick,red, directed=.85]  (-2.25,1) to (-3,1);
	\draw[very thick,red] (2.25,1) to (-2.25,1);
	\draw[very thick,red, directed=.55]  (1.25,1) to [out=180,in=0](0.5,0) ;
	\draw[very thick,red, directed=.55] (-0.5,0) to [out=180,in=0](-1.25,1) ;
	\draw[very thick,red, directed=.55]  (2,1) to (-2,1);
	\draw[very thick,red, directed=.55]  (0.5,0) to (-0.5,0);
	\draw[very thick,blue, directed=.55] (3,-1) to [out=180,in=0](-3,-1) ;
\end{tikzpicture}
};
\endxy
}

  \newcommand{\eqnAP}{
\xy
(0,0)*{
\begin{tikzpicture} [fill opacity=0.2,decoration={markings, mark=at position 0.6 with {\arrow{>}}; }, scale=.4]
	\draw [fill=green] (1,1) -- (-1,2) -- (-1,-1) -- (1,-2) -- cycle;
	\node [opacity=1] at (0,0) {$\bullet^{N_\lambda}$};
	\node[ opacity=1] at (-.75,1.5) {\tiny{$1$}};
	\node [opacity=1] at (0.5,-1) {$\idem_\lambda$};
\end{tikzpicture}};
\endxy
 \quad = \quad -\sum_{i=0}^{N_\lambda-1} {N_\lambda \choose i}(-\lambda)^{N_\lambda-i} ~\xy
(0,0)*{
\begin{tikzpicture} [fill opacity=0.2,decoration={markings, mark=at position 0.6 with {\arrow{>}}; }, scale=.4]
	\draw [fill=green] (1,1) -- (-1,2) -- (-1,-1) -- (1,-2) -- cycle;
	\node [opacity=1] at (0,0) {$\bullet^{i}$};
	\node[ opacity=1] at (-.75,1.5) {\tiny{$1$}};
	\node [opacity=1] at (0.5,-1) {$\idem_\lambda$};
\end{tikzpicture}};
\endxy
 }

\newcommand{\smallfacet}[3][1]{
\xy
(0,0)*{
\begin{tikzpicture} [fill opacity=0.2,decoration={markings, mark=at position 0.6 with {\arrow{>}}; }, scale=.4]
	\draw [fill=green] (1,1) -- (-1,2) -- (-1,-1) -- (1,-2) -- cycle;
	\node [opacity=1] at (0.5,-1) {$#2$};
	\node[opacity=1] at (-.75,1.5) {\tiny{$#3$}};
\end{tikzpicture}};
\endxy
}

\newcommand{\smalldotfacet}[3][1]{
\xy
(0,0)*{
\begin{tikzpicture} [fill opacity=0.2,decoration={markings, mark=at position 0.6 with {\arrow{>}}; }, scale=.4]
	\draw [fill=green] (1,1) -- (-1,2) -- (-1,-1) -- (1,-2) -- cycle;
	\node [opacity=1] at (0,0) {$\bullet$};
	\node [opacity=1] at (0.5,-1) {$#2$};
	\node[opacity=1] at (-.75,1.5) {\tiny{$#3$}};
\end{tikzpicture}};
\endxy
}
\newcommand{\eqnAR}{
\xy
(0,0)*{
\begin{tikzpicture} [scale=.5,fill opacity=0.2]
	\path[fill=green] (2.25,3) to (.75,3) to (.75,0) to (2.25,0);
	\path[fill=green] (.75,3) to [out=225,in=0] (-.5,2.5) to (-.5,-.5) to [out=0,in=225] (.75,0);
	\path[fill=green] (.75,3) to [out=135,in=0] (-1,3.5) to (-1,.5) to [out=0,in=135] (.75,0);	
	\draw [very thick,directed=.55] (2.25,0) to (.75,0);
	\draw [very thick,directed=.55] (.75,0) to [out=135,in=0] (-1,.5);
	\draw [very thick,directed=.55] (.75,0) to [out=225,in=0] (-.5,-.5);
	\draw[very thick, red, directed=.55] (.75,0) to (.75,3);
	\draw [very thick] (2.25,3) to (2.25,0);
	\draw [very thick] (-1,3.5) to (-1,.5);
	\draw [very thick] (-.5,2.5) to (-.5,-.5);
	\draw [very thick,directed=.55] (2.25,3) to (.75,3);
	\draw [very thick,directed=.55] (.75,3) to [out=135,in=0] (-1,3.5);
	\draw [very thick,directed=.55] (.75,3) to [out=225,in=0] (-.5,2.5);
	\node [opacity=1]  at (1.4,2.65) {\tiny{$_{a+b}$}};
	\node[opacity=1] at (-.75,3.25) {\tiny{$b$}};
	\node[opacity=1] at (-.25,2.25) {\tiny{$a$}};	
	\node[opacity=1] at (1.5,1.5) {\small $e_s$};
\end{tikzpicture}
};
\endxy
\quad = \sum_{l = 0}^s \quad
\xy
(0,0)*{
\begin{tikzpicture} [scale=.5,fill opacity=0.2]
	\path[fill=green] (2.25,3) to (.75,3) to (.75,0) to (2.25,0);
	\path[fill=green] (.75,3) to [out=225,in=0] (-.5,2.5) to (-.5,-.5) to [out=0,in=225] (.75,0);
	\path[fill=green] (.75,3) to [out=135,in=0] (-1,3.5) to (-1,.5) to [out=0,in=135] (.75,0);	
	\draw [very thick,directed=.55] (2.25,0) to (.75,0);
	\draw [very thick,directed=.55] (.75,0) to [out=135,in=0] (-1,.5);
	\draw [very thick,directed=.55] (.75,0) to [out=225,in=0] (-.5,-.5);
	\draw[very thick, red, directed=.55] (.75,0) to (.75,3);
	\draw [very thick] (2.25,3) to (2.25,0);
	\draw [very thick] (-1,3.5) to (-1,.5);
	\draw [very thick] (-.5,2.5) to (-.5,-.5);
	\draw [very thick,directed=.55] (2.25,3) to (.75,3);
	\draw [very thick,directed=.55] (.75,3) to [out=135,in=0] (-1,3.5);
	\draw [very thick,directed=.55] (.75,3) to [out=225,in=0] (-.5,2.5);
	\node [opacity=1]  at (1.4,2.65) {\tiny {$_{a+b}$}};
	\node[opacity=1] at (-.75,3.25) {\tiny {$b$}};
	\node[opacity=1] at (-.25,2.25) {\tiny {$a$}};	
	\node[opacity=1] at (.125,1) {\small $e_{s-l}$};
	\node[opacity=1] at (-.25,3) {\small $e_{l}$};	
\end{tikzpicture}
};
\endxy}

\newcommand{\eqnAS}{
\xy
(0,0)*{
\begin{tikzpicture} [scale=.4,fill opacity=0.2]
	\path[fill=green] (2,2) to (2,-2) to (-2,-2) to (-2,2); 
	\draw[very thick, directed=.55] (2,-2) to (-2,-2);
	\draw[very thick] (2,-2) to (2,2);
	\draw[very thick] (-2,-2) to (-2,2);
	\draw[very thick, directed=.55] (2,2) to (-2,2);
	\node[opacity=1] at (1.4,1.65) {\tiny $_{k+1}$};
\end{tikzpicture}
};
\endxy
\quad = \quad
\xy
(0,0)*{
\begin{tikzpicture} [scale=.4,fill opacity=0.2]
	\path[fill=green,opacity=0.2] (1.25,0) to [out=90,in=0] (0,1.25) to [out=180,in=90] (-1.25,0) to [out=270,in=180] (0,-1.25) to [out=0,in=270] (1.25,0);
	\path[fill=green,opacity=0.2] (1.25,0) to [out=90,in=0] (0,1.25) to [out=180,in=90] (-1.25,0) to [out=270,in=180] (0,-1.25) to [out=0,in=270] (1.25,0);
	\path[fill=green] (2,0) to (1.25,0) to [out=90,in=0] (0,1.25) to [out=180,in=90] (-1.25,0) to (-2,0) to (-2,2) to (2,2);
	\path[fill=green] (2,0) to (1.25,0) to [out=270,in=0] (0,-1.25) to [out=180,in=270] (-1.25,0) to (-2,0) to (-2,-2) to (2,-2);	
	\draw[very thick, directed=.55] (2,-2) to (-2,-2);
	\draw[very thick] (2,-2) to (2,2);
	\draw[very thick] (-2,-2) to (-2,2);
	\draw[very thick, red, directed=.5] (1.25,0) to [out=90,in=0] (0,1.25) to [out=180,in=90] (-1.25,0);
	\draw[very thick, red] (1.25,0) to [out=270,in=0] (0,-1.25) to [out=180,in=270] (-1.25,0);	
	\draw[green] (1.25,0) to [out=225,in=315] (-1.25,0);
	\draw[dashed,green] (1.25,0) to [out=135,in=45] (-1.25,0);
	\draw[very thick, directed=.55] (2,2) to (-2,2);
	\node[opacity=1] at (1.4,1.65) {\tiny $_{k+1}$};
	\node[opacity=1] at (.25,-.75) {\tiny $1$};
	\node[opacity=1] at (-.25,.75) {\tiny $k$};
	\node[opacity=1] at (-.15,-.375) {\small $\bullet^{k}$};
\end{tikzpicture}
};
\endxy
}

\newcommand{\eqnAT}{
\xy
(0,0)*{
\begin{tikzpicture} [scale=.4,fill opacity=0.2]
	\path[fill=green] (2,2) to (2,-2) to (-2,-2) to (-2,2); 
	\draw[very thick, directed=.55] (2,-2) to (-2,-2);
	\draw[very thick] (2,-2) to (2,2);
	\draw[very thick] (-2,-2) to (-2,2);
	\draw[very thick, directed=.55] (2,2) to (-2,2);
	\node[opacity=1] at (1.4,1.65) {\tiny $_{k+1}$};
	\node [opacity=1] at (1.5,-1.5) {$\idem_\lambda$};
\end{tikzpicture}
};
\endxy
\quad = \quad \sum_{i=0}^k {k \choose i}(-\lambda)^{i}~
\xy
(0,0)*{
\begin{tikzpicture} [scale=.4,fill opacity=0.2]
	\path[fill=green] (1.25,0) to [out=90,in=0] (0,1.25) to [out=180,in=90] (-1.25,0) to [out=270,in=180] (0,-1.25) to [out=0,in=270] (1.25,0);
	\path[fill=green] (1.25,0) to [out=90,in=0] (0,1.25) to [out=180,in=90] (-1.25,0) to [out=270,in=180] (0,-1.25) to [out=0,in=270] (1.25,0);
	\path[fill=green] (2,0) to (1.25,0) to [out=90,in=0] (0,1.25) to [out=180,in=90] (-1.25,0) to (-2,0) to (-2,2) to (2,2);
	\path[fill=green] (2,0) to (1.25,0) to [out=270,in=0] (0,-1.25) to [out=180,in=270] (-1.25,0) to (-2,0) to (-2,-2) to (2,-2);	
	\draw[very thick, directed=.55] (2,-2) to (-2,-2);
	\draw[very thick] (2,-2) to (2,2);
	\draw[very thick] (-2,-2) to (-2,2);
	\draw[very thick, red, directed=.5] (1.25,0) to [out=90,in=0] (0,1.25) to [out=180,in=90] (-1.25,0);
	\draw[very thick, red] (1.25,0) to [out=270,in=0] (0,-1.25) to [out=180,in=270] (-1.25,0);	
	\draw[green] (1.25,0) to [out=225,in=315] (-1.25,0);
	\draw[dashed,green] (1.25,0) to [out=135,in=45] (-1.25,0);
	\draw[very thick, directed=.55] (2,2) to (-2,2);
	\node[opacity=1] at (1.4,1.65) {\tiny $_{k+1}$};
	\node[opacity=1] at (.25,-.75) {\tiny $1$};
	\node[opacity=1] at (-.25,.75) {\tiny $k$};
	\node[opacity=1] at (-.15,-.375) {\small $\bullet^{k-i}$};
	\node [opacity=1] at (1.5,-1.5) {$\idem_\lambda$};
\end{tikzpicture}
};
\endxy \quad = \quad
\xy
(0,0)*{
\begin{tikzpicture} [scale=.4,fill opacity=0.2]
	\path[fill=green] (1.25,0) to [out=90,in=0] (0,1.25) to [out=180,in=90] (-1.25,0) to [out=270,in=180] (0,-1.25) to [out=0,in=270] (1.25,0);
	\path[fill=green] (1.25,0) to [out=90,in=0] (0,1.25) to [out=180,in=90] (-1.25,0) to [out=270,in=180] (0,-1.25) to [out=0,in=270] (1.25,0);
	\path[fill=green] (2,0) to (1.25,0) to [out=90,in=0] (0,1.25) to [out=180,in=90] (-1.25,0) to (-2,0) to (-2,2) to (2,2);
	\path[fill=green] (2,0) to (1.25,0) to [out=270,in=0] (0,-1.25) to [out=180,in=270] (-1.25,0) to (-2,0) to (-2,-2) to (2,-2);	
	\draw[very thick, directed=.55] (2,-2) to (-2,-2);
	\draw[very thick] (2,-2) to (2,2);
	\draw[very thick] (-2,-2) to (-2,2);
	\draw[very thick, red, directed=.5] (1.25,0) to [out=90,in=0] (0,1.25) to [out=180,in=90] (-1.25,0);
	\draw[very thick, red] (1.25,0) to [out=270,in=0] (0,-1.25) to [out=180,in=270] (-1.25,0);	
	\draw[green] (1.25,0) to [out=225,in=315] (-1.25,0);
	\draw[dashed,green] (1.25,0) to [out=135,in=45] (-1.25,0);
	\draw[very thick, directed=.55] (2,2) to (-2,2);
	\node[opacity=1] at (1.4,1.65) {\tiny $_{k+1}$};
	\node[opacity=1] at (.25,-.75) {\tiny $1$};
	\node[opacity=1] at (-.25,.75) {\tiny $k$};
	\node[opacity=1] at (-.15,-.375) {\small $\bullet^{k}$};
	\node [opacity=1] at (1.5,-1.5) {$\idem_\lambda$};
\end{tikzpicture}
};
\endxy
}

\newcommand{\eqnAU}{
\xy
(0,0)*{
\begin{tikzpicture} [scale=.6,fill opacity=0.2]
	\node[opacity=1] at (1.5,1.65) {\tiny $_{a+b}$};
	\node[opacity=1] at (.25,1.4) {\tiny $a$};
	\node[opacity=1] at (-.25,2.1) {\tiny $b$};
	\path[fill=green] (2,2) to (.75,2) to (.75,-2) to (2,-2);
	\path[fill=green] (-2,2) to (-.75,2) to (-.75,-2) to (-2,-2);
	\path[fill=green] (.75,2) .. controls (.5,2.5) and (-.5,2.5) .. (-.75,2) to (-.75,-2) .. controls (-.5,-1.5) and (.5,-1.5) .. (.75,-2);
	\path[fill=green] (.75,2) .. controls (.5,1.5) and (-.5,1.5) .. (-.75,2) to (-.75,-2) .. controls (-.5,-2.5) and (.5,-2.5) .. (.75,-2);
	\draw[very thick, directed=.55] (2,-2) -- (.75,-2);
	\draw[very thick, directed=.55] (-.75,-2) -- (-2,-2);
	\draw[very thick, directed=.55] (.75,-2) .. controls (.5,-2.5) and (-.5,-2.5) .. (-.75,-2);
	\draw[very thick, directed=.55] (.75,-2) .. controls (.5,-1.5) and (-.5,-1.5) .. (-.75,-2);	
	\draw[very thick] (2,-2) to (2,2);
	\draw[very thick] (-2,-2) to (-2,2);
	\draw[very thick, red, directed=.5] (.75,-2) to (.75,2);
	\draw[very thick, red, directed=.5] (-.75,2) to (-.75,-2);	
	\draw[very thick, directed=.55] (2,2) -- (.75,2);
	\draw[very thick, directed=.55] (-.75,2) -- (-2,2);
	\draw[very thick, directed=.55] (.75,2) .. controls (.5,2.5) and (-.5,2.5) .. (-.75,2);
	\draw[very thick, directed=.55] (.75,2) .. controls (.5,1.5) and (-.5,1.5) .. (-.75,2);	
\end{tikzpicture}
};
\endxy
\quad = \sum_{\alpha \in P(a,b)} (-1)^{|\hat{\alpha}|} \quad
\xy
(0,0)*{
\begin{tikzpicture} [scale=.6,fill opacity=0.2]
	\node[opacity=1] at (1.5,1.65) {\tiny $_{a+b}$};
	\node[opacity=1] at (0,-2) {$_{\pi_{\alpha}}$};
	\node[opacity=1] at (0,2) {$_{\pi_{\hat{\alpha}}}$};
	\path[fill=green] (.75,-2) to [out=90,in=0] (0,-.5) to [out=180,in=90] (-.75,-2) .. controls (-.5,-1.5) and (.5,-1.5) .. (.75,-2);
	\path[fill=green] (.75,-2) to [out=90,in=0] (0,-.5) to [out=180,in=90] (-.75,-2) .. controls (-.5,-2.5) and (.5,-2.5) .. (.75,-2);
	\path[fill=green] (2,-2) to (.75,-2) to [out=90,in=0] (0,-.5) to [out=180,in=90] (-.75,-2) to (-2,-2) to (-2,2) to (-.75,2) 
		to [out=270,in=180] (0,.5) to [out=0,in=270] (.75,2) to (2,2);
	\path[fill=green] (.75,2) .. controls (.5,2.5) and (-.5,2.5) .. (-.75,2) to [out=270,in=180] (0,.5) to [out=0,in=270] (.75,2);
	\path[fill=green] (.75,2) .. controls (.5,1.5) and (-.5,1.5) .. (-.75,2) to [out=270,in=180] (0,.5) to [out=0,in=270] (.75,2);
	\draw[very thick, directed=.55] (2,-2) -- (.75,-2);
	\draw[very thick, directed=.55] (-.75,-2) -- (-2,-2);
	\draw[very thick, directed=.55] (.75,-2) .. controls (.5,-2.5) and (-.5,-2.5) .. (-.75,-2);
	\draw[very thick, directed=.55] (.75,-2) .. controls (.5,-1.5) and (-.5,-1.5) .. (-.75,-2);	
	\draw[very thick] (2,-2) to (2,2);
	\draw[very thick] (-2,-2) to (-2,2);
	\draw[very thick, red, directed=.5] (.75,-2) to [out=90,in=0] (0,-.5) to [out=180,in=90] (-.75,-2);
	\draw[very thick, red, directed=.5] (-.75,2) to [out=270,in=180] (0,.5) to [out=0,in=270] (.75,2);	
	\draw[very thick, directed=.55] (2,2) -- (.75,2);
	\draw[very thick, directed=.55] (-.75,2) -- (-2,2);
	\draw[very thick, directed=.55] (.75,2) .. controls (.5,2.5) and (-.5,2.5) .. (-.75,2);
	\draw[very thick, directed=.55] (.75,2) .. controls (.5,1.5) and (-.5,1.5) .. (-.75,2);	
\end{tikzpicture}
};
\endxy
 }
 
 \newcommand{\eqnAV}{
 \xy
(0,0)*{
\begin{tikzpicture} [fill opacity=0.2,decoration={markings, mark=at position 0.6 with {\arrow{>}}; }, scale=.4]
	\draw [fill=green] (1,1) -- (-1,2) -- (-1,-1) -- (1,-2) -- cycle;
	\node [opacity=1] at (0,0) {$\bullet^{N_\lambda}$};
	\node[ opacity=1] at (-.75,1.5) {\tiny{$1$}};
\end{tikzpicture}};
\endxy
\quad = \quad 0
 }
 
 \newcommand{\eqnAW}{
 \sum_{i=0}^{N_\lambda} {N_\lambda \choose i} (-\lambda)^{N_\lambda-i} ~\xy
(0,0)*{
\begin{tikzpicture} [fill opacity=0.2,decoration={markings, mark=at position 0.6 with {\arrow{>}}; }, scale=.4]
	\draw [fill=green] (1,1) -- (-1,2) -- (-1,-1) -- (1,-2) -- cycle;
	\node [opacity=1] at (0,0) {$\bullet^{i}$};
	\node [opacity=1] at (0.5,-1) {$\idem_\lambda$};
	\node[ opacity=1] at (-.75,1.5) {\tiny{$1$}};
\end{tikzpicture}};
\endxy
\quad = \quad 0
 }
 
  \newcommand{\eqnAX}{
\xy
(0,0)*{
\begin{tikzpicture} [fill opacity=0.2,decoration={markings, mark=at position 0.6 with {\arrow{>}}; }, scale=.4]
	\draw [fill=green] (1,1) -- (-1,2) -- (-1,-1) -- (1,-2) -- cycle;
	\node [opacity=1] at (0,0) {$\bullet^{N}$};
	\node[ opacity=1] at (-.75,1.5) {\tiny{$1$}};
\end{tikzpicture}};
\endxy
 \quad = \quad  \sum_{i=0}^{N-1} (-1)^{N-i-1} e_{N-i}(\Sigma) ~\xy
(0,0)*{
\begin{tikzpicture} [fill opacity=0.2,decoration={markings, mark=at position 0.6 with {\arrow{>}}; }, scale=.4]
	\draw [fill=green] (1,1) -- (-1,2) -- (-1,-1) -- (1,-2) -- cycle;
	\node [opacity=1] at (0,0) {$\bullet^{i}$};
	\node[ opacity=1] at (-.75,1.5) {\tiny{$1$}};
\end{tikzpicture}};
\endxy
 }
 
  \newcommand{\eqnAY}{
\xy
(0,0)*{
\begin{tikzpicture} [fill opacity=0.2,decoration={markings, mark=at position 0.6 with {\arrow{>}}; }, scale=.4]
	\draw [fill=green] (1,1) -- (-1,2) -- (-1,-1) -- (1,-2) -- cycle;
	\node [opacity=1] at (0,0) {$\bullet^{N}$};
	\node[ opacity=1] at (-.75,1.5) {\tiny{$1$}};
\end{tikzpicture}};
\endxy
 \quad = \quad 0
 }

\newcommand{\LH}[2]{\mathrm{KhR}^{#2}(#1)} 
\newcommand{\LC}[3]{\llbracket #1_{#3} \rrbracket ^{#2}} 
\newcommand{\taut}{\mathrm{taut}} 
\newcommand{\Homcat}{\Hom} 
\newcommand{\idem}{\mathbbm{1}}
\newcommand{\extp}{\bV^a \C[\xi]}

\author{David E.~V.~Rose}
\address{Department of Mathematics, University of Southern California, Los Angeles, CA 90089, USA}
\email{davidero@usc.edu}

\author{Paul Wedrich}
\address{Centre for Mathematical Sciences, University of Cambridge, CB3 0WB, England}
\email{p.wedrich@gmail.com}

\title{Deformations of colored $\protect\slnn{N}$ link homologies via foams}
\begin{abstract}
We generalize results of Lee, Gornik and Wu on the structure of deformed colored $\slnn{N}$ link homologies to the case of non-generic deformations. 
To this end, we use foam technology to give a completely combinatorial construction of Wu's deformed colored $\slnn{N}$ link homologies. 
By studying the underlying deformed higher representation theoretic structures and generalizing the Karoubi envelope approach of Bar-Natan and Morrison 
we explicitly compute the deformed invariants in terms of undeformed type A link homologies of lower rank and color.
\end{abstract}
\maketitle
\tableofcontents 

\section{Introduction}
\subsection{Statement of Results}

In \cite{Kh1}, Khovanov introduced a homology theory categorifying the Jones polynomial. 
This homology theory for links in $S^3$ has proven to be a powerful topological invariant, leading e.g. to Rasmussen's combinatorial proof of the Milnor conjecture on the slice genus of torus knots \cite{Ras2}.
Rasmussen's work built on earlier results of Lee, \cite{Lee}, who studied a deformed version of Khovanov's link invariant. 
Khovanov's theory is controlled by the Frobenius algebra $\C[X]/\langle X^2 \rangle$, which appears as the invariant of the unknot, and Lee showed that deforming this algebra to $\C[X]/\langle X^2-1 \rangle$ leads to a trivial link homology theory which assigns two copies of the vector space $\C$ to any knot.

Khovanov and Rozansky \cite{KR} used the theory of matrix factorizations to generalize Khovanov homology to a link homology theory 
(now called Khovanov-Rozansky homology) which categorifies the $\slnn{N}$ link polynomial.
This was later extended by Wu~\cite{Wu1} and Yonezawa~\cite{Yon} to a categorified invariant of links whose components are colored by fundamental representations $\bV^k\C^N$ of $\slnn{N}$ for $0\leq k \leq N$. 
In these theories, the underlying Frobenius algebra is isomorphic to $\C[X]/\langle X^N \rangle$.
Following work of Gornik \cite{Gor}, Rasmussen \cite{Ras2}, and Krasner \cite{Kra}, Wu defined deformed versions of $\slnn{N}$ link homology~\cite{Wu2}, in which this algebra is deformed to 
$\C[X]/\langle P(X)\rangle$, where $P(X)$ is an arbitrary degree-$N$ polynomial. 
Gornik \cite{Gor} and Wu \cite{Wu3} showed that if $P(X)$ has simple roots, this invariant assigns the direct sum of $N$ copies of the vector space $\C$ to any $1$-colored knot.
This result as well as Lee's, and their generalizations to the case of links, can be interpreted as saying that when $P(X)$ has simple roots, the 1-colored deformed homology of a link decomposes into the direct sum of $\slnn{1}$ homologies of various sub-links, which are always 1-dimensional. 
Other deformations have been studied for $N=2$ by Khovanov \cite{Kh2} and $N=3$ by Mackaay-Vaz \cite{MV}.

In this paper, we prove a vast generalization of these results, showing that the deformation of colored $\slnn{N}$ link homology corresponding to a general degree-$N$ monic polynomial $P(X)$ with root multiset $\Sigma$ decomposes into type A link homologies of lower rank and color. To this end, we use foam technology to define deformed colored $\slnn{N}$ link homologies $\mathrm{KhR}^\Sigma(-)$ and compare them to the undeformed colored $\slnn{M}$ link homologies $\mathrm{KhR}^{\slnn{M}}(-)$ constructed in \cite{QR}. Precisely, we show:
\begin{thm}
\label{mainthm}

Let $\cal{L}(a_1,\dots,a_k)$ be a $k$-component oriented, framed link with the $i^{\mathrm{th}}$ component colored by the fundamental $\slnn{N}$ representations $\bV^{a_i}\C^N$. Let $\Sigma$ be an $N$-element multiset of complex numbers consisting of $l$ distinct numbers occurring with multiplicities $N_1,\dots,N_l$. There is an isomorphism of vector spaces:
\begin{equation}
\LH{\cal{L}(a_1,\dots,a_k)}{\Sigma}\cong \bigoplus_{\substack{\sum_{j=1}^l b_{i,j}= a_i \\ 0\leq b_{i,j} \leq N_j}} \bigotimes_{j=1}^l  \LH{\cal{L}(b_{1,j},\dots, b_{k,j})}{\slnn{N_{j}}}
\end{equation}
which preserves the homological grading.
\end{thm}

\begin{exa-nono} Let $K$ be a 1-colored knot. Then the $\Sigma$-deformed $\slnn{N}$ homology of $K$ splits into the direct sum of undeformed $\slnn{M}$ homologies of $K$, and there is one $\slnn{M}$ summand for every root of multiplicity $M$ in $\Sigma$:
\begin{equation*}
\LH{K}{\Sigma} \cong \bigoplus_{j=1}^l \LH{K}{\slnn{N_{j}}}.
\end{equation*}
This seems to be a folk theorem in the link homology community, see e.g. \cite{GW}. The authors, however, could not find a proof in the literature.
\end{exa-nono}

\begin{exa-nono} Let $K$ be a knot we write $K^0$, $K^1$ and $K^2$ for its 0-, 1- and 2-colored variants respectively. Let $\Sigma=\{\lambda_1,\lambda_1, \lambda_2,\lambda_2,\lambda_2\}$ for complex numbers $\lambda_1 \neq \lambda_2$. Then the $\Sigma$-deformed $\slnn{5}$ homology of $K^2$ is:
\begin{align*}
\LH{K^2}{\Sigma} \cong& \Big( \LH{K^2}{\slnn{2}} \otimes \LH{K^0}{\slnn{3}} \Big) \oplus \Big( \LH{K^1}{\slnn{2}} \otimes \LH{K^1}{\slnn{3}} \Big) \\
& \quad \quad \quad \oplus \Big(  \LH{K^0}{\slnn{2}} \otimes \LH{K^2}{\slnn{3}} \Big) \\
\cong& \; \C \oplus \Big( \LH{K^1}{\slnn{2}} \otimes \LH{K^1}{\slnn{3}} \Big) \oplus \LH{K^2}{\slnn{3}}
\end{align*}
up to shifts in homological degree on the first and third direct summand.
\end{exa-nono}

Our main tool is the $\slnn{N}$ foam 2-category $\Foam{N}$ constructed in~\cite{QR}, as well as its relation to the Khovanov-Lauda~\cite{KhL3} 
diagrammatic categorification of quantum $\slm$ (see also work of Rouquier~\cite{Rou}).
The former can be viewed as the universal framework for the definition of categorified 
$\slnn{N}$ Reshetikhin-Turaev invariants of tangles $\tau$ colored by the fundamental representations of $\slnn{N}$.
More specifically, given any colored tangle $\tau$, there exists an invariant $\LC{\tau}{}{}$ taking values in the homotopy category of chain complexes over $\Hom$-categories 
in $\Foam{N}$, which consist of trivalent graphs called webs and decorated, singular cobordisms between them called foams. 
Passing to a quotient 2-category $\foam{N}{}^\bullet$ obtained by introducing an additional foam relation on decorated 1-labeled foam facets:
\begin{equation}
\label{eqn-undef}
\eqnAY
\end{equation}
it is shown in~\cite{QR} that the resulting bi-graded link invariant  $\mathrm{KhR}^{\slnn{N}}(-)$ essentially agrees with 
Wu's and Yonezawa's colored generalization of $\slnn{N}$ Khovanov-Rozansky link homology.
Equation \eqref{eqn-undef} corresponds to the fact that $X^N$ is the derivative of the polynomial used to give the potentials for the matrix factorizations in their construction.

In this paper, we analogously define the deformed colored $\slnn{N}$ link invariants $\LH{-}{\Sigma}$ for an $N$-element multiset $\Sigma$ of complex deformation parameters by working in a deformed foam 2-category $\Foam{N}^\Sigma$. It is defined  as the quotient 2-category of $\Foam{N}$ by the additional relation:
\begin{equation}
\label{eqn-def}
\eqnAX
\end{equation}
where $e_i(\Sigma)$ denotes the $i^{th}$ elementary symmetric polynomial in $N$ variables, evaluated at the multiset $\Sigma$. This is motivated by the relation between $\Foam{N}$ and categorified quantum groups and by Wu's construction of deformed $\slnn{N}$ link homology, which utilizes matrix factorizations whose potential is built from a polynomial with derivative $P(X)=\sum_{i=0}^N (-1)^{N-i}e_{N-i}(\Sigma)X^i$ with root multiset $\Sigma$. 

We prove Theorem \ref{mainthm} for the invariants constructed via the $\Sigma$-deformed $\slnn{N}$ foam 2-categories $\Foam{N}^\Sigma$ and undeformed $\slnn{N_j}$ foam 2-categories $\Foam{N_j}^\bullet$. To this end, we adapt Bar-Natan and Morrison's Karoubi envelope technology, originally used to give a ``local'' proof of Lee's deformation result \cite{BNM}, to the setting of foams, see Section \ref{subsubs_BNM}. The relation to Wu's deformed Khovanov-Rozansky link homology is then provided by the following generalization of Theorem 4.11 in \cite{QR}:

\begin{thm}
\label{mainthm2}
The invariant $\LH{\mathcal{L}}{\Sigma}$ constructed from $\foam{N}{}^\Sigma$ is (up to grading shifts) isomorphic to Wu's colored, deformed Khovanov-Rozansky homology of the mirror link $\cal{L}'$ with respect to deformation parameters $\Sigma$.
\end{thm}

In \cite{QR}, the identification of the link invariants defined via foams and and matrix factorizations is proven using results of Mackaay and Yonezawa \cite{MY} 
which imply the existence of a 2-representation of $\Foam{N}$ on $\slnn{N}$ matrix factorizations. 
Rather than adapt their results to the deformed case, we instead give a new, streamlined proof utilizing the theory of stabilization of matrix factorizations to 
give a 2-representation of $\foam{N}{}^\Sigma$ on a 2-category of deformed matrix factorizations. 
We believe this result might be of independent interest, see Section \ref{section-defmf}.

\subsection{Outlook}
There are several possible applications of the results in this paper. 

The first concerns the definition and study of concordance invariants in the spirit of Rasmussen's s-invariant \cite{Ras2}. 
Lobb \cite{Lo1, Lo2} has used Gornik's generic deformation of $\slnn{N}$ link homology to define concordance invariants that are analogous to Rasmussen's invariant. 
Lewark has recently proved independence results for these concordance invariants \cite{Lew}. 
It would be interesting to see whether deformations of colored $\slnn{N}$ link homologies also give rise to concordance invariants and 
whether foam technology can be used to prove (in)dependence properties between them. 

The next application concerns relations between type A link homology theories of different rank and color. 
Both experimental computations and physical reasons suggest that the type A link homology package carries a very rigid structure, 
which is only partially visible on the decategorified level of Reshetikhin-Turaev $\slnn{N}$ invariants, see \cite{DGR}, \cite{GW}, \cite{GS}, \cite{GGS} and \cite{GNSS}. One feature of this structure is the stabilization of $\slnn{N}$ link homologies as $N \to \infty$ to a triply-graded link homology theory that categorifies the HOMFLY-PT polynomial. 
The flip-side of this feature provides specialization spectral sequences from the triply-graded homology to $\slnn{N}$ homology for every $N$. 
Both of these features have been proven for the 1-colored case by Rasmussen \cite{Ras1}. 

Many other aspects of the conjectured structure have not been rigorously proven yet. One, however, that seems to be in reach is the existence of spectral sequences, or ``differentials'', between 
$\slnn{N}$ and $\slnn{M}$ link homologies for $N>M$. 
In analogy to the Lee-Rasmussen spectral sequence that links Khovanov homology to Lee's deformation, 
Wu has defined spectral sequences connecting the ordinary $\slnn{N}$ link homology to its deformations \cite{Wu2}. 
Together with Theorems \ref{mainthm} and \ref{mainthm2} that identify Wu's deformed invariants in terms of undeformed invariants, 
it should be possible to construct the desired spectral sequences.

Wu \cite{Wu2} has further proved that the deformed $\slnn{N}$ link homologies inherit a quantum filtration from the bi-graded undeformed invariant. 
We have ignored this filtration in this paper, but tracking it through the computation of the deformed invariants should significantly improve our understanding of the 
Rasmussen-type concordance invariants and the deformation spectral sequences. 

We also note that there are bi-graded equivariant versions of $\slnn{N}$ link homology, in which the deformation parameters are not specialized to complex numbers but kept as graded variables. 
The $\slnn{2}$ and $\slnn{3}$ equivariant theories have been studied in \cite{Kh1} and \cite{MV}. Krasner \cite{Kra} has introduced a version for 1-colored $\slnn{N}$ link homology for general $N$, 
which has been subsequently generalized by Wu \cite{Wu2} to arbitrary colorings by fundamental representations.  
It is an interesting question whether these equivariant theories also admit a definition via foam technology. This in turn would help to understand the quantum filtration on the deformed invariants.

Daniel Tubbenhauer has informed us that the deformations studied in this paper could be useful for writing down explicit isomorphisms between the centers of $\slnn{N}$ web algebras \cite{Ma} 
and cohomology rings of certain generalizations of Springer fibers, whose existence is guaranteed by \cite[Corollary 7.10]{Ma}. This would generalize work of Mackaay-Pan-Tubbenhauer \cite{MPT} on the case of $\slnn{3}$. \\

\textbf{Structure of this paper:}
We begin by introducing the necessary technology and graphical calculi in Section \ref{section-tech}. 
In particular, we discuss foams, categorified quantum groups, and the Karoubi envelope technology of Bar-Natan and Morrison. 
In Section \ref{section-defunknot} we study deformations of the higher representation theoretic structures that control deformed link invariants and prove a version of 
Theorem \ref{mainthm} for the unknot. 
Armed with this tool, we prove splitting relations in the deformed foam 2-category $\foam{N}{}^\Sigma$ and introduce a suitable idempotent completion 
$\hat{\foam{N}{}^\Sigma}$ in Section \ref{section-foamcat}. 
This section also establishes a 2-representation of the deformed foam 2-category on matrix factorizations, which is necessary for the proof of Theorem \ref{mainthm2}. 
Finally, Section \ref{section-decomp} contains the definition of the deformed link invariants $\LH{\mathcal{L}}{\Sigma}$ and the proofs of Theorems \ref{mainthm} and \ref{mainthm2}.\\

\textbf{Acknowledgements:}
The authors would like to thank Hanno Becker, Sabin Cautis, Eugene Gorsky, Aaron Lauda, Daniel Murfet, Jake Rasmussen and Ben Webster for helpful conversations and email exchanges during the course of this work. 
Special thanks go to Marko Sto\v{s}i\'c and Daniel Tubbenhauer for numerous useful comments on a draft of this paper.
D. R. would like to especially thank Sabin Cautis and Aaron Lauda, as preliminary work on this topic began in collaboration with them. 
P. W. would like to especially thank Jake Rasmussen for his support and guidance\footnote{P. W.'s PhD studies at the Department of Pure Mathematics and Mathematical Statistics, 
University of Cambridge, have been supported by the ERC grant ERC-2007-StG-205349 held by Ivan Smith and an EPSRC department doctoral training grant. The proof of Theorem \ref{mainthm} in this paper is part of P. W.'s thesis.}.

\section{Technology review}
\label{section-tech}
In this section, we recall the relevant machinery needed to prove Theorems \ref{mainthm} and \ref{mainthm2}. Explicitly, we discuss $\slnn{N}$ foams, categorified quantum groups, and their deformations, as well as the Karoubi envelope technology used in \cite{BNM}.

\subsection{Foams}
\label{subsubsec_foams}
Recall from \cite{QR} that a natural setting for a combinatorial formulation for Khovanov-Rozansky's $\slnn{N}$ link homology is the 2-category $\foam{N}$. 
In this 2-category, objects are given by sequences $\mathbf{a} = (a_1,\ldots, a_m)$ for $m > 0$ with $a_i \in \{1,\ldots,N\}$, 1-morphisms are formal direct sums of enhanced $\slnn{N}$ webs -- 
leftward oriented, labeled\footnote{These labels correspond to the ``colorings'' of tangle components by fundamental representations of $\slnn{N}$. 
We reserve the word ``color'' for certain idempotent decorations on foams and webs, see below.} trivalent graphs generated by
\[
\xy
(0,0)*{
\begin{tikzpicture}[scale=.4]
	\draw [very thick,directed=.55] (2.25,0) to (.75,0);
	\draw [very thick,directed=.55] (.75,0) to [out=135,in=0] (-1,.75);
	\draw [very thick,directed=.55] (.75,0) to [out=225,in=0] (-1,-.75);
	\node at (3.25,0) {\tiny $a+b$};
	\node at (-1.5,.75) {\tiny $a$};
	\node at (-1.5,-.75) {\tiny $b$};
\end{tikzpicture}
};
\endxy
\quad , \quad
\xy
(0,0)*{
\begin{tikzpicture}[scale=.4]
	\draw [very thick,rdirected=.55] (-2.25,0) to (-.75,0);
	\draw [very thick,rdirected=.55] (-.75,0) to [out=45,in=180] (1,.75);
	\draw [very thick,rdirected=.55] (-.75,0) to [out=315,in=180] (1,-.75);
	\node at (-3.25,0) {\tiny $a+b$};
	\node at (1.5,.75) {\tiny $a$};
	\node at (1.5,-.75) {\tiny $b$};
\end{tikzpicture}
};
\endxy
\]
which we view as mapping from the sequence determined by the labeled points on the right boundary to the one determined by the left. 
The 2-morphisms are matrices of enhanced $\slnn{N}$ foams, singular cobordisms between such webs generated by
\[
\xy
(0,0)*{
\begin{tikzpicture} [scale=.5,fill opacity=0.2]
	\path[fill=blue] (2.25,3) to (.75,3) to (.75,0) to (2.25,0);
	\path[fill=red] (.75,3) to [out=225,in=0] (-.5,2.5) to (-.5,-.5) to [out=0,in=225] (.75,0);
	\path[fill=red] (.75,3) to [out=135,in=0] (-1,3.5) to (-1,.5) to [out=0,in=135] (.75,0);	
	\draw [very thick,directed=.55] (2.25,0) to (.75,0);
	\draw [very thick,directed=.55] (.75,0) to [out=135,in=0] (-1,.5);
	\draw [very thick,directed=.55] (.75,0) to [out=225,in=0] (-.5,-.5);
	\draw[very thick, red, directed=.55] (.75,0) to (.75,3);
	\draw [very thick] (2.25,3) to (2.25,0);
	\draw [very thick] (-1,3.5) to (-1,.5);
	\draw [very thick] (-.5,2.5) to (-.5,-.5);
	\draw [very thick,directed=.55] (2.25,3) to (.75,3);
	\draw [very thick,directed=.55] (.75,3) to [out=135,in=0] (-1,3.5);
	\draw [very thick,directed=.55] (.75,3) to [out=225,in=0] (-.5,2.5);
	\node [blue, opacity=1]  at (1.5,2.5) {\tiny{$_{a+b}$}};
	\node[red, opacity=1] at (-.75,3.25) {\tiny{$b$}};
	\node[red, opacity=1] at (-.25,2.25) {\tiny{$a$}};		
\end{tikzpicture}
};
\endxy
\quad , \quad 
\xy
(0,0)*{
\begin{tikzpicture} [scale=.5,fill opacity=0.2]
	\path[fill=blue] (-2.25,3) to (-.75,3) to (-.75,0) to (-2.25,0);
	\path[fill=red] (-.75,3) to [out=45,in=180] (.5,3.5) to (.5,.5) to [out=180,in=45] (-.75,0);
	\path[fill=red] (-.75,3) to [out=315,in=180] (1,2.5) to (1,-.5) to [out=180,in=315] (-.75,0);	
	\draw [very thick,rdirected=.55] (-2.25,0) to (-.75,0);
	\draw [very thick,rdirected=.55] (-.75,0) to [out=315,in=180] (1,-.5);
	\draw [very thick,rdirected=.55] (-.75,0) to [out=45,in=180] (.5,.5);
	\draw[very thick, red, rdirected=.55] (-.75,0) to (-.75,3);
	\draw [very thick] (-2.25,3) to (-2.25,0);
	\draw [very thick] (1,2.5) to (1,-.5);
	\draw [very thick] (.5,3.5) to (.5,.5);
	\draw [very thick,rdirected=.55] (-2.25,3) to (-.75,3);
	\draw [very thick,rdirected=.55] (-.75,3) to [out=315,in=180] (1,2.5);
	\draw [very thick,rdirected=.55] (-.75,3) to [out=45,in=180] (.5,3.5);
	\node [blue, opacity=1]  at (-1.5,2.5) {\tiny{$_{a+b}$}};
	\node[red, opacity=1] at (.25,3.25) {\tiny{$b$}};
	\node[red, opacity=1] at (.75,2.25) {\tiny{$a$}};		
\end{tikzpicture}
};
\endxy
\quad , \quad
\xy
(0,0)*{
\begin{tikzpicture} [scale=.5,fill opacity=0.2]
	\path [fill=red] (4.25,-.5) to (4.25,2) to [out=165,in=15] (-.5,2) to (-.5,-.5) to 
		[out=0,in=225] (.75,0) to [out=90,in=180] (1.625,1.25) to [out=0,in=90] 
			(2.5,0) to [out=315,in=180] (4.25,-.5);
	\path [fill=red] (3.75,.5) to (3.75,3) to [out=195,in=345] (-1,3) to (-1,.5) to 
		[out=0,in=135] (.75,0) to [out=90,in=180] (1.625,1.25) to [out=0,in=90] 
			(2.5,0) to [out=45,in=180] (3.75,.5);
	\path[fill=blue] (.75,0) to [out=90,in=180] (1.625,1.25) to [out=0,in=90] (2.5,0);
	\draw [very thick,directed=.55] (2.5,0) to (.75,0);
	\draw [very thick,directed=.55] (.75,0) to [out=135,in=0] (-1,.5);
	\draw [very thick,directed=.55] (.75,0) to [out=225,in=0] (-.5,-.5);
	\draw [very thick,directed=.55] (3.75,.5) to [out=180,in=45] (2.5,0);
	\draw [very thick,directed=.55] (4.25,-.5) to [out=180,in=315] (2.5,0);
	\draw [very thick, red, directed=.75] (.75,0) to [out=90,in=180] (1.625,1.25);
	\draw [very thick, red] (1.625,1.25) to [out=0,in=90] (2.5,0);
	\draw [very thick] (3.75,3) to (3.75,.5);
	\draw [very thick] (4.25,2) to (4.25,-.5);
	\draw [very thick] (-1,3) to (-1,.5);
	\draw [very thick] (-.5,2) to (-.5,-.5);
	\draw [very thick,directed=.55] (4.25,2) to [out=165,in=15] (-.5,2);
	\draw [very thick, directed=.55] (3.75,3) to [out=195,in=345] (-1,3);
	\node [blue, opacity=1]  at (1.625,.5) {\tiny{$_{a+b}$}};
	\node[red, opacity=1] at (3.5,2.65) {\tiny{$b$}};
	\node[red, opacity=1] at (4,1.85) {\tiny{$a$}};		
\end{tikzpicture}
};
\endxy
\quad , \quad 
\xy
(0,0)*{
\begin{tikzpicture} [scale=.5,fill opacity=0.2]
	\path [fill=red] (4.25,2) to (4.25,-.5) to [out=165,in=15] (-.5,-.5) to (-.5,2) to
		[out=0,in=225] (.75,2.5) to [out=270,in=180] (1.625,1.25) to [out=0,in=270] 
			(2.5,2.5) to [out=315,in=180] (4.25,2);
	\path [fill=red] (3.75,3) to (3.75,.5) to [out=195,in=345] (-1,.5) to (-1,3) to [out=0,in=135]
		(.75,2.5) to [out=270,in=180] (1.625,1.25) to [out=0,in=270] 
			(2.5,2.5) to [out=45,in=180] (3.75,3);
	\path[fill=blue] (2.5,2.5) to [out=270,in=0] (1.625,1.25) to [out=180,in=270] (.75,2.5);
	\draw [very thick,directed=.55] (4.25,-.5) to [out=165,in=15] (-.5,-.5);
	\draw [very thick, directed=.55] (3.75,.5) to [out=195,in=345] (-1,.5);
	\draw [very thick, red, directed=.75] (2.5,2.5) to [out=270,in=0] (1.625,1.25);
	\draw [very thick, red] (1.625,1.25) to [out=180,in=270] (.75,2.5);
	\draw [very thick] (3.75,3) to (3.75,.5);
	\draw [very thick] (4.25,2) to (4.25,-.5);
	\draw [very thick] (-1,3) to (-1,.5);
	\draw [very thick] (-.5,2) to (-.5,-.5);
	\draw [very thick,directed=.55] (2.5,2.5) to (.75,2.5);
	\draw [very thick,directed=.55] (.75,2.5) to [out=135,in=0] (-1,3);
	\draw [very thick,directed=.55] (.75,2.5) to [out=225,in=0] (-.5,2);
	\draw [very thick,directed=.55] (3.75,3) to [out=180,in=45] (2.5,2.5);
	\draw [very thick,directed=.55] (4.25,2) to [out=180,in=315] (2.5,2.5);
	\node [blue, opacity=1]  at (1.625,2) {\tiny{$_{a+b}$}};
	\node[red, opacity=1] at (3.5,2.65) {\tiny{$b$}};
	\node[red, opacity=1] at (4,1.85) {\tiny{$a$}};		
\end{tikzpicture}
};
\endxy
\]
\[
\xy
(0,0)*{
\begin{tikzpicture} [scale=.5,fill opacity=0.2]
	\path[fill=blue] (-.75,4) to [out=270,in=180] (0,2.5) to [out=0,in=270] (.75,4) .. controls (.5,4.5) and (-.5,4.5) .. (-.75,4);
	\path[fill=red] (-.75,4) to [out=270,in=180] (0,2.5) to [out=0,in=270] (.75,4) -- (2,4) -- (2,1) -- (-2,1) -- (-2,4) -- (-.75,4);
	\path[fill=blue] (-.75,4) to [out=270,in=180] (0,2.5) to [out=0,in=270] (.75,4) .. controls (.5,3.5) and (-.5,3.5) .. (-.75,4);
	\draw[very thick, directed=.55] (2,1) -- (-2,1);
	\path (.75,1) .. controls (.5,.5) and (-.5,.5) .. (-.75,1); 
	\draw [very thick, red, directed=.7] (-.75,4) to [out=270,in=180] (0,2.5) to [out=0,in=270] (.75,4);
	\draw[very thick] (2,4) -- (2,1);
	\draw[very thick] (-2,4) -- (-2,1);
	\draw[very thick,directed=.55] (2,4) -- (.75,4);
	\draw[very thick,directed=.55] (-.75,4) -- (-2,4);
	\draw[very thick,directed=.55] (.75,4) .. controls (.5,3.5) and (-.5,3.5) .. (-.75,4);
	\draw[very thick,directed=.55] (.75,4) .. controls (.5,4.5) and (-.5,4.5) .. (-.75,4);
	\node [red, opacity=1]  at (1.5,3.5) {\tiny{$_{a+b}$}};
	\node[blue, opacity=1] at (-.25,3.375) {\tiny{$a$}};
	\node[blue, opacity=1] at (-.25,4.1) {\tiny{$b$}};	
\end{tikzpicture}
};
\endxy
\quad , \quad
\xy
(0,0)*{
\begin{tikzpicture} [scale=.5,fill opacity=0.2]
	\path[fill=blue] (-.75,-4) to [out=90,in=180] (0,-2.5) to [out=0,in=90] (.75,-4) .. controls (.5,-4.5) and (-.5,-4.5) .. (-.75,-4);
	\path[fill=red] (-.75,-4) to [out=90,in=180] (0,-2.5) to [out=0,in=90] (.75,-4) -- (2,-4) -- (2,-1) -- (-2,-1) -- (-2,-4) -- (-.75,-4);
	\path[fill=blue] (-.75,-4) to [out=90,in=180] (0,-2.5) to [out=0,in=90] (.75,-4) .. controls (.5,-3.5) and (-.5,-3.5) .. (-.75,-4);
	\draw[very thick, directed=.55] (2,-1) -- (-2,-1);
	\path (.75,-1) .. controls (.5,-.5) and (-.5,-.5) .. (-.75,-1); 
	\draw [very thick, red, directed=.7] (.75,-4) to [out=90,in=0] (0,-2.5) to [out=180,in=90] (-.75,-4);
	\draw[very thick] (2,-4) -- (2,-1);
	\draw[very thick] (-2,-4) -- (-2,-1);
	\draw[very thick,directed=.55] (2,-4) -- (.75,-4);
	\draw[very thick,directed=.55] (-.75,-4) -- (-2,-4);
	\draw[very thick,directed=.55] (.75,-4) .. controls (.5,-3.5) and (-.5,-3.5) .. (-.75,-4);
	\draw[very thick,directed=.55] (.75,-4) .. controls (.5,-4.5) and (-.5,-4.5) .. (-.75,-4);
	\node [red, opacity=1]  at (1.25,-1.25) {\tiny{$_{a+b}$}};
	\node[blue, opacity=1] at (-.25,-3.4) {\tiny{$b$}};
	\node[blue, opacity=1] at (-.25,-4.1) {\tiny{$a$}};
\end{tikzpicture}
};
\endxy
\quad , \quad
\xy
(0,0)*{
\begin{tikzpicture} [scale=.5,fill opacity=0.2]
	\path[fill=red] (-2.5,4) to [out=0,in=135] (-.75,3.5) to [out=270,in=90] (.75,.25)
		to [out=135,in=0] (-2.5,1);
	\path[fill=blue] (-.75,3.5) to [out=270,in=125] (.29,1.5) to [out=55,in=270] (.75,2.75) 
		to [out=135,in=0] (-.75,3.5);
	\path[fill=blue] (-.75,-.5) to [out=90,in=235] (.29,1.5) to [out=315,in=90] (.75,.25) 
		to [out=225,in=0] (-.75,-.5);
	\path[fill=red] (-2,3) to [out=0,in=225] (-.75,3.5) to [out=270,in=125] (.29,1.5)
		to [out=235,in=90] (-.75,-.5) to [out=135,in=0] (-2,0);
	\path[fill=red] (-1.5,2) to [out=0,in=225] (.75,2.75) to [out=270,in=90] (-.75,-.5)
		to [out=225,in=0] (-1.5,-1);
	\path[fill=red] (2,3) to [out=180,in=0] (.75,2.75) to [out=270,in=55] (.29,1.5)
		to [out=305,in=90] (.75,.25) to [out=0,in=180] (2,0);
	\draw[very thick, directed=.55] (2,0) to [out=180,in=0] (.75,.25);
	\draw[very thick, directed=.55] (.75,.25) to [out=225,in=0] (-.75,-.5);
	\draw[very thick, directed=.55] (.75,.25) to [out=135,in=0] (-2.5,1);
	\draw[very thick, directed=.55] (-.75,-.5) to [out=135,in=0] (-2,0);
	\draw[very thick, directed=.55] (-.75,-.5) to [out=225,in=0] (-1.5,-1);
	\draw[very thick, red, rdirected=.85] (-.75,3.5) to [out=270,in=90] (.75,.25);
	\draw[very thick, red, rdirected=.75] (.75,2.75) to [out=270,in=90] (-.75,-.5);	
	\draw[very thick] (-1.5,-1) -- (-1.5,2);	
	\draw[very thick] (-2,0) -- (-2,3);
	\draw[very thick] (-2.5,1) -- (-2.5,4);	
	\draw[very thick] (2,3) -- (2,0);
	\draw[very thick, directed=.55] (2,3) to [out=180,in=0] (.75,2.75);
	\draw[very thick, directed=.55] (.75,2.75) to [out=135,in=0] (-.75,3.5);
	\draw[very thick, directed=.65] (.75,2.75) to [out=225,in=0] (-1.5,2);
	\draw[very thick, directed=.55]  (-.75,3.5) to [out=225,in=0] (-2,3);
	\draw[very thick, directed=.55]  (-.75,3.5) to [out=135,in=0] (-2.5,4);
	\node[red, opacity=1] at (-2.25,3.375) {\tiny$c$};
	\node[red, opacity=1] at (-1.75,2.75) {\tiny$b$};	
	\node[red, opacity=1] at (-1.25,1.75) {\tiny$a$};
	\node[blue, opacity=1] at (0,2.75) {\tiny$_{b+c}$};
	\node[blue, opacity=1] at (0,.25) {\tiny$_{a+b}$};
	\node[red, opacity=1] at (1.35,2.5) {\tiny$_{a+b}$};	
	\node[red, opacity=1] at (1.35,2) {\tiny$_{+c}$};	
\end{tikzpicture}
};
\endxy
\quad , \quad
\xy
(0,0)*{
\begin{tikzpicture} [scale=.5,fill opacity=0.2]
	\path[fill=red] (-2.5,4) to [out=0,in=135] (.75,3.25) to [out=270,in=90] (-.75,.5)
		 to [out=135,in=0] (-2.5,1);
	\path[fill=blue] (-.75,2.5) to [out=270,in=125] (-.35,1.5) to [out=45,in=270] (.75,3.25) 
		to [out=225,in=0] (-.75,2.5);
	\path[fill=blue] (-.75,.5) to [out=90,in=235] (-.35,1.5) to [out=315,in=90] (.75,-.25) 
		to [out=135,in=0] (-.75,.5);	
	\path[fill=red] (-2,3) to [out=0,in=135] (-.75,2.5) to [out=270,in=125] (-.35,1.5) 
		to [out=235,in=90] (-.75,.5) to [out=225,in=0] (-2,0);
	\path[fill=red] (-1.5,2) to [out=0,in=225] (-.75,2.5) to [out=270,in=90] (.75,-.25)
		to [out=225,in=0] (-1.5,-1);
	\path[fill=red] (2,3) to [out=180,in=0] (.75,3.25) to [out=270,in=45] (-.35,1.5) 
		to [out=315,in=90] (.75,-.25) to [out=0,in=180] (2,0);				
	\draw[very thick, directed=.55] (2,0) to [out=180,in=0] (.75,-.25);
	\draw[very thick, directed=.55] (.75,-.25) to [out=135,in=0] (-.75,.5);
	\draw[very thick, directed=.55] (.75,-.25) to [out=225,in=0] (-1.5,-1);
	\draw[very thick, directed=.45]  (-.75,.5) to [out=225,in=0] (-2,0);
	\draw[very thick, directed=.35]  (-.75,.5) to [out=135,in=0] (-2.5,1);	
	\draw[very thick, red, rdirected=.75] (-.75,2.5) to [out=270,in=90] (.75,-.25);
	\draw[very thick, red, rdirected=.85] (.75,3.25) to [out=270,in=90] (-.75,.5);
	\draw[very thick] (-1.5,-1) -- (-1.5,2);	
	\draw[very thick] (-2,0) -- (-2,3);
	\draw[very thick] (-2.5,1) -- (-2.5,4);	
	\draw[very thick] (2,3) -- (2,0);
	\draw[very thick, directed=.55] (2,3) to [out=180,in=0] (.75,3.25);
	\draw[very thick, directed=.55] (.75,3.25) to [out=225,in=0] (-.75,2.5);
	\draw[very thick, directed=.55] (.75,3.25) to [out=135,in=0] (-2.5,4);
	\draw[very thick, directed=.55] (-.75,2.5) to [out=135,in=0] (-2,3);
	\draw[very thick, directed=.55] (-.75,2.5) to [out=225,in=0] (-1.5,2);
	\node[red, opacity=1] at (-2.25,3.75) {\tiny$c$};
	\node[red, opacity=1] at (-1.75,2.75) {\tiny$b$};	
	\node[red, opacity=1] at (-1.25,1.75) {\tiny$a$};
	\node[blue, opacity=1] at (-.125,2.25) {\tiny$_{a+b}$};
	\node[blue, opacity=1] at (-.125,.75) {\tiny$_{b+c}$};
	\node[red, opacity=1] at (1.35,2.75) {\tiny$_{a+b}$};
	\node[red, opacity=1] at (1.35,2.25) {\tiny$_{+c}$};	
\end{tikzpicture}
};
\endxy
\]
modulo isotopy and local relations\footnote{The colors red and blue in the foam graphics here and in \cite{QR} have no special significance. 
Later we will use specific colorings of foam facets to indicate decorations by idempotents, see Convention \ref{conv-coloring}. 
We strongly recommend reading this paper in its pdf version, as we make heavy use of colors.}. 
By convention, we view foams as mapping from the web determined by the bottom boundary to that on the top.
The facets of these foams again carry labelings by elements in $\{1,\ldots,N\}$, and a $k$-labeled facet may also be decorated by 
elements from the ring of symmetric functions in $k$ variables. 
Note that in \cite{QR}, the authors utilize the fact that this 2-category admits a grading; 
however, as we'll eventually pass to quotients $\foam{N}{}^\Sigma$ where this grading is broken, we won't concern ourselves with these issues.

Rather than recall the complete list of local relations, we refer the reader to \cite{QR} for full details, and 
list only a few that will play a substantial role in this paper:
\begin{equation}\label{FoamRel1}
\xy
(0,-8.8)*{
\begin{tikzpicture} [scale=.6,fill opacity=0.2]
	\path[fill=red] (1.5,4) to [out=180,in=45] (.75,3.5) to [out=270,in=90] (-.75,.25) 
		to [out=45,in=180] (1.5,1);
	\path[fill=blue] (-.75,2.75) to [out=270,in=135] (-.29,1.5) to [out=55,in=270] (.75,3.5) 
		to [out=180,in=45] (-.75,2.75);
	\path[fill=blue] (-.75,.25) to [out=90,in=240] (-.29,1.5) to [out=305,in=90] (.75,-.5) 
		to [out=180,in=315] (-.75,.25);
	\path[fill=red] (2,3) to [out=180,in=315] (.75,3.5) to [out=270,in=55] (-.29,1.5) 
		to [out=305,in=90] (.75,-.5) to [out=45,in=180] (2,0);
	\path[fill=red] (2.5,2) to [out=180,in=315] (-.75,2.75) to [out=270,in=90] (.75,-.5) 
		to [out=315,in=180] (2.5,-1);
	\path[fill=red] (-2,3) to [out=0,in=180] (-.75,2.75) to [out=270,in=125] (-.29,1.5) 
		to [out=235,in=90] (-.75,.25) to [out=180,in=0] (-2,0);
	\draw[very thick, directed=.75] (1.5,1) to [out=180,in=45] (-.75,.25);
	\draw[very thick, directed=.55] (2,0) to [out=180,in=45] (.75,-.5);
	\draw[very thick, directed=.55] (2.5,-1) to [out=180,in=315] (.75,-.5);
	\draw[very thick, directed=.55] (.75,-.5) to [out=180,in=315] (-.75,.25);
	\draw[very thick, directed=.55] (-.75,.25) to [out=180,in=0] (-2,0);
	\draw[very thick, red, directed=.75] (-.75,2.75) to [out=270,in=90] (.75,-.5);
	\draw[very thick, red, directed=.85] (.75,3.5) to [out=270,in=90] (-.75,.25);
	\draw[very thick] (1.5,1) -- (1.5,4);	
	\draw[very thick] (2,0) -- (2,3);
	\draw[very thick] (2.5,-1) -- (2.5,2);	
	\draw[very thick] (-2,3) -- (-2,0);
	\node[blue, opacity=1] at (0,2.75) {\tiny$_{b+c}$};
	\node[blue, opacity=1] at (0,.25) {\tiny$_{a+b}$};
\end{tikzpicture}
};
(-.3,8.75)*{
\begin{tikzpicture} [scale=.6,fill opacity=0.2]
	\path[fill=red] (1.5,4) to [out=180,in=45] (-.75,3.25) to [out=270,in=90] (.75,.5) 
		to [out=45,in=180] (1.5,1);
	\path[fill=blue] (-.75,3.25) to [out=270,in=135] (.35,1.5) to [out=45,in=270] (.75,2.5) 
		to [out=180,in=315] (-.75,3.25);
	\path[fill=blue] (-.75,-.25) to [out=90,in=225] (.35,1.5) to [out=315,in=90] (.75,.5) 
		to [out=180,in=45] (-.75,-.25);
	\path[fill=red] (2,3) to [out=180,in=45] (.75,2.5) to [out=270,in=45] (.35,1.5) 
		to [out=315,in=90] (.75,.5) to [out=315,in=180] (2,0);
	\path[fill=red] (2.5,2) to [out=180,in=315] (.75,2.5) to [out=270,in=90] (-.75,-.25) 
		to [out=315,in=180] (2.5,-1);
	\path[fill=red] (-2,3) to [out=0,in=180] (-.75,3.25) to [out=270,in=135] (.35,1.5) 
		to [out=225,in=90] (-.75,-.25) to [out=180,in=0] (-2,0);
	\draw[red, dashed] (1.5,1) to [out=180,in=45] (.75,.5);
	\draw[red, dashed] (2,0) to [out=180,in=315] (.75,.5);
	\draw[red,dashed] (2.5,-1) to [out=180,in=315] (-.75,-.25);	
	\draw[blue, dashed] (.75,.5) to [out=180,in=45] (-.75,-.25);
	\draw[red,dashed] (-.75,-.25) to [out=180,in=0] (-2,0);
	\draw[very thick, red, directed=.85] (-.75,3.25) to [out=270,in=90] (.75,.5);
	\draw[very thick, red, directed=.75] (.75,2.5) to [out=270,in=90] (-.75,-.25);
	\draw[very thick] (1.5,1) -- (1.5,4);	
	\draw[very thick] (2,0) -- (2,3);
	\draw[very thick] (2.5,-1) -- (2.5,2);	
	\draw[very thick] (-2,3) -- (-2,0);
	\draw[very thick, directed=.55] (1.5,4) to [out=180,in=45] (-.75,3.25);
	\draw[very thick, directed=.75] (2,3) to [out=180,in=45] (.75,2.5);
	\draw[very thick, directed=.75] (2.5,2) to [out=180,in=315] (.75,2.5);
	\draw[very thick, directed=.55] (.75,2.5) to [out=180,in=315] (-.75,3.25);
	\draw[very thick, directed=.55] (-.75,3.25) to [out=180,in=0] (-2,3);
	\node[red, opacity=1] at (1.25,3.75) {\tiny$c$};
	\node[red, opacity=1] at (1.75,2.75) {\tiny$b$};	
	\node[red, opacity=1] at (2.25,1.75) {\tiny$a$};
	\node[blue, opacity=1] at (.125,2.25) {\tiny$_{a+b}$};
	\node[red, opacity=1] at (-1.35,2.5) {\tiny$_{a+b+c}$};
\end{tikzpicture}
};
\endxy
\;\; = \;\;
\xy
(0,0)*{
\begin{tikzpicture} [scale=.6,fill opacity=0.2]
	\path[fill=red] (1.5,7) to [out=180,in=45] (-.75,6.25) to (-.75,.25) to [out=45,in=180] (1.5,1);
	\path[fill=red] (2,6) to [out=180,in=45] (.75,5.5) to (.75,-.5) to [out=45,in=180] (2,0);
	\path[fill=red] (2.5,5) to [out=180,in=315] (.75,5.5) to (.75,-.5) to [out=315,in=180] (2.5,-1);
	\path[fill=red] (-.75,6.25) to [out=180,in=0] (-2,6) to (-2,0) to [out=0,in=180] (-.75,.25);
	\path[fill=blue] (.75,5.5) to [out=180,in=315] (-.75,6.25) to (-.75,.25) to [out=315,in=180] (.75,-.5);
	\draw[very thick, directed=.55] (1.5,1) to [out=180,in=45] (-.75,.25);
	\draw[very thick, directed=.55] (2,0) to [out=180,in=45] (.75,-.5);
	\draw[very thick, directed=.55] (2.5,-1) to [out=180,in=315] (.75,-.5);
	\draw[very thick, directed=.55] (.75,-.5) to [out=180,in=315] (-.75,.25);
	\draw[very thick, directed=.55] (-.75,.25) to [out=180,in=0] (-2,0);
	\draw[very thick, red, directed=.55] (.75,5.5) to (.75,-.5);
	\draw[very thick, red, directed=.55] (-.75,6.25) to (-.75,.25);
	\draw[very thick] (1.5,1) -- (1.5,7);	
	\draw[very thick] (2,0) -- (2,6);
	\draw[very thick] (2.5,-1) -- (2.5,5);	
	\draw[very thick] (-2,6) -- (-2,0);
	\draw[very thick, directed=.55] (1.5,7) to [out=180,in=45] (-.75,6.25);
	\draw[very thick, directed=.75] (2,6) to [out=180,in=45] (.75,5.5);
	\draw[very thick, directed=.75] (2.5,5) to [out=180,in=315] (.75,5.5);
	\draw[very thick, directed=.55] (.75,5.5) to [out=180,in=315] (-.75,6.25);
	\draw[very thick, directed=.55] (-.75,6.25) to [out=180,in=0] (-2,6);	
	\node[red, opacity=1] at (1.25,6.75) {\tiny$c$};
	\node[red, opacity=1] at (1.75,5.75) {\tiny$b$};	
	\node[red, opacity=1] at (2.25,4.75) {\tiny$a$};
	\node[blue, opacity=1] at (.125,5) {\tiny$_{a+b}$};
	\node[red, opacity=1] at (-1.37,5.5) {\tiny$_{a+b+c}$};
\end{tikzpicture}
};
\endxy
\quad , \quad 
\xy
(0,-9)*{
\begin{tikzpicture} [scale=.6,fill opacity=0.2]
	\path[fill=red] (1.5,4) to [out=180,in=45] (-.75,3.25) to [out=270,in=90] (.75,.5) 
		to [out=45,in=180] (1.5,1);
	\path[fill=blue] (-.75,3.25) to [out=270,in=135] (.35,1.5) to [out=45,in=270] (.75,2.5) 
		to [out=180,in=315] (-.75,3.25);
	\path[fill=blue] (-.75,-.25) to [out=90,in=225] (.35,1.5) to [out=315,in=90] (.75,.5) 
		to [out=180,in=45] (-.75,-.25);
	\path[fill=red] (2,3) to [out=180,in=45] (.75,2.5) to [out=270,in=45] (.35,1.5) 
		to [out=315,in=90] (.75,.5) to [out=315,in=180] (2,0);
	\path[fill=red] (2.5,2) to [out=180,in=315] (.75,2.5) to [out=270,in=90] (-.75,-.25) 
		to [out=315,in=180] (2.5,-1);
	\path[fill=red] (-2,3) to [out=0,in=180] (-.75,3.25) to [out=270,in=135] (.35,1.5) 
		to [out=225,in=90] (-.75,-.25) to [out=180,in=0] (-2,0);
	\draw[very thick, directed=.55] (1.5,1) to [out=180,in=45] (.75,.5);
	\draw[very thick, directed=.75] (2,0) to [out=180,in=315] (.75,.5);
	\draw[very thick, directed=.55] (2.5,-1) to [out=180,in=315] (-.75,-.25);	
	\draw[very thick, directed=.55] (.75,.5) to [out=180,in=45] (-.75,-.25);
	\draw[very thick, directed=.55] (-.75,-.25) to [out=180,in=0] (-2,0);
	\draw[very thick, red, directed=.85] (-.75,3.25) to [out=270,in=90] (.75,.5);
	\draw[very thick, red, directed=.75] (.75,2.5) to [out=270,in=90] (-.75,-.25);
	\draw[very thick] (1.5,1) -- (1.5,4);	
	\draw[very thick] (2,0) -- (2,3);
	\draw[very thick] (2.5,-1) -- (2.5,2);	
	\draw[very thick] (-2,3) -- (-2,0);
	\node[blue, opacity=1] at (.125,.75) {\tiny$_{b+c}$};
\end{tikzpicture}
};
(-.3,9.5)*{
\begin{tikzpicture} [scale=.6,fill opacity=0.2]
	\path[fill=red] (1.5,4) to [out=180,in=45] (.75,3.5) to [out=270,in=90] (-.75,.25) 
		to [out=45,in=180] (1.5,1);
	\path[fill=blue] (-.75,2.75) to [out=270,in=135] (-.29,1.5) to [out=55,in=270] (.75,3.5) 
		to [out=180,in=45] (-.75,2.75);
	\path[fill=blue] (-.75,.25) to [out=90,in=240] (-.29,1.5) to [out=305,in=90] (.75,-.5) 
		to [out=180,in=315] (-.75,.25);
	\path[fill=red] (2,3) to [out=180,in=315] (.75,3.5) to [out=270,in=55] (-.29,1.5) 
		to [out=305,in=90] (.75,-.5) to [out=45,in=180] (2,0);
	\path[fill=red] (2.5,2) to [out=180,in=315] (-.75,2.75) to [out=270,in=90] (.75,-.5) 
		to [out=315,in=180] (2.5,-1);
	\path[fill=red] (-2,3) to [out=0,in=180] (-.75,2.75) to [out=270,in=125] (-.29,1.5) 
		to [out=235,in=90] (-.75,.25) to [out=180,in=0] (-2,0);
	\draw[red, dashed] (1.5,1) to [out=180,in=45] (-.75,.25);
	\draw[red, dashed] (2,0) to [out=180,in=45] (.75,-.5);
	\draw[red,dashed] (2.5,-1) to [out=180,in=315] (.75,-.5);
	\draw[blue,dashed] (.75,-.5) to [out=180,in=315] (-.75,.25);
	\draw[red,dashed] (-.75,.25) to [out=180,in=0] (-2,0);
	\draw[very thick, red, directed=.75] (-.75,2.75) to [out=270,in=90] (.75,-.5);
	\draw[very thick, red, directed=.85] (.75,3.5) to [out=270,in=90] (-.75,.25);
	\draw[very thick] (1.5,1) -- (1.5,4);	
	\draw[very thick] (2,0) -- (2,3);
	\draw[very thick] (2.5,-1) -- (2.5,2);	
	\draw[very thick] (-2,3) -- (-2,0);
	\draw[very thick, directed=.55] (1.5,4) to [out=180,in=45] (.75,3.5);
	\draw[very thick, directed=.75] (2,3) to [out=180,in=315] (.75,3.5);
	\draw[very thick, directed=.55] (2.5,2) to [out=180,in=315] (-.75,2.75);	
	\draw[very thick, directed=.55] (.75,3.5) to [out=180,in=45] (-.75,2.75);
	\draw[very thick, directed=.55] (-.75,2.75) to [out=180,in=0] (-2,3);
	\node[red, opacity=1] at (1.25,3.5) {\tiny$c$};
	\node[red, opacity=1] at (1.75,2.75) {\tiny$b$};	
	\node[red, opacity=1] at (2.25,1.75) {\tiny$a$};
	\node[blue, opacity=1] at (0,2.75) {\tiny$_{b+c}$};
	\node[blue, opacity=1] at (0,.25) {\tiny$_{a+b}$};
	\node[red, opacity=1] at (-1.35,2.5) {\tiny$_{a+b+c}$};
\end{tikzpicture}
};
\endxy
\;\; = \;\;
\xy
(0,0)*{
\begin{tikzpicture} [scale=.6,fill opacity=0.2]
	\path[fill=red] (1.5,7) to [out=180,in=45] (.75,6.5) to (.75,.5) to [out=45,in=180] (1.5,1);
	\path[fill=red] (2,6) to [out=180,in=315] (.75,6.5) to (.75,.5) to [out=315,in=180] (2,0);
	\path[fill=red] (2.5,5) to [out=180,in=315] (-.75,5.75) to (-.75,-.25) to [out=315,in=180] (2.5,-1);
	\path[fill=red] (-.75,5.75) to [out=180,in=0] (-2,6) to (-2,0) to [out=0,in=180] (-.75,-.25);
	\path[fill=blue] (.75,6.5) to [out=180,in=45] (-.75,5.75) to (-.75,-.25) to [out=45,in=180] (.75,.5);
	\draw[very thick, directed=.55] (1.5,1) to [out=180,in=45] (.75,.5);
	\draw[very thick, directed=.75] (2,0) to [out=180,in=315] (.75,.5);
	\draw[very thick, directed=.55] (2.5,-1) to [out=180,in=315] (-.75,-.25);	
	\draw[very thick, directed=.55] (.75,.5) to [out=180,in=45] (-.75,-.25);
	\draw[very thick, directed=.55] (-.75,-.25) to [out=180,in=0] (-2,0);
	\draw[very thick, red, directed=.55] (-.75,5.75) to (-.75,-.25);
	\draw[very thick, red, directed=.55] (.75,6.5) to (.75,.5);
	\draw[very thick] (1.5,1) -- (1.5,7);	
	\draw[very thick] (2,0) -- (2,6);
	\draw[very thick] (2.5,-1) -- (2.5,5);	
	\draw[very thick] (-2,6) -- (-2,0);
	\draw[very thick, directed=.55] (1.5,7) to [out=180,in=45] (.75,6.5);
	\draw[very thick, directed=.75] (2,6) to [out=180,in=315] (.75,6.5);
	\draw[very thick, directed=.75] (2.5,5) to [out=180,in=315] (-.75,5.75);	
	\draw[very thick, directed=.55] (.75,6.5) to [out=180,in=45] (-.75,5.75);
	\draw[very thick, directed=.55] (-.75,5.75) to [out=180,in=0] (-2,6);	
	\node[red, opacity=1] at (1.25,6.5) {\tiny$c$};
	\node[red, opacity=1] at (1.75,5.75) {\tiny$b$};	
	\node[red, opacity=1] at (2.25,4.75) {\tiny$a$};
	\node[blue, opacity=1] at (.125,5.75) {\tiny$_{b+c}$};
	\node[red, opacity=1] at (-1.37,5.5) {\tiny$_{a+b+c}$};	
\end{tikzpicture}
};
\endxy
\end{equation}
\begin{equation}\label{FoamRel2}
\xy
(0,0)*{
\begin{tikzpicture} [scale=.5,fill opacity=0.2]
	\node[red,opacity=1] at (1.5,1.65) {\tiny$_{a+b}$};
	\node[blue,opacity=1] at (.25,1.4) {\tiny$a$};
	\node[blue,opacity=1] at (-.25,2.1) {\tiny$b$};
	\path[fill=red] (2,2) to (.75,2) to (.75,-2) to (2,-2);
	\path[fill=red] (-2,2) to (-.75,2) to (-.75,-2) to (-2,-2);
	\path[fill=blue] (.75,2) .. controls (.5,2.5) and (-.5,2.5) .. (-.75,2) to (-.75,-2) .. controls (-.5,-1.5) and (.5,-1.5) .. (.75,-2);
	\path[fill=blue] (.75,2) .. controls (.5,1.5) and (-.5,1.5) .. (-.75,2) to (-.75,-2) .. controls (-.5,-2.5) and (.5,-2.5) .. (.75,-2);
	\draw[very thick, directed=.55] (2,-2) -- (.75,-2);
	\draw[very thick, directed=.55] (-.75,-2) -- (-2,-2);
	\draw[very thick, directed=.55] (.75,-2) .. controls (.5,-2.5) and (-.5,-2.5) .. (-.75,-2);
	\draw[very thick, directed=.55] (.75,-2) .. controls (.5,-1.5) and (-.5,-1.5) .. (-.75,-2);	
	\draw[very thick] (2,-2) to (2,2);
	\draw[very thick] (-2,-2) to (-2,2);
	\draw[very thick, red, directed=.5] (.75,-2) to (.75,2);
	\draw[very thick, red, directed=.5] (-.75,2) to (-.75,-2);	
	\draw[very thick, directed=.55] (2,2) -- (.75,2);
	\draw[very thick, directed=.55] (-.75,2) -- (-2,2);
	\draw[very thick, directed=.55] (.75,2) .. controls (.5,2.5) and (-.5,2.5) .. (-.75,2);
	\draw[very thick, directed=.55] (.75,2) .. controls (.5,1.5) and (-.5,1.5) .. (-.75,2);	
\end{tikzpicture}
};
\endxy
\;\; = \sum_{\alpha \in P(a,b)} (-1)^{|\hat{\alpha}|} \;\;
\xy
(0,0)*{
\begin{tikzpicture} [scale=.5,fill opacity=0.2]
	\node[red,opacity=1] at (1.5,1.65) {\tiny$_{a+b}$};
	\node[opacity=1] at (0,-2) {$_{\pi_{\alpha}}$};
	\node[opacity=1] at (0,2) {$_{\pi_{\hat{\alpha}}}$};
	\path[fill=blue] (.75,-2) to [out=90,in=0] (0,-.5) to [out=180,in=90] (-.75,-2) .. controls (-.5,-1.5) and (.5,-1.5) .. (.75,-2);
	\path[fill=blue] (.75,-2) to [out=90,in=0] (0,-.5) to [out=180,in=90] (-.75,-2) .. controls (-.5,-2.5) and (.5,-2.5) .. (.75,-2);
	\path[fill=red] (2,-2) to (.75,-2) to [out=90,in=0] (0,-.5) to [out=180,in=90] (-.75,-2) to (-2,-2) to (-2,2) to (-.75,2) 
		to [out=270,in=180] (0,.5) to [out=0,in=270] (.75,2) to (2,2);
	\path[fill=blue] (.75,2) .. controls (.5,2.5) and (-.5,2.5) .. (-.75,2) to [out=270,in=180] (0,.5) to [out=0,in=270] (.75,2);
	\path[fill=blue] (.75,2) .. controls (.5,1.5) and (-.5,1.5) .. (-.75,2) to [out=270,in=180] (0,.5) to [out=0,in=270] (.75,2);
	\draw[very thick, directed=.55] (2,-2) -- (.75,-2);
	\draw[very thick, directed=.55] (-.75,-2) -- (-2,-2);
	\draw[very thick, directed=.55] (.75,-2) .. controls (.5,-2.5) and (-.5,-2.5) .. (-.75,-2);
	\draw[very thick, directed=.55] (.75,-2) .. controls (.5,-1.5) and (-.5,-1.5) .. (-.75,-2);	
	\draw[very thick] (2,-2) to (2,2);
	\draw[very thick] (-2,-2) to (-2,2);
	\draw[very thick, red, directed=.5] (.75,-2) to [out=90,in=0] (0,-.5) to [out=180,in=90] (-.75,-2);
	\draw[very thick, red, directed=.5] (-.75,2) to [out=270,in=180] (0,.5) to [out=0,in=270] (.75,2);	
	\draw[very thick, directed=.55] (2,2) -- (.75,2);
	\draw[very thick, directed=.55] (-.75,2) -- (-2,2);
	\draw[very thick, directed=.55] (.75,2) .. controls (.5,2.5) and (-.5,2.5) .. (-.75,2);
	\draw[very thick, directed=.55] (.75,2) .. controls (.5,1.5) and (-.5,1.5) .. (-.75,2);	
\end{tikzpicture}
};
\endxy
\quad , \quad
\xy
(0,0)*{
\begin{tikzpicture} [scale=.6,fill opacity=0.2]
	\path[fill=blue] (2.25,3) to (.75,3) to (.75,0) to (2.25,0);
	\path[fill=red] (.75,3) to [out=225,in=0] (-.5,2.5) to (-.5,-.5) to [out=0,in=225] (.75,0);
	\path[fill=red] (.75,3) to [out=135,in=0] (-1,3.5) to (-1,.5) to [out=0,in=135] (.75,0);	
	\draw [very thick,directed=.55] (2.25,0) to (.75,0);
	\draw [very thick,directed=.55] (.75,0) to [out=135,in=0] (-1,.5);
	\draw [very thick,directed=.55] (.75,0) to [out=225,in=0] (-.5,-.5);
	\draw[very thick, red, directed=.55] (.75,0) to (.75,3);
	\draw [very thick] (2.25,3) to (2.25,0);
	\draw [very thick] (-1,3.5) to (-1,.5);
	\draw [very thick] (-.5,2.5) to (-.5,-.5);
	\draw [very thick,directed=.55] (2.25,3) to (.75,3);
	\draw [very thick,directed=.55] (.75,3) to [out=135,in=0] (-1,3.5);
	\draw [very thick,directed=.55] (.75,3) to [out=225,in=0] (-.5,2.5);
	\node [blue, opacity=1]  at (1.5,2.75) {\tiny{$_{a+b}$}};
	\node[red, opacity=1] at (-.75,3.25) {\tiny{$b$}};
	\node[red, opacity=1] at (-.25,2.25) {\tiny{$a$}};	
	\node[opacity=1] at (1.5,1.5) {\small$\pi_\gamma$};
\end{tikzpicture}
};
\endxy
\quad = \sum_{\alpha,\beta} c_{\alpha,\beta}^\gamma \quad
\xy
(0,0)*{
\begin{tikzpicture} [scale=.6,fill opacity=0.2]
	\path[fill=blue] (2.25,3) to (.75,3) to (.75,0) to (2.25,0);
	\path[fill=red] (.75,3) to [out=225,in=0] (-.5,2.5) to (-.5,-.5) to [out=0,in=225] (.75,0);
	\path[fill=red] (.75,3) to [out=135,in=0] (-1,3.5) to (-1,.5) to [out=0,in=135] (.75,0);	
	\draw [very thick,directed=.55] (2.25,0) to (.75,0);
	\draw [very thick,directed=.55] (.75,0) to [out=135,in=0] (-1,.5);
	\draw [very thick,directed=.55] (.75,0) to [out=225,in=0] (-.5,-.5);
	\draw[very thick, red, directed=.55] (.75,0) to (.75,3);
	\draw [very thick] (2.25,3) to (2.25,0);
	\draw [very thick] (-1,3.5) to (-1,.5);
	\draw [very thick] (-.5,2.5) to (-.5,-.5);
	\draw [very thick,directed=.55] (2.25,3) to (.75,3);
	\draw [very thick,directed=.55] (.75,3) to [out=135,in=0] (-1,3.5);
	\draw [very thick,directed=.55] (.75,3) to [out=225,in=0] (-.5,2.5);
	\node [blue, opacity=1]  at (1.5,2.75) {\tiny{$_{a+b}$}};
	\node[red, opacity=1] at (-.75,3.25) {\tiny{$b$}};
	\node[red, opacity=1] at (-.25,2.25) {\tiny{$a$}};	
	\node[opacity=1] at (.125,1.25) {\small$\pi_\alpha$};
	\node[opacity=1] at (-.25,3) {\small$\pi_\beta$};	
\end{tikzpicture}
};
\endxy
\end{equation}
\begin{equation}\label{FoamRel3}
\xy
(0,0)*{
\begin{tikzpicture} [scale=.45,fill opacity=0.2]
	\path[fill=red] (2.5,3) to [out=180,in=315] (-.5,3.25) to [out=270,in=90]  (1,.5) 
		to [out=315,in=180] (2.5,0);
	\path[fill=red] (2,4) to [out=180,in=0] (1,3.75) to [out=270,in=45] (.13,2) to [out=315,in=90]
		(1,.5) to [out=45,in=180] (2,1);
	\path[fill=blue] (1,3.75) to [out=225,in=45] (-.5,3.25) to [out=270,in=135] (.13,2) to 
		[out=45,in=270] (1,3.75);
	\path[fill=blue] (-.5,.5) to [out=90,in=235] (.13,2) to [out=315,in=90] (1,.5);
	\path[fill=red] (-2,1) to [out=0,in=135] (-.5,.5) to [out=90,in=270] (1,3.75) 
		to [out=135,in=0] (-2,4);
	\path[fill=red] (-1.5,3) to [out=0,in=180] (-.5,3.25) to [out=270,in=135] (.13,2)
		to [out=235,in=90] (-.5,.5) to [out=225,in=0] (-1.5,0);
	\path[fill=red] (2.5,6) to [out=180,in=315] (1,6.5) to [out=270,in=90]  (-.5,3.25)
		to [out=315,in=180] (2.5,3);
	\path[fill=red] (2,7) to [out=180,in=45] (1,6.5) to [out=270,in=55] (.37,5) to 
		[out=315,in=90] (1,3.75) to [out=0,in=180] (2,4);
	\path[fill=blue] (1,6.5) to [out=270,in=55] (.37,5) to [out=135,in=270] (-.5,6.5);
	\path[fill=blue] (1,3.75) to [out=225,in=45] (-.5,3.25) to [out=90,in=225] (.37,5) to
		[out=315,in=90] (1,3.75);
	\path[fill=red] (-2,4) to [out=0,in=135] (1,3.75) to [out=90,in=270] (-.5,6.5)
		 to [out=135,in=0] (-2,7);
	\path[fill=red] (-1.5,6) to [out=0,in=225] (-.5,6.5) to [out=270,in=135] (.37,5) to 
		[out=225,in=90] (-.5,3.25) to [out=180,in=0] (-1.5,3);
	\draw[very thick, directed=.55] (2,1) to [out=180,in=45] (1,.5);
	\draw[very thick,directed=.55] (2.5,0) to [out=180,in=315] (1,.5);
	\draw[very thick, directed=.55] (1,.5) to (-.5,.5);
	\draw[very thick, directed=.55] (-.5,.5) to [out=225,in=0] (-1.5,0);
	\draw[very thick, directed=.55] (-.5,.5) to [out=135,in=0] (-2,1);
	\draw[very thick, red, directed=.65] (-.5,.5) to [out=90,in=270] (1,3.75);
	\draw[very thick, red, directed=.65] (-.5,3.25) to [out=270,in=90]  (1,.5);
	\draw[very thick] (2,1) to (2,4);
	\draw[very thick] (2.5,0) to (2.5,3);
	\draw[very thick] (-1.5,0) to (-1.5,3);
	\draw[very thick] (-2,1) to (-2,4);	
	\draw[red, dashed] (2,4) to [out=180,in=0] (1,3.75);
	\draw[red, dashed] (1,3.75) to [out=135,in=0] (-2,4);
	\draw[red, dashed] (2.5,3) to [out=180,in=315] (-.5,3.25);
	\draw[red, dashed] (-.5,3.25) to [out=180,in=0] (-1.5,3);
	\draw[blue, dashed] (1,3.75) to [out=225,in=45] (-.5,3.25);
	\draw[very thick, red, directed=.65] (1,3.75) to [out=90,in=270] (-.5,6.5);
	\draw[very thick, red, directed=.65] (1,6.5) to [out=270,in=90]  (-.5,3.25);
	\draw[very thick] (2,4) to (2,7);
	\draw[very thick] (2.5,3) to (2.5,6);
	\draw[very thick] (-1.5,3) to (-1.5,6);
	\draw[very thick] (-2,4) to (-2,7);	
	\draw[very thick, directed=.55] (2,7) to [out=180,in=45] (1,6.5);
	\draw[very thick,directed=.55] (2.5,6) to [out=180,in=315] (1,6.5);
	\draw[very thick, directed=.55] (1,6.5) to (-.5,6.5);
	\draw[very thick, directed=.55] (-.5,6.5) to [out=225,in=0] (-1.5,6);
	\draw[very thick, directed=.55] (-.5,6.5) to [out=135,in=0] (-2,7);
	\node[blue,opacity=1] at (.25,.75) {\tiny$_{a+b}$};
	\node[red,opacity=1] at (1.75,6.75) {\tiny$b$};
	\node[red,opacity=1] at (2.25,5.75) {\tiny$a$};
	\node[red,opacity=1] at (-1.75,6.75) {\tiny$c$};
	\node[blue,opacity=1] at (.25,3.75) {\tiny$_{b-c}$};
	\node[blue,opacity=1] at (.25,6.25) {\tiny$_{a+b}$};
	\node[red,opacity=1] at (-1,5.75) {\tiny$_{a+b}$};
	\node[red,opacity=1] at (-1,5.5) {\tiny$_{-c}$};
\end{tikzpicture}
};
\endxy
\;\; = \sum_{\alpha \in P(a,c)} (-1)^{|\hat{\alpha}|} \;\;
\xy
(0,0)*{
\begin{tikzpicture} [scale=.45,fill opacity=0.2]
	\path[fill=red] (2,7) to [out=180,in=45] (1,6.5) to (1,.5) to [out=45,in=180] (2,1);
	\path[fill=red] (2.5,6) to [out=180,in=315] (1,6.5) to (1,.5) to [out=315,in=180] (2.5,0);
	\path[fill=blue] (1,6.5) to (-.5,6.5) to (-.5,.5) to (1,.5);
	\path[fill=red] (-1.5,0) to [out=0,in=225] (-.5,.5) to (-.5,6.5) to [out=225,in=0] (-1.5,6);
	\path[fill=red] (-2,1) to [out=0,in=135] (-.5,.5) to (-.5,6.5) to [out=135,in=0] (-2,7);
	\draw[very thick, directed=.55] (2,1) to [out=180,in=45] (1,.5);
	\draw[very thick,directed=.55] (2.5,0) to [out=180,in=315] (1,.5);
	\draw[very thick, directed=.55] (1,.5) to (-.5,.5);
	\draw[very thick, directed=.55] (-.5,.5) to [out=225,in=0] (-1.5,0);
	\draw[very thick, directed=.55] (-.5,.5) to [out=135,in=0] (-2,1);
	\draw[very thick, red, directed=.55] (1,6.5) to (1,.5);
	\draw[very thick, red, directed=.55] (-.5,.5) to (-.5,6.5);
	\draw[very thick] (2,1) to (2,7);
	\draw[very thick] (2.5,0) to (2.5,6);
	\draw[very thick] (-1.5,0) to (-1.5,6);
	\draw[very thick] (-2,1) to (-2,7);	
	\draw[very thick, directed=.55] (2,7) to [out=180,in=45] (1,6.5);
	\draw[very thick,directed=.55] (2.5,6) to [out=180,in=315] (1,6.5);
	\draw[very thick, directed=.55] (1,6.5) to (-.5,6.5);
	\draw[very thick, directed=.55] (-.5,6.5) to [out=225,in=0] (-1.5,6);
	\draw[very thick, directed=.65] (-.5,6.5) to [out=135,in=0] (-2,7);
	\node[red,opacity=1] at (1.75,6.75) {\tiny$b$};
	\node[red,opacity=1] at (2.25,5.75) {\tiny$a$};
	\node[red,opacity=1] at (-1.75,6.75) {\tiny$c$};
	\node[blue,opacity=1] at (.25,6.25) {\tiny$_{a+b}$};
	\node[red,opacity=1] at (-1,5.75) {\tiny$_{a+b}$};
	\node[red,opacity=1] at (-1,5.5) {\tiny$_{-c}$};
	\node[opacity=1] at (2,.5) {$_{\pi_{\alpha}}$};
	\node[opacity=1] at (-1.25,6.5) {$_{\pi_{\hat{\alpha}}}$};
\end{tikzpicture}
};
\endxy
\quad , \quad
\xy
(0,0)*{
\begin{tikzpicture} [scale=.45,fill opacity=0.2]
	\path[fill=red] (2.5,6) to [out=180,in=315] (-.5,6.25) to [out=270,in=90]  (1,3.5) 
		to [out=315,in=180] (2.5,3);
	\path[fill=red] (2,7) to [out=180,in=0] (1,6.75) to [out=270,in=45] (.13,5) to [out=315,in=90]
		(1,3.5) to [out=45,in=180] (2,4);
	\path[fill=blue] (1,6.75) to [out=225,in=45] (-.5,6.25) to [out=270,in=135] (.13,5) to 
		[out=45,in=270] (1,6.75);
	\path[fill=blue] (-.5,3.5) to [out=90,in=235] (.13,5) to [out=315,in=90] (1,3.5);
	\path[fill=red] (-2,4) to [out=0,in=135] (-.5,3.5) to [out=90,in=270] (1,6.75) 
		to [out=135,in=0] (-2,7);
	\path[fill=red] (-1.5,6) to [out=0,in=180] (-.5,6.25) to [out=270,in=135] (.13,5)
		to [out=235,in=90] (-.5,3.5) to [out=225,in=0] (-1.5,3);
	\path[fill=red] (2.5,3) to [out=180,in=315] (1,3.5) to [out=270,in=90]  (-.5,.25)
		to [out=315,in=180] (2.5,0);
	\path[fill=red] (2,4) to [out=180,in=45] (1,3.5) to [out=270,in=55] (.37,2) to 
		[out=315,in=90] (1,.75) to [out=0,in=180] (2,1);
	\path[fill=blue] (1,3.5) to [out=270,in=55] (.37,2) to [out=135,in=270] (-.5,3.5);
	\path[fill=blue] (1,.75) to [out=225,in=45] (-.5,.25) to [out=90,in=225] (.37,2) to
		[out=315,in=90] (1,.75);
	\path[fill=red] (-2,1) to [out=0,in=135] (1,.75) to [out=90,in=270] (-.5,3.5)
		 to [out=135,in=0] (-2,4);
	\path[fill=red] (-1.5,3) to [out=0,in=225] (-.5,3.5) to [out=270,in=135] (.37,2) to 
		[out=225,in=90] (-.5,.25) to [out=180,in=0] (-1.5,0);
	\draw[very thick, directed=.65] (2,1) to [out=180,in=0] (1,.75);
	\draw[very thick, directed=.55] (1,.75) to [out=135,in=0] (-2,1);
	\draw[very thick, directed=.55] (2.5,0) to [out=180,in=315] (-.5,.25);
	\draw[very thick, directed=.55] (-.5,.25) to [out=180,in=0] (-1.5,0);
	\draw[very thick, directed=.55] (1,.75) to [out=225,in=45] (-.5,.25);
	\draw[very thick, red, directed=.65] (1,.75) to [out=90,in=270] (-.5,3.5);
	\draw[very thick, red, directed=.65] (1,3.5) to [out=270,in=90]  (-.5,.25);
	\draw[very thick] (2,1) to (2,4);
	\draw[very thick] (2.5,0) to (2.5,3);
	\draw[very thick] (-1.5,0) to (-1.5,3);
	\draw[very thick] (-2,1) to (-2,4);	
	\draw[dashed, red] (2,4) to [out=180,in=45] (1,3.5);
	\draw[red, dashed] (2.5,3) to [out=180,in=315] (1,3.5);
	\draw[blue, dashed] (1,3.5) to (-.5,3.5);
	\draw[red, dashed] (-.5,3.5) to [out=225,in=0] (-1.5,3);
	\draw[dashed, red] (-.5,3.5) to [out=135,in=0] (-2,4);
	\draw[very thick, red, directed=.65] (-.5,3.5) to [out=90,in=270] (1,6.75);
	\draw[very thick, red, directed=.65] (-.5,6.25) to [out=270,in=90]  (1,3.5);
	\draw[very thick] (2,4) to (2,7);
	\draw[very thick] (2.5,3) to (2.5,6);
	\draw[very thick] (-1.5,3) to (-1.5,6);
	\draw[very thick] (-2,4) to (-2,7);	
	\draw[very thick, directed=.65] (2,7) to [out=180,in=0] (1,6.75);
	\draw[very thick, directed=.55] (1,6.75) to [out=135,in=0] (-2,7);
	\draw[very thick, directed=.55] (2.5,6) to [out=180,in=315] (-.5,6.25);
	\draw[very thick, directed=.55] (-.5,6.25) to [out=180,in=0] (-1.5,6);
	\draw[very thick, directed=.55] (1,6.75) to [out=225,in=45] (-.5,6.25);
	\node[blue,opacity=1] at (.25,.825) {\tiny$_{b-c}$};
	\node[red,opacity=1] at (1.75,6.75) {\tiny$b$};
	\node[red,opacity=1] at (2.25,5.75) {\tiny$a$};
	\node[red,opacity=1] at (-1.75,6.75) {\tiny$c$};
	\node[blue,opacity=1] at (.25,6.25) {\tiny$_{b-c}$};
	\node[blue,opacity=1] at (.25,3.75) {\tiny$_{a+b}$};
	\node[red,opacity=1] at (-1,5.75) {\tiny$_{a+b}$};
	\node[red,opacity=1] at (-1,5.5) {\tiny$_{-c}$};
\end{tikzpicture}
};
\endxy
\;\; =  \sum_{\alpha \in P(a,c)} (-1)^{|\hat{\alpha}|} \;\;
\xy
(0,0)*{
\begin{tikzpicture} [scale=.45,fill opacity=0.2]
	\path[fill=red] (2,7) to [out=180,in=0] (1,6.75) to (1,.75) to [out=0,in=180] (2,1);
	\path[fill=red] (2.5,6) to [out=180,in=315] (-.5,6.25) to (-.5,.25) to [out=315,in=180] (2.5,0);
	\path[fill=blue] (1,6.75) to [out=225,in=45] (-.5,6.25) to (-.5,.25) to [out=45,in=225] 
		 (1,.75) to (1,6.75);
	\path[fill=red] (-1.5,6) to [out=0,in=180] (-.5,6.25) to (-.5,.25) to [out=180,in=0] (-1.5,0);
	\path[fill=red] (-2,1) to [out=0,in=135] (1,.75) to (1,6.75) to [out=135,in=0] (-2,7);
	\draw[very thick, directed=.65] (2,1) to [out=180,in=0] (1,.75);
	\draw[very thick, directed=.65] (1,.75) to [out=135,in=0] (-2,1);
	\draw[very thick, directed=.55] (2.5,0) to [out=180,in=315] (-.5,.25);
	\draw[very thick, directed=.55] (-.5,.25) to [out=180,in=0] (-1.5,0);
	\draw[very thick, directed=.55] (1,.75) to [out=225,in=45] (-.5,.25);
	\draw[very thick, red, directed=.45] (-.5,6.25) to (-.5,.25);
	\draw[very thick, red, directed=.45] (1,.75) to (1,6.75);
	\draw[very thick] (2,1) to (2,7);
	\draw[very thick] (2.5,0) to (2.5,6);
	\draw[very thick] (-1.5,0) to (-1.5,6);
	\draw[very thick] (-2,1) to (-2,7);	
	\draw[very thick, directed=.65] (2,7) to [out=180,in=0] (1,6.75);
	\draw[very thick, directed=.55] (1,6.75) to [out=135,in=0] (-2,7);
	\draw[very thick, directed=.65] (2.5,6) to [out=180,in=315] (-.5,6.25);
	\draw[very thick, directed=.55] (-.5,6.25) to [out=180,in=0] (-1.5,6);
	\draw[very thick, directed=.55] (1,6.75) to [out=225,in=45] (-.5,6.25);
	\node[red,opacity=1] at (1.75,6.75) {\tiny$b$};
	\node[red,opacity=1] at (2.25,5.75) {\tiny$a$};
	\node[red,opacity=1] at (-1.75,6.75) {\tiny$c$};
	\node[blue,opacity=1] at (.25,6.25) {\tiny$_{b-c}$};
	\node[red,opacity=1] at (-1,5.75) {\tiny$_{a+b}$};
	\node[red,opacity=1] at (-1,5.5) {\tiny$_{-c}$};
	\node[opacity=1] at (1.75,.25) {$_{\pi_{\alpha}}$};
	\node[opacity=1] at (-1,6.75) {$_{\pi_{\hat{\alpha}}}$};
\end{tikzpicture}
};
\endxy
\end{equation}
where here $P(a,b)$ denotes the set of partitions of length $\leq a$ with each part $\leq b$, $\pi_{\alpha}$ denotes the Schur function corresponding to the partition $\alpha$ and the 
$c_{\alpha,\beta}^\gamma$ are the corresponding Littlewood-Richardson coefficients.

A tangle diagram whose components are labeled by elements in $\{1,\ldots,N\}$ determines a complex in $\foam{N}$, 
which is, up to homotopy equivalence, an invariant of the underlying framed tangle. In the case that the tangle is a link, passing to the quotient $\foam{N}^{\bullet}$ and applying a representable functor yields a complex of vector spaces whose homology is isomorphic (up to shifts and grading conventions) 
to the $\slnn{N}$ link homology defined by Khovanov and Rozansky and generalized to the colored case by Wu and Yonezawa. 

\subsection{Higher representation theory}

The construction of $\foam{N}$ was motivated by a desired relation to higher representation theory. 
The categorified quantum group $\cal{U}_Q(\slm)$ is the 2-category whose objects are given by $\slm$ weights $\lambda$, 
and whose $1$-morphisms are formal direct sums of (shifts $\{k\}$ of) compositions of
\[
\onel, \quad \onenn{\l+\alpha_i} \sE_i = \onenn{\l+\alpha_i} \sE_i\onel = \sE_i \onel, \quad \text{ and }\quad 
\onenn{\lambda-\alpha_i} \sF_i = \onenn{\lambda-\alpha_i} \sF_i\onel = \sF_i\onel
\]
for $i \in \{1,\ldots,m-1\}$ and where the $\alpha_i$ are the simple $\slm$ roots. 
The 2-morphisms are given by matrices of linear combinations of (degree zero) string diagrams -- dotted, immersed oriented curves 
colored by elements $i \in \{1,\ldots,m-1\}$ with top and bottom boundary, e.g.:
\[
\xy
(0,0)*{
\begin{tikzpicture} [scale=.75]
	\draw[thick, ->] (0,2) to [out=270,in=90] (2,0);
	\draw[thick, ->] (1,0) to [out=90,in=270] (2,2);
	\draw[thick, ->] (3,2) to [out=270,in=0] (2,1) to [out=180,in=270] (1,2);
	\node at (2.5,.5) {$\lambda$};
	\node at (2,1) {$\bullet$};
	\node at (.5,1.25) {$\bullet$};
	\node at (.875,.5) {\tiny$i$};
	\node at (-.125,1.5) {\tiny$j$};
	\node at (3,1.25) {\tiny$k$};
\end{tikzpicture}
};
\endxy
\]
modulo local relations. 
The domain 1-morphism of such a diagram is given (up to grading shifts) by considering the orientations and 
labelings of the strands incident upon the bottom boundary, reading upward strands as $\cal{E}$'s and downward strands as $\cal{F}$'s, 
and similarly for the codomain by considering the top boundary. For example, the domain and codomain of the above string diagram are 
(up to shifts) $\cal{E}_i \cal{F}_j \onel$ and $\cal{F}_j \cal{E}_k \cal{E}_i \cal{F}_k \onel$. 

We refer the reader to the work of Lauda~\cite{Lau1} and Khovanov-Lauda~\cite{KhL1,KhL2,KhL3} (see also independent work of Rouquier~\cite{Rou})
for a detailed discussion on categorified quantum $\slm$.
The main result of \cite{KhL3} is that the 2-category $\UcatD_Q(\slm)$, obtained by passing to the Karoubi envelope in each $\Hom$-category 
of $\Ucat_Q(\slm)$, categorifies quantum $\slm$. Explicitly, they show that the Lusztig idempotent form $\U_q(\slm)$ of the quantum group 
is isomorphic to the category obtained by taking the Grothendieck group $\mathrm{K}_0$ 
in each $\Hom$-category of $\UcatD_Q(\slm)$.
We'll assume some familiarity with categorified quantum groups for the duration, and utilize the conventions and notation from \cite{QR}.

The 2-category $\foam{N}$ is constructed to give a 2-representation of $\cal{U}_Q(\slm)$ via categorical skew Howe duality. 
Recall that work of Cautis-Kamnitzer-Licata~\cite{CKL} and Cautis-Kamnitzer-Morrison~\cite{CKM} shows that the commuting 
(skew Howe dual) actions of quantum $\glm$ and $\slnn{N}$ on the vector spaces $\bV_q^k(\C_q^m \otimes \C_q^N)$ induce a functor
\[
\varphi_m \colon U_q(\glm) \to \mathrm{Rep}(U_q(\slnn{N}))
\]
which sends a $\glm$ weight $\mathbf{a} = (a_1,\ldots,a_m)$ to the tensor product of fundamental quantum $\slnn{N}$ representations
$\bV_q^{a_1} \C_q^N \otimes \cdots \otimes \bV_q^{a_m} \C_q^N$. 
In fact, Cautis-Kamnitzer-Morrison use this to give a completely combinatorial description for the full subcategory 
of quantum $\slnn{N}$ representations generated by the fundamental representations. 
In their description, objects are given as in $\foam{N}$ and morphisms are given by linear combinations of 
$\slnn{N}$ webs, modulo planar isotopy and relations.

By design, the 2-category $\foam{N}$ gives a categorification of this result, i.e. 
it admits a 2-functor $\Ucat_Q(\glm) \xrightarrow{\Phi_m} \foam{N}$
so that the diagram
\[
\xymatrix{
\Ucat_Q(\glm) \ar[d]^-{\mathrm{K}_0} \ar[r]^-{\Phi_m} & \foam{N} \ar[d]^-{\mathrm{K}_0} \\
\U_q(\glm) \ar[r]^-{\varphi_m} & \mathrm{Rep}(U_q(\slnn{N}))
}
\]
commutes, where $\Ucat_Q(\glm)$ is the direct sum of an infinite number of copies of $\Ucat_Q(\slm)$ and admits a similar 
description in which $\slm$ weights are replaced by $\glm$ weights. 

\subsection{Thick calculus}
\label{section-thick}
The 2-functor $\Ucat_Q(\glm) \to \foam{N}$ actually extends to a certain full 2-subcategory $\Ucatc_Q(\glm) \subset \UcatD_Q(\glm)$.
In the case $m=2$, $\Ucatc_Q(\glnn{2}) = \UcatD_Q(\glnn{2})$, and this category is described\footnote{Technically, they describe $\UcatD_Q(\slnn{2})$, 
but the only difference in passing to $\UcatD_Q(\glnn{2})$ is that we use $\glnn{2}$ weights.} graphically by Khovanov-Lauda-Mackaay-Sto\v{s}i\'c in \cite{KLMS}.
Recall that the objects in the Karoubi envelope of a category $\mathbf{C}$ are given by pairs $(c,e)$ where $c \in \mathrm{Ob}(\mathbf{C})$ 
and $c \xrightarrow{e} c$ is an idempotent morphism. 
In the case of $\Ucatc_Q(\glnn{2})$, consider the idempotent morphism $\cal{E}^a \onel \xrightarrow{\mathbf{e}_a} \cal{E}^a \onel$ where $\mathbf{e}_a$ is 
given by decorating any string diagram giving a reduced expression for the longest word in the symmetric group on $a$ elements
with a specific pattern of dots, starting with $a-1$ dots on the top left-most strand, and placing one fewer dot on each strand as we 
head to the right\footnote{We use the boldface notation $\mathbf{e}_a$ for the nilHecke idempotents to distinguish them clearly from elementary symmetric polynomials $e_i$.}. 
The following depicts the case $a=4$:
\begin{equation}\label{eqn-e_aEx}
\xy 0;/r.15pc/:
 (-12,-20)*{}; (12,20) **\crv{(-12,-8) & (12,8)}?(1)*\dir{>};
 (-4,-20)*{}; (4,20) **\crv{(-4,-13) & (12,2) & (12,8)&(4,13)}?(1)*\dir{>};?(.88)*\dir{}+(0.1,0)*{\bullet};
 (4,-20)*{}; (-4,20) **\crv{(4,-13) & (12,-8) & (12,-2)&(-4,13)}?(1)*\dir{>}?(.86)*\dir{}+(0.1,0)*{\bullet};
 ?(.92)*\dir{}+(0.1,0)*{\bullet};
 (12,-20)*{}; (-12,20) **\crv{(12,-8) & (-12,8)}?(1)*\dir{>}?(.70)*\dir{}+(0.1,0)*{\bullet};
 ?(.90)*\dir{}+(0.1,0)*{\bullet};?(.80)*\dir{}+(0.1,0)*{\bullet};
 \endxy
 \qquad =: \qquad   \xy
 (0,0)*{\includegraphics[scale=0.4]{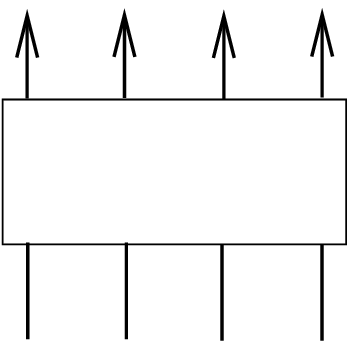}};
 (0,-0.5)*{\mathbf{e}_a};
  \endxy
\end{equation}
where we use the box notation from \cite{KLMS} for the 2-morphism (here, we do not depict the strand labels, as there is only one possible in the $\glnn{2}$ case).

Khovanov-Lauda-Mackaay-Sto\v{s}i\'c 
show that the 1-morphisms $\cal{E}^{(a)}\onel:= (\cal{E}^a\onel \left\{ \frac{a(a-1)}{2} \right\} , \mathbf{e}_a)$ and their duals $\onel \cal{F}^{(a)}$ generate 
$\UcatD_Q(\glnn{2})$, and also introduce a ``thick calculus'' to describe this 2-category. 
In the Karoubi envelope of a category $\cat{C}$, 
a morphism between two objects $(c,e)$ and $(c',e')$ is given by a 1-morphism $c \xrightarrow{f} c'$ in $\cat{C}$ so that $e'f=f=fe$. 
The map $\mathbf{e}_a$, which gives the identity 2-morphism on $\cal{E}^{(a)}\onel$ in $\UcatD_Q(\glnn{2})$, is depicted by a thick, colored upward strand
\begin{equation}\label{eqn-ea}
\xy
 (0,0)*{\includegraphics[scale=0.5]{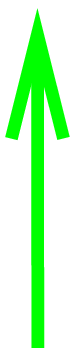}};
 (-2.5,-6)*{a}; 
  \endxy
  \quad : = \quad
 \xy
 (0,0)*{\includegraphics[scale=0.5]{figs/c1-1.eps}};
 (0,0)*{\mathbf{e}_a}; 
  \endxy
\end{equation}
and the remainder of the 2-morphisms in $\UcatD_Q(\glnn{2})$ are generated by splitter and merger maps
\begin{equation}\label{eqn-splitmerge}
\xy
 (0,0)*{\includegraphics[scale=0.5]{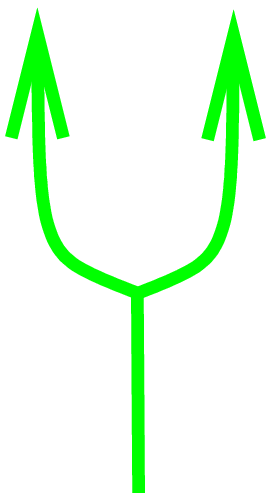}};
 (-5,-10)*{a+b};(-8,4)*{a};(8,4)*{b};
  \endxy
 \;\; :=\;\;
     \xy
 (0,0)*{\includegraphics[scale=0.5]{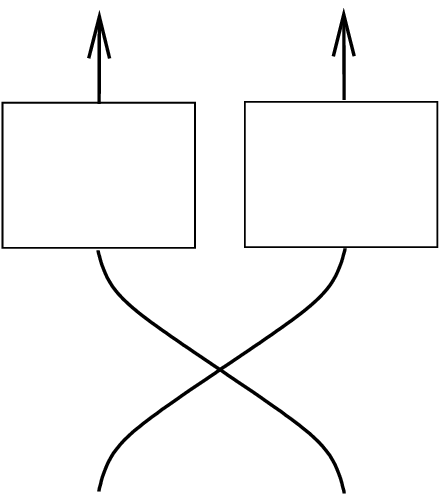}};
 (-7,-10)*{b};(7,-10)*{a};(-6,4)*{\mathbf{e}_a};(6,4)*{\mathbf{e}_b}; 
  \endxy
\qquad , \qquad 
    \xy
 (0,0)*{\includegraphics[scale=0.5,angle=180]{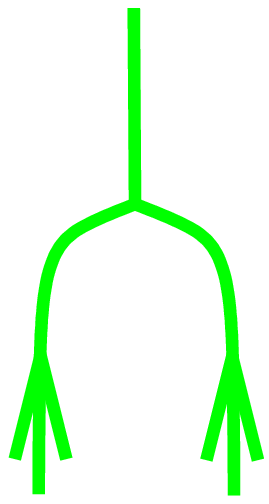}};
 (-5,10)*{a+b};(-8,-4)*{a};(8,-4)*{b}; 
  \endxy
  \;\; :=\;\;
     \xy
 (0,0)*{\includegraphics[scale=0.5,angle=180]{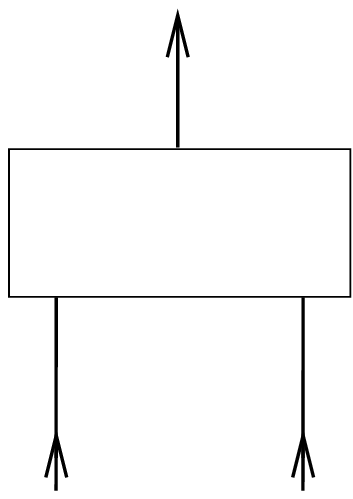}};
 (-8,-7)*{a};(8,-7)*{b};(0,1)*{\mathbf{e}_{a+b}};(-5,10)*{a+b}; 
  \endxy 
\end{equation}
where
\[
  \xy
 (0,0)*{\includegraphics[scale=0.5]{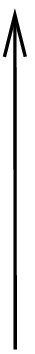}};
 (-3,-3)*{a};
  \endxy
  \quad : = \quad
  \xy
 (0,0)*{\includegraphics[scale=0.5]{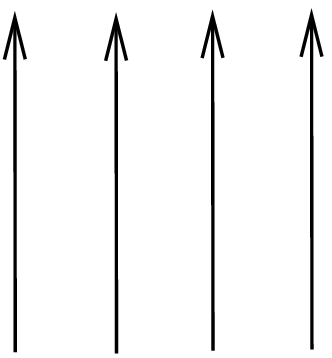}};
 (0,-11)*{\underbrace{\hspace{0.7in}}};  (0,-14)*{a};
  \endxy
  \quad \text{ and }  \quad
  \xy
 (0,0)*{\includegraphics[scale=0.5]{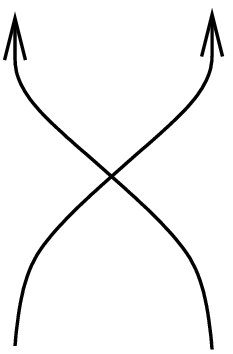}};
 (-7,-6)*{a};(7,-6)*{b};
  \endxy
  \quad := \quad
  \xy
 (0,0)*{\includegraphics[scale=0.5]{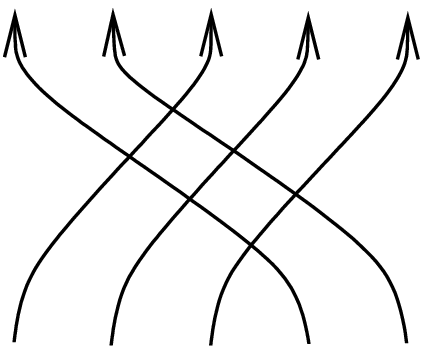}};
 (-5.5,-11)*{\underbrace{\hspace{0.45in}}};  (-5.5,-14)*{a};
 (7.5,-11)*{\underbrace{\hspace{0.25in}}};  (7.5,-14)*{b};
  \endxy 
\]
which are maps $\cal{E}^{(a+b)}\onel \to \cal{E}^{(a)} \cal{E}^{(b)}\onel \{-ab\}$ and $\cal{E}^{(a)} \cal{E}^{(b)}\onel \to \cal{E}^{(a+b)}\onel \{-ab\}$. 
Thick strands also may carry decorations by elements of the ring of symmetric functions in $a$ variables 
(depicted by placing a box containing the function on such a strand). 
Schur functors $\pi_{\alpha}$ satisfy the following relation
\begin{equation}\label{eqn-explode}
\xy
 (0,0)*{\includegraphics[scale=0.5]{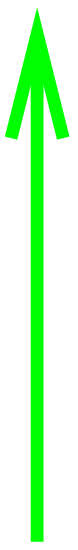}};
 (-2.5,-11)*{a};(0,-2)*{\bigb{\pi_{\alpha}}};
  \endxy
 \quad = \quad
 \xy
 (0,0)*{\includegraphics[scale=0.5]{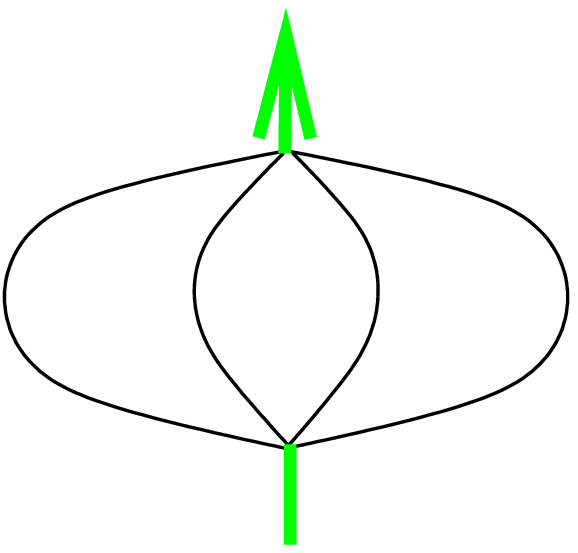}};
 (-3,-12)*{a};(-14,0)*{\bullet}+(-4.5,1)*{\scs \alpha_1+}+(0,-3)*{\scs a-1};
 (-4.5,0)*{\bullet}+(-4.5,1)*{\scs \alpha_2+}+(0,-3)*{\scs a-2};
 (4.5,0)*{\bullet}+(4,1)*{\scs \alpha_{a-1}}+(0,-3)*{\scs + 1};
 (14,0)*{\bullet}+(3,1)*{\scs \alpha_{a}};
 (0,-2)*{\cdots};
  \endxy
\end{equation}
in which the morphisms which split and merge thickness $a$ strands into thin (thickness $1$) strands are given by 
any of the possible compositions of the above mergers and splitters -- the relations for $\UcatD_Q(\glnn{2})$ given in \cite{KLMS} 
guarantee that they are the same.

There is not currently a completely diagrammatic description for $\UcatD_Q(\glm)$ for $m \geq 3$, hence we instead work 
with $\Ucatc_Q(\glm)$, the full 2-subcategory generated by 
$\cal{E}^{(a)}_i\onel:= (\cal{E}_i^a\onel \{\frac{a(a-1)}{2}\} , \mathbf{e}_a)$ 
and their adjoints, where here $\mathbf{e}_a$ is as above, but with all strands $i$-labeled. 
We refer the reader to \cite[Section 3.2]{QR} for details about the $2$-functor 
$\Phi_m \colon \Ucatc_Q(\glm) \to \foam{N}$, but note here that it acts on splitter/merger 
morphisms in $\Ucatc_Q(\glnn{2})$ via
\begin{equation}\label{PhMS}
\xy
(0,0)*{
\begin{tikzpicture} [scale=.4,fill opacity=0.2]
\draw[ultra thick, green] (0,-1.5) to (0,0);
\draw[ultra thick, green, directed=1] (0,0) to [out=30,in=270] (1,1.5);
\draw[ultra thick, green, directed=1] (0,0) to [out=150,in=270] (-1,1.5);
\node[green, opacity=1] at (0,-2) {\tiny$_{a+b}$};
\node[green, opacity=1] at (-1,2) {\tiny$a$};
\node[green, opacity=1] at (1,2) {\tiny$b$};
\end{tikzpicture}
}
\endxy
\mapsto \quad
\xy
(0,0)*{
\begin{tikzpicture} [scale=.4,fill opacity=0.2]
	\path[fill=red] (3.5,2.5) to (-2.5,2.5) to (-2.5,7.5) to (3.5,7.5);
	\path[fill=red] (2.5,3.5) to (-3.5,3.5) to (-3.5,8.5) to (2.5,8.5);
	\path[fill=blue] (.5,2.5) to (-.5,3.5) to (-.5,5.75) to (.5,4.75);
	\path[fill=blue] (-.75,7.5) to [out=270,in=130] (-.125,4.91) to (-1.125,5.91) to [out=130,in=270] (-1.75,8.5);
	\path[fill=blue] (-.125,4.91) to [out=320,in=180] (.5,4.75) to (-.5,5.75) to [out=180,in=320] (-1.125,5.91);
	\path[fill=blue] (.5,4.75) to [out=0,in=270] (1.75,7.5) to (.75,8.5) to [out=270,in=0] (-.5,5.75);
	\draw[very thick, directed=.25, directed=.75] (3.5,2.5) to (-2.5,2.5);
	\draw[very thick, directed=.25, directed=.75] (2.5,3.5) to (-3.5,3.5);
	\draw[very thick, directed=.55] (.5,2.5) to (-.5,3.5);	
	\draw[very thick, red] (.5,2.5) to [out=90,in=270] (.5,4.75);
	\draw[very thick, red] (.5,4.75) to [out=0,in=270] (1.75,7.5);
	\draw[very thick, red] (.5,4.75) to [out=180,in=270] (-.75,7.5);
	\draw[very thick, red] (.5,4.75) to (-.5,5.75);
	\draw[very thick, red] (-.5,3.5) to [out=90,in=270] (-.5,5.75);
	\draw[very thick, red] (-.5,5.75) to [out=0,in=270] (.75,8.5);
	\draw[very thick, red] (-.5,5.75) to [out=180,in=270] (-1.75,8.5);
	\draw[very thick] (3.5,2.5) to (3.5,7.5);	
	\draw[very thick] (2.5,3.5) to (2.5,8.5);
	\draw[very thick] (-2.5,2.5) to (-2.5,7.5);	
	\draw[very thick] (-3.5,3.5) to (-3.5,8.5);
	\draw[very thick, directed=.5, directed=.9, directed=.1] (3.5,7.5) to (-2.5,7.5);
	\draw[very thick, directed=.5, directed=.9, directed=.1] (2.5,8.5) to (-3.5,8.5);
	\draw[very thick, directed=.5] (1.75,7.5) to (.75,8.5);	
	\draw[very thick, directed=.5] (-.75,7.5) to (-1.75,8.5);
	\node[blue, opacity=1] at (-1.2,7) {\tiny$a$};
	\node[blue, opacity=1] at (1.25,7) {\tiny$b$};
	\node[blue, opacity=1] at (0,4.25) {\tiny$_{a+b}$};
	\end{tikzpicture}
};
\endxy
\quad \quad , \quad \quad
\xy
(0,0)*{
\begin{tikzpicture} [scale=.4,fill opacity=0.2]
\draw[ultra thick, green, directed=1] (0,0) to (0,1.5);
\draw[ultra thick, green] (1,-1.5) to [out=90,in=330] (0,0);
\draw[ultra thick, green] (-1,-1.5) to [out=90,in=210] (0,0);
\node[green, opacity=1] at (0,2) {\tiny$_{a+b}$};
\node[green, opacity=1] at (-1,-2) {\tiny$a$};
\node[green, opacity=1] at (1,-2) {\tiny$b$};
\end{tikzpicture}
}
\endxy
\mapsto \quad
\xy
(0,0)*{
\begin{tikzpicture} [scale=.4,fill opacity=0.2]
	\path[fill=red] (3.5,0) to (-2.5,0) to (-2.5,5) to (3.5,5);
	\path[fill=red] (2.5,1) to (-3.5,1) to (-3.5,6) to (2.5,6);
	\path[fill=blue] (.5,2.75) to (-.5,3.75) to (-.5,6) to (.5,5);
	\path[fill=blue] (-.75,0) to [out=90,in=180] (.5,2.75) to (-.5,3.75) to [out=180,in=90] (-1.75,1);
	\path[fill=blue] (.5,2.75) to [out=0,in=140] (1.125,2.59) to (.125,3.59) to [out=140,in=0] (-.5,3.75);
	\path[fill=blue] (1.125,2.59) to [out=320,in=270] (1.75,0) to (.75,1) to [out=270,in=320] (.125,3.59);
	\draw[very thick, directed=.5, directed=.9, directed=.1] (3.5,0) to (-2.5,0);
	\draw[very thick, directed=.5, directed=.9, directed=.1] (2.5,1) to (-3.5,1);
	\draw[very thick, directed=.5] (1.75,0) to (.75,1);	
	\draw[very thick, directed=.5] (-.75,0) to (-1.75,1);		
	\draw[very thick, red] (-.75,0) to [out=90,in=180] (.5,2.75);
	\draw[very thick, red] (1.75,0) to [out=90,in=0] (.5,2.75);
	\draw[very thick, red] (.5,2.75) to [out=90,in=270] (.5,5);
	\draw[very thick, red] (.5,2.75) to (-.5,3.75);
	\draw[very thick, red] (-1.75,1) to [out=90,in=180] (-.5,3.75);
	\draw[very thick, red] (.75,1) to [out=90,in=0] (-.5,3.75);
	\draw[very thick, red] (-.5,3.75) to [out=90,in=270] (-.5,6);
	\draw[very thick] (3.5,0) to (3.5,5);	
	\draw[very thick] (2.5,1) to (2.5,6);
	\draw[very thick] (-2.5,0) to (-2.5,5);	
	\draw[very thick] (-3.5,1) to (-3.5,6);
	\draw[very thick, directed=.25, directed=.75] (3.5,5) to (-2.5,5);
	\draw[very thick, directed=.25, directed=.75] (2.5,6) to (-3.5,6);
	\draw[very thick, directed=.55] (.5,5) to (-.5,6);
	\node[blue, opacity=1] at (-1.25,1.5) {\tiny$a$};
	\node[blue, opacity=1] at (0,4.5) {\tiny$_{a+b}$};
	\node[blue, opacity=1] at (1.25,1.5) {\tiny$b$};
	\end{tikzpicture}
};
\endxy
\end{equation}
and on (thin) cap/cup morphisms by
\begin{equation}\label{PhCC}
\xy
(0,0)*{
\begin{tikzpicture}[scale=.75]
\draw[thick, <-] (0,-1) to [out=90,in=180] (.5,0) to [out=0,in=90] (1,-1);
\node at (1.5,.25) {\tiny$(a,b)$};
\end{tikzpicture}
};
\endxy
\mapsto
 \quad \capFEfoam{a}{b} \quad , \qquad
\xy
(0,0)*{
\begin{tikzpicture}[scale=.75]
\draw[thick, ->] (0,-1) to [out=90,in=180] (.5,0) to [out=0,in=90] (1,-1);
\node at (1.5,.25) {\tiny$(a,b)$};
\end{tikzpicture}
};
\endxy
\mapsto
 (-1)^{b} \quad \capEFfoam{a}{b}
\end{equation}
\[
\xy
(0,0)*{
\begin{tikzpicture}[scale=.75]
\draw[thick, ->] (0,1) to [out=270,in=180] (.5,0) to [out=0,in=270] (1,1);
\node at (1.5,-.25) {\tiny$(a,b)$};
\end{tikzpicture}
};
\endxy
\mapsto
(-1)^{b + 1} \quad \cupFEfoam{a}{b} \quad , \qquad
\xy
(0,0)*{
\begin{tikzpicture}[scale=.75]
\draw[thick, <-] (0,1) to [out=270,in=180] (.5,0) to [out=0,in=270] (1,1);
\node at (1.5,-.25) {\tiny$(a,b)$};
\end{tikzpicture}
};
\endxy
\mapsto
\quad \cupEFfoam{a}{b}
\]
since these will be explicitly used later in our description of the link invariant.

The $2$-representation $\Ucatc_Q(\glm) \xrightarrow{\Phi_m} \foam{N}$ necessarily maps $\glm$ weights whose entries don't lie in $\{0,\ldots,N\}$ to zero, 
hence factors through the quotient $\Ucatc^{0\leq N}_Q(\glm)$ where we kill (the identity 2-morphism on the identity 1-morphism of) 
these weights. As we'll see in Section \ref{section-NHquot}, 
it is exactly the procedure of taking this quotient which gives rise to deformation parameters controlling the deformed link invariants.

\subsection{Quantum Weyl group action and Rickard complexes}

A crucial observation of Cautis-Kamnitzer-Licata \cite{CKL} is that the braiding on the category of 
quantum $\slnn{N}$ representations (which gives rise to $\slnn{N}$ link polynomials) can be recovered from 
the functor $\phi_m\colon \U_q(\glm) \to \mathrm{Rep}(U_q(\slnn{N}))$. 
Indeed, Lusztig's ``quantum Weyl group'' elements 
\[
T_i 1_{\mathbf{a}} = 
\begin{cases} 
{\displaystyle \sum_{\begin{smallmatrix} j_1,j_2\geq 0 \\ j_1-j_2=a_i-a_{i+1} \end{smallmatrix}}} (-q)^{j_2}F_i^{(j_1)}E_i^{(j_2)}1_{\mathbf{a}} & \text{ if $a_i \geq a_{i+1}$} \\
{\displaystyle \sum_{\begin{smallmatrix} j_1,j_2\geq 0 \\ j_1-j_2=a_i-a_{i+1} \end{smallmatrix}}} (-q)^{j_1}E_i^{(j_2)}F_i^{(j_1)}1_{\mathbf{a}} & \text{if $a_i \leq a_{i+1}$}
\end{cases}
\]
generate a braid group action on any finite-dimensional representation of quantum $\slm$, see \cite[Subsection 5.1.1]{Lus} and \cite{CKM}.
Under $\varphi_m$, these elements map to the braiding between fundamental $\slnn{N}$ representations; 
explicitly, the element $T 1_{(a,b)}$ gives the braiding 
$\bV_q^a \C_q^N \otimes \bV_q^b \C_q^N \to \bV_q^b \C_q^N \otimes \bV_q^a \C_q^N$.

The Rickard complexes, introduced in the $q=1$ case by Chuang and Rouquier \cite{ChR}, categorify these elements, and generate a 
categorical braid group action on any (integrable) 2-representation of $\Ucatc_Q(\glm)$. 
These complexes $\cal{T}_i\onea$ take the form
\begin{equation} \label{Rickardp}
\cal{T}_i \onea =
\xymatrix{ \uwave{\cal{F}_i^{(a_i-a_{i+1})} \onea} \ar[r]^-{d_1} & \cal{F}_i^{(a_i-a_{i+1}+1)} \cal{E}_i \onea \{1\} \ar[r]^-{d_2} & \cdots
\ar[r]^-{d_s} & \cal{F}_i^{(a_i-a_{i+1}+s)} \cal{E}_i^{(s)} \onea \{s\} \ar[r]^-{d_{s+1}} &\cdots}
\end{equation}
when $a_i \geq a_{i+1}$ and
\begin{equation} \label{Rickardn}
\cal{T}_i \onea =
\xymatrix{ \uwave{\cal{E}_i^{(a_{i+1}-a_i)} \onea} \ar[r]^-{d_1} & \cal{E}_i^{(a_{i+1}-a_i+1)} \cal{F}_i \onea \{1\} \ar[r]^-{d_2} & \cdots
\ar[r]^-{d_s} & \cal{E}_i^{(a_{i+1}-a_i+s)} \cal{F}_i^{(s)} \onea \{s\} \ar[r]^-{d_{s+1}} &\cdots}
\end{equation}
when $a_i \leq a_{i+1}$. 
Here and throughout, we've \uwave{underlined} the term in homological degree zero.
The differential $d_k$ that appears in the second complex is conveniently expressed in thick calculus as
\[
d_k=
\xy
(0,0)*{
\begin{tikzpicture}[scale=1]
\draw [thick, directed= .55] (1,1.2) .. controls (1,.5) and (0,.5) .. (0,1.2);
\draw [ultra thick, green, ->] (0,0) -- (0,2);
\draw [ultra thick, green, ->] (1,2) -- (1,0);
\node [green] at (0,2.5) {\footnotesize $-\lambda+k$}; 
\node at (.5,1) {\footnotesize $1$}; 
\node [green] at (1,2.5) {\footnotesize $k$}; 
\node at (1.7,1) {$\mathbf{a}$};
\end{tikzpicture}
};
\endxy
\]
where all strands are colored by the index $i \in I$ and $\lambda =a_i-a_{i+1}$. The differential in the first complex is defined similarly, 
and in both cases the equality $d^2=0$ follows directly from thick calculus relations.

Recall that the images of the Rickard complexes under any integrable $2$-representation are invertible, up to homotopy, with inverses given by the images of the complexes
\begin{equation}
\label{RickardInvp}
\onea \cal{T}_i^{-1} =
\xymatrix{ \cdots \ar[r]^-{d_{s+1}^*} & \onea \cal{F}_i^{(s)} \cal{E}_i^{(a_i-a_{i+1}+s)} \{-s\} \ar[r]^-{d_s^*}
& \cdots \ar[r]^-{d_2^*}
& \onea \cal{F}_i \cal{E}_i^{(a_i-a_{i+1}+1)} \{-1\} \ar[r]^-{d_1^*} &\uwave{ \onea \cal{E}_i^{(a_i-a_{i+1})}} }
\end{equation}
when $a_i \geq a_{i+1}$ and
\begin{equation}
\label{RickardInvn}
\onea \cal{T}_i^{-1} =
\xymatrix{ \cdots \ar[r]^-{d_{s+1}^*} & \onea \cal{E}_i^{(s)} \cal{F}_i^{(a_{i+1}-a_i+s)} \{-s\} \ar[r]^-{d_s^*}
& \cdots \ar[r]^-{d_2^*}
& \onea \cal{E}_i \cal{F}_i^{(a_{i+1}-a_i+1)} \{-1\}  \ar[r]^-{d_1^*} & \uwave{\onea \cal{F}_i^{(a_{i+1}-a_i)}} }
\end{equation}
when $a_i \leq a_{i+1}$. In both cases the differential is given by a composition of splitters with a 
thickness $1$ cap 2-morphisms, e.g. for \eqref{RickardInvn}:
\[
d^*_k=
\xy
(0,0)*{
\begin{tikzpicture}[scale=1]
\draw [thick, rdirected= .55] (1,0.8) .. controls (1,1.5) and (0,1.5) .. (0,0.8);
\draw [ultra thick, green, ->] (0,0) -- (0,2);
\draw [ultra thick, green, ->] (1,2) -- (1,0);
\node [green] at (1,-.5) {\footnotesize $\lambda+k$}; 
\node at (.5,1) {\footnotesize $1$}; 
\node [green] at (0,-.5) {\footnotesize $k$}; 
\node at (1.7,1) {$\mathbf{a}$};
\end{tikzpicture}
};
\endxy
\]
In Section \ref{section-decomp}, we'll use these complexes to define our tangle invariant.

\subsection{Karoubi envelope technology}
\label{subsubs_BNM}
In his famous paper \cite{BN} Bar-Natan shows that Khovanov homology can be constructed locally, 
by working in the homotopy category of chain complexes over a certain ($1+1$)-dimensional cobordism category. 
Objects of this category are formal direct sums of 1-manifolds embedded in the plane (possibly with boundary) and equipped with a formal $\Z$-grading. 
Morphisms are matrices of linear combinations of cobordisms between 1-manifolds, decorated with dots, modulo the following local relations: 
\begin{equation*}
\thinsphere[.5]=0 \quad , \quad \thindottedsphere[.5]{}=1 \quad , \quad  
\thincylinder[.5]{} \;\; = \;\; \grbambooRHSone[.5]+\grbambooRHStwo[.5] \quad , \quad
\twodotsheet[.5]{2} \;\; = \;\; 0.
\end{equation*}
In \cite{BNM} Bar-Natan and Morrison explain that Lee's deformed $\slnn{2}$ link homology \cite{Lee} arises from the same kind of construction, 
after modifying the final ``sheet'' relation above to
\begin{equation*} 
\twodotsheet[.5]{2} = \grsheet[.5]
\end{equation*}
so that the operator given by adding a dot to a cobordism is no longer nilpotent. 

To analyze the effects of this deformation, consider the algebra of endomorphisms of a strand, 
denoting the identity by $\idem$, a sheet decorated by a dot by $X$, 
and extend linearly so that polynomials in $X$ denote linear combinations of decorated sheets. 
The undeformed sheet relation can then be expressed as $X^2=0$ and the deformed relation is $X^2-\idem=0$. 
From this it is clear that in the deformed case the operator of placing a dot on a sheet has eigenvalues $1$ and $-1$ with corresponding eigenvectors $\idem_{+1}:= \frac{\idem+X}{2}$ and $\idem_{-1}:= \frac{\idem-X}{2}$. 

The decomposition into eigenspaces for the action of adding a dot splits the deformed cobordism category: 
every connected component of a cobordism can be written as a sum of the two decorations 
$\idem= \idem_{+1} + \idem_{-1}$, which are orthogonal (i.e. $\idem_{+1} \idem_{-1}=0$), 
idempotent (i.e. $\idem_{\pm 1} \idem_{\pm 1} = \idem_{\pm 1}$), and obviously commute.

Next, Bar-Natan and Morrison enlarge the cobordism category by proceeding to its Karoubi envelope\footnote{See the explanation in Section \ref{section-thick}.}. 
Practically, this means allowing objects, i.e. planar 1-manifolds, to be ``colored'' by $\idem_{+1}$ and $\idem_{-1}$ as well. 
Any uncolored 1-manifold is isomorphic to the direct sum of the $\idem_{+1}$ and $\idem_{-1}$ versions, 
and colored cobordisms between colored 1-manifolds are only non-zero if the corresponding idempotent decorations agree. 

Using this splitting of the deformed cobordism category, Bar-Natan and Morrison compute a decomposition for the chain complexes arising in the definition of the deformed link invariant, 
whose objects are planar 1-manifolds that arise as resolutions of the link diagram. 
The first result is that the only non-zero contributions to the deformed link invariant come from colorings of the planar 1-manifolds by $\idem_{+1}$ and $\idem_{-1}$ 
which are consistent along link components. 
The second result is that each such coloring contributes only one generator to the link homology. 
This reproduces Lee's result that the deformed $\slnn{2}$ homology of a $l$-component link is $2^l$-dimensional. 
Alternatively, we could say that it is a direct sum of tensor products of $\slnn{1}$ homologies, 
where $\slnn{1}$ homology assigns the 1-dimensional vector space $\C$ to any link. 
More precisely, we have one summand for each coloring of components of the link by $\idem_{+1}$ or $\idem_{-1}$, 
and the tensorands are the $\slnn{1}$ homologies of the $\idem_{+1}$- and $\idem_{-1}$-colored sublinks, respectively.

Gornik's generalization \cite{Gor} of the generic deformation result to $\slnn{N}$ can be understood along very similar lines. 
Again, there is a Frobenius algebra $\C[X]/\langle X^N \rangle$ of local decorations which is being deformed to 
$\C[X]/ \langle X^N-\beta^N \rangle \cong \C\oplus \cdots \oplus \C$, with one summand for each of the $N$ roots of the polynomial $X^N- \beta^N$.
The idempotents that project onto the $N$ summands then split the category underlying the chain complexes in the construction of the link homology. 
The resulting invariant for a knot is a direct sum of $N$ copies of its $\slnn{1}$ homology. 
Similarly, for $l$-component links one gets a $N^l$-dimensional vector space which can be understood as a direct sum 
over possible root-colorings of components of 1-dimensional tensor products of $\slnn{1}$ homologies of sublinks, one for each different root. 

In order to prove our decomposition result Theorem \ref{mainthm}, we start by computing the algebra of decorations on foams facets in the deformed foam 2-category $\foam{N}{}^\Sigma$. 
In fact, the algebra of decorations on a $k$-labeled facet is isomorphic to the deformed link homology of the $\bV^k \C^N$-colored unknot. 
We compute it, and hence prove Theorem \ref{mainthm} in the special case of the unknot, in Section \ref{section-defunknot}. 
In particular, the algebra of decorations on a $k$-labeled foam facet decomposes into a direct sum of local pieces indexed by $k$-element multisubsets of the set of roots $\Sigma$. 
This gives idempotent foam decorations along which the link invariant splits into a direct sum, which is proved in Section \ref{section-sum}. 
Similarly as for the generic deformation in the 1-colored case, 
the only non-zero contributions to the deformed link invariant come from idempotent colorings that are consistent along link components. 
This is shown in Lemma \ref{sumdecomp}. 
However, in the case of general deformations of colored invariants there are two new features that have not been rigorously addressed in the literature. 
One appears because we allow higher colors, the other because we allow non-generic deformations.

\textbf{Higher color:} Foam facets are not colored by roots, i.e. elements of $\Sigma$, anymore, but by multisubsets of $\Sigma$ of size corresponding to the label of the facet. 
Such a multisubset can contain several different roots and in this case we need a new way to split this facet into parts colored by single roots. 
This is where we use the full power of the foam technology, see Section \ref{section-foamcat} for preparatory work and Section \ref{section-tensor} for the actual tensor product decomposition of the link invariant.

\textbf{Non-generic deformation:} A root $\lambda$ can occur in $\Sigma$ with a multiplicity $N_\lambda \geq 1$. 
The $\lambda$-colored part of a direct summand of the deformed link invariants is essentially the $\slnn{N_\lambda}$ homology of the $\lambda$-colored sub-link, i.e. in particular it is usually not trivially 1-dimensional. 
To see this we need to check that after all splitting procedures, the $\lambda$-colored foams behave like $\slnn{N_\lambda}$-foams. This is done in Section \ref{section-unicoloredfoam}.

\section{Deforming nilHecke algebra quotients}
\label{section-defunknot}

The nilHecke algebra $\BNC_a$ plays a fundamental role in higher representation theory. 
Indeed, this algebra is given by $\mathrm{END}(\cal{E}_i^a \onel)$, the endomorphism algebra of (not necessarily degree zero) 
morphisms between a composition of $a$ $\cal{E}_i$'s in $\Ucat_Q(\slm)$. 
In this section, we'll review the nilHecke algebra, and then proceed to study certain deformations of its cyclotomic quotients, which control the deformed Khovanov-Rozansky homologies of colored unknots.
\newpage 
\subsection{The nilHecke algebra}
\begin{defi}
The nilHecke algebra on $a$ strands, $\BNC_a$ admits an algebraic presentation 
as the graded $\C$-algebra of endomorphisms of the abelian group $\C[X_1,\dots, X_a]$ generated by operators:
\begin{itemize}
\item $\xi_i$ of degree $2$ for $1\leq i \leq a$ acting by multiplication by $X_i$,
\item $\partial_i$ of degree $-2$ for $1 \leq i \leq a-1 $ acting as divided difference, \\i.e. for $p(X_1,\dots, X_a)\in \C[X_1,\dots, X_a]$:
\begin{equation*}\partial_i(p(X_1,\cdots X_a))= \frac{p(\cdots, X_i, X_{i+1},\cdots)-p(\cdots, X_{i+1},X_i,\cdots)}{X_i-X_{i+1}}.
\end{equation*}
\end{itemize} 
which satisfy the following complete set of relations:
\begin{itemize}
\item $\xi_i \xi_j = \xi_j \xi_i$,
\item $\xi_i \partial_j = \partial_j \xi_i\quad$ if $i\notin \{j,j+1\}$,
\item $\partial_i \partial_i =0$,
\item $\partial_i \partial_{i+1} \partial_{i} =\partial_{i+1} \partial_{i}\partial_{i+1} $,
\item $\xi_i\partial_i - \partial_i \xi_{i+1}= 1 = \partial_i \xi_i - \xi_{i+1}\partial_i$.
\end{itemize}
\end{defi}

The following result can be found in \cite[Proposition 3.5]{Lau1}:
\begin{prop}
\label{nilHeckeprop}
\begin{enumerate}
\item The center of $\BNC_a$ is $Z(\BNC_a)\cong \C[\xi_1,\dots,\xi_a]^{S_a}=: \Sym(\xi_1,\dots, \xi_a)$.
\item $\BNC_a$ is graded isomorphic to the algebra of $a!\times a!$ matrices over its center:
\begin{equation*}\BNC_a \cong Mat(a!, Z(\BNC_a)).
\end{equation*}
\end{enumerate}
\end{prop}

The isomorphism $\BNC_a \xrightarrow{\cong} \mathrm{END}(\cal{E}_i^a \onel)$ is given by identifying the generator
$\xi_i$ with the string diagram consisting of $a$ upward strands with a dot on the $i^{th}$ strand 
and the generator $\partial_i$ with a crossing between the $i^{th}$ and $(i+1)^{st}$ strands:
\[
 1 \quad \mapsto \quad
\xy
(0,0)*{
\begin{tikzpicture}[scale=.5]
	\draw[thick,->] (0,0) to (0,2);
	\node at (1,1) {$\cdots$};
	\draw[thick, ->] (2,0) to (2,2);
\end{tikzpicture}
};
\endxy
\qquad , \qquad
\xi_i \quad \mapsto \quad 
\xy
(0,0)*{
\begin{tikzpicture}[scale=.5]
	\draw[thick,->] (0,0) to (0,2);
	\node at (1,1) {$\cdots$};
	\draw[thick, ->] (2,0) to (2,2);
	\node at (2,1) {$\bullet$};
	\node at (3,1) {$\cdots$};
	\draw[thick, ->] (4,0) to (4,2);
\end{tikzpicture}
};
\endxy
\qquad , \qquad 
\partial_i \quad \mapsto \quad  
\xy
(0,0)*{
\begin{tikzpicture}[scale=.5]
	\draw[thick,->] (0,0) to (0,2);
	\node at (1,1) {$\cdots$};
	\draw[thick, ->] (2,0) to [out=90,in=270] (3,2);
	\draw[thick, ->] (3,0) to [out=90,in=270] (2,2);
	\node at (4,1) {$\cdots$};
	\draw[thick, ->] (5,0) to (5,2);
\end{tikzpicture}
};
\endxy.
\]

Multiplication is given by composition of 2-morphisms in $\Ucat_Q(\slm)$, i.e. 
by stacking diagrams vertically. 
An arbitrary element of $\BNC_a$ can be written as a $\C$-linear combination of such stacked string diagrams.

We'll also utilize the ``thick calculus'' for the nilHecke algebra, detailed in \cite{KLMS}, which corresponds to the algebra of 
upward strands in $\Ucatc_Q(\slm)$ having varying thickness.
Set
\[
D_a:= (\partial_1 \partial_2 \cdots \partial_{a-1})(\partial_1\cdots \partial_{a-2})\cdots (\partial_1)
\]
and let $\Delta_X= \prod_{1\leq i<j\leq a}(X_i-X_j)$ be the Vandermonde determinant.
The action of $D_a$ on polynomials $p\in \C[X_1,\dots, X_a]$ is given by 
\begin{equation*}
D_a(p(X_1,\dots,X_a))=\frac{1}{\Delta_X} \sum_{w\in S_a} \epsilon(w) p(X_{w(1)},\dots,X_{w(a)} )
\end{equation*}
where $\epsilon(w) \in \{ \pm1\}$ is the sign of the permutation $w$.
In other words, $D_a$ anti-symmetrizes a polynomial and then divides by the Vandermonde determinant, resulting in a symmetric polynomial. 
Divided differences not only act on elements of $\C[X_1,\dots, X_a]$, but also on the subring $\C[\xi_1,\dots, \xi_a]$ of $\BNC_a$.  
In particular, if $f\in \C[\xi_1,\dots, \xi_a]$, we denote by $D_a(f)$ the action of the product of divided differences on $f$. The following compatibility relation holds:
\begin{equation*}D_a f(\xi_1,\dots, \xi_a) D_a = D_a(f)(\xi_1,\dots, \xi_a) D_a.
\end{equation*}
However, we point out that this is only true in the presence of the $D_a$ on the right.

Define $\delta_a := \xi_1^{a-1} \xi_2^{a-2} \cdots \xi_{a-1}$. It is easy to compute that 
\begin{equation*}
\Delta_\xi = \prod_{1\leq i<j\leq a}(\xi_i-\xi_j) = \sum_{w\in S_a} \epsilon(w) \xi_{w(1)}^{a-1} \xi_{w(2)}^{a-2} \cdots \xi_{w(a-1)}
\end{equation*}
 and hence $D_a(\delta_a) = \Delta_\xi / \Delta_\xi=1$ and $\mathbf{e}_a = \delta_a D_a$ is idempotent in $\BNC_a$:
\begin{equation*}
\mathbf{e}_a^2 =\delta_a D_a \delta_a D_a= \delta_a D_a(\delta_a) D_a = \delta_a D_a = \mathbf{e}_a.
\end{equation*}
In fact, this is exactly the idempotent $\mathbf{e}_a$ defined in the introduction, and depicted graphically (in the case $a=4$) in 
equation \eqref{eqn-e_aEx}.

One can use this idempotent to explicitly describe the isomorphism between 
the center \\$Z(\BNC_a) \cong \C[\xi_1,\dots, \xi_a]^{S_a}$ and the direct summand $\mathbf{e}_a \BNC_a \mathbf{e}_a \subset \BNC_a$ via:
\begin{align*}Z(\BNC_a) &\cong Z(\BNC_a) \mathbf{e}_a = \mathbf{e}_a \BNC_a \mathbf{e}_a\\
y &\mapsto y \mathbf{e}_a.
\end{align*} 
If $\alpha = (\alpha_1,\dots, \alpha_a)$ is a partition of length $\leq a$
and $\pi_\alpha(\xi_1,\dots, \xi_a)$ is the Schur polynomial associated to $\alpha$, 
then $D_a(\xi_1^{a-1+\alpha_1} \xi_2^{a-2+\alpha_2} \cdots  \xi_a^{\alpha_a})=\pi_\alpha(\xi_1,\dots, \xi_a)$. 
Hence, under the above isomorphism we have:
\begin{equation}
\label{eqn-explosion} 
\begin{array}{r l}
\pi_\alpha(\xi_1,\dots, \xi_a) \mapsto &\pi_\alpha(\xi_1,\dots, \xi_a) \mathbf{e}_a =\delta_a D_a(\xi_1^{a-1+\alpha_1} \xi_2^{a-2+\alpha_2} \cdots  \xi_a^{\alpha_a}) D_a   \\~
\\
 & \quad = \delta_a D_a \xi_1^{\alpha_1} \xi_2^{\alpha_2} \cdots  \xi_a^{\alpha_a} \delta_a D_a= \mathbf{e}_a \xi_1^{\alpha_1} \xi_2^{\alpha_2} \cdots  \xi_a^{\alpha_a} \mathbf{e}_a.
 \end{array} 
\end{equation}
Compare this with the thick calculus relation \eqref{eqn-explode}.

\subsection{Quotients of the nilHecke algebra}
\label{section-NHquot}

A certain quotient of the nilHecke algebra will be relevant to our study of deformed link homology. 
Recall that the 2-functor $\Ucatc_Q(\glm) \to \foam{N}$ factors through the quotient $\Ucatc_Q^{0\leq N}(\glm)$, 
where we kill $\glm$ weights whose entries lie outside the set $\{0,\ldots,N\}$. 
Consider $\mathbf{h} = (N,\ldots,N,0,\ldots,0)$, which is a highest weight in this quotient, and note that
\[
0 = 
\xy
(0,0)*{
\begin{tikzpicture}[scale=.75]
	\draw[thick] (0,0) to (0,1);
	\draw[thick] (0,1) to [out=90,in=270] (-.5,2);
	\draw[thick] (-.5,1) to  [out=90,in=270] (0,2);
	\draw[thick] (-.5,2) to [out=90,in=0] (-.75,2.25) to [out=180,in=90] (-1,2) to (-1,1) to [out=270,in=180] 
		(-.75,.75) to [out=0,in=270] (-.5,1);
	\draw[thick,->] (0,2) to (0,3);
	\node at (-1.25,2.75) {$\mathbf{h}$};
\end{tikzpicture}
};
\endxy
=
\sum_{i=0}^{N} 
\xy
(0,0)*{
\begin{tikzpicture}[scale=.75]
	\draw[thick,->] (0,0) to (0,3);
	\draw[thick,directed=.55] (-1,1.5) circle (.5);
	\node at (-1,2.75) {$\mathbf{h}$};
	\node at (0,2.25) {$\bullet$};
	\node at (.5,2.5) {\tiny$N-i$};
	\node at (-1,1) {$\bullet$};
	\node at (-1,.625) {\tiny$\begin{smallmatrix} -N-1 \\ +i \end{smallmatrix}$};
\end{tikzpicture}
};
\endxy
\]
where the string diagrams are colored by the number of $N$'s in $\mathbf{h}$. 
The first equality holds since the region inside the ``left-curl'' is zero in $\Ucatc_Q^{0\leq N}(\glm)$. 
Note that the individual summands on the right-hand side are not (necessarily) zero, 
since these bubbles are fake, in the sense of \cite{Lau1}. 
The infinite Grassmannian relation \cite{Lau1} implies that these bubbles generate the endomorphism 
algebra of the highest weight object $\mathbf{h}$, 
hence we can view the positive degree fake bubbles as (graded) parameters.

Under the 2-functor $\Ucatc_Q^{0\leq N}(\glm) \to \foam{N}$, 
this highest weight endomorphism algebra maps to the endomorphism algebra of an $N$-labeled facet, 
with the fake bubbles of degree $i$ mapping to the decoration by a signed elementary symmetric polynomial $(-1)^i e_i$ by \cite[Equation 3.33]{QR}.
This endomorphism algebra in turn determines the ground ring over which the link homology theory is defined, 
i.e. the invariant of a link will be a module over this algebra.
In \cite{QR}, it is shown that by setting the (images of) the bubble deformation parameters 
to zero yields a link homology theory isomorphic to Khovanov-Rozansky homology. 
Setting these parameters to other values should thus correspond to a deformed version of Khovanov-Rozansky homology. 

This hence gives the relation
\begin{equation}\label{eq-doteq}
\sum_{i=0}^{N} 
c_i
\xy
(0,0)*{
\begin{tikzpicture}[scale=.5]
	\draw[thick,->] (0,0) to (0,3);
	\node at (-.5,2.875) {$\mathbf{h}$};
	\node at (0,2.25) {$\bullet$};
	\node at (.75,2.625) {\tiny$N-i$};
\end{tikzpicture}
};
\endxy
=0
\end{equation}
where the $c_i \in \C$ are the specializations of the fake bubbles. 
This corresponds to a relation on 1-labeled foams facets which meet $N$-labeled foam facets 
(see \cite[Section 4.1]{QR} for a general discussion about $N$-labeled facets).
Since we'd like foam relations to be local, this motivates studying this relation in all weights, 
not just for $\mathbf{a} = \mathbf{h}$.

To this end, let $\Sigma$ be a multiset of $N$ complex numbers and 
\begin{equation}\label{eq-polynomial}
P(X) = \prod_{s\in \Sigma}^N (X-s) = X^N + \sum_{i=0}^{N-1} c_{N-i} X^i
\end{equation}
be the monic degree $N$ polynomial with root multiset $\Sigma$ and coefficients $c_{i}=(-1)^{i}e_{i}(\Sigma)$. 
\begin{defi}
The $\Sigma$-deformed quotient of the nilHecke algebra $\BNC_a^\Sigma$ is the quotient algebra of $\BNC_a$ modulo the ideal generated by $P(\xi_1)$. 
\end{defi}

In the case where $\Sigma =\{0^N\}$, (i.e. $P(X)=X^N$), this algebra is known as the level $N$ cyclotomic quotient of $\BNC_a$, 
which we denote by $\BNC_a^N$. We aim to now generalize the following result of Lauda \cite[Proposition 5.3]{Lau3}:

\begin{prop} There are isomorphisms of graded algebras:
\begin{enumerate}
\item $Z(\BNC_a^N) \cong \mathrm{H}^*(Gr(a,N))$,
\item $\BNC_a^N \cong Mat(a!, Z(\BNC_a^N))$.
\end{enumerate}
Here $\mathrm{H}^*(Gr(a,N))$ denotes the cohomology ring with coefficients in $\C$ of the Grassmannian of complex $a$-planes in $\C^N$.
\end{prop}

To generalize this to arbitrary $\Sigma$, we'll adapt Lauda's method of proof to our setting.

\begin{defi} Let $\X=\{\xi_1,\dots, \xi_a\}$ and $\Y= \{y_1,\dots y_b\}$ be two alphabets of variables. \\ We denote the ring of symmetric polynomials in $\X$ by $\Sym(\X)$ 
and the ring of polynomials separately symmetric in $\X$ and $\Y$ by $\Sym(\X | \Y)$. 
The \emph{complete symmetric polynomials} $h_i(\X)$ in $\X$ can be defined via their generating function:

\begin{equation*} 
\sum_{i=0}^\infty h_i(\X)t^i = \prod_{\xi\in \X}(1-t \xi)^{-1}.
\end{equation*}
The \emph{elementary symmetric polynomials} $e_i(\X)$ in $\X$ are defined by:
\begin{equation*} 
\sum_{i=0}^\infty e_i(\X)t^i = \prod_{\xi\in \X}(1+t \xi),
\end{equation*}
and finally we define the \emph{complete symmetric functions in} $\X-\Y$, denoted $h_i(\X-\Y)$, by: 
\begin{equation*} 
\sum_{i=0}^\infty h_i(\X-\Y)t^i = \frac{\prod_{y\in \Y}(1-t y)}{\prod_{\xi\in \X}(1-t \xi)}. 
\end{equation*}
Note that this gives the explicit formula:
\[
h_k(\X - \Y) = \sum_{i=0}^k (-1)^i e_i(\Y)h_{k-i}(\X) .
\]
\end{defi}

\begin{defi}
Let $\X=\{\xi_1,\dots, \xi_a\}$ be an alphabet of $a$ variables (of degree 2) and $\Bb= \{b_1,\dots, b_N\}$ an alphabet of $N$ variables (of degree 2). The following is an explicit description of the $GL(N)$-equivariant cohomology (with $\C$ coefficients) of the Grassmannian $Gr(a,N)$ of complex $a$-planes in $\C^N$:
\begin{equation*} \mathrm{H}^*_{GL(N)}(Gr(a,N))\cong \frac{\Sym(\X | \Bb)}{\langle h_{N-a+1}(\X-\Bb),\dots, h_{N}(\X-\Bb)\rangle} .
\end{equation*}
This is a rank ${N \choose a}$ graded free $\Sym(\Bb)$ module 
and $\Sym(\Bb)\cong \mathrm{H}^*_{GL(N)}(*)$, see \cite[Section 2.3]{Wu2} and references therein. If we quotient $\mathrm{H}^*_{GL(N)}(Gr(a,N))$ by the relations $b_i=0$ we recover the well-known description of the ordinary cohomology ring of the Grassmannian:
\begin{equation*}\mathrm{H}^*(Gr(a,N))\cong \frac{\Sym(\X)}{\langle h_{N-a+1}(\X),\dots, h_{N}(\X)\rangle}.
\end{equation*}

We can also quotient $\mathrm{H}^*_{GL(N)}(Gr(a,N))$ by sending $\Bb$ to $\Sigma$, an arbitrary multisubset of $N$ complex numbers. The result is the $\C$-algebra
\begin{equation*}H_a^\Sigma:= \frac{\Sym(\X)}{\langle h_{N-a+1}(\X-\Sigma),\dots, h_{N}(\X-\Sigma)\rangle}
\end{equation*} 
which we call the \emph{$\Sigma$-deformed cohomology ring of $Gr(a,N)$}. It is a flat deformation of $\mathrm{H}^*(Gr(a,N))$, in particular it has complex dimension ${N \choose a}$. We use the following notation for its defining ideal:
\begin{equation*}
I_a^\Sigma := \langle h_{N-a+1}(\X-\Sigma),\dots, h_{N}(\X-\Sigma)\rangle \subset \Sym(\X).
\end{equation*}
\end{defi}

\begin{prop}
\label{BNCq}
There are isomorphisms of algebras:
\begin{enumerate}
\item $Z(\BNC_a^\Sigma) \cong H_a^\Sigma$,
\item $\BNC_a^\Sigma \cong Mat(a!, H_a^\Sigma)$.
\end{enumerate}
\end{prop}
\begin{proof} \label{BNCqproof}
To explain the context we first go through a proof of Proposition \ref{nilHeckeprop} along the exposition in Section 5 of \cite{Lau3}.

Let $\X:=\{\xi_1,\dots,\xi_a\}$ be an alphabet of $a$ variables and denote by $\mathcal{H}_a$ the abelian subgroup of $\C[\X]:=\C[\xi_1,\dots,\xi_a]$ generated by all monomials $\xi_1^{\alpha_1}\cdots\xi_a^{\alpha_a}$ with $0\leq \alpha_i\leq a-i$. 
$\mathcal{H}_a$ has rank $a!$ and $\C[\X] \cong \mathcal{H}_a\otimes \Sym(\X)$ as graded $\Sym(\X)$-modules. 
In particular, the generators of $\mathcal{H}_a$ give a basis for $\C[\X]$ as a free graded $\Sym(\X)$-module and $\End_{\Sym(\X)}(\C[\X])\cong Mat(a!,\Sym(\X))$. 
It is easy to check that the nilHecke generators $\xi_i$ and $\partial_i$ act as $\Sym(\X)$-module endomorphisms of $\C[\X]$ and hence there is a homomorphism:
\begin{equation*}\theta\colon \BNC_a \to Mat(a!,\Sym(\X)).
\end{equation*}
Lauda has shown that this is an isomorphism of graded algebras \cite{Lau1}, which proves Proposition \ref{nilHeckeprop}.

Let $\alpha=(\alpha_2,\dots,\alpha_a)$ be a sequence with $0\leq \alpha_i\leq a-i$ and write $\overline{\xi}^\alpha:= \xi_2^{\alpha_2}\cdots\xi_a^{\alpha_a}$, 
then we can partition the above basis for $\mathcal{H}_a$ into $(a-1)!$ ordered subsets $B_\alpha:=\{\xi_1^{a-1}\overline{\xi^\alpha}, \dots,\xi_1\overline{\xi^\alpha}, \overline{\xi^\alpha}\}$ indexed by sequences $\alpha$ as above. 
The orders on $B_\alpha$ extend to a total order on the basis of $\mathcal{H}_\alpha$ and with respect to this ordered basis the action of $\xi_1$ under the isomorphism 
$\theta$ is given by a block diagonal matrix of $(a-1)!$ identical blocks (the restriction to the span of $B_\alpha$'s) of the form:
\begin{equation*}\theta(\xi_1)=  \left(\begin{array}{ccccc}
e_1 &1 & 0 &\dots & 0\\
-e_2 & 0 & 1 & \ddots & \vdots \\
\vdots & \vdots & \ddots & \ddots & 0\\
\vdots & \vdots & \ddots & 0 & 1\\
(-1)^{a-1}e_a & 0 & \dots & \dots & 0 \end{array}\right) 
\end{equation*}
where we write $e_i:=e_i(\X)$ for the $i^{th}$ elementary symmetric polynomials in $\X$.

The image of the ideal $\langle P(\xi_1) \rangle$ under the isomorphism $\theta$ is determined by the matrix equation $0 = \theta(P(\xi_1))= P(\theta(\xi_1))$. To explicitly compute the right-hand side of this matrix equation we first describe powers of $\theta(\xi_1)$.

\begin{lem}
If we write \begin{equation*}\theta(\xi_1)^k=  \left(\begin{array}{ccc}
b^{k}_{1,1} & \dots & b^{k}_{1,a}\\
\vdots & & \vdots \\

b^{k}_{a,1}&  \dots & b^{k}_{a,a} \end{array}\right),
\end{equation*}
then the $b^{k}_{i,j}$ are completely determined by the relations:
\begin{equation*}b^{k}_{i,j}= \begin{cases} h_{k+i-j} -\sum_{l=1}^{i-1} h_{i-l} b^{k}_{l,j}\quad &\text{ for } j\leq k \\
 \delta_{i+k,j}\quad &\text{ for } j> k \end{cases}.
\end{equation*}
Here we use the shorthand $h_{i}:=h_{i}(\X)$ for the $i^{th}$ complete symmetric polynomial in $\X$. In particular, the first row of $\theta(\xi_1)^k$ has entries $h_k,\dots, h_{k+1-a}$.
\end{lem}
\begin{proof}
Proof by induction on $k$. For $k=1$ we only have to check the first relation for $j=1$; we compute the right-hand side:
\begin{equation*}h_{1+i-1} -\sum_{l=1}^{i-1} h_{i-l} (-1)^{l-1}e_l= \underbrace{\sum_{l=0}^{i} h_{i-l} (-1)^{l}e_l}_{=0} - (-1)^i e_i=(-1)^{i-1}e_i.
\end{equation*}

For the induction step we assume the relations hold for $b^{k}_{i,j}$ and will deduce the relations for $b^{k+1}_{i,j}$. 
First note that since $\theta(\xi_1)$ has the identity matrix as $(a-1)\times (a-1)$-minor we get $b^{k+1}_{i,j+1}=b^{k}_{i,j}$ for all $0\leq i\leq a$ and $0\leq j< a$. 
It is then immediate that $b^{k+1}_{i,j}$ for $j\geq 2$ satisfy the required relations and we only have to check the relations between the entries $b^{k+1}_{i,1}$ in the first column by induction in $i$. 
The case $i=1$ is given as follows:
\begin{equation*}b^{k+1}_{1,1}= \sum_{l=1}^a (-1)^{l-1} e_l b^{k}_{1,l} = \sum_{l=1}^a (-1)^{l-1} e_l h_{k+1-l} = 0 - (-1)^{-1} e_0 h_{k+1}= h_{k+1}
\end{equation*}
and assuming it holds for all smaller indices, the case for $i+1$ is given by:
\begin{align*}
b^{k+1}_{i+1,1} &= \sum_{l=1}^a (-1)^{l-1} e_l b^{k}_{i+1,l}=\sum_{l=1}^a (-1)^{l-1} e_l \left(h_{k+i+1-l} - \sum_{r=1}^i h_{i+1-r} b^k_{r,l} \right) \\
& =\sum_{l=1}^a (-1)^{l-1} e_l (h_{k+i+1-l}) -  \sum_{r=1}^i h_{i+1-r} \left(\sum_{l=1}^a (-1)^{l-1} e_l b^k_{r,l} \right) \\
& =h_{(k+1)+(i+1)-1} -  \sum_{r=1}^i h_{i+1-r} b^{k+1}_{r,1}.
\end{align*}
\end{proof}

\begin{lem}
\label{ThetaPLem}
Denoting \begin{equation*}\theta(P(\xi_1))=  \left(\begin{array}{ccc}
c_{1,1} & \dots & c_{1,a}\\
\vdots & & \vdots \\
c_{a,1}&  \dots & c_{a,a} \end{array}\right),
\end{equation*}
we have $c_{1,i}=h_{N+1-i}(\X-\Sigma)$, and all other $c_{i,j}$ lie in the ideal $I_a^\Sigma$ generated by the entries of the first row.
\end{lem}
\begin{proof} Since $\theta$ is an algebra isomorphism we have:
\begin{equation*}\theta(P(\xi_1))= P(\theta(\xi_1))= \sum_{l=0}^N (-1)^{l}e_l(\Sigma) \theta(\xi_1)^{N-l},
\end{equation*}
and hence the entries are given by
\begin{equation*}c_{i,j}= \sum_{l=0}^N (-1)^{l}e_l(\Sigma) b^{N-l}_{i,j}.
\end{equation*}
We then compute that the entries in the first row are:
\begin{equation*}c_{1,j}= \sum_{l=0}^N (-1)^{l}e_l(\Sigma) b^{N-l}_{1,j} = \sum_{l=0}^N (-1)^{l}e_l(\Sigma) h_{N-l+1-j}= h_{N-j+1}(\X-\Sigma). 
\end{equation*}
All other entries $c_{i,j}$ are determined by the entries in the first row by a similar recursion as in the case of $\theta(\xi_1)^k$; 
assume $i>1$, then we have:
\begin{align*}
c_{i,j} &= \sum_{l=0}^N (-1)^{l}e_l(\Sigma) b^{N-l}_{i,j}= \sum_{l=0}^N (-1)^{l}e_l (\Sigma) \left(h_{N-l+i-j} - \sum_{r=1}^{i-1} h_{i-r} b^{N-l}_{r,j} \right) \\
& = h_{N+i-j}(\X-\Sigma) -  \sum_{r=1}^{i-1} h_{i-r} \sum_{l=0}^N (-1)^{l}e_l(\Sigma)  b^{N-l}_{r,j} \\
& = h_{N+i-j}(\X-\Sigma) -  \sum_{r=1}^{i-1} h_{i-r} c_{r,j}.
\end{align*}

The following (by induction on $s$) shows that $h_{N+s}(\X-\Sigma)\in I_a^\Sigma$ for every $s>0$:
\begin{align*}
h_{N+s}(\X-\Sigma) & = \sum_{l=0}^{N+s} (-1)^l e_l(\Sigma) h_{N+s-l} =  \sum_{l=0}^N (-1)^l e_l(\Sigma) h_{N+s-l} \\
& =  \sum_{l=0}^N (-1)^l e_l(\Sigma) \left(-\sum_{r=1}^N (-1)^r e_r h_{N+s-l-r}\right) \\
& =  -\sum_{r=1}^N (-1)^r e_r  h_{N+s-r}(\X-\Sigma).
\end{align*}

It then follows (again by induction on $i$) that $c_{i,j} \in I_a^\Sigma$ for all $i>1$.
\end{proof}

Since the (two sided) ideal generated by a matrix $A$ is equal to the ideal of matrices with entries taking values in the ideal generated by the entries of $A$, 
Lemma \ref{ThetaPLem} shows that taking the quotient of $Mat(a!,\Sym(\X))$ by the ideal $\theta(\langle P(\xi_1) \rangle)$ is equal to the quotient of 
$Mat(a!,\Sym(\X))$ by matrices with entries in the ideal $I_a^\Sigma$. 
This shows that $\BNC_a^\Sigma \cong Mat(a!,H_a^\Sigma)$. 
Moreover $Z(\BNC_a^\Sigma)$ is isomorphic via $\theta$ to $Z(Mat(a!,H_a^\Sigma))= H_a^\Sigma \id_{a!} \cong H_a^\Sigma$. 
\end{proof}

\begin{rem} Note that as far as the center $Z(\BNC_a^\Sigma) = \mathbf{e}_a \BNC_a^\Sigma \mathbf{e}_a$ is concerned, there is nothing special about $\xi_1$: 
in $\mathbf{e}_a \BNC_a^\Sigma \mathbf{e}_a$ the relation $\mathbf{e}_a P(\xi_j) \mathbf{e}_a =0$ holds for every $1\leq j\leq a$.
\end{rem}

\subsection{Decomposing the $\Sigma$-deformed Grassmannian cohomology ring}
The following is equivalent to Theorem \ref{mainthm} in the special case of the $\bV^a \C^N$-colored unknot.
\begin{thm}
\label{unknotdecomp} Let $\lambda_1,\dots, \lambda_l$ be pairwise distinct complex numbers and $N_1, \dots, N_l$ natural numbers such that $\sum_{i=1}^l N_i= N$ and let 
$\Sigma=\{\lambda_1^{N_1},\dots, \lambda_l^{N_l}\}$ be the multiset containing $\lambda_i$ exactly $N_i$ times.  There is an isomorphism of $\C$-algebras:
\begin{equation*} H_a^\Sigma \cong \bigoplus_{\substack{\sum a_j = a \\ 0 \leq a_j \leq N_j} }\bigotimes_{j=1}^l H_{a_j}^{N_j}.
\end{equation*}

\end{thm}
\begin{defi}
Let $\X=\{\xi_1,\dots, \xi_a\}$ be an alphabet of $a$ variables and $H_1^\Sigma= \frac{\C[\xi]}{\la P(\xi)\ra}$. Then we define $T_a^\Sigma:=\langle P(\xi_1),\dots P(\xi_a)\rangle$ and identify 
\begin{equation*}
\bigotimes_{i=1}^a H_1^\Sigma \cong \frac{\C[\X]}{T_a^\Sigma}=:R_a^\Sigma.
\end{equation*} 

The symmetric group $S_a$ acts on this by permuting tensor factors or, in other words, by permuting the $\xi_i$. Denote by $\bV^a H_1^\Sigma$ the vector space of anti-symmetric tensors in $\bigotimes_{i=1}^a H_1^\Sigma$ and by $\extp$ the vector space of anti-symmetric tensors in $\bigotimes_{i=1}^a\C[\xi]$. The latter we identify with anti-symmetric polynomials in $\C[\X]$. In both cases, we denote the anti-symmetrization map by
\begin{equation*}\ASym(-) = \frac{1}{a!}\sum_{w\in S_a} \epsilon(w) w(-).
\end{equation*}
Recall that $\Delta_\xi = \prod_{1\leq i<j\leq a} (\xi_j-\xi_i)$ denotes the Vandermonde determinant. 
Multiplying by $\Delta_\xi$ is a vector space isomorphism from $\Sym(\X)$ to $\extp$ and equips the latter with the pushforward algebra structure: 
if $\Delta_\xi f, \Delta_\xi g \in \extp $ for $f,g \in \Sym(\X)$, then
\begin{equation*}(\Delta_\xi f)*(\Delta_\xi g):= \Delta_\xi (f g).
\end{equation*}
\end{defi}

\begin{lem} \label{extlem} The pushforward algebra structure on $\extp$ descends to the quotient $\bV^a H_1^\Sigma$, and multiplication by $\Delta_\xi$ descends to an algebra isomorphism:
\begin{equation*}H_a^\Sigma \xrightarrow{\cong} \bV^a H_1^\Sigma.
\end{equation*} 
\end{lem}
\begin{proof}
It suffices to check that $\Delta_\xi \cdot I_a^\Sigma \subset T_a^\Sigma$. 
We then have the composition of linear maps
\begin{equation*}
H_a^\Sigma = \frac{\Sym(\X)}{I_a^\Sigma} \overset{\Delta_\xi}{\cong} \frac{\extp}{\Delta_\xi \cdot I_a^\Sigma}  \twoheadrightarrow \frac{\extp}{\extp \cap T_a^\Sigma}\cong \frac{\extp +T_a^\Sigma }{ T_a^\Sigma} 
= \ASym \left(\bigotimes_{i=1}^a H_1^\Sigma \right) = \bV^a H_1^\Sigma 
\end{equation*}
which is surjective, and hence must be an isomorphism for dimensional reasons.

We now check that the generators of $I_a^\Sigma$ are mapped into $T_a^\Sigma$ under multiplication by $\Delta_\xi$. Let  $1\leq j \leq a$, then we have:
\begin{align*} 
\Delta_\xi h_{N-a+j}(\X-\Sigma) &=  \sum_{i=0}^N (-1)^ie_i(\Sigma) \Delta_\xi h_{N-a+j-i}(\X) \\
& = \sum_{i=0}^N (-1)^ie_i(\Sigma) \sum_{w \in S_a} \epsilon(w) \xi_{w(a)}^{a-1+N-a+j-i} \xi_{w(a-1)}^{a-2}\cdots \xi_{w(2)}^1 \\
& = \sum_{w \in S_a} \epsilon(w) \left(\xi_{w(a)}^{j-1}P(\xi_{w(a)}) \right) \xi_{w(a-1)}^{a-2}\cdots \xi_{w(2)}^1  \in  T_a^\Sigma .
\end{align*}

Here we have used the identity:
$\Delta_{\xi} h_k(\X) = \sum_{w \in S_a} \epsilon(w) \xi_{w(a)}^{a-1+k}\xi_{w(a-1)}^{a-2}\cdots \xi_{w(2)}^1$ which is clear from the defining formula for the Schur polynomials 
\begin{equation*}\pi_{\alpha}(\X):= \frac{\det_{1\leq i,j,\leq m}(\xi_i^{\alpha_j+a-j})}{\det_{1\leq i,j,\leq m}(\xi_i^{a-j})}= \frac{\det_{1\leq i,j,\leq m}(\xi_i^{\alpha_j+a-j})}{\Delta_\xi} 
\end{equation*}
and the identity $h_k(\X)= \pi_{(k)}(\X)$.
\end{proof}

We can now prove Theorem \ref{unknotdecomp}.
\begin{proof}[Proof of Theorem \ref{unknotdecomp}]
By the Chinese Remainder Theorem we know that $\frac{\C[\xi]}{\la P(\xi)\ra}\cong \bigoplus_{i=1}^l \frac{\C[\xi]}{\la (\xi-\lambda_i)^{N_i}\ra }$, 
so let $\idem_\mu(\xi)\in \C[\xi]$ be a representative for the idempotent that picks out the summand corresponding to the root $\mu\in \Sigma$. Thus we get the algebra isomorphism:
 \begin{equation*} H_1^\Sigma= \frac{\C[\xi]}{\langle P(\xi)\rangle } \cong \bigoplus_{i=1}^l \idem_{\lambda_i}(\xi) H_1^\Sigma.
\end{equation*}
 In the following we make liberal use of the canonical isomorphism $\bigotimes_{j=1}^a H_1^\Sigma\cong \frac{\C[\X]}{T_a^\Sigma} = R_a^\Sigma$. 
 A set of minimal idempotents in $R_a^\Sigma$ is given by $\{ \idem_{\overline{\mu}}:=\prod_{j=1}^a \idem_{\mu_j}(\xi_j)\}$ where $\overline{\mu}=(\mu_1,\dots, \mu_a)$ ranges of all $a$-tuples of roots appearing in $\Sigma$. 
 $S_a$ acts on $\C[\X]$ and $R_a^\Sigma$ by permuting the indices of the variables $\xi_i$ and on tuples $\overline{\mu}$ by permuting roots. 
 For $w\in S_a$ we have:
\begin{equation*}w(\idem_{\overline{\mu}})= \prod_{j=1}^a \idem_{\mu_j}(\xi_{w(j)})= \prod_{j=1}^a \idem_{\mu_{w^{-1}(j)}}(\xi_{j}) = \idem_{w^{-1}(\overline{\mu})}.
\end{equation*}
Given an $a$-element multiset $A=\{\lambda_1^{a_1},\dots, \lambda_l^{a_l}\}$ of roots we write $\mu_A:=(\lambda_1,\dots,\lambda_1,\lambda_2,\dots, \lambda_l)$ for the corresponding tuple ordered by index. 
Every $\overline{\mu}$ can be written as $\overline{\mu}=\tau^{-1}(\mu_A)$ for a $\tau\in S_a$ and a multiset $A$, and this presentation is unique if we restrict the choice of $\tau$ to a set of 
coset representatives\footnote{For convenience we choose $T$ to be simultaneously a set of right and left coset representatives.} $T$ of $\prod_{i=1}^l S_{a_i}$ in $S_a$. 

With these conventions in place, we can decompose $R_a^\Sigma$ into $S_a$-invariant direct summands:
\begin{equation}
\label{eqn-decomp}
R_a^\Sigma\cong \bigoplus_{\substack{a\text{-element multisets A} \\ \text{ of roots }}}~\underbrace{\bigoplus_{\tau \in T}~\tau (\idem_{\mu_A})R_a^\Sigma}_{S_a\text{-invariant}}. 
\end{equation}
Taking anti-symmetric components respects the decomposition on the right-hand side into $S_a$-invariant direct summands. 
Thus, our goal is to compute the anti-symmetric component of an (outer) summand on the right-hand side. Consider the projection
\begin{equation*}\bigoplus_{\substack{\tau \in T}}\tau (\idem_{\mu_A})R_a^\Sigma \to  \idem_{\mu_A} R_a^\Sigma
\end{equation*}
which is given by multiplying by the idempotent $\idem_{\mu_A}$. An elementary computation shows that this restricts to a vector space isomorphism
\begin{equation}
\label{eqn-Xone} \pi \colon X_1:=  \ASym_{S_a} \left ( \bigoplus_{\substack{\tau \in T}}\tau (\idem_{\mu_A})R_a^\Sigma \right ) \to \ASym_{\prod_{i=1}^l S_{a_l}} \left ( \idem_{\mu_A} R_a^\Sigma \right ) =:X_2
\end{equation}
where the right-hand side denotes the vector space of tensors $y$ in $\idem_{\mu_A} R_a^\Sigma$ which are anti-symmetric for the action of $\prod_{i=1}^l S_{a_l}\subset S_a$, that is $w(y)=\epsilon(w) y$ for all $w\in \prod_{i=1}^l S_{a_l}$. The inverse for $\pi$ is given by $\psi(y):=\sum_{\tau \in T} \epsilon(\tau)\tau(\idem_{\mu_A} y)$.
 
Fix $A=\{\lambda_1^{a_1},\dots, \lambda_l^{a_l}\}$ as above, then for $1\leq i \leq l$ we denote: 
\[
\X_i:=\left \{ \xi_{1+ \sum_{k=1}^{i-1}a_k},\dots, \xi_{\sum_{k=1}^{i}a_k} \right \}, \quad  T_a^{\lambda_i\in\Sigma}:= \langle P(\xi)\mid \xi \in \X_i \rangle,\]
\[ R_i:= \frac{\C[\X_i]}{T_a^{\lambda_i\in\Sigma}}\;\text{ and }\idem_{\lambda_i}:=\prod_{\xi \in \X_i} \idem_{\lambda_i}(\xi).\]
Under the canonical isomorphism $R_a^\Sigma\cong \otimes_{i=1}^l R_i$ we have $\idem_{\mu_A} R_a^\Sigma\cong \otimes_{i=1}^l \idem_{\lambda_i}R_i$ and
\begin{equation}
\label{eqn-Xtwo} X_2 = \ASym_{\prod_{i=1}^l S_{a_l}}\left ( \idem_{\mu_A} R_a^\Sigma \right ) \cong \bigotimes_{i=1}^l \ASym_{S_{a_i}}\idem_{\lambda_i}R_i =:X_3.
\end{equation}
Note for later use that since $\idem_A:=\sum_{\tau\in T}\tau(\idem_{\mu_A})$ is $S_a$-invariant and $\idem_{\mu_A}$ is $\prod_{i=1}^l S_{a_i}$-invariant, we also have 
\begin{eqnarray*}X_1=& \ASym_{S_a} \left ( \bigoplus_{\substack{\tau \in T}}\tau (\idem_{\mu_A})R_a^\Sigma \right ) = \idem_A  \ASym_{S_a} \left ( R_a^\Sigma \right ), \\
X_2=& \ASym_{\prod_{i=1}^l S_{a_l}} \left ( \idem_{\mu_A} R_a^\Sigma \right ) = \idem_{\mu_A} \ASym_{\prod_{i=1}^l S_{a_i}} \left( R_a^\Sigma \right )
\end{eqnarray*} with respect to the multiplication in $R_a^\Sigma$.

From the Chinese Remainder Theorem we know that $\xi \mapsto w+\lambda_i$ gives an algebra isomorphism
\begin{equation*}\phi\colon \idem_{\lambda_i}(\xi) \frac{\C[\xi]}{\langle{P(\xi)}\rangle } \to \frac{\C[\xi]}{\langle{(\xi-\lambda_i)^{N_i}}\rangle } \to \frac{\C[w]}{\langle{w^{N_i}}\rangle }
\end{equation*} and this extends to an $S_{a_i}$-equivariant algebra isomorphism:
\begin{equation*}\phi\colon  \idem_{\lambda_i} R_i = \idem_{\lambda_i} \frac{\C[\X_i]}{T_a^{\lambda_i\in \Sigma} } \to \frac{\C[\X_i]}{\langle (\xi-\lambda_i)^{N_i}\mid \xi\in \X_i \rangle } \to \frac{\C[\W_i]}{\langle w^{N_i} \mid w\in \W_i \rangle }
\end{equation*} where $\W_i=\{w_1,\dots, w_{a_i}\}$ is an auxiliary alphabet. It follows from Lemma \ref{extlem} that, when restricted to the anti-symmetric component, $\phi$ gives the vector space isomorphism
\begin{equation}
\label{eqn-Xthree} \phi\colon X_3=\bigotimes_{i=1}^l \ASym_{S_{a_i}}\idem_{\lambda_i}R_i \xrightarrow{} \bigotimes_{i=1}^l H_{a_i}^{N_i}.
\end{equation}

The composition of the vector space isomorphisms in equations \eqref{eqn-Xone}, \eqref{eqn-Xtwo} and \eqref{eqn-Xthree} thus gives a decomposition of the $S_a$-invariant direct summands of equation \eqref{eqn-decomp} 
as required by the statement of the theorem. 
However, we further must check that the composition is an algebra isomorphism. 
In fact it is not, but it is close and the discrepancy is not hard to fix.

To see this, we compute the pushforward of the multiplication $*$ on $X_1$ under $\pi$. 
Let $x,y\in X_1$ be represented by anti-symmetric polynomials in $\C[\X]$ and denote by $x y$  their product in $R_a^\Sigma$ and by $x*y$ their product in $\bV^a H_1^\Sigma$.
We compute:
\begin{equation*}\pi(x)*\pi(y):=\pi(x*y)= \pi(\frac{x y}{\Delta_\xi})= \idem_{\mu_A} \frac{ x y}{\Delta_\xi} = c \frac{(\idem_{\mu_A} x) (\idem_{\mu_A} y)}{\prod_{i=1}^l\Delta_i}
\end{equation*}
where we write $\Delta_i:=\prod_{1+\sum_{k=1}^{i-1}a_k \leq r<s  \leq \sum_{k=1}^{i}a_k}(\xi_r-\xi_s)$ for the Vandermonde determinants in the subalphabets $\X_i\subset \X$ and $c= 1_{\mu_A}(\prod_{i=1}^l \Delta_l)/\Delta_\xi$. 
We will see in Lemma \ref{ArtLem} that $c$ represents a unit in $\idem_{\mu_A}R_a^\Sigma$ and clearly it is $\prod_{i=1}^l S_{a_i}$-invariant.
It follows that $\pi/c$ is still a vector space isomorphism $X_1\to X_2$, and the pushforward of the multiplication $*$ on $X_1$ under it is given by:
\begin{equation}
\label{PiAISO}(\pi/c)(x)* (\pi/c)(y):=  (\pi/c)(x*y) = \frac{(\idem_{\mu_A} x) (\idem_{\mu_A} y)}{\prod_{i=1}^l\Delta_i}.
\end{equation}
We now equip each tensorand $\ASym_{S_{a_i}}\idem_{\lambda_i}R_i$ of $X_3$, see \eqref{eqn-Xtwo}, with the multiplication $*$ given by multiplying representing anti-symmetric polynomials and then dividing by the appropriate Vandermonde determinant $\Delta_i$.
Then equation \eqref{PiAISO} says that $\pi/c\colon X_1\to X_2$ composed with the canonical isomorphism $X_2\to X_3$ is an algebra isomorphism with respect to the tensor product algebra structure on $X_3$. 
Since $\phi$ sends $\Delta_{i}$ to $\prod_{0\leq r<s\leq a_i}(w_r+\lambda_i-w_s-\lambda_i)=\Delta_w$, an easy check shows that $\phi$ in equation \eqref{eqn-Xthree} is also an algebra isomorphism. 

To summarize the proof, we assemble the algebra isomorphisms:
\begin{equation*}H_a^\Sigma \cong \bV^a H_1^\Sigma \cong \bigoplus_{\substack{a\text{-element multisets } A \\\text{ of roots} }} \idem_A \ASym_{S_a} (\idem_A R_a^\Sigma) 
\end{equation*}
\begin{equation*}\cong \bigoplus_{\substack{\sum a_j = a  \\\ A = \{\lambda_1^{a_1},\dots, \lambda_l^{a_l}\} }}\bigotimes_{i=1}^l \ASym_{S_{a_i}}\idem_{\lambda_i}R_i \cong \bigoplus_{\substack{\sum a_j = a \\ 0 \leq a_j \leq N_j  }}\bigotimes_{i=1}^l H_{a_i}^{N_i}.
\end{equation*}
The first isomorphism was introduced in Lemma \ref{extlem}, and the second one comes from the direct sum decomposition of $(H_1^\Sigma)^{\otimes a}$ into $S_a$-invariant summands. 
The third isomorphism is assembled from the isomorphisms $\pi/c$ from equation \eqref{PiAISO} on summands composed with the canonical isomorphism in equation \eqref{eqn-Xtwo}, 
and the last one comes from the Chinese Remainder Theorem and the inverse of the isomorphism from Lemma \ref{extlem}, see equation \eqref{eqn-Xthree}. 
The last isomorphism also shows that a summand indexed by a multiset $A$ of roots is non-zero if and only if $A$ is actually a multisubset of $\Sigma$.
 \end{proof}

In the proof we have claimed that $c= \idem_{\mu_A}\frac{\prod_{i=1}^l \Delta_l}{\Delta_\xi}$ represents a unit in $\idem_{\mu_A}R_a^\Sigma$. This is clear from the following useful lemma.
\begin{lem}
\label{ArtLem}
Let $R$ be a finite dimensional quotient of a polynomial ring 
$R=\frac{\C[x_1,\dots, x_a]}{I}$ and let $V(I)\subset \C^a$ be the vanishing set of $I$. Then we have the decomposition
\begin{equation*}R\cong \bigoplus_{v\in V(I)}\idem_v R 
\end{equation*} where $\idem_v$ are minimal idempotents and $\idem_v R$ is isomorphic to $R_{p_v}$, the localization of $R$ at the complement of the maximal ideal $(x_1-v_1,\dots, x_a-v_a)/I$. For elements $\bar{f} \in \idem_v R$ we have
\begin{equation}
\label{preidunit}
\bar{f} \text{ is not a unit} \iff \bar{f} \text{ is a zero divisor} \iff f(v)=0
\end{equation} 
where $f$ is any lift of $\bar{f}$ to $\C[x_1,\dots, x_a]$.
\end{lem}
\begin{proof}
Since $R$ is a commutative Artinian ring, it decomposes uniquely into local commutative Artinian rings, one for each maximal ideal of $R$. 
Maximal ideals of $R$ are in bijection with maximal ideals of $\C[x_1,\dots,x_a]$ that contain $I$. 
The maximal ideals of $\C[x_1,\dots,x_a]$ are exactly $I_v:=(x_1-v_1,\dots,x_a-v_a)$ for $v\in \C^a$ and $I\subset I_v \iff f(v)=0\quad \forall f\in I \iff v\in V(I)$. 
It follows that 
\begin{equation*}R\cong \bigoplus_{v\in V(I)} R_{p_v} \cong \bigoplus_{v\in V(I)}\idem_v R 
\end{equation*} 
where $p_v:=I_v/I$, $R_{p_v}$ denotes the localization of $R$ at $R\setminus p_v$, and $\idem_v \in R$ is the idempotent corresponding to the summand $R_{p_v}$. 
The statement about non-units is then clear from the explicit description of the local ring $R_{p_v}$.
\end{proof}

Now in the case of $c \in \idem_{\mu_A}R_a^\Sigma$ for $\mu_A=(\mu_1,\dots,\mu_a)=(\lambda_1,\dots, \lambda_1,\lambda_2,\dots, \lambda_l)$ we have
\begin{equation*}
c^{-1} |_{\xi_i \mapsto \mu_i}= \idem_{\mu_A}\frac{\Delta_\xi}{\prod_{i=1}^l \Delta_l} |_{\xi_i \mapsto \mu_i}=\idem_{\mu_A} \prod_{\mu_i\neq \mu_j, i<j}(\mu_i-\mu_j)\neq 0 
\end{equation*} and equation \eqref{preidunit} shows that $c^{-1}$, hence also $c$, is a unit.

\begin{rem}
\label{idconsistency}
We have the isomorphism 
\begin{equation*}H_a^\Sigma = \frac{\Sym(\X)}{\langle h_{N-a+1}(\X-\Sigma),\dots, h_{N}(\X-\Sigma)\rangle}\cong \frac{\C[e_1(\X),\dots, e_a(\X)]}{\langle h_{N-a+1}(\X-\Sigma),\dots, h_{N}(\X-\Sigma)\rangle}
\end{equation*} 
and it follows by considering the generating function of $h_j(\X-\Sigma)$ that the vanishing set of this ideal is given by $\{(e_1(A),\dots, e_a(A))\mid A\subset \Sigma,~ |A|=a \}\subset \C^a$. 
Applying Lemma \ref{ArtLem} reproves the fact that the minimal idempotents of $H_a^\Sigma$ are indexed by $a$-element multisubsets $A\subset \Sigma$. 
However, we should check that the idempotent corresponding to $A$ identified in this remark -- call it $\idem'_A$ -- equals $\idem_A$ as defined in the proof of Theorem \ref{unknotdecomp}. 
For this it suffices to check that $\idem_A(\X)|_{\X\mapsto A}~\neq 0 \in \C$. 
Recall that by definition $\idem_{\lambda_i}(\xi)=\idem'_{\lambda_i}(\xi)$ in $\frac{\C[\xi]}{\langle P(\xi) \rangle}$, and hence $\idem_{\lambda_i}(\lambda_j)=\delta_{i}^j$. 
Further, $\mu_A=(\mu_1,\dots,\mu_a)$ was defined as the $a$-tuple consisting of elements $\lambda_i$ of $A$, ordered by index $i$, so we compute
\begin{equation*}
\idem_A(\X) |_{\X\mapsto A}~=~ \idem_A(\X) |_{(\xi_1,\dots,\xi_a)\mapsto \mu_A} ~=~\sum_{\tau \in T} \tau(\idem_{\mu_A})(\mu_A)= \sum_{\tau \in T}\prod_{j=1}^a \idem_{\mu_{\tau(j)}}(\mu_j) = \prod_{j=1}^a \idem_{\mu_{j}}(\mu_j) =1.
\end{equation*}
\end{rem}

\begin{cor}
\label{idunit}
Let $A$ be an $a$-element multisubset of $\Sigma$ and $f\in \Sym(\X)$, then $f$ represents a unit in $\idem_A H_a^\Sigma$  if and only if $f(A)\neq 0$.
\end{cor}
\begin{proof} Immediate from equation \eqref{preidunit} and Remark \ref{idconsistency}.
\end{proof}

\subsection{Thick calculus for nilHecke quotients}
We now deduce relations for the nilHecke quotients $\BNC_a^\Sigma$ using the thick graphical calculus introduced in \cite{KLMS} and detailed above in Section \ref{section-thick}. 
Note that in the quotients $\BNC_a^\Sigma$ the element $\mathbf{e}_a$ is still an idempotent, and it projects onto a direct summand isomorphic to $Z(\BNC_a^\Sigma)\cong H_a^\Sigma$, 
but in general it is not a minimal idempotent due to the decomposition of $H_a^\Sigma$ given in Theorem \ref{unknotdecomp}.

\begin{cor}
\label{cor-idem} The collection of symmetric polynomials 
\begin{equation*} 
    \xy
 (0,0)*{\includegraphics[scale=0.5]{figs/single-tup.eps}};
 (-2.5,-8)*{a};(0,-2)*{\bigb{A}};
  \endxy\quad :=  \idem_{A} = \sum_{\tau\in T}{\tau(\idem_{\mu_A})}\in \Sym(\X)
\end{equation*}
for $A \subset \Sigma$ and $|A|=a$ which were introduced in the proof of Theorem \ref{unknotdecomp}, 
give a complete collection of commuting, minimal, orthogonal idempotents of $Z(\BNC_a^\Sigma)\cong H_a^\Sigma$. 
In other words, in $H_a^\Sigma$ we have that for $a$-element multisubsets $A$ and $B$ of $\Sigma$,
\begin{equation} 
\label{orthogonal}
 \xy
 (0,0)*{\includegraphics[scale=0.5]{figs/tlong-up.eps}};
 (-2.5,-12)*{a};(0,1)*{\bigb{A}};(0,-6)*{\bigb{B}};
  \endxy
\quad =\quad 
\begin{cases} \xy
 (0,0)*{\includegraphics[scale=0.5]{figs/single-tup.eps}};
 (-2.5,-8)*{a};(0,-2)*{\bigb{A}};
  \endxy
   &\text{ if } A = B\\ ~0 &\text{ if } A \neq B\end{cases} 
\end{equation}
and the thick edge decomposes in $H_a^\Sigma$ into a sum of $\idem_A$-decorated thick edges:
\begin{equation}
\label{edgedecomp}
 \xy
 (0,0)*{\includegraphics[scale=0.5]{figs/single-tup.eps}};(-2.5,-8)*{a};
  \endxy \quad = \quad \sum_{\substack{a\text{-element multisets } \\ A~\subset ~\Sigma }}  \xy
 (0,0)*{\includegraphics[scale=0.5]{figs/single-tup.eps}};
 (-2.5,-8)*{a};(0,-2)*{\bigb{A}};
  \endxy
  \end{equation}
\end{cor} 
\begin{proof} Immediate.
\end{proof}

\begin{prop} \emph{(Non-admissible colorings by multisubsets)}\\
\label{adm}
Let $A$, $B$ and $C$ be $a$-, $b$- and $(a+b)$-element multisubsets of $\Sigma$, then in $\BNC_{a+b}^\Sigma$ we have
  \begin{equation}
 \label{nonadm}
  \xy
 (0,0)*{\includegraphics[scale=0.5]{figs/tsplit.eps}};
(-5,-11)*{a+b};(-8,8)*{a};(8,8)*{b};
 (-5,2)*{\bigb{A}};
 (5,2)*{\bigb{B}};
  (0,-6)*{\bigb{C}};
  \endxy
\quad =\quad  
\xy
 (0,0)*{\includegraphics[scale=0.5,angle=180]{figs/tsplitd.eps}};
 (-5,11)*{a+b};(-8,-8)*{a};(8,-8)*{b};
 (-5,-2)*{\bigb{A}};
 (5,-2)*{\bigb{B}};
 (0,6)*{\bigb{C}};
  \endxy
\quad =\quad 
0 \qquad \text{ if } A\uplus B \neq C
    \end{equation}
and we call such a coloring \emph{non-admissible}. Here $\uplus$ denotes the multiset sum (disjoint union) of multisets.
\end{prop}
Labelings by idempotents corresponding to multisubsets of $\Sigma$ that ``add up'' at mergers and splitters (i.e. $A\uplus B = C$) are called \emph{admissible}.

\begin{proof}
Denote by $\X_1$, $\X_2$ and $\X$ the alphabets of operators $\xi_j$ on the thickness $a$, $b$ and $a+b$ strands respectively 
and by $H_a^\Sigma(\X_1)$, $H_b^\Sigma(\X_2)$ and $H_{a+b}^\Sigma(\X)$ the algebras of decorations on these strands.

Equation (2.67) in \cite{KLMS} then implies that the algebras of decorations on the diagrams
\begin{equation*} 
\xy
 (0,0)*{\includegraphics[scale=0.5]{figs/tsplit.eps}};
(-5,-11)*{a+b};(-8,8)*{a};(8,8)*{b};
  \endxy \qquad \text{and}\qquad 
\xy
 (0,0)*{\includegraphics[scale=0.5,angle=180]{figs/tsplitd.eps}};
 (-5,11)*{a+b};(-8,-8)*{a};(8,-8)*{b};
  \endxy
  \end{equation*}
  are both given by 
  \begin{equation*}
  \frac{H_{a+b}^\Sigma(\X)\otimes H_{a}^\Sigma(\X_1) \otimes H_{b}^\Sigma(\X_2)}{\la e_i(\X)-e_i(\X_1\sqcup \X_2)\mid i>0 \ra } \quad .
  \end{equation*}

Now let $A$, $B$ and $C$ be $a$-, $b$- and $(a+b)$-element multisubsets of $\Sigma$ respectively. Then the algebra of additional decorations on the idempotent-decorated diagrams in \eqref{nonadm} is
  \begin{equation}
  \label{algsplitter} 
  \frac{\idem_C(\X)H_{a+b}^\Sigma(\X)\otimes \idem_A(\X_1) H_{a}^\Sigma(\X_1) \otimes \idem_B(\X_2) H_{b}^\Sigma(\X_2)}{\la e_i(\X)-e_i(\X_1\sqcup \X_2)\mid i>0 \ra \cap (\idem_C(\X)H_{a+b}^\Sigma(\X)\otimes \idem_A(\X_1) H_{a}^\Sigma(\X_1) \otimes \idem_B(\X_2) H_{b}^\Sigma(\X_2)) }.
\end{equation} 
The numerator here is a direct summand of $H_{a+b}^\Sigma(\X)\otimes H_{a}^\Sigma(\X_1) \otimes H_{b}^\Sigma(\X_2)$ that can be picked out by localizing at the complement of the maximal ideal 
$\la e_i(\X)-e_i(C),e_i(\X_1)-e_i(A), e_i(\X_2)-e_i(B)\mid i>0 \ra$. 
If $C\neq A\uplus B$ then there is a $j\in \N$ such that $e_j(C)-e_j(A\uplus B) \neq 0 \in \C$, thus by Lemma \ref{idunit} $e_j(\X)-e_j(\X_1\sqcup \X_2)$ is a unit in the numerator. 
Taking the quotient in equation \eqref{algsplitter} then collapses the direct summand, and equation \eqref{nonadm} then follows.
\end{proof}

\begin{cor}\emph{(Idempotent decoration migration)}\\
\label{adm2}
Let $A$ be an $(a+b)$-element multisubset of $\Sigma$, then in $\BNC_{a+b}^\Sigma$ we have:

\begin{equation} 
\label{splitteridem}
    \xy
 (0,0)*{\includegraphics[scale=0.5]{figs/tsplit.eps}};
 (-5,-11)*{a+b};(-8,8)*{a};(8,8)*{b};
  (0,-6)*{\bigb{A}};
  \endxy
  \quad =\quad 
  \sum_{\substack{A_1\uplus A_2 = A\\ |A_1|=a   }} \xy
 (0,0)*{\includegraphics[scale=0.5]{figs/tsplit.eps}};
 (-5,-11)*{a+b};(-8,8)*{a};(8,8)*{b};
 (-5,2)*{\bigb{A_1}};(5,2)*{\bigb{A_2}};
 (0,-6)*{\bigb{A}};
  \endxy \quad =
 \quad
 \sum_{\substack{A_1\uplus A_2 = A\\ |A_1|=a   }} \;
    \xy
 (0,0)*{\includegraphics[scale=0.5]{figs/tsplit.eps}};
(-5,-11)*{a+b};(-8,8)*{a};(8,8)*{b};
 (-5,2)*{\bigb{A_1}};(5,2)*{\bigb{A_2}};
  \endxy \quad ,
\end{equation}
\begin{equation}
\label{mergeidem}
\xy
 (0,0)*{\includegraphics[scale=0.5,angle=180]{figs/tsplitd.eps}};
 (-5,11)*{a+b};(-8,-8)*{a};(8,-8)*{b};
  (0,6)*{\bigb{A}};
  \endxy
  \quad =
 \quad
 \sum_{\substack{A_1\uplus A_2 = A\\ |A_1|=a   }} \;
    \xy
 (0,0)*{\includegraphics[scale=0.5,angle=180]{figs/tsplitd.eps}};
 (-5,11)*{a+b};(-8,-8)*{a};(8,-8)*{b};
 (-5,-2)*{\bigb{A_1}};
 (5,-2)*{\bigb{A_2}};
 (0,6)*{\bigb{A}};
  \endxy
  \quad =
 \quad
 \sum_{\substack{A_1\uplus A_2 = A\\ |A_1|=a   }} \;
    \xy
 (0,0)*{\includegraphics[scale=0.5,angle=180]{figs/tsplitd.eps}};
 (-5,11)*{a+b};(-8,-8)*{a};(8,-8)*{b};
 (-5,-2)*{\bigb{A_1}};
 (5,-2)*{\bigb{A_2}};
  \endxy \quad .
\end{equation}
In particular, for multisubsets $A,B\subset \Sigma$, $|A|=a$, $|B|=b$ we have: 
\begin{equation} 
\label{splitteridem2}
\xy
 (0,0)*{\includegraphics[scale=0.5]{figs/tsplit.eps}};
(-5,-11)*{a+b};(-8,8)*{a};(8,8)*{b};
 (-5,2)*{\bigb{A}};(5,2)*{\bigb{B}};
  \endxy
\quad =\quad 
\begin{cases} \xy
 (0,0)*{\includegraphics[scale=0.5]{figs/tsplit.eps}};
(-5,-11)*{a+b};(-8,8)*{a};(8,8)*{b};
 (-5,2)*{\bigb{A}};(5,2)*{\bigb{B}};
 (0,-6)*{\bigb{A\uplus B}};
  \endxy
   &\text{ if } A\uplus B \subset \Sigma, \\ 0 &\text{ otherwise. }\end{cases} 
   \end{equation}
   
   \begin{equation}
   \label{mergeidem2}
  \xy
 (0,0)*{\includegraphics[scale=0.5,angle=180]{figs/tsplitd.eps}};
 (-5,11)*{a+b};(-8,-8)*{a};(8,-8)*{b};
 (-5,-2)*{\bigb{A}};
 (5,-2)*{\bigb{B}};
  \endxy
\quad =\quad 
\begin{cases} \xy
 (0,0)*{\includegraphics[scale=0.5,angle=180]{figs/tsplitd.eps}};
 (-5,11)*{a+b};(-8,-8)*{a};(8,-8)*{b};
 (-5,-2)*{\bigb{A}};
 (5,-2)*{\bigb{B}};
 (0,6)*{\bigb{A\uplus B}};
  \endxy
   &\text{ if } A\uplus B \subset \Sigma, \\ 0 &\text{ otherwise. }\end{cases}
\end{equation}
\end{cor}

\begin{proof} For equation \eqref{splitteridem} we compute
\begin{equation*} 
    \xy
 (0,0)*{\includegraphics[scale=0.5]{figs/tsplit.eps}};
 (-5,-11)*{a+b};(-8,8)*{a};(8,8)*{b};
  (0,-6)*{\bigb{A}};
  \endxy
  \quad \overset{\eqref{edgedecomp}}{=}\quad 
  \sum_{\substack{A_1,A_2\subset \Sigma\\ |A_1|=a,~|A_2|=b    }} \xy
 (0,0)*{\includegraphics[scale=0.5]{figs/tsplit.eps}};
 (-5,-11)*{a+b};(-8,8)*{a};(8,8)*{b};
 (-5,2)*{\bigb{A_1}};(5,2)*{\bigb{A_2}};
 (0,-6)*{\bigb{A}};
  \endxy \quad \overset{\eqref{nonadm}}{=}\quad 
  \sum_{\substack{A_1\uplus A_2 = A\\ |A_1|=a    }} \xy
 (0,0)*{\includegraphics[scale=0.5]{figs/tsplit.eps}};
 (-5,-11)*{a+b};(-8,8)*{a};(8,8)*{b};
 (-5,2)*{\bigb{A_1}};(5,2)*{\bigb{A_2}};
 (0,-6)*{\bigb{A}};
  \endxy
\end{equation*}
\begin{equation*}
 \overset{\eqref{nonadm}}{=}\quad
 \sum_{\substack{A_1\uplus A_2 = A\\ |A_1|=a \\ B\subset \Sigma \\
 |B|=a+b    }} \;
    \xy
 (0,0)*{\includegraphics[scale=0.5]{figs/tsplit.eps}};
(-5,-11)*{a+b};(-8,8)*{a};(8,8)*{b};
 (-5,2)*{\bigb{A_1}};(5,2)*{\bigb{A_2}};
 (0,-6)*{\bigb{B}};
  \endxy
  \quad \overset{\eqref{edgedecomp}}{=} \quad \sum_{\substack{A_1\uplus A_2 = A\\ |A_1|=a   }} \;
    \xy
 (0,0)*{\includegraphics[scale=0.5]{figs/tsplit.eps}};
(-5,-11)*{a+b};(-8,8)*{a};(8,8)*{b};
 (-5,2)*{\bigb{A_1}};(5,2)*{\bigb{A_2}};
  \endxy.
\end{equation*}
and the proof of equation \eqref{mergeidem} is analogous. 
Equations \eqref{splitteridem2} and \eqref{mergeidem2} follow similarly. 
\end{proof}

\section{The $\Sigma$-deformed foam category $\Foam{N}^\Sigma$}
\label{section-foamcat}
We define the 2-category $\foam{N}{}^\Sigma$ of $\Sigma$-deformed $\slnn{N}$-foams as the quotient of the foam 2-category
$\foam{N}{}$, described in Section \ref{subsubsec_foams}, by the following additional relation on $1$-labeled foam facets:
\begin{equation}
\label{dotrel}
\eqnAX \quad .
\end{equation}
Since this equation is not degree-homogeneous, we hence ignore the grading on foams (i.e. to be precise we first pass to 
the ungraded version of $\Foam{N}$, then impose this relation to pass to $\Foam{N}^\Sigma$).

This quotient is motivated by the deformed nilHecke algebra quotient introduced in the last section. 
Indeed, the 2-representation $\Ucatc_Q(\glm) \to \Foam{N}$ gives an action of the nilHecke algebra 
on the latter, and in order to obtain an action of the $\Sigma$-deformed nilHecke quotient, 
we impose this local foam analog of equation \eqref{eq-doteq}.

\begin{defi} We define $\Phi_\Sigma\colon \Ucatc_Q(\glm) \to \foam{N}^\Sigma$ as the composition of the foamation 2-functor $\Phi_m\colon \Ucatc_Q(\glm) \to \foam{N}$ 
and the quotient 2-functor $\foam{N}\to \foam{N}^\Sigma$.
\end{defi}

It follows that the 2-functor $\Phi_\Sigma \colon  \Ucatc_Q(\glm) \to \foam{N}^\Sigma$ factors through the quotient of $\Ucatc_Q(\glm)$ in which we've imposed the relation that 
dots satisfy the equation $
P\left(
\xy
(0,0)*{
\begin{tikzpicture} [scale=.3]
	\draw[thick, directed = .99] (0,0) to (0,2);
	\node at (0,1) {$\bullet$};
\end{tikzpicture}
};
\endxy
\right)
=0$, hence, the thick calculus equations in Corollaries \ref{cor-idem} and \ref{adm2} and Proposition \ref{adm} 
correspond to analogous foam relations in $\foam{N}^\Sigma$. In fact, the thick calculus relations can be seen as intersections of foam relations with planes.
More precisely, we get:
\begin{lem} The algebra of decorations of a $k$-labeled foam facet, or alternatively, the endomorphism algebra of the $k$-labeled web edge, 
carries an action of $H_k^\Sigma$. In fact, from the 2-representation on deformed matrix factorizations in Section \ref{section-defmf} it follows that there is an isomorphism \[\End \left(
\xy
(0,0)*{
\begin{tikzpicture}
\node[rotate=-45] at (0,0){\xy
(0,0)*{
\begin{tikzpicture} [scale=.2]
	\draw[very thick, directed=.55] (2,-2) -- (-2,-2);
\end{tikzpicture}
};
\endxy};
\node at (.25,.25) {\small$k$};
\end{tikzpicture}
};
\endxy
\right) \cong H_k^\Sigma .
\]
\end{lem}
Compare with Remark 4.1 in \cite{QR}. Moreover, we have the following important consequences.
\begin{itemize}
\item Every $k$-labeled foam facet in $\foam{N}{}^\Sigma$ splits into a sum over foam facets colored by minimal idempotent decorations corresponding to 
$k$-element multisubsets of $\Sigma$.
\item Equation \eqref{nonadm} then implies that a foam is zero whenever it contains a seam whose adjacent facets are non-admissibly colored by idempotents. 
Here, similar to the case of thick calculus diagrams, we say that a foam is {\it admissibly colored} precisely when around any seam the sum of the multisets of the idempotents 
coloring two of the facets equals the multiset coloring the third.
Consequently, foam relations analogous to those of Corollary \ref{adm2} hold in a neighborhood of any seam. 
\end{itemize}

\subsection{Foam splitting relations}
\label{relations}
\begin{conv}
\label{conv-coloring}
Let $A,B\subset \Sigma$ be disjoint multisubsets of roots:  
\begin{equation*}
\lambda\in A \Rightarrow \lambda \notin B \text{ and } \mu \in B \Rightarrow \mu \notin A.
\end{equation*}
For the duration, unless otherwise stated, we use red and blue colored foam facets to denote facets decorated by the orthogonal idempotents $\idem_A$ and $\idem_B$, respectively. 
We use green as a generic color for both undecorated foam facets and for decorations by $\idem_{A\uplus B}$.
\end{conv}

\begin{lem} \label{split1} The following foams are invertible as 2-morphisms in $\foam{N}{}^\Sigma$: 
\begin{equation*}
\unzip[1]{}{}
\quad,\quad
\zip[1]{}{}.
\end{equation*}
\end{lem}
\begin{proof}
Decorating the $b=c$ case of the foam relations in equation \eqref{FoamRel3} by red and blue idempotents we get:
\begin{equation}\label{iso1}
\unzipzip[1]{}{}{} \quad = \sum_{\alpha \in P(a,b)} (-1)^{|\hat{\alpha}|} \quad 
\zipthrough[1]{\pi_{\alpha}}{\pi_{\hat{\alpha}}}{}{}
\end{equation}
and
\begin{equation}\label{iso2}
\zipunzip[1]{}{}
\quad =  \sum_{\alpha \in P(a,b)} (-1)^{|\hat{\alpha}|} \quad 
\parallelsheets[1]{\pi_{\alpha}}{\pi_{\hat{\alpha}}}.
\end{equation}

Let $\X$ and $\Y$ be the alphabets assigned to the red and blue foam facets where $\pi_\alpha$ and $\pi_{\hat{\alpha}}$ are placed. 
We check that $ \sum_{\alpha \in P(a,b)} (-1)^{|\hat{\alpha}|} \pi_{\alpha}(\X)\pi_{\hat{\alpha}}(\Y)$ represents a unit in $\idem_A H_a^\Sigma(\X)\otimes \idem_B H_b^\Sigma(\Y)$ by using the criterion in Corollary \ref{idunit}:

\begin{equation*}\sum_{\alpha \in P(a,b)} (-1)^{|\hat{\alpha}|} \pi_{\alpha}(\X)\pi_{\hat{\alpha}}(\Y) |_{\substack{\X\mapsto A \\ \Y \mapsto B}} = \sum_{\alpha \in P(a,b)} (-1)^{|\hat{\alpha}|} \pi_{\alpha}(A)\pi_{\hat{\alpha}}(B)
= \prod_{\lambda\in A}\prod_{\mu \in B} (\mu - \lambda) \neq 0 \in \C.
\end{equation*}
A proof for the second equality can e.g. be found in \cite[p. 65, Example 5]{M}, and the product is non-zero because $A$ and $B$ consist of distinct roots.

Let $\sum_r f_r(\X)g_r(\Y)$ be a representative of $(\sum_{\alpha \in P(a,b)} (-1)^{|\hat{\alpha}|}\pi_\alpha(\X)\pi_{\hat{\alpha}}(\Y))^{-1}$ in $H_a^\Sigma(\X)\otimes H_b^\Sigma(\Y)$, 
then the following are explicit inverses for the decorated unzip and zip foams:

\begin{equation*}
\left ( \; \unzip[1]{}{} \; \right )^{-1}= \sum_r ~\zip[1]{f_r}{g_r}
\quad,\quad
\left ( \; \zip[1]{}{} \; \right )^{-1}=\sum_r ~\unzip[1]{f_r}{g_r} \quad .
\end{equation*}
\end{proof}

\begin{lem} Let $p$ and $q$ be symmetric polynomials in $a$ and $b$ variables respectively. Then the following relations hold:

\begin{equation}\label{dotmigration1}
\zipthrough[1]{p}{}{q}{}  \quad = \quad \zipthrough[1]{}{q}{}{p}\quad, \quad \zipthroughflipcol[1]{p}{}{q}{}\quad = \quad \zipthroughflipcol[1]{}{p}{}{q} \quad ,
\end{equation}

\begin{equation}
\label{iso5}
 \sum_{\alpha \in P(a,b)} (-1)^{|\hat{\alpha}|}~  
 \splitterbubble[1]{\pi_{\alpha}}{\pi_{\hat{\alpha}}}{} \quad = \quad 
\splitter[1] \quad .
\end{equation}

\end{lem}
\begin{proof}
We again use  $\sum_r f_r(\X)g_r(\Y)$, which is a representative of $(\sum_{\alpha \in P(a,b)} (-1)^{|\hat{\alpha}|}\pi_\alpha(\X)\pi_{\hat{\alpha}}(\Y))^{-1}$ in $H_a^\Sigma(\X)\otimes H_b^\Sigma(\Y)$. 
For the first relation in equation \eqref{dotmigration1} we compute:

\begin{equation*}
\zipthrough[1]{p}{}{q}{}  \quad \overset{\eqref{iso1}}{=} \quad  \sum_r ~ \unzipzip[1]{f_r p}{g_r}{q}{} \quad = \quad \sum_r ~ \unzipzip[1]{f_r}{g_r q}{}{p} \quad \overset{\eqref{iso1}}{=} \quad \zipthrough[1]{}{q}{}{p} \;\; .
\end{equation*}
Equation \eqref{iso5} then follows via:
 \begin{align*}
 \sum_{\alpha \in P(a,b)} (-1)^{|\hat{\alpha}|}~  
 \splitterbubble[1]{\pi_{\alpha}}{\pi_{\hat{\alpha}}}{}
  &=  \sum_{\alpha \in P(a,b)} (-1)^{|\hat{\alpha}|}~
 \splitterbubble[1]{}{\pi_{\hat{\alpha}}}{\pi_{\alpha}} \\
& \overset{\eqref{iso1}}{=}
 \fatsplitter[1] \quad = \quad 
\splitter[1] \quad .
\end{align*}

For the second relation in \eqref{dotmigration1} we now have:

\begin{align*}
\zipthroughflipcol[1]{p}{}{q}{}
& = \sum_{\alpha \in P(a,b)} (-1)^{|\hat{\alpha}|}~  
\zipthroughspecA[1]{\pi_{\hat{\alpha}}}{\pi_{\alpha}}{}{}{p}{q}
~ = \sum_{\alpha \in P(a,b)} (-1)^{a b + |\hat{\alpha}|}~  
\zipthroughspecB[1]{\pi_{\hat{\alpha}}}{\pi_{\alpha}}{}{}{q}{p} \\
&= \sum_{\alpha \in P(a,b)} (-1)^{a b + |\hat{\alpha}|}~  
\zipthroughspecB[1]{\pi_{\hat{\alpha}}}{\pi_{\alpha}}{q}{p}{}{}
~ = \sum_{\alpha \in P(a,b)} (-1)^{|\hat{\alpha}|}~  
\zipthroughspecA[1]{\pi_{\hat{\alpha}}}{\pi_{\alpha}}{q}{p}{}{} \\
& = ~
 \zipthroughflipcol[1]{}{p}{}{q} \quad .
\end{align*}
\end{proof}

\begin{lem} \label{split2} The following foams are invertible as 2-morphisms in $\foam{N}{}^\Sigma$: 
\begin{equation*}
\blisterclose[1]{}{}
\quad,\quad
\blisteropen[1]{}{}.
\end{equation*}
\end{lem}
\begin{proof}
Given the relations\footnote{In relation \eqref{iso4} the green shading is meant to indicate a decoration by the mixed idempotent $\idem_{A\uplus B}$} 
\begin{equation}\label{iso3}
\tube[1]{}{}
\quad = \sum_{\alpha \in P(a,b)} (-1)^{|\hat{\alpha}|} \quad 
\blistercloseopen[1]{\pi_{\alpha}}{\pi_{\hat{\alpha}}}{},
\end{equation}
\begin{equation}
\label{iso4}
\flatsheet[1]
\quad = \quad \sum_{\alpha \in P(a,b)}(-1)^{|\hat{\alpha}|}~
\greenbubble[1]{\pi_{\alpha}}{\pi_{\hat{\alpha}}}
\end{equation}
and the decoration migration relations \eqref{dotmigration1}, it follows immediately that
\begin{equation*}
\sum_{\alpha \in P(a,b)} (-1)^{|\hat{\alpha}|}~\blisteropen[1]{\pi_{\alpha}}{\pi_{\hat{\alpha}}}
\quad\text{and}\quad
\sum_{\alpha \in P(a,b)} (-1)^{|\hat{\alpha}|}~\blisterclose[1]{\pi_{\alpha}}{\pi_{\hat{\alpha}}}
\end{equation*}
are inverse to the digon removal and creation foams, respectively.

Equation \eqref{iso3} is just an idempotent decorated version of relation \eqref{FoamRel2}. 
Equation \eqref{iso4} is a stronger, more local version of \eqref{iso5}, but we cannot use the same trick to deduce it. 
To de-clutter the pictures, we compute this relation in $\BNC_{a+b}^\Sigma$; the result can then be transferred using the foamation functor $\Phi_\Sigma$. 
Alternatively, one can interpret the following nilHecke pictures as $2d$-slices through the corresponding foams.

We begin by using equation \eqref{eqn-explode} to explode a thick edge into thin edges, and then combine this relation with Corollary \ref{adm2} 
to slide the decoration by the multiset onto the thin edges. In the simplest case where the multiset contains only one root $\nu$ we have:
\begin{equation*} 
 \xy
 (0,0)*{\includegraphics[scale=0.5]{figs/tlong-up.eps}};
 (-3,-12)*{a};
  (0,0)*{ \bigb{\{\nu,\dots, \nu\}}};
  \endxy
   \qquad = \quad
  \xy
 (0,0)*{\includegraphics[scale=0.5]{figs/texplode.eps}};
 (-3,-12)*{a};
 (-14,1)*{ \bigb{\nu}};
 (-4.5,1)*{ \bigb{\nu}};
 (4.5,1)*{ \bigb{\nu}};
 (14,1)*{ \bigb{\nu}};
 (-13,-5)*{\bullet}+(-4.5,1)*{\scs a-1};
 (-3.5,-5)*{\bullet}+(-4.5,1)*{\scs a-2};
 (3.5,-5)*{\bullet}+(-2.5,1)*{\scs 1};
 (0,-2)*{\cdots};
  \endxy \quad  .
\end{equation*}
Now suppose that $A=\{\lambda,\dots,\lambda\}$ and $B=\{\mu,\dots, \mu\}$ with $\lambda\neq \mu$, then similarly we have:

\begin{equation*}
\xy
 (0,0)*{\includegraphics[scale=0.6]{figs/tlong-up.eps}};
 (-6,-17)*{a+b};
  (0,0)*{ \bigb{A \uplus B}};
  \endxy
   \qquad = \quad \sum_{\substack{ (\nu_{a+b},\dots, \nu_1) \\ \text{is a re-ordering of }\\ (\lambda,\dots, \lambda,\mu,\dots, \mu)}}  
 \xy
 (0,0)*{\includegraphics[scale=0.7]{figs/texplode.eps}};
 (-6,-17)*{a+b};
 (-19,1)*{ \bigb{\nu_{a+b}}};
 (-5.5,1)*{ \bigb{\nu_{b+1}}};
 (5.5,1)*{ \bigb{\nu_{b}}};
 (19,1)*{ \bigb{\nu_1}};
 (-19.5,-5)*{\bullet}+(-5.5,1)*{\scs a+b-1};
 (-5.6,-5)*{\bullet}+(-2.5,1)*{\scs b};
 (5.6,-5)*{\bullet}+(4.5,1)*{\scs b-1};
 (-10,-7)*{\cdots};
  (10,-7)*{\cdots};
  \endxy \quad .
\end{equation*}
Next we reorder the decorations on the strands (at the expense of signs) so that all $\lambda$ idempotents lie on the left, all $\mu$ idempotents on the right, and in both groups of strands the number of additional dots decreases from left to right:
\begin{equation*}
= \sum_{\substack{ l_1 >\cdots > l_a }}
 \pm 
  \xy
 (0,0)*{\includegraphics[scale=0.7]{figs/texplode.eps}};
 (-6,-17)*{a+b};
 (-19,1)*{ \bigb{\lambda}};
 (-5.5,1)*{ \bigb{\lambda}};
 (5.5,1)*{ \bigb{\mu}};
 (19,1)*{ \bigb{\mu}};
 (-19.5,-5)*{\bullet}+(-3.5,1)*{\scs l_1};
 (-5.6,-5)*{\bullet}+(-3.5,1)*{\scs l_a};
 (5.6,-5)*{\bullet}+(3.5,1)*{\scs r_1};
  (19.5,-5)*{\bullet}+(3.5,1)*{\scs r_b};
   (-10,-7)*{\cdots};
  (10,-7)*{\cdots};
  \endxy .
\end{equation*}
Here the sum is taken over all strictly decreasing sequences $a+b-1 \geq l_1>\dots > l_a \geq 0$ and $r_1,\dots, r_b$ are the remaining $b$ numbers between $0$ and $a+b-1$ in decreasing order. 
Clearly the set of such sequences $(l_1,\dots,l_a)$ is in bijection with partitions $(l_1-(a-1), l_2-(a-2), \dots, l_a)$ whose Young diagrams fit into a $a\times b$ box. 
If $(l_1,\dots, l_a)$ corresponds to a partition $\alpha \in P(a,b)$, then it is easy to check that $(r_1,\dots,r_b)$ corresponds to $\hat{\alpha}\in P(b,a)$ and the sign introduced by reordering decorations on strands is $(-1)^{|\hat{\alpha}|}$. 
Finally, we use equation \eqref{eqn-explode} to express the $a$ strands on the left and the $b$ strands on the right in terms of strands of thickness $a$ and $b$ respectively. 
This expresses the decorations $(l_1,\dots, l_a)$ and $(r_1,\dots, r_b)$ on the thin strands as Schur polynomials $\pi_\alpha$ and $\pi_{\hat{\alpha}}$, 
and using Corollary \ref{adm2} we can slide the idempotents onto the thick strands:

\begin{equation*}
= \quad \sum_{\alpha \in P(a,b)} (-1)^{|\hat{\alpha}|}\quad
  \xy
 (0,0)*{\includegraphics[scale=0.6]{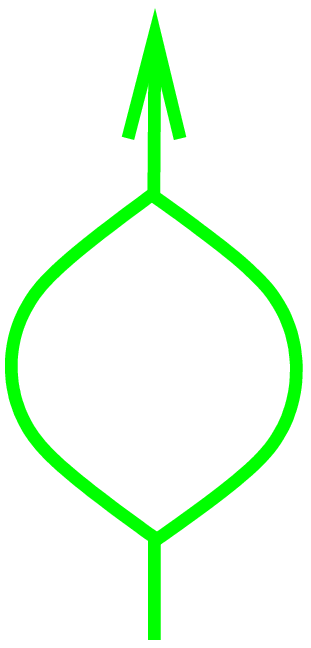}};
 (-6,-17)*{a+b};
  (-7,2)*{\bigb{A}};
  (7,2)*{\bigb{B}};
 (-7,-7)*{\bigb{\pi_{\alpha}}};
 (7,-7)*{\bigb{\pi_{\hat{\alpha}}}};
  \endxy .
\end{equation*}
This gives the thick calculus version of equation \eqref{iso4} for this choice of $A$ and $B$.

The case of general $A$ and $B$ is very similar. The main difference is that there are more possible re-orderings of the decoration by roots on thin strands. 
However, if we interpret a nilHecke picture decorated by idempotents $\lambda$ and $\mu$ as sum over all possible ways of replacing $\lambda$'s by elements of $A$ and $\mu$'s by elements of $B$, 
then the proof of the special case immediately carries over to the general setting.
\end{proof}

\subsection{Karoubi envelope technology} 
Let $W$ be a web, i.e. a 1-morphism in $\foam{N}{}^\Sigma$, then the foam versions of Corollary \ref{adm} and Proposition \ref{adm2} show that the identity 2-morphism $\id_W$ decomposes into a sum of idempotent foams -- 
one for each coloring of the edges of $W$ by multisubsets of roots that is compatible at vertices. We now proceed to a 2-category $\hat{\foam{N}{}^\Sigma}$ in which these idempotents split.

\begin{defi}
Let $\mathrm{Kar}(\Foam{N}^\Sigma)$ denote the 2-category obtained by passing to the Karoubi envelope in each $\Hom$-category of $\Foam{N}^{\Sigma}$. 
We define $\hat{\foam{N}{}^\Sigma}$ to be a certain full 2-subcategory of $\mathrm{Kar}(\Foam{N}^\Sigma)$ that contains as $1$-morphisms all the pairs $(W,F_W)$ where $W$ is a web in $\Foam{N}^\Sigma$ and $F_W$ is a decorated identity foam on $W$ in $\Foam{N}^\Sigma$ such that each $a$-labeled facet is decorated by an idempotent $\idem_A$ corresponding to an $a$-element multisubset $A \subset \Sigma$. More precisely, $\hat{\foam{N}{}^\Sigma}$ has the same objects as $\Foam{N}^\Sigma$ and has $\Hom$-categories given by the full subcategories of the corresponding $\Hom$-categories of $\mathrm{Kar}(\Foam{N}^\Sigma)$ that contain all formal direct sums of pairs $(W,F_W)$.
\end{defi}
Note that, in particular, $\Foam{N}^\Sigma$ embeds as a full 2-subcategory of $\hat{\foam{N}{}^\Sigma}$, since the identity foam over any web can be expressed as the sum over all possible colorings of its facets.
Practically speaking, $\hat{\foam{N}{}^\Sigma}$ can be viewed as the 2-category in which:
\begin{itemize}
\item objects are sequences $\mathbf{a} = (a_1,\ldots, a_m)$ for $m \geq 0$ as in $\Foam{N}^\Sigma$, 
\item 1-morphisms are formal direct sums of webs where in addition to a labeling, each $a$-labeled edge is colored by an 
idempotent $\idem_A$ corresponding to an $a$-element multisubset $A \subset \Sigma$, and
\item 2-morphisms are matrices of linear combinations of foams as in $\Foam{N}^\Sigma$, but with each facet incident upon a web edge decorated by the idempotent coloring the edge.
\end{itemize}

As in the case of thick calculus diagrams and foams, we call a web {\it admissibly colored} if at each trivalent vertex the union of the multisets coloring two of the edges equals the third. 
Since non-admissibly colored foams are zero, it follows that a non-admissibly colored web is isomorphic to the ``zero web'' (i.e. the zero object in the relevant $\Hom$-category).

We now point out that in $\hat{\foam{N}{}^\Sigma}$  there are three\footnote{This is due to the fact that $\Foam{N}$ is secretly a 3-category.} ways of composing morphisms, 
and establish our notation for them.
\begin{itemize}
\item Sequences, webs, and foams can be placed side by side (i.e. on objects this is concatenation of sequences). We denote this operation by $\sqcup$.
\item Webs and foams can be composed in the 1-morphism direction, i.e. glued horizontally along their left and right boundaries, and we denote this by $\otimes$.
\item Foams can be composed in the 2-morphism direction by gluing vertically, and we write $\circ$ for this operation. 
\end{itemize}

We will also utilize two notions of ``equivalence'' for colored webs in $\hat{\foam{N}{}^\Sigma}$.
A $1$-morphism $W\colon \mathbf{o}_1 \to \mathbf{o}_2$ in $\hat{\foam{N}{}^\Sigma}$ is \emph{isomorphic} to a $1$-morphism $V\colon \mathbf{o}_1 \to \mathbf{o}_2$ 
if there exist 2-morphisms $F_1\colon W \to V$ and $F_2\colon V \to W$ in $\hat{\foam{N}{}^\Sigma}$ such that:
\begin{equation*}
F_2\circ F_1 = \id_W \text{ and } F_1\circ F_2 = \id_V .
\end{equation*} 
In this case we write $V \cong W$. 
Next, a $1$-morphism $W\colon \mathbf{o}_1 \to \mathbf{o}_2$ in $\hat{\foam{N}{}^\Sigma}$ is \emph{weakly equivalent} to a $1$-morphism $V\colon \mathbf{u}_1 \to \mathbf{u}_2$ if there exist $1$-morphisms 
$L \colon \mathbf{o}_2 \to \mathbf{u}_2$, $L^{-1} \colon \mathbf{u}_2 \to \mathbf{o}_2$, $R\colon \mathbf{u}_1 \to \mathbf{o}_1$ and $R^{-1}\colon \mathbf{o}_1 \to \mathbf{u}_1$ such that 
\begin{equation*}
L\otimes W \otimes R \cong V,
\end{equation*}
\begin{equation*}
L^{-1}\otimes L \cong 1_{\mathbf{o}_2}, \;L\otimes L^{-1} \cong 1_{\mathbf{u}_2},\; R\otimes R^{-1} \cong 1_{\mathbf{o}_1},\; R^{-1}\otimes R \cong 1_{\mathbf{u}_1}. 
\end{equation*}

We now aim to use these notions of equivalence to ``split'' the foam 2-category $\hat{\foam{N}{}^\Sigma}$ into pieces in which webs and foams are colored by multisubsets of $\Sigma$ containing only 
one root $\lambda \in \Sigma$. Although we do not prove a full decomposition theorem (see Remark \ref{rem-conj} below), we'll see in Section \ref{section-tensor} that the splitting results 
obtained here suffice to decompose the link invariant as in Theorem \ref{mainthm}.

Let $F$ be a foam with an admissible coloring of facets by multisubsets of $\Sigma$ and let $\lambda\in \Sigma$ be a root. 
We want to define the foam $F_\lambda$ that results from forgetting everything in $F$ that is not colored by $\lambda$. 
More precisely, consider the underlying CW-complex of $F$; in it we erase all 2-cells that are colored with multisubsets not containing $\lambda$ and smoothen out all seams that have become obsolete. 
We define a foam structure on the resulting CW-complex by setting the label of each remaining 2-cell to be the (positive) multiplicity of $\lambda$ in the corresponding color on $F$. 
This is again a foam by admissibility of the original coloring. 
Finally we decorate each facet with the idempotent of the multisubset containing only $\lambda$'s.
\begin{defi}
The \emph{$\lambda$-component} of an admissibly colored foam $F$, denoted by $F_\lambda$, is the foam in $\hat{\foam{N}{}^\Sigma}$ constructed via the procedure just described.
\end{defi}
In the following, we'll use the shorthand $\bigsqcup_\lambda F_\lambda := F_{\lambda_l}\sqcup \cdots \sqcup F_{\lambda_1}$.

\begin{exa}\label{exa-foamsplit} If $\lambda_1 \neq \lambda_2$ are two roots in $\Sigma$ and colors red, blue and green indicate decorations with idempotents corresponding to multisets 
$\{\lambda_1^a \}$, $\{\lambda_2^b \}$ and $\{\lambda_1^a, \lambda_2^b\}$ respectively, then we have, for example:
\begin{equation*}
F = \zipunzip[1]{}{}\quad , \quad F_ {\lambda_2} \sqcup F_{\lambda_1} =
\parallelsheets[1]{}{}.
\end{equation*}
\end{exa}
\newpage 
\begin{defi} Let $W$ be a colored web in $\hat{\foam{N}{}^\Sigma}$.
\begin{itemize}
\item The \emph{$\lambda$-component $W_\lambda$ of $W$} is the (co-)domain of $(\id_W)_\lambda$, the $\lambda$-component of the identity foam on $W$. 
As for foams, we define the shorthand: $\bigsqcup_\lambda W_\lambda := W_{\lambda_l}\sqcup \cdots \sqcup W_{\lambda_1}$.
\item $W$ is called \emph{split} if $W= \bigsqcup_\lambda W_\lambda$. More generally, 
for any colored web $W'$, the split web $\bigsqcup_\lambda W'_\lambda$ is called the \emph{split web associated to} $W'$.
\end{itemize}
\end{defi}

\begin{exa}\label{exa-websplit} With coloring conventions as in Example \ref{exa-foamsplit} we have, for example:
\begin{equation*}
W = \xy
(0,0)*{
\begin{tikzpicture} [scale=.5,fill opacity=0.2]
	\draw[blue, very thick, directed=.55] (-2,3) to [out=180, in=45] (-3,2.5);
	\draw[red, very thick, directed=.55]  (-1.5,2) to [out=180, in=315] (-3,2.5);
	\draw[green, very thick, directed=.55] (-3,2.5) to (-4.5,2.5);
	\draw[blue, very thick, directed=.55] (-4.5,2.5) to [out=135,in=0] (-6,3);
	\draw[red, very thick, directed=.55] (-4.5,2.5) to [out=225,in=0] (-5.5,2);	
\end{tikzpicture}
};
\endxy  
\quad , \quad W_ {\lambda_2} \sqcup W_{\lambda_1} =
\xy
(0,0)*{
\begin{tikzpicture} [scale=.5,fill opacity=0.2]
	\draw[blue, very thick, directed=.55] (-2,3) to (-6,3);
	\draw[red, very thick, directed=.55] (-1.5,2) to  (-5.5,2);	
\end{tikzpicture}
};
\endxy \quad .
\end{equation*}
\end{exa}

Next, let $\mathbf{o}=(a_1,\dots, a_m)$ be an object in $\hat{\foam{N}{}^\Sigma}$ and suppose that for every entry $a_i$ of $\mathbf{o}$ we are given an $a_i$-element multisubset $A_{i}=\{\lambda_1^{a_{i,1}},\dots, \lambda_l^{a_{i,l}}\}\subset \Sigma$ -- 
we call such a collection $\mathbf{A}=(A_1,\dots, A_m)$ an \emph{incidence condition} for $\mathbf{o}$. We then consider the identity web on $\mathbf{o}$ with strands colored by multisubsets $A_i$, 
and use the following notation for the (co-)domain of the associated split web: 
\begin{equation}\label{eq-ICob}
\bigsqcup_\lambda \mathbf{o}_\lambda:= (a_{1,l},\dots, a_{m,l}, \dots, a_{1,1}, \dots, a_{m,1}).
\end{equation}

\begin{defi} \label{LRdef} 
Let $L \colon \mathbf{o} \to \bigsqcup_\lambda \mathbf{o}_\lambda$ be the combinatorially simplest web from $\mathbf{o}$ to $\bigsqcup_\lambda \mathbf{o}_\lambda$ that is colored with the multiset $A_i$ on the strand starting at the entry $a_i$ of $\mathbf{o}$ 
and colored with the multiset $\{\lambda_j^{a_{i,j}}\}$ on the strand terminating at the entry $a_{i,j}$ of $\bigsqcup_\lambda \mathbf{o}_\lambda$. 
Analogously we define $R\colon \bigsqcup_\lambda \mathbf{o}_\lambda \to \mathbf{o}$ to be the combinatorially simplest web from $\bigsqcup_\lambda \mathbf{o}_\lambda$ to $\mathbf{o}$ that is colored with 
$\{\lambda_j^{a_{i,j}}\}$ on the strand starting at the entry $a_{i,j}$ of $\bigsqcup_\lambda \mathbf{o}_\lambda$ and colored with $A_i$ on the strand terminating at the entry $a_i$ of $\mathbf{o}$.
 \end{defi}

Explicitly, $L$ is given as a composition $L:=L_{l-1}\otimes \cdots \otimes L_1$ with one component $L_j$ for each root $\lambda_j$, except the last one. 
Each $L_j$ itself can be decomposed as $L_j=L_{1,j}\otimes \cdots \otimes L_{m,j} $, where $L_{m,1}$ splits off the $\lambda_1$-component from the strand coming out of $a_m$ and continues it below the remainder of the $a_m$ strand. 
$L_{m-1,1}$ splits off the $\lambda_1$-component from the strand coming out of $a_{m-1}$, merges it with the remainder of the $a_m$ strand, which contains no $\lambda_1$'s any more, and splits it off on the other side. 
In general $L_{i,j}$ splits the $\lambda_j$-component off the remainder of the $a_i$-strand, and passes it through the remainders of all $a_k$-strands with $k>i$, which contain no $\lambda_j$'s any more. 
The composite $L_j$ thus is the combinatorially simplest web that splits the $\lambda_j$-components off all $a_k$ strands and continues them as a bundle of parallel strands below the $a_i$ remainder strands and above the bundles of 
$\lambda_{j'}$-colored parallel strands for $j'<j$ that have been split off by $L_{j'}$. 
It is not hard to see that the composition $L$ is, up to planar isotopy, the combinatorially simplest web from $\mathbf{o}$ to $ \bigsqcup_\lambda \mathbf{o}_\lambda$ with the prescribed boundary colorings.
The colored web $R$ can be obtained similarly, or simply by reflecting $L$ horizontally.

Given a colored web $W\colon \mathbf{o}_1 \to \mathbf{o}_2$, we will mostly be interested in the webs $L\colon \mathbf{o}_2 \to \bigsqcup_\lambda \mathbf{o}_{2,\lambda}$ and 
$R\colon \bigsqcup_\lambda \mathbf{o}_{1,\lambda} \to \mathbf{o}_1$ constructed from the incidence conditions for $\mathbf{o}_2$ and $\mathbf{o}_1$ determined by the coloring of left and right boundary edges of the colored web $W$. 
In particular, we can then consider the colored web $L\otimes W\otimes R$.

\begin{exa}\label{exa-LR} In the case of the identity web $1_{\mathbf{o}}$ on $\mathbf{o}=(a+b)$, which is colored by the multisubset $\{\lambda_1^a, \lambda_2^b\} \subset \Sigma$, 
and using the coloring conventions from Example \ref{exa-foamsplit}, we have the following prototypical example:
\begin{equation*}
L = \xy
(0,0)*{
\begin{tikzpicture} [scale=.5,fill opacity=0.2]
	\draw[green, very thick, directed=.55] (-3,2.5) to (-4.5,2.5);
	\draw[blue, very thick, directed=.55] (-4.5,2.5) to [out=135,in=0] (-6,3);
	\draw[red, very thick, directed=.55] (-4.5,2.5) to [out=225,in=0] (-5.5,2);	
\end{tikzpicture}
};
\endxy  
\quad , \quad R =
\xy
(0,0)*{
\begin{tikzpicture} [scale=.5,fill opacity=0.2]
	\draw[blue, very thick, directed=.55] (-2,3) to [out=180, in=45] (-3,2.5);
	\draw[red, very thick, directed=.55]  (-1.5,2) to [out=180, in=315] (-3,2.5);
	\draw[green, very thick, directed=.55] (-3,2.5) to (-4.5,2.5);
\end{tikzpicture}
};
\endxy \quad .
\end{equation*}
\end{exa}

\begin{lem} \label{lem-LR} Suppose $\mathbf{o}$ is an object in $\hat{\foam{N}{}^\Sigma}$ and fix an incidence condition for $\mathbf{o}$. 
Let $L$ and $R$ be the corresponding webs constructed in Definition \ref{LRdef}. Then we have 
\begin{equation*}
R\otimes L \cong 1_{\mathbf{o}} \quad \text{ and } \quad L\otimes R \cong 1_{\bigsqcup_\lambda \mathbf{o}_\lambda}.
\end{equation*}
\end{lem}
\begin{proof}
$L$ and $R$ are both compositions of mergers (and splitters) whose two incoming (outgoing) strands are colored with disjoint multisubsets. Moreover, splitters and mergers in $R$ are paired up with mergers and splitters in $L$ -- 
in reverse order. Repeated application of Lemmata \ref{split1} and \ref{split2} allows the construction of foams giving the isomorphism (see also equation \eqref{webiso1} below).
\end{proof}

\begin{defi} Let $W$ be a colored web in $\hat{\foam{N}{}^\Sigma}$.
$W$ is called \emph{boundary-split}, if it is of the form $L\otimes W'\otimes R$ for some colored web $W'\colon \mathbf{o}_1 \to \mathbf{o}_2$ in $\hat{\foam{N}{}^\Sigma}$ and for $L\colon \mathbf{o}_2\to \bigsqcup_\lambda \mathbf{o}_{2,\lambda}$ 
and $R\colon \bigsqcup_\lambda \mathbf{o}_{1,\lambda}\to \mathbf{o}_1$ as in Definition \ref{LRdef}. $L\otimes W'\otimes R$ is then called the \emph{boundary-split web associated to} $W'$.
\end{defi}
\begin{rem} Lemma \ref{lem-LR} shows that every web $W'$ in $\hat{\foam{N}{}^\Sigma}$ is weakly equivalent to its associated boundary-split web $L\otimes W'\otimes R$.
\end{rem}

Our goal is now to show that a boundary-split web $W$ is isomorphic to its associated split web $\bigsqcup_\lambda W_\lambda$. 
Unless stated otherwise, we use red and blue colors to denote colorings of web edges with disjoint multisubsets of $\Sigma$. 
Green denotes mixed or arbitrary colorings.

\begin{lem}
\label{webisolem}
The following isomorphisms hold in $\hat{\foam{N}{}^\Sigma}$:
\begin{equation}
\label{webiso1}
\eqnA 
\end{equation}
\begin{equation}
\label{webiso2}
\eqnB 
\end{equation}
\begin{equation}
\label{webiso3}
\eqnC 
\end{equation}
The reflections of these relations across the horizontal axis in the plane also hold.
\end{lem}
\begin{proof} Equation \eqref{webiso1} follows from Lemmata \ref{split1} and \ref{split2}. For \eqref{webiso2} we have:

\begin{eqnarray*}
\eqnD \\
\eqnE
\end{eqnarray*}
where we have used (the splitter version of) relation \eqref{FoamRel1} three times.
\end{proof}

\begin{conv}
We define the following shorthand for ``crossings'' of web edges colored by disjoint multisubsets:
\begin{equation}
\label{4valdef}
\eqnF.
\end{equation}
\end{conv}
Using this, equations \eqref{webiso2} and \eqref{webiso3} take the form:
\begin{equation*}
\eqnG ,
\end{equation*}
\begin{equation*}
\eqnH .
\end{equation*}

\begin{defi}
We call a web \emph{semi-split} if it is boundary-split and each edge is either colored by a multisubset containing a single root or a multisubset containing exactly two distinct roots, 
in which case the edge (green) is required to have a neighborhood as on the left-hand side of \eqref{4valdef}. 
\end{defi}

\begin{lem}
\label{websemisplitlem}
Every boundary-split web $L\otimes W \otimes R$ in $\hat{\foam{N}{}^\Sigma}$ is isomorphic to a semi-split web $W'$. 
\end{lem}
\begin{proof} 
We inductively split off roots, starting with $\lambda_1$. 
For this we draw in red edges colored with multisubsets of the single root $\lambda_1$, in blue edges colored with multisubsets not containing $\lambda_1$, 
and in green edges colored by mixed multisubsets. 
By \eqref{webiso1} we can open a red-blue digon in every green edge of  $L\otimes W \otimes R$ and get an isomorphic web. 
Next we replace all vertices that are adjacent to at least two green edges with isomorphic webs that only contain green edges of type \eqref{4valdef}, e.g. for an all-green merger web:
\begin{equation*}
\eqnI
\end{equation*}
\begin{equation*}
\cong \eqnJ .
\end{equation*}
The case of an all-green splitter web is completely analogous, and vertices with only two adjacent green edges are even easier to split. 
The local replacements of green vertices as above patch together to give an isomorphism to a web in which green edges are flanked by red-blue mergers and splitters in 
the crossing configuration from equation \eqref{4valdef}.

For the induction step $i-1\mapsto i$ we use the following coloring on edges:\\
\begin{center}
\begin{tabular}{c | c}
color & roots occuring in multisubset (with some multiplicity)\\ \hline
red & only $\lambda_i$\\
orange & exactly one $\lambda_k$ with $k<i$\\
magenta & exactly two distinct $\lambda_k$ with $k\leq i$\\
blue & some roots $\lambda_k$ with $k>i$\\
green & $\lambda_i$ and at least one $\lambda_k$ with $k>i$\\
cyan & exactly one $\lambda_k$ with $k<i$ and at least one with $k>i$\\
black & $\lambda_i$, exactly one $\lambda_k$ with $k<i$ and at least one with $k>i$\\
\end{tabular}
\end{center}
We can assume that only these colorings are present. Moreover, orange strands can interact with \{other orange, red, blue, green\} strands only in crossing configurations around \{magenta, magenta, cyan, black\} edges respectively. Furthermore, such crossing configurations are the only occurrences of magenta, cyan and black edges. 

The goal for the induction step is to split red edges off green and black edges. As before we introduce red-blue digons in every green edge and locally replace green vertices. 
Every remaining green edge is in red-blue crossing configuration or bounds red-blue on one side and orange-black on the other side (and every black bounds orange-green on both sides). 
We get rid of all black edges by splitting off their red component:
\begin{equation*}
\eqnK
\end{equation*}
\begin{equation*}
\eqnL
\end{equation*}
Note that now red and orange strands interact with each other and with strands that contain higher index roots (blue) only in crossing configurations (around magenta, green and cyan edges), as required in the induction step. 
For the next step old \{blue, green, cyan\} edges become \{green, black, black\} or \{blue, cyan, cyan\} depending on whether they contain $\lambda_{i+1}$ or not. Orange stays orange, red becomes orange, and magenta stays magenta. 
This colored web satisfies the induction hypothesis for the next step. After repeating this process for each root, it terminates in a semi-split web $W'$. 
\end{proof}

\begin{prop}
\label{webisoprop} Every boundary-split web $L\otimes W \otimes R$ is isomorphic to its associated split web $\bigsqcup_\lambda W_\lambda$.
\end{prop}

\begin{proof} The proof proceeds in two steps, first we use Lemma \ref{websemisplitlem} to find an isomorphism from $L\otimes W \otimes R$ to a semi-split web $W'$. Clearly $W$, $L\otimes W \otimes R$ and $W'$ have equal associated split webs. It remains to completely separate the $\lambda_i$-components in $W'$. 
Again we proceed by induction and start by peeling off the $\lambda_1$-component $W'_{\lambda_1}$. 
For this, consider a web-isotopy $t\mapsto W'_{\lambda_1}(t)$ for $t\in [0,1]$, i.e. an ambient isotopy of $W'_{\lambda_1}$ in the plane which preserves the left-directedness of web edges and which moves $W'_{\lambda_1}$ off the rest  
$W'\setminus W'_{\lambda_1}$. 
If we superimpose $W'_{\lambda_1}(t)$ and $W'\setminus W'_{\lambda_1}$ we get a homotopy $t\mapsto W'(t)$ of graphs of valence $\leq 6$. If the original web-isotopy is generic, the graphs $W'(t)$ actually are of valence $\leq 5$ and there are only finitely many $t$ for which the valence is $5$ -- these correspond to the moves in \eqref{webiso2} and \eqref{webiso3}. 
Furthermore, $4$-valent vertices in $W(t)$ should be understood as composition of a merge- and a split- $3$-valent vertex, either in crossing configuration as in Definition \ref{4valdef}, or splittable as in \eqref{webiso1}. 
Thus, $t\mapsto W'(t)$ is a web-isotopy except in finitely many points $t$ where the number and valence of vertices changes locally. 
It is not hard to see that the possible local changes are exactly the ones from Lemma \ref{webisolem} and hence can be realized by isomorphism foams. 
A composition of the appropriate local isomorphism foams, thus, splits off $W'_{\lambda_1}$ from $W'$. 
One then proceeds to split off, in exactly the same way, $W'_{\lambda_2}$ and so forth up to $W'_{\lambda_{l-1}}$. The result then follows since $W'_{\lambda_i}=W_{\lambda_i}$ for $1\leq i\leq l$.
\end{proof}

\begin{rem} Proposition \ref{webisoprop} together with Lemma \ref{lem-LR} show that every web $W$ in $\hat{\foam{N}{}^\Sigma}$ is weakly equivalent to its associated split web $\bigsqcup_\lambda W_\lambda$.
\end{rem}

\subsection{A web splitting functor}
\label{2hom}

We now extend Proposition \ref{webisoprop} to the 2-categorical level. Ideally, we'd like a 2-endofunctor of $\hat{\foam{N}{}^\Sigma}$ which fully splits foams into 
pieces carrying colorings of only one root, but due to coherence issues, we are currently unable to deduce such a result. 
Instead, we define a family of functors between $\Hom$-categories in $\hat{\foam{N}{}^\Sigma}$ which will suffice to split the complex assigned to a tangle.

We begin by fixing, for each colored web, an isomorphism between its associated boundary split and split webs. 
Precisely, let $W\colon \mathbf{o}_1 \to \mathbf{o}_2$ be a colored web in $\hat{\foam{N}{}^\Sigma}$ and suppose that $L\colon \mathbf{o}_2 \to \bigsqcup_\lambda \mathbf{o}_{2,\lambda}$ and 
$R\colon \bigsqcup_\lambda \mathbf{o}_{1,\lambda} \to \mathbf{o}_1$ are the webs given in Definition \ref{LRdef}. 
Proposition \ref{webisoprop} guarantees that there is an isomorphism $T_W\colon L\otimes W \otimes R \to \bigsqcup_\lambda W_\lambda$, so fix one and denote its inverse by $B_W$. 
We have some freedom in choosing $T_W$, and in Section \ref{section-tensor} we will specify a convenient choice for webs that arise as resolutions of tangle diagrams. 

For the next definition,
suppose $F\colon W_1 \to W_2$ is a foam between colored webs $W_1,W_2\colon \mathbf{o}_1 \to \mathbf{o}_2$ in $\hat{\foam{N}{}^\Sigma}$ with identical incident conditions on the boundary sequences $\mathbf{o}_1$ and $\mathbf{o}_2$ respectively. 
Further, consider the webs $L$ and $R$ and the isomorphism foams $T_{W_2}$ and $B_{W_1}$ described above.

\begin{defi}
\label{def-phi}
Let $\phi:= \phi_2\circ \phi_1$ for
\begin{align*}
\phi_1\colon& \Hom(W_1,W_2) \to \Hom(L\otimes W_1 \otimes R, L\otimes W_2 \otimes R),  \quad F\mapsto \id_L \otimes F \otimes \id_R, \\
\phi_2\colon& \Hom(L\otimes W_1 \otimes R, L\otimes W_2 \otimes R) \to \Hom(\bigsqcup_\lambda W_{1, \lambda}, \bigsqcup_\lambda W_{2,\lambda}),  \quad F\mapsto T_{W_2} \circ F \circ B_{W_1}.
\end{align*}
\end{defi}

\begin{prop}\label{2homiso} Fix objects $\mathbf{o}_1,\mathbf{o}_2 \in \hat{\foam{N}{}^\Sigma}$. Then the maps 
\begin{equation*}\phi \colon \Hom(W_1,W_2) \to \Hom(\bigsqcup_\lambda W_{1, \lambda}, \bigsqcup_\lambda W_{2,\lambda})
\end{equation*}
for colored webs $W_1,W_2\colon \mathbf{o}_1 \to \mathbf{o}_2$ with identical incident conditions on $\mathbf{o}_1$ and $\mathbf{o}_2$ respectively, are vector space isomorphisms that respect composition $\circ$ of foams.
\end{prop}
\begin{proof}
It is clear that the maps $\phi_1$ respect composition of foams and for $\phi_2$ it follows from the definition of $B_W$ as the inverse of $T_W$.

Next, note that $\phi_2$ is clearly a vector space isomorphism, since it is pre- and post-composition with isomorphism foams. 
To see that $\phi_1$ is as well, let $L^{-1}$ and $R^{-1}$ be the webs obtained by reflecting $L$ and $R$ horizontally. 
We have isomorphism foams $\phi_L \colon L^{-1}\otimes L \to 1_{\mathbf{o}_2}$ and $\phi_R\colon R\otimes R^{-1} \to 1_{\mathbf{o}_1}$, and an inverse for $\phi_1$ is then given by $\psi\colon G\mapsto (\phi_L\otimes \id_{W_2} \otimes \phi_R)\circ(\id_{L^{-1}}\otimes G \otimes \id_{R^{-1}})\circ (\phi^{-1}_L \otimes \id_{W_1} \otimes \phi^{-1}_R)$. 
\end{proof}

Finally, suppose that $\mathbf{A}$ and $\mathbf{B}$ are incidence conditions for objects $\mathbf{o}_1$ and $\mathbf{o}_2$ in $\hat{\foam{N}{}^\Sigma}$, respectively. 
By expressing each facet incident upon a left or right boundary as a sum over colorings, we see that the $\Hom$-categories in $\hat{\foam{N}{}^\Sigma}$ split into direct sums: 
\begin{equation}
\Hom(\mathbf{o}_1,\mathbf{o}_2) \cong \bigoplus_{\mathbf{A}, \mathbf{B}}\Homcat^{\mathbf{A}\to\mathbf{B}}(\mathbf{o}_1,\mathbf{o}_2)
\end{equation}
where the sum is over all incidence conditions $\mathbf{A}$ and $\mathbf{B}$, 
and $\Hom^{\mathbf{A}\to\mathbf{B}}(\mathbf{o}_1,\mathbf{o}_2)$ denotes the full subcategory of $\Hom(\mathbf{o}_1,\mathbf{o}_2)$ generated by webs that are colored with the multisets prescribed by 
$\mathbf{A}$ and $\mathbf{B}$ on the right and left boundary edges respectively.

\begin{defi}
\label{def-phifunct}
Let $W$ be a web in $\Homcat^{\mathbf{A}\to\mathbf{B}}(\mathbf{o}_1,\mathbf{o}_2)$ and suppose that $\bigsqcup_\lambda \mathbf{o}_{1,\lambda}$ and $\bigsqcup_\lambda \mathbf{o}_{2,\lambda}$ 
are the objects given in equation \eqref{eq-ICob}, then we also denote by $\phi$ the functor 
\[
\Homcat^{\mathbf{A}\to\mathbf{B}}(\mathbf{o}_1,\mathbf{o}_2) \to \Homcat(\bigsqcup_\lambda \mathbf{o}_{1,\lambda}, \bigsqcup_\lambda \mathbf{o}_{2,\lambda})
\]
defined on webs $W$ and foams $F$ in $\Homcat^{\mathbf{A}\to\mathbf{B}}(\mathbf{o}_1,\mathbf{o}_2)$ by:
\begin{align*}
\phi(W) &:= \;\bigsqcup_\lambda W_{\lambda},\\
\phi(F) &:= T_{W_2}\circ (\id_L \otimes F \otimes \id_R) \circ B_{W_1} .
\end{align*}
\end{defi}

We call the functors $\phi$ \emph{web splitting functors}. Note that their definition depends on our choice of isomorphism foam $T_W$ for every colored web $W$ in $\Homcat^{\mathbf{A}\to\mathbf{B}}(\mathbf{o}_1,\mathbf{o}_2)$. 
In Section \ref{section-tensor} we show that with a suitable choice of $T_W$ the functors $\phi$ not only split webs, but also certain foams between them. We give a prototypical example of this:

\begin{exa} With coloring conventions as in Example \ref{exa-foamsplit} we have
\begin{equation}
\phi\left (\flatsheet[1] \right ) \quad = \quad
\parallelsheets[1]{}{}.
\end{equation}
\end{exa}
Indeed, this follows from Lemma \ref{split1} using
\begin{equation*}
\id_L:= \leftsplit[1]~, \quad \id_R := 
\rightsplit[1] ~, \quad  T:= \unzip[1]{}{}, \quad B:= 
\sum_r ~ \zip[1]{f_r}{g_r}.
\end{equation*}

\begin{rem}\label{rem-conj}
For colored webs $W$ in $\hat{\foam{N}{}^\Sigma}$ we conjecture that the map 
\begin{equation*}
\bigotimes_{\lambda} \Hom( W_{1,\lambda}, W_{2,\lambda}) \to \Hom \left ( \bigsqcup_\lambda W_{1,\lambda}, \bigsqcup_\lambda W_{2,\lambda}\right ),
\end{equation*} given by placing foams colored by individual roots side by side, is an isomorphism of vector spaces. 
In particular, this would mean that every foam between split webs can be split into non-interacting colored components, possibly with additional decorations.
\end{rem}

\subsection{A 2-representation of $\Foam{N}^\Sigma$}
\label{section-defmf}
In this section, we prove that the deformed $\slnn{N}$ foam 2-category $\Foam{N}^\Sigma$ is sufficiently non-degenerate, 
by constructing a 2-representation onto a version of Wu's deformed matrix factorizations \cite{Wu2}. 
Indeed, let $\HMF$ denote the 2-category 
given as follows:
\begin{itemize}

\item objects are pairs $(R,w)$ where $R$ is a $\C$-algebra and $w\in R$,

\item 1-morphisms $(R,w) \to (S,v)$ are matrix factorizations $X$ over $R \otimes_\C S$ with potential $v-w$, and

\item 2-morphisms $X \to Y$ are morphisms in the homotopy category of matrix factorizations.
\end{itemize}
We'll assume the basics concerning matrix factorizations, which can e.g. be found in \cite{KR}; 
see \cite{CM2} for details about the 2-category of matrix factorizations.

Our result is the following:

\begin{thm}\label{thm-nondegen}
There is a 2-representation from the deformed foam 2-category $\Foam{N}^\Sigma$ to the 2-category of matrix factorizations. 
Moreover, this 2-representation assigns to a web in $\Foam{N}^\Sigma$ the same matrix factorization as in 
Wu's construction of deformed link homology.
\end{thm}

Of course, it suffices to assign pairs $(R,w)$ to sequences, 
the same matrix factorizations as in \cite{Wu2} to generating webs, and 
morphisms of matrix factorizations to generating foams, 
and then check that the images of the foam relations hold in $\HMF$. 
However, we can simplify this check using an argument similar to that in \cite{QR}. 
Indeed, there it is shown that the undeformed foam category $\Foam{N}$ is equivalent 
to a certain 2-subcategory of the quotient of $\dot{\cal{U}}_Q(\glnn{\infty})$ by the $N$-bounded weights. 
Since the 2-category of matrix factorizations is idempotent complete, it suffices to construct 
a 2-representation of $\cal{U}_Q(\glnn{\infty})$ sending non-$N$-bounded weights to zero 
and satisfying (the preimage of) the additional foam relation in $\Foam{N}^\Sigma$, which then 
induces a 2-functor from $\Foam{N}^\Sigma$.

Practically speaking, this shows that we need only check the foam relations coming from relations in 
$\cal{U}_Q(\glnn{\infty})$ and not those coming from the thick calculus in $\dot{\cal{U}}_Q(\glnn{\infty})$, 
(which are used to split certain idempotent foams in $\Foam{N}$). This simplifies the number of 
relations needed to be checked (more details below).

We hence begin by following Wu, assigning a pair $(R,w)$ to an object $(a_1,\ldots,a_k)$ in $\Foam{N}^{\Sigma}$.
We set $R=\Sym(\X_1| \cdots |\X_k)$, the $\C$-algebra of partially symmetric functions in the alphabets $\X_1, \ldots , \X_k$, 
where $\X_i$ consists of $a_i$ variables. 
We let $w=Q(\X_1 \cup \cdots \cup \X_k)$ where $Q'(X)=(N+1)P(X)$ with $P(X)$ as in equation \eqref{eq-polynomial}, 
$Q(0)=0$, and for a polynomial $T(X)=\sum_{i=0}^k c_i X^i \in \C[X]$ we set $T(\X) = \sum_{i=0}^k c_i p_i(\X)$ 
where $p_i(\X)$ denotes the $i^{th}$ power sum symmetric polynomial in the alphabet $\X$.

Given sequences $\mathbf{a}$ and $\mathbf{b}$ of elements of a $\C$-algebra, we'll follow \cite{KR} and 
denote by $\{\mathbf{a},\mathbf{b}\}$ the Koszul matrix factorization they determine.
We then assign the Koszul matrix factorizations:
\begin{equation}\label{MFgenweb}
\Big\{
(U_i)_{i=1}^{k+l} \; , \; (e_i(\W \cup \X) - e_i(\Y))_{i=1}^{k+l}
\Big\}
\quad \text{and} \quad
\Big\{
(-U_i)_{i=1}^{k+l} \; , \; (e_i(\Y) - e_i(\W \cup \X))_{i=1}^{k+l}
\Big\}
\end{equation}
over $\Sym(\W|\X|\Y)$ with potentials $Q(\W \cup \X) - Q(\Y)$ and $Q(\Y) - Q(\W \cup \X)$ (respectively)
to the generating webs
\[
\xy
(0,0)*{
\begin{tikzpicture}[scale=.5]
	\draw [very thick,directed=.55] (2.25,0) to (.75,0);
	\draw [very thick,directed=.55] (.75,0) to [out=135,in=0] (-1,.75);
	\draw [very thick,directed=.55] (.75,0) to [out=225,in=0] (-1,-.75);
	\node at (3.125,0) {\tiny $k+l$};
	\node at (-1.5,.75) {\tiny $k$};
	\node at (-1.5,-.75) {\tiny $l$};
\end{tikzpicture}
};
\endxy
\quad , \quad
\xy
(0,0)*{
\begin{tikzpicture}[scale=.5]
	\draw [very thick,rdirected=.55] (-2.25,0) to (-.75,0);
	\draw [very thick,rdirected=.55] (-.75,0) to [out=45,in=180] (1,.75);
	\draw [very thick,rdirected=.55] (-.75,0) to [out=315,in=180] (1,-.75);
	\node at (-3.125,0) {\tiny $k+l$};
	\node at (1.5,.75) {\tiny $k$};
	\node at (1.5,-.75) {\tiny $l$};
\end{tikzpicture}
};
\endxy
\]
where $|\W| = k$, $|\X| = l$, and $|\Y| = k+l$.
Here the polynomials $U_i$ are chosen so that 
\[
Q(\W \cup \X) - Q(\Y) = \sum_{i=1}^{k+l} \big(e_i(\W \cup \X) - e_i(\Y)\big)U_i.
\]
Note that these are the same matrix factorizations that Wu assigns to trivalent vertices.

We now assign a morphism of matrix factorizations to each generating foam. 
To do so, we utilize the concept of stabilization of linear factorizations. 
Recall from \cite{CM} that a linear 
factorization $L$ over a ring $R$ with potential $w\in R$ is a $\Z/2\Z$-graded $R$-module, 
equipped with an odd degree differential $d$ satisfying $d^2 = w~\id$. Informally, a linear factorization 
is a matrix factorization where we loosen the requirement that the $R$-module be free. In particular, 
matrix factorizations give examples of linear factorizations.

Following \cite{CM}, define the {\it stabilization} of a linear factorization $L$ over $(R,w)$ to be a finite-rank 
matrix factorization $M_L$ over $(R,w)$ together with a morphism of linear factorizations $\pi: M_L \to L$ 
inducing a quasi-isomorphism of $\Z/2\Z$-graded complexes:
\begin{equation}\label{eq-qi_stab}
\Hom_R(K,M_L) \xrightarrow{\pi \circ} \Hom_R(K,L)
\end{equation}
for any finite-rank matrix factorization $K$ over $(R,w)$.

We use stabilizations as follows: 
suppose that we are given linear factorizations $L_1$ and $L_2$ with corresponding stabilizations $M_{L_i}$, 
then, provided $M_{L_1}$ is homotopy equivalent to a finite-rank matrix factorization, the diagram
\begin{equation}\label{eq-ind}
\xymatrix{
M_{L_1}\ar[r]^{\pi_1} & L_1 \ar[d] \\
M_{L_2} \ar[r]^{\pi_2} & L_2}
\end{equation}
induces a map on homology 
$\mathrm{H}_*(\Hom_R(L_1,L_2)) \to \mathrm{H}_*(\Hom_R(M_{L_1},L_2)) \to
\mathrm{H}_*(\Hom_R(M_{L_1},M_{L_2}))$. 

Since $\mathrm{H}_0$ gives the morphisms in 
the homotopy category of matrix (or linear) factorizations, we can construct a morphism 
$\stab(\varphi) \in \Hom_{\HMF}(M_{L_1},M_{L_2})$ from a morphism 
$\varphi: L_1 \to L_2$
which is the unique (up to homotopy) morphism so that the diagram
\[
\xymatrix{
M_{L_1}\ar[r]^{\pi_1} \ar[d]_{\stab(\varphi)} & L_1 \ar[d]^{\varphi} \\
M_{L_2} \ar[r]^{\pi_2} & L_2}
\] 
commutes. We'll use this to define the morphisms of matrix factorizations assigned to generating foams, 
and to check that the foam relations are satisfied.

In doing so, we utilize facts about the stabilization of Koszul matrix factorizations. 
Let $\{\mathbf{a},\mathbf{b}\}$ be a Koszul factorization over a $\C$-algebra $R$, then there 
exists a morphism of linear factorizations $\{\mathbf{a},\mathbf{b}\} \to R/(\mathbf{b})$, where the 
latter is viewed as a linear factorization concentrated in degree zero.

\begin{prop}[{\cite[Corollary D.3]{CM}}] \label{prop-stab}
If $\mathbf{b}$ is a regular sequence in $R$, then $\{\mathbf{a},\mathbf{b}\} \to R/(\mathbf{b})$ 
is a stabilization.
\end{prop}

\begin{conv}
In the following we use a large number of quotient rings of the form 
\begin{equation*}
\frac{\Sym(\X_1|\cdots|\X_a|\X_{a+1}|\cdots|\X_{a+b})}{\langle e_i(\X_1 \cup \cdots \cup \X_{a}) - e_i(\X_{a+1} \cup \cdots \cup \X_{a+b})| i>0 \rangle}
\end{equation*}
where $\Sym(\X_1|\cdots|\X_a|\X_{a+1}|\cdots|\X_{a+b})$ denotes the subring of polynomials in $\C[\X_1\cup \cdots \cup \X_{a+b}]$ symmetric in each of the alphabets $\X_1, \dots, \X_{a+b}$ separately. Since the quotient has the effect of identifying symmetric polynomials in the alphabets $\X_1\cup \cdots \cup \X_{a}$ and $\X_{a+1}\cup \cdots \cup \X_{a+b}$, we use the shorthand 
\begin{equation*}
\frac{\Sym(\X_1|\cdots|\X_a|\X_{a+1}|\cdots|\X_{a+b})}{\langle \X_1 \cup \cdots \cup \X_{a} = \X_{a+1} \cup \cdots \cup \X_{a+b}\rangle}
\end{equation*}
for such a quotient ring. We further abbreviate by writing $m$ for a $1$-element alphabet $\X=\{m\}$ in this notation.
\end{conv}

Proposition \ref{prop-stab} implies that the matrix factorizations in equation \eqref{MFgenweb} are stabilizations of the 
linear factorizations $\Sym(\W|\X|\Y)/\langle \W \cup \X = \Y \rangle$ and $\Sym(\W|\X|\Y)/ \langle \Y = \W \cup \X \rangle$. 
Moreover, denoting the matrix factorization associated to a web $W$ by $\MF(W)$, 
we have that for the following maps:
\begin{gather}
\nonumber
\MF \big(
\xy
(0,0)*{
\begin{tikzpicture} [scale=.3,fill opacity=0.2]
	\draw[very thick, directed=.55] (1,1) -- (-1,1);
\end{tikzpicture}
};
\endxy
\big)
\xrightarrow{\pi} 
\frac{\Sym(\V | \Y)}{\langle \V =\Y \rangle} \\
\MF \big(
\xy
(0,0)*{
\begin{tikzpicture} [scale=.3,fill opacity=0.2]
	\draw[very thick,directed=.55] (2,4) -- (.75,4);
	\draw[very thick,directed=.55] (-.75,4) -- (-2,4);
	\draw[very thick,directed=.55] (.75,4) .. controls (.5,3.5) and (-.5,3.5) .. (-.75,4);
	\draw[very thick,directed=.55] (.75,4) .. controls (.5,4.5) and (-.5,4.5) .. (-.75,4);
\end{tikzpicture}
};
\endxy
\big)
\xrightarrow{\pi}
\left( \frac{\Sym(\V | \L | \M)}{\langle \V = \L \cup \M \rangle} \right) \otimes_{\Sym(\L | \M)} 
\left( \frac{\Sym(\L | \M | \Y)}{\langle \L \cup \M = \Y \rangle} \right) \\
\nonumber
\MF \big(
\xy
(0,0)*{
\begin{tikzpicture} [scale=.3,fill opacity=0.2]
	\draw [very thick,directed=.55] (.75,0) to (-.75,0);
	\draw [very thick,directed=.75] (-.75,0) to [out=135,in=0] (-2.25,.5);
	\draw [very thick,directed=.75] (-.75,0) to [out=225,in=0] (-2.25,-.5);
	\draw [very thick,directed=.45] (2.25,.5) to [out=180,in=45] (.75,0);
	\draw [very thick,directed=.45] (2.25,-.5) to [out=180,in=315] (.75,0);
\end{tikzpicture}
};
\endxy
\big)
\xrightarrow{\pi}
\left( \frac{\Sym(\V | \W | \L)}{\langle \V \cup \W = \L \rangle} \right) \otimes_{\Sym(\L)} 
\left( \frac{\Sym(\L | \X | \Y)}{\langle \L = \X \cup \Y \rangle} \right) \\
\label{eq-linfac}
\MF \big(
\xy
(0,0)*{
\begin{tikzpicture} [scale=.3,fill opacity=0.2]
	\draw [very thick,directed=.55] (1,.5) to [out=190,in=350] (-1,.5);
	\draw [very thick,directed=.55] (1,-.5) to [out=170,in=10] (-1,-.5);
\end{tikzpicture}
};
\endxy
\big)
\xrightarrow{\pi} 
\frac{\Sym(\V | \X)}{\langle \V =\X \rangle} \otimes_{\C}
\frac{\Sym(\W | \Y)}{\langle \W= \Y \rangle}  \\
\nonumber
\MF \left(
\xy
(0,0)*{
\begin{tikzpicture} [scale=.3,fill opacity=0.2]
	\draw[very thick, directed=.65] (2,3) to [out=180,in=0] (.75,2.75);
	\draw[very thick, directed=.65] (.75,2.75) to [out=135,in=0] (-.75,3.5);
	\draw[very thick, directed=.65] (.75,2.75) to [out=225,in=0] (-2,2);
	\draw[very thick, directed=.75]  (-.75,3.5) to [out=225,in=0] (-2,3);
	\draw[very thick, directed=.75]  (-.75,3.5) to [out=135,in=0] (-2,4);
\end{tikzpicture}
};
\endxy
\right) 
\xrightarrow{\pi}
\left( \frac{\Sym(\V | \W | \L)}{\langle \V \cup \W =\L \rangle} \right) \otimes_{\Sym(\L)} 
\left( \frac{\Sym(\L | \X | \Y)}{\langle \L \cup \X = \Y \rangle} \right) \\
\nonumber
\MF \left(
\xy
(0,0)*{
\begin{tikzpicture} [scale=.3,fill opacity=0.2]
	\draw[very thick, directed=.55] (2,3) to [out=180,in=0] (.75,3.25);
	\draw[very thick, directed=.65] (.75,3.25) to [out=225,in=0] (-.75,2.5);
	\draw[very thick, directed=.65] (.75,3.25) to [out=135,in=0] (-2,4);
	\draw[very thick, directed=.75] (-.75,2.5) to [out=135,in=0] (-2,3);
	\draw[very thick, directed=.75] (-.75,2.5) to [out=225,in=0] (-2,2);
\end{tikzpicture}
};
\endxy
\right)
\xrightarrow{\pi}
\left( \frac{\Sym(\W | \X | \M)}{\langle \W \cup \X = \M \rangle} \right) \otimes_{\Sym(\M)} 
\left( \frac{\Sym(\V | \M | \Y)}{\langle \V \cup \M =\Y \rangle} \right)
\end{gather}
the matrix factorizations are (homotopy equivalent to) stabilizations of the indicated linear factorizations.
The fact that these maps are stabilizations follows in each case, except for the digon web in the second line, 
since the matrix factorizations are homotopy equivalent to Koszul factorizations,
and the indicated linear factorization is isomorphic to the corresponding linear factorization 
which the Koszul factorization stabilizes.

The matrix factorization assigned to the digon web is a tensor product of Koszul factorizations, 
and we must slightly generalize Proposition \ref{prop-stab} 
to show that it stabilizes the tensor product of the corresponding linear factorizations. 
Recall from \cite[Proposition D.1]{CM} that Proposition \ref{prop-stab} can be proven as follows. 
One first considers the Koszul complex $\{ \mathbf{b} \}$ over $R$ given by the regular sequence $\mathbf{b}$. 
There exists a homotopy equivalence (over $\C$) between $\{ \mathbf{b} \}$ and $R/ (\mathbf{b})$ which specifies 
a deformation retract datum. 
Tensoring with the finite-rank matrix factorization $K^{\vee}$ (the dual of the matrix factorization $K$) and applying perturbation gives a deformation 
retract datum over $\C$ between $K^{\vee} \otimes R/(\mathbf{b})$ and $K^{\vee} \otimes \{ \mathbf{a}, \mathbf{b} \}$ 
which gives the quasi-isomorphism in equation \eqref{eq-qi_stab}. Here we utilize the isomorphism of matrix factorizations  
$K^{\vee} \otimes_R M \cong \Hom_R(K,M)$.

This same method (which is adapted from the results in \cite{DM}) shows that the stabilization result for the digon web 
follows provided the tensor product of Koszul complexes associated to the web only has homology in degree zero, 
and which equals the corresponding tensor product of linear factorizations. 
We hence consider the Koszul complexes $C_1 = \{e_i(\V) - e_i(\L \cup \M) \}$ and $C_2 = \{ e_i(\L \cup \M) - e_i(\Y) \}$ 
over the rings $\Sym(\V | \L | \M)$ and $\Sym(\L | \M | \Y)$, respectively. 
Let $S=\Sym(\L|\M)$, then the homology of $C_1 \otimes_S C_2$ is computed using the K{\"u}nneth spectral sequence to be 
\[
\mathrm{H}_{i} ( C_1 \otimes_S C_2 ) 
\cong 
\left( \mathrm{H}_i(C_1) \otimes_S \frac{\Sym(\L | \M | \Y)}{\langle \L \cup \M = \Y \rangle} \right) 
\oplus \mathrm{Tor}_1^S\left(\mathrm{H}_{i-1}(C_1), \frac{\Sym(\L | \M | \Y)}{\langle \L \cup \M =\Y \rangle} \right)
\]
which is only non-zero when $i=0$ (since $\frac{\Sym(\V | \L | \M)}{\langle \V =\L \cup \M \rangle}$ is a free 
$S$-module) in which case it equals 
$
\left( \frac{\Sym(\V | \L | \M)}{\langle \V = \L \cup \M \rangle} \right) \otimes_S
\left( \frac{\Sym(\L | \M | \Y)}{\langle \L \cup \M = \Y \rangle} \right)
$
as desired.

We now use equation \eqref{eq-ind} to assign a morphism of matrix factorizations to each generating foam, noting 
that the domain web of each generator is mapped to a matrix factorization which is homotopy equivalent to a 
finite-rank matrix factorization. We send:
\[
\xy
(0,0)*{
\begin{tikzpicture} [scale=.5,fill opacity=0.2]
	\path[fill=green] (-.75,4) to [out=270,in=180] (0,2.5) to [out=0,in=270] (.75,4) .. controls (.5,4.5) and (-.5,4.5) .. (-.75,4);
	\path[fill=green] (-.75,4) to [out=270,in=180] (0,2.5) to [out=0,in=270] (.75,4) -- (2,4) -- (2,1) -- (-2,1) -- (-2,4) -- (-.75,4);
	\path[fill=green] (-.75,4) to [out=270,in=180] (0,2.5) to [out=0,in=270] (.75,4) .. controls (.5,3.5) and (-.5,3.5) .. (-.75,4);
	\draw[very thick, directed=.55] (2,1) -- (-2,1);
	\path (.75,1) .. controls (.5,.5) and (-.5,.5) .. (-.75,1); 
	\draw [very thick, red, directed=.65] (-.75,4) to [out=270,in=180] (0,2.5) to [out=0,in=270] (.75,4);
	\draw[very thick] (2,4) -- (2,1);
	\draw[very thick] (-2,4) -- (-2,1);
	\draw[very thick,directed=.55] (2,4) -- (.75,4);
	\draw[very thick,directed=.55] (-.75,4) -- (-2,4);
	\draw[very thick,directed=.55] (.75,4) .. controls (.5,3.5) and (-.5,3.5) .. (-.75,4);
	\draw[very thick,directed=.55] (.75,4) .. controls (.5,4.5) and (-.5,4.5) .. (-.75,4);
	\node [opacity=1]  at (1.5,3.5) {\tiny{$_{a+b}$}};
	\node[opacity=1] at (.25,3.4) {\tiny{$a$}};
	\node[opacity=1] at (-.25,4.1) {\tiny{$b$}};	
\end{tikzpicture}
};
\endxy 
\mapsto
\big(
\overline{1} \mapsto \overline{1} \otimes \overline{1}
\big)
\quad , \quad
\xy
(0,0)*{
\begin{tikzpicture} [scale=.5,fill opacity=0.2]
	\path[fill=green] (-.75,-4) to [out=90,in=180] (0,-2.5) to [out=0,in=90] (.75,-4) .. controls (.5,-4.5) and (-.5,-4.5) .. (-.75,-4);
	\path[fill=green] (-.75,-4) to [out=90,in=180] (0,-2.5) to [out=0,in=90] (.75,-4) -- (2,-4) -- (2,-1) -- (-2,-1) -- (-2,-4) -- (-.75,-4);
	\path[fill=green] (-.75,-4) to [out=90,in=180] (0,-2.5) to [out=0,in=90] (.75,-4) .. controls (.5,-3.5) and (-.5,-3.5) .. (-.75,-4);
	\draw[very thick, directed=.55] (2,-1) -- (-2,-1);
	\path (.75,-1) .. controls (.5,-.5) and (-.5,-.5) .. (-.75,-1); 
	\draw [very thick, red, directed=.65] (.75,-4) to [out=90,in=0] (0,-2.5) to [out=180,in=90] (-.75,-4);
	\draw[very thick] (2,-4) -- (2,-1);
	\draw[very thick] (-2,-4) -- (-2,-1);
	\draw[very thick,directed=.55] (2,-4) -- (.75,-4);
	\draw[very thick,directed=.55] (-.75,-4) -- (-2,-4);
	\draw[very thick,directed=.55] (.75,-4) .. controls (.5,-3.5) and (-.5,-3.5) .. (-.75,-4);
	\draw[very thick,directed=.55] (.75,-4) .. controls (.5,-4.5) and (-.5,-4.5) .. (-.75,-4);
	\node [opacity=1]  at (1.25,-1.25) {\tiny{$_{a+b}$}};
	\node[opacity=1] at (-.25,-3.4) {\tiny{$b$}};
	\node[opacity=1] at (.25,-4.1) {\tiny{$a$}};
\end{tikzpicture}
};
\endxy
\mapsto
\left(
\overline{\pi_{\lambda}^{\L}} \otimes \overline{1} \mapsto 
\left\{ \begin{array}{l} 
\overline{1} \quad \text{if} \quad \lambda=b^a, \\
0 \quad \text{all other } \lambda \in P(a,b),
\end{array} \right.
\right)
\]
\begin{equation}\label{eq-2MorMap}
\xy
(0,0)*{
\begin{tikzpicture} [scale=.5,fill opacity=0.2]
	\path [fill=green] (4.25,-.5) to (4.25,2) to [out=165,in=15] (-.5,2) to (-.5,-.5) to 
		[out=0,in=225] (.75,0) to [out=90,in=180] (1.625,1.25) to [out=0,in=90] 
			(2.5,0) to [out=315,in=180] (4.25,-.5);
	\path [fill=green] (3.75,.5) to (3.75,3) to [out=195,in=345] (-1,3) to (-1,.5) to 
		[out=0,in=135] (.75,0) to [out=90,in=180] (1.625,1.25) to [out=0,in=90] 
			(2.5,0) to [out=45,in=180] (3.75,.5);
	\path[fill=green] (.75,0) to [out=90,in=180] (1.625,1.25) to [out=0,in=90] (2.5,0);
	\draw [very thick,directed=.55] (2.5,0) to (.75,0);
	\draw [very thick,directed=.55] (.75,0) to [out=135,in=0] (-1,.5);
	\draw [very thick,directed=.55] (.75,0) to [out=225,in=0] (-.5,-.5);
	\draw [very thick,directed=.55] (3.75,.5) to [out=180,in=45] (2.5,0);
	\draw [very thick,directed=.55] (4.25,-.5) to [out=180,in=315] (2.5,0);
	\draw [very thick, red, directed=.75] (.75,0) to [out=90,in=180] (1.625,1.25);
	\draw [very thick, red] (1.625,1.25) to [out=0,in=90] (2.5,0);
	\draw [very thick] (3.75,3) to (3.75,.5);
	\draw [very thick] (4.25,2) to (4.25,-.5);
	\draw [very thick] (-1,3) to (-1,.5);
	\draw [very thick] (-.5,2) to (-.5,-.5);
	\draw [very thick,directed=.55] (4.25,2) to [out=165,in=15] (-.5,2);
	\draw [very thick, directed=.55] (3.75,3) to [out=195,in=345] (-1,3);
	\node[opacity=1]  at (1.625,.5) {\tiny{$_{a+b}$}};
	\node[opacity=1] at (3.5,2.65) {\tiny{$b$}};
	\node[opacity=1] at (4,1.85) {\tiny{$a$}};		
\end{tikzpicture}
};
\endxy
\mapsto
\big(
\overline{1} \mapsto \overline{1} \otimes \overline{1}
\big)
\quad , \quad
\xy
(0,0)*{
\begin{tikzpicture} [scale=.5,fill opacity=0.2]
	\path [fill=green] (4.25,2) to (4.25,-.5) to [out=165,in=15] (-.5,-.5) to (-.5,2) to
		[out=0,in=225] (.75,2.5) to [out=270,in=180] (1.625,1.25) to [out=0,in=270] 
			(2.5,2.5) to [out=315,in=180] (4.25,2);
	\path [fill=green] (3.75,3) to (3.75,.5) to [out=195,in=345] (-1,.5) to (-1,3) to [out=0,in=135]
		(.75,2.5) to [out=270,in=180] (1.625,1.25) to [out=0,in=270] 
			(2.5,2.5) to [out=45,in=180] (3.75,3);
	\path[fill=green] (2.5,2.5) to [out=270,in=0] (1.625,1.25) to [out=180,in=270] (.75,2.5);
	\draw [very thick,directed=.55] (4.25,-.5) to [out=165,in=15] (-.5,-.5);
	\draw [very thick, directed=.55] (3.75,.5) to [out=195,in=345] (-1,.5);
	\draw [very thick, red, directed=.75] (2.5,2.5) to [out=270,in=0] (1.625,1.25);
	\draw [very thick, red] (1.625,1.25) to [out=180,in=270] (.75,2.5);
	\draw [very thick] (3.75,3) to (3.75,.5);
	\draw [very thick] (4.25,2) to (4.25,-.5);
	\draw [very thick] (-1,3) to (-1,.5);
	\draw [very thick] (-.5,2) to (-.5,-.5);
	\draw [very thick,directed=.55] (2.5,2.5) to (.75,2.5);
	\draw [very thick,directed=.55] (.75,2.5) to [out=135,in=0] (-1,3);
	\draw [very thick,directed=.55] (.75,2.5) to [out=225,in=0] (-.5,2);
	\draw [very thick,directed=.55] (3.75,3) to [out=180,in=45] (2.5,2.5);
	\draw [very thick,directed=.55] (4.25,2) to [out=180,in=315] (2.5,2.5);
	\node [opacity=1]  at (1.625,2) {\tiny{$_{a+b}$}};
	\node[opacity=1] at (3.5,2.65) {\tiny{$b$}};
	\node[opacity=1] at (4,1.85) {\tiny{$a$}};		
\end{tikzpicture}
};
\endxy
\mapsto
\left(
\overline{1} \otimes \overline{1} \mapsto
\sum_{\alpha \in P(a,b)} (-1)^{|\hat{\alpha}|} \overline{\pi_{\hat{\alpha}}^{\V}} \otimes \overline{\pi_{\alpha}^{\Y}}
\right)
\end{equation}
\[
\xy
(0,0)*{
\begin{tikzpicture} [scale=.5,fill opacity=0.2]
	\path[fill=green] (-2.5,4) to [out=0,in=135] (-.75,3.5) to [out=270,in=90] (.75,.25)
		to [out=135,in=0] (-2.5,1);
	\path[fill=green] (-.75,3.5) to [out=270,in=125] (.29,1.5) to [out=55,in=270] (.75,2.75) 
		to [out=135,in=0] (-.75,3.5);
	\path[fill=green] (-.75,-.5) to [out=90,in=235] (.29,1.5) to [out=315,in=90] (.75,.25) 
		to [out=225,in=0] (-.75,-.5);
	\path[fill=green] (-2,3) to [out=0,in=225] (-.75,3.5) to [out=270,in=125] (.29,1.5)
		to [out=235,in=90] (-.75,-.5) to [out=135,in=0] (-2,0);
	\path[fill=green] (-1.5,2) to [out=0,in=225] (.75,2.75) to [out=270,in=90] (-.75,-.5)
		to [out=225,in=0] (-1.5,-1);
	\path[fill=green] (2,3) to [out=180,in=0] (.75,2.75) to [out=270,in=55] (.29,1.5)
		to [out=305,in=90] (.75,.25) to [out=0,in=180] (2,0);
	\draw[very thick, directed=.55] (2,0) to [out=180,in=0] (.75,.25);
	\draw[very thick, directed=.55] (.75,.25) to [out=225,in=0] (-.75,-.5);
	\draw[very thick, directed=.55] (.75,.25) to [out=135,in=0] (-2.5,1);
	\draw[very thick, directed=.55] (-.75,-.5) to [out=135,in=0] (-2,0);
	\draw[very thick, directed=.55] (-.75,-.5) to [out=225,in=0] (-1.5,-1);
	\draw[very thick, red, rdirected=.85] (-.75,3.5) to [out=270,in=90] (.75,.25);
	\draw[very thick, red, rdirected=.75] (.75,2.75) to [out=270,in=90] (-.75,-.5);	
	\draw[very thick] (-1.5,-1) -- (-1.5,2);	
	\draw[very thick] (-2,0) -- (-2,3);
	\draw[very thick] (-2.5,1) -- (-2.5,4);	
	\draw[very thick] (2,3) -- (2,0);
	\draw[very thick, directed=.55] (2,3) to [out=180,in=0] (.75,2.75);
	\draw[very thick, directed=.55] (.75,2.75) to [out=135,in=0] (-.75,3.5);
	\draw[very thick, directed=.65] (.75,2.75) to [out=225,in=0] (-1.5,2);
	\draw[very thick, directed=.55]  (-.75,3.5) to [out=225,in=0] (-2,3);
	\draw[very thick, directed=.55]  (-.75,3.5) to [out=135,in=0] (-2.5,4);
	\node[opacity=1] at (-2.25,3.375) {\tiny$c$};
	\node[opacity=1] at (-1.75,2.75) {\tiny$b$};	
	\node[opacity=1] at (-1.25,1.75) {\tiny$a$};
	\node[opacity=1] at (0,2.75) {\tiny$_{b+c}$};
	\node[opacity=1] at (0,.25) {\tiny$_{a+b}$};
\end{tikzpicture}
};
\endxy
\mapsto
\big(
\overline{1} \otimes \overline{1} \mapsto \overline{1} \otimes \overline{1}
\big)
\quad , \quad
\xy
(0,0)*{
\begin{tikzpicture} [scale=.5,fill opacity=0.2]
	\path[fill=green] (-2.5,4) to [out=0,in=135] (.75,3.25) to [out=270,in=90] (-.75,.5)
		 to [out=135,in=0] (-2.5,1);
	\path[fill=green] (-.75,2.5) to [out=270,in=125] (-.35,1.5) to [out=45,in=270] (.75,3.25) 
		to [out=225,in=0] (-.75,2.5);
	\path[fill=green] (-.75,.5) to [out=90,in=235] (-.35,1.5) to [out=315,in=90] (.75,-.25) 
		to [out=135,in=0] (-.75,.5);	
	\path[fill=green] (-2,3) to [out=0,in=135] (-.75,2.5) to [out=270,in=125] (-.35,1.5) 
		to [out=235,in=90] (-.75,.5) to [out=225,in=0] (-2,0);
	\path[fill=green] (-1.5,2) to [out=0,in=225] (-.75,2.5) to [out=270,in=90] (.75,-.25)
		to [out=225,in=0] (-1.5,-1);
	\path[fill=green] (2,3) to [out=180,in=0] (.75,3.25) to [out=270,in=45] (-.35,1.5) 
		to [out=315,in=90] (.75,-.25) to [out=0,in=180] (2,0);				
	\draw[very thick, directed=.55] (2,0) to [out=180,in=0] (.75,-.25);
	\draw[very thick, directed=.55] (.75,-.25) to [out=135,in=0] (-.75,.5);
	\draw[very thick, directed=.55] (.75,-.25) to [out=225,in=0] (-1.5,-1);
	\draw[very thick, directed=.45]  (-.75,.5) to [out=225,in=0] (-2,0);
	\draw[very thick, directed=.35]  (-.75,.5) to [out=135,in=0] (-2.5,1);	
	\draw[very thick, red, rdirected=.75] (-.75,2.5) to [out=270,in=90] (.75,-.25);
	\draw[very thick, red, rdirected=.85] (.75,3.25) to [out=270,in=90] (-.75,.5);
	\draw[very thick] (-1.5,-1) -- (-1.5,2);	
	\draw[very thick] (-2,0) -- (-2,3);
	\draw[very thick] (-2.5,1) -- (-2.5,4);	
	\draw[very thick] (2,3) -- (2,0);
	\draw[very thick, directed=.55] (2,3) to [out=180,in=0] (.75,3.25);
	\draw[very thick, directed=.55] (.75,3.25) to [out=225,in=0] (-.75,2.5);
	\draw[very thick, directed=.55] (.75,3.25) to [out=135,in=0] (-2.5,4);
	\draw[very thick, directed=.55] (-.75,2.5) to [out=135,in=0] (-2,3);
	\draw[very thick, directed=.55] (-.75,2.5) to [out=225,in=0] (-1.5,2);
	\node[opacity=1] at (-2.25,3.75) {\tiny$c$};
	\node[opacity=1] at (-1.75,2.75) {\tiny$b$};	
	\node[opacity=1] at (-1.25,1.75) {\tiny$a$};
	\node[opacity=1] at (-.125,2.25) {\tiny$_{a+b}$};
	\node[opacity=1] at (-.125,.75) {\tiny$_{b+c}$};
\end{tikzpicture}
};
\endxy
\mapsto
\big(
\overline{1} \otimes \overline{1} \mapsto \overline{1} \otimes \overline{1}
\big)
\]
where in each case the map on the right-hand side describes a morphism 
between the linear factorizations from equation \eqref{eq-linfac} corresponding to the 
top and bottom webs, and $\bar{f}$ denotes the equivalence class of $f$ in the quotient.
In these formulae, $\pi_{\lambda}^{\W}$ denotes the Schur 
polynomial in the alphabet $\W$ corresponding to the partition $\lambda$, 
and $b^a = (b,\ldots,b)$, the partition of $ab$ given by a sequence of $b$'s of 
length $a$. We can now proceed with the proof of Theorem \ref{thm-nondegen}.

\begin{proof}
It suffices to show that the foam relations hold in $\HMF$. 
As we mentioned above, rather than check them all by hand, we'll instead adopt a method of proof from \cite{QR}. 
By an argument similar to that in Section 4 of that paper, it suffices to construct a family of 
2-functors $\cal{U}_Q(\glm) \xrightarrow{\Phi_m} \HMF$ which kill non-$N$-bounded weights and the 2-morphism
$
P\left(
\xy
(0,0)*{
\begin{tikzpicture} [scale=.3]
	\draw[thick, directed = .99] (0,0) to (0,2);
	\node at (0,1) {$\bullet$};
\end{tikzpicture}
};
\endxy
\right)
$,
and so that the triangles 
\[
\xymatrix{
\cal{U}_Q(\glm) \ar[r] \ar[rd]_{\Gamma_{m}} & \cal{U}_Q(\glnn{m+1}) \ar[d]^{\Gamma_{m+1}} \\
& \HMF
}
\]
commute. From the definition of the foamation 2-functor in \cite{QR} and our above assignments to 
webs and foams, it is clear how such 2-functors should be defined. 
To see that they are well-defined, we must check that all relations 
in $\cal{U}_Q(\glm)$ are satisfied. This in turn implies that we need only check the foam relations which 
are the analogs of the relations in $\cal{U}_Q(\glm)$. 
Since equation \eqref{eq-2MorMap} implies that the image in $\HMF$ of the ``M-P foam relations'' 
from equation \eqref{FoamRel1} are satisfied,
things simplify even more, and we finally deduce that we need only check a subset of the general foam 
relations, which we verify below.

To do so, we'll again employ stabilization. The matrix factorizations through which the 
(images of the) foam relations factor are all given as tensor products of Koszul factorizations assigned to 
trivalent webs, and we can consider the corresponding tensor product of the linear factorizations they stabilize. 
This gives a diagram
\[
\xymatrix{
\bigotimes_{i} M_{L_{i,1}}\ar[r]^{\pi_1} \ar[d]_{\stab(\varphi_1)} & \bigotimes_{i} L_{i,1} \ar[d]^{\varphi_1} \\
\bigotimes_{j} M_{L_{j,2}}\ar[r]^{\pi_2} \ar[d]_{\stab(\varphi_2)} & \bigotimes_{j} L_{j,2} \ar[d]^{\varphi_2} \\
\vdots \ar[d]_{\stab(\varphi_{l-1})} & \vdots \ar[d]^{\varphi_{l-1}} \\
\bigotimes_{k} M_{L_{k,l}}\ar[r]^{\pi_k} & \bigotimes_{k} L_{k,l}  \\
}
\] 
which commutes up to homotopy. Each side of a foam relation gives rise to such a diagram, and the morphism 
of matrix factorizations is uniquely determined by the morphism of linear factorizations, 
provided the matrix factorizations assigned to the bottom webs are homotopic to ones which are finite-rank, 
and provided that the matrix factorizations assigned to the top webs (i.e. the bottom left in the above diagram)
are homotopic to ones which stabilize the corresponding tensor product of linear factorizations.

The finite-rank condition for the bottom webs follows similarly to results of Wu \cite{Wu1} in the undeformed case. 
To see that the matrix factorizations corresponding to the top webs (are homotopic to ones which) 
stabilize the corresponding linear factorizations, we note that we've already shown this for 
\cite[Equations (3.9) -- (3.12)]{QR}. 
For the remainder of the relations we argue as for the digon web above. 
It again suffices to show that the tensor product of Koszul complexes associated to the top web has homology 
only in degree zero, and equal to the corresponding tensor product of linear factorizations. 
In each case, this follows from (possibly repeated) use of the K{\"u}nneth spectral sequence, 
and the fact that
$
\frac{\Sym(\V|\W|\X)}{\langle \V \cup \W =\X \rangle}
$
is a free module, over both $\Sym(\V|\W)$ and $\Sym(\X)$.
Note that this is essentially a version of results of Becker \cite[Theorem 2]{Bec} and Webster \cite[Theorem 2.5]{Web} for deformed potentials. 

We now check the requisite foam relations (with numbering and notation from \cite{QR} for the remainder of this section)
by confirming that the corresponding maps of linear factorizations agree. 
\newline

\noindent\underline{Equation (QR 3.9)}: 
By \cite[Remark 3.2]{QR}, this only needs to be checked when $\pi_\gamma = e_s$.
This relation then follows since multiplication by $e_i(\X)$ on $\frac{\Sym(\V | \W | \X)}{\langle \V \cup \W =\X \rangle}$ 
is equal to multiplication by $e_i(\V \cup \W)$.
\newline

\noindent\underline{Equation (QR 3.10), first relation}: 
Again by \cite[Remark 3.2]{QR}, it suffices to check the case when $\pi_\alpha = 1$.
Let $\V$ and $\Y$ be alphabets with $k+1$ variables.
It suffices to show that the morphism corresponding to the right-hand side is the identity, hence we compute:
\[
\xymatrix{
\frac{\Sym(\V | \Y)}{\langle \V =\Y \rangle} \ar[r] 
\POS p-(0,7.5)*+{\overline{1}}="a"
&\frac{\Sym(\V | m | \M)}{\langle \V = m \cup \M \rangle} \otimes
 \frac{\Sym(m | \M | \Y)}{\langle m \cup \M = \Y \rangle} \ar[r]
\POS p-(0,7.5)*+{\overline{1}\otimes \overline{1}}="b"
&\frac{\Sym(\V | m | \M)}{\langle \V = m \cup \M \rangle} \otimes
 \frac{\Sym(m | \M | \Y)}{\langle m \cup \M =\Y \rangle}  \ar[r]
\POS p-(0,7.5)*+{\overline{m^k}\otimes \overline{1}}="c" 
&  \frac{\Sym(\V | \Y)}{\langle \V =\Y \rangle}
\POS p-(0,7.5)*+{\overline{1}}="d"
\ar@{|->} "a";"b"
\ar@{|->} "b";"c"
\ar@{|->} "c";"d"
}
\]
which verifies the relation.
\newline

\noindent \underline{Equation (QR 3.10), second relation}:
It suffices to check the case when with $a = 1 = b$, by the M-P relation, isotopy, and equation (QR 3.9). We compute the 
left-hand side:
\[
\xymatrix{
\frac{\Sym(\W | \X)} {\langle \W = \X \rangle} \ar[r] 
\POS p-(0,10)*+{\overline{1}}="a"
&  \frac{\Sym(\W|m|n)}{\left\langle \W= \{m,n\} \right\rangle}  \otimes
\frac{\Sym(m|n|\X)}{\left\langle \{m,n\}=\X \right\rangle}  \ar[r]
\POS p-(0,10)*+{\overline{1} \otimes \overline{1}}="b"
&  \frac{\Sym(\W|m|n)}{\left\langle \W= \{m,n\} \right\rangle}  \otimes
\frac{\Sym(m|n|\X)}{\left\langle \{m,n\}=\X \right\rangle} \ar[r]
\POS p-(0,10)*+{\overline{f(m)g(n)}\otimes \overline{1}}="c" 
\POS p+(0,-15)*+{\overline{m} \otimes \overline{1}}="d"
\POS p+(0,-20)*+{\overline{1} \otimes \overline{1}}="e"
&
\frac{\Sym(\W | \X)} {\langle \W = \X \rangle}
\POS p+(0,-15)*+{\overline{1}}="f"
\POS p+(0,-20)*+{0}="g"
\ar@{|->} "a";"b"
\ar@{|->} "b";"c"
\ar@{|->} "d";"f"
\ar@{|->} "e";"g"
}
\]
while the right-hand side is the negative of the map which is the same as the above, 
but with the second map given instead by 
$\overline{1} \otimes \overline{1} \mapsto \overline{g(m)f(n)}\otimes \overline{1}$.
Equivalently, this is the negative of the map which is the same as the above, but instead with the 
third map given by $\overline{n} \otimes \overline{1} \mapsto \overline{1}$ and 
$\overline{1} \otimes \overline{1} \mapsto 0$. 
Since $\overline{n} \otimes \overline{1} = \overline{e_1(\W)} \otimes \overline{1} - \overline{m} \otimes \overline{1}$
and $\overline{e_1(\W)} \otimes \overline{1} \mapsto 0$ under the final map, this confirms the relation.
\newline

\noindent \underline{Equation (QR 3.11)}:
The $a=1=b$ case of this relation is used to deduce that the image of the $3^{\mathrm{rd}}$ nilHecke relation 
is satisfied. A careful analysis of the proof of \cite[Lemma 3.7]{QR} shows that the only remaining version of this 
relation required are those when $a=1, b=2$ and $a=2,b=1$, which are used to prove the $2^{\mathrm{nd}}$, 
Reidemeister 3-like nilHecke relation.

In the $a=1=b$ case, the right-hand side corresponds to the sum of the map
\[
\xymatrix{
\frac{\Sym(\W|m|n)}{\left\langle \W=\{m,n\} \right\rangle}  \otimes
\frac{\Sym(m|n|\X)}{\left\langle \{m,n\}=\X \right\rangle} \ar[r]
\POS p-(0,10)*+{\overline{1} \otimes \overline{1}}="a"
\POS p-(0,20)*+{\overline{m} \otimes \overline{1}}="b"
& 
\frac{\Sym(\W | \X)} {\langle \W = \X \rangle} \ar[r]
\POS p-(0,10)*+{\overline{1}}="c"
\POS p-(0,20)*+{\overline{e_1(\W)}}="d"
& \frac{\Sym(\W|m|n)}{\left\langle \W=\{m,n\} \right\rangle}  \otimes
\frac{\Sym(m|n|\X)}{\left\langle \{m,n\}=\X  \right\rangle} 
\POS p-(0,10)*+{\overline{1} \otimes \overline{1}}="e"
\POS p-(0,20)*+{\overline{m+n} \otimes \overline{1}}="f"
\ar@{|->} "a";"c"
\ar@{|->} "b";"d"
\ar@{|->} "c";"e"
\ar@{|->} "d";"f"
}
\]
and the negative of the map
\[
\xymatrix{
\frac{\Sym(\W|m|n)}{\left\langle\W=\{m,n\} \right\rangle}  \otimes
\frac{\Sym(m|n|\X)}{\left\langle \{m,n\}=\X \right\rangle} \ar[r]
\POS p-(0,10)*+{\overline{1} \otimes \overline{1}}="a"
\POS p-(0,20)*+{\overline{m} \otimes \overline{1}}="b"
&
\frac{\Sym(\W | \X)} {\langle \W = \X \rangle} \ar[r]
\POS p-(0,10)*+{0}="c"
\POS p-(0,20)*+{\overline{1}}="d"
& \frac{\Sym(\W|m|n)}{\left\langle \W=\{m,n\} \right\rangle}  \otimes
\frac{\Sym(m|n|\X)}{\left\langle \{m,n\}=\X \right\rangle} 
\POS p-(0,10)*+{0}="e"
\POS p-(0,20)*+{\overline{n} \otimes \overline{1}}="f"
\ar@{|->} "a";"c"
\ar@{|->} "b";"d"
\ar@{|->} "c";"e"
\ar@{|->} "d";"f"
}
\]
which confirms that this map equals the identity, as desired.

For the $a=1,b=2$ case, let $|\V|=3=|\Y|$ and $|\M|=2$.
The right-hand side corresponds to the sum of the map, 
\[
\xymatrix{
\frac{\Sym(\V | m | \M)}
{\left\langle \V= m\cup \M \right\rangle} \otimes 
\frac{\Sym(m | \M | \Y)}
{\left\langle m\cup \M = \Y \right\rangle} \ar[r]
\POS p-(0,12.5)*+{\overline{1} \otimes \overline{1}}="a"
\POS p-(0,20)*+{\overline{m} \otimes \overline{1}}="d"
\POS p-(0,27.5)*+{\overline{m^2} \otimes \overline{1}}="g"
& \frac{\Sym(\V|\Y)} {\langle \V = \Y \rangle} \ar[r]
\POS p-(0,12.5)*+{\overline{1}}="b"
\POS p-(0,20)*+{\overline{e_1(\V)}}="e"
\POS p-(0,27.5)*+{\overline{e_1(\V)^2 - e_2(\V)}}="h"
& \frac{\Sym(\V | m | \M)}
{\left\langle \V= m\cup \M \right\rangle} \otimes 
\frac{\Sym(m | \M | \Y)}
{\left\langle m\cup \M = \Y \right\rangle}
\POS p-(0,12.5)*+{\overline{1}}="c"
\POS p-(0,20)*+{\overline{m+e_1(\M)} \otimes \overline{1}}="f"
\POS p-(-10,27.5)*+{\overline{m^2+me_1(\M)+e_1(\M)^2-e_2(\M)} \otimes \overline{1}}="i" 
\ar@{|->} "a";"b"
\ar@{|->} "b";"c"
\ar@{|->} "d";"e"
\ar@{|->} "e";"f"
\ar@{|->} "g";"h"
\ar@{|->} "h";"i"
}
\]
the negative of the map
\[
\xymatrix{
\frac{\Sym(\V | m | \M)}
{\left\langle \V = m\cup \M  \right\rangle} \otimes 
\frac{\Sym(m | \M | \Y)}
{\left\langle m\cup\M = Z \right\rangle} \ar[r]
\POS p-(0,12.5)*+{\overline{1} \otimes \overline{1}}="a"
\POS p-(0,20)*+{\overline{m} \otimes \overline{1}}="d"
\POS p-(0,27.5)*+{\overline{m^2} \otimes \overline{1}}="g"
& \frac{\Sym(\V|\Y)} {\langle \V = \Y \rangle} \ar[r]
\POS p-(0,12.5)*+{0}="b"
\POS p-(0,20)*+{\overline{1}}="e"
\POS p-(0,27.5)*+{\overline{e_1(\V)}}="h"
& \frac{\Sym(\V | m | \M)}
{\left\langle \V = m\cup \M \right\rangle} \otimes 
\frac{\Sym(m | \M | \Y)}
{\left\langle m\cup\M = Z \right\rangle}
\POS p-(0,12.5)*+{0}="c"
\POS p-(0,20)*+{\overline{e_1(\M)} \otimes \overline{1}}="f"
\POS p-(0,27.5)*+{\overline{me_1(\M)+e_1(\M)^2} \otimes \overline{1}}="i"
\ar@{|->} "a";"b"
\ar@{|->} "b";"c"
\ar@{|->} "d";"e"
\ar@{|->} "e";"f"
\ar@{|->} "g";"h"
\ar@{|->} "h";"i"
}
\]
and the map
\[
\xymatrix{
\frac{\Sym(\V | m | \M)}
{\left\langle \V = m\cup \M \right\rangle} \otimes 
\frac{\Sym(m | \M | \Y)}
{\left\langle m\cup\M = Z \right\rangle}\ar[r]
\POS p-(0,12.5)*+{\overline{1} \otimes \overline{1}}="a"
\POS p-(0,20)*+{\overline{m} \otimes \overline{1}}="d"
\POS p-(0,27.5)*+{\overline{m^2} \otimes \overline{1}}="g"
& \frac{\Sym(\V|\Y)} {\langle \V = \Y \rangle} \ar[r]
\POS p-(0,12.5)*+{0}="b"
\POS p-(0,20)*+{0}="e"
\POS p-(0,27.5)*+{\overline{1}}="h"
& \frac{\Sym(\V | m | \M)}
{\left\langle \V = m\cup \M \right\rangle} \otimes 
\frac{\Sym(m | \M | \Y)}
{\left\langle m\cup\M = Z \right\rangle}\POS p-(0,12.5)*+{0}="c"
\POS p-(0,20)*+{0}="f"
\POS p-(0,27.5)*+{\overline{e_2(\M)} \otimes \overline{1}}="i"
\ar@{|->} "a";"b"
\ar@{|->} "b";"c"
\ar@{|->} "d";"e"
\ar@{|->} "e";"f"
\ar@{|->} "g";"h"
\ar@{|->} "h";"i"
}
\]
which confirms that this map is the identity. The case  $a=2,b=1$ follows similarly.
\newline

\noindent \underline{Equation (QR 3.12)}:
Both sides of this relation are given  by
\[
\xymatrix{
\frac{\Sym(\A|\L|\V)}{\langle \A = \L \cup \V \rangle} \otimes
\frac{\Sym(\L|\W|\M)}{\langle \L=\W \cup \M \rangle} \otimes
\frac{\Sym(\M|\X|\Y)}{\langle \M = \X \cup \Y \rangle} \ar[r]
\POS p-(0,7.5)*+{\overline{1} \otimes \overline{1} \otimes \overline{1}}="a"
& 
\frac{\Sym(\A|\S|\Y)}{\langle \A = \S \cup \Y \rangle} \otimes
\frac{\Sym(\S|\T|\X)}{\langle \S = \T \cup \X \rangle} \otimes
\frac{\Sym(\T|\V|\W)}{\langle \T = \V \cup \W \rangle}
\POS p-(0,7.5)*+{\overline{1} \otimes \overline{1} \otimes \overline{1}}="b"
\ar@{|->} "a";"b"
}.
\]
Hence, they are equal. In the above, the tensor products are each taken over symmetric polynomials in the 
common alphabets between the tensor factors.
\newline

\noindent \underline{Equations (QR 3.13) and (QR 3.14)}:
It suffices to prove these relations in the case when $a=1=c$; however, it isn't much more difficult to verify the general relation. 
To check this, we first note that both of the possible ways to construct the following crossing correspond 
to the morphism of linear factorizations:
\[
\xy
(0,0)*{
\begin{tikzpicture} [scale=.5,fill opacity=0.2]
	\path[fill=green] (2.5,3) to [out=180,in=315] (-.5,3.25) to [out=270,in=90]  (1,.5) 
		to [out=315,in=180] (2.5,0);
	\path[fill=green] (2,4) to [out=180,in=0] (1,3.75) to [out=270,in=45] (.13,2) to [out=315,in=90]
		(1,.5) to [out=45,in=180] (2,1);
	\path[fill=green] (1,3.75) to [out=225,in=45] (-.5,3.25) to [out=270,in=135] (.13,2) to 
		[out=45,in=270] (1,3.75);
	\path[fill=green] (-.5,.5) to [out=90,in=235] (.13,2) to [out=315,in=90] (1,.5);
	\path[fill=green] (-2,1) to [out=0,in=135] (-.5,.5) to [out=90,in=270] (1,3.75) 
		to [out=135,in=0] (-2,4);
	\path[fill=green] (-1.5,3) to [out=0,in=180] (-.5,3.25) to [out=270,in=135] (.13,2)
		to [out=235,in=90] (-.5,.5) to [out=225,in=0] (-1.5,0);
	\draw[very thick, directed=.55] (2,1) to [out=180,in=45] (1,.5);
	\draw[very thick,directed=.55] (2.5,0) to [out=180,in=315] (1,.5);
	\draw[very thick, directed=.55] (1,.5) to (-.5,.5);
	\draw[very thick, directed=.55] (-.5,.5) to [out=225,in=0] (-1.5,0);
	\draw[very thick, directed=.55] (-.5,.5) to [out=135,in=0] (-2,1);
	\draw[very thick, red, directed=.65] (-.5,.5) to [out=90,in=270] (1,3.75);
	\draw[very thick, red, directed=.65] (-.5,3.25) to [out=270,in=90]  (1,.5);
	\draw[very thick] (2,1) to (2,4);
	\draw[very thick] (2.5,0) to (2.5,3);
	\draw[very thick] (-1.5,0) to (-1.5,3);
	\draw[very thick] (-2,1) to (-2,4);	
	\draw[very thick, directed=.65] (2,4) to [out=180,in=0] (1,3.75);
	\draw[very thick, directed=.55] (1,3.75) to [out=135,in=0] (-2,4);
	\draw[very thick, directed=.55] (2.5,3) to [out=180,in=315] (-.5,3.25);
	\draw[very thick, directed=.55] (-.5,3.25) to [out=180,in=0] (-1.5,3);
	\draw[very thick, directed=.55] (1,3.75) to [out=225,in=45] (-.5,3.25);
	\node[opacity=1] at (1.75,3.75) {\tiny$b$};
	\node[opacity=1] at (2.25,2.75) {\tiny$a$};
	\node[opacity=1] at (-1.75,3.75) {\tiny$c$};
	\node[opacity=1] at (.25,3.25) {\tiny$_{b-c}$};
	\node[opacity=1] at (.25,.875) {\tiny$_{a+b}$};
	\node[opacity=1] at (-1,2.75) {\tiny$_{a+b}$};
	\node[opacity=1] at (-.75,2.5) {\tiny$_{-c}$};
\end{tikzpicture}
};
\endxy
\mapsto
\left(
\xymatrix{
\frac{\Sym(\V | \W | \L)}{\langle \V \cup \W  =\L \rangle} \otimes \frac{\Sym(\L | \X | \Y)}{\langle \L = \X \cup \Y \rangle}
\ar[r]
\POS p-(0,7.5)*+{\overline{1} \otimes \overline{1}}="b"
&
\frac{\Sym(\W|\M|\Y)}{\langle \W= \M \cup \Y \rangle} \otimes \frac{\Sym(\V|\M|\X)}{\langle \V \cup \M=\X \rangle}
\POS p-(0,7.5)*+{\overline{1} \otimes \overline{1}}="a"
 \ar@{|->} "b";"a"
}
\right)
\]
and similarly both ways of constructing the crossing give the indicated map:
\[
\xy
(0,0)*{
\begin{tikzpicture} [scale=.5,fill opacity=0.2]
	\path[fill=green] (2.5,3) to [out=180,in=315] (1,3.5) to [out=270,in=90]  (-.5,.25)
		to [out=315,in=180] (2.5,0);
	\path[fill=green] (2,4) to [out=180,in=45] (1,3.5) to [out=270,in=55] (.37,2) to 
		[out=315,in=90] (1,.75) to [out=0,in=180] (2,1);
	\path[fill=green] (1,3.5) to [out=270,in=55] (.37,2) to [out=135,in=270] (-.5,3.5);
	\path[fill=green] (1,.75) to [out=225,in=45] (-.5,.25) to [out=90,in=225] (.37,2) to
		[out=315,in=90] (1,.75);
	\path[fill=green] (-2,1) to [out=0,in=135] (1,.75) to [out=90,in=270] (-.5,3.5)
		 to [out=135,in=0] (-2,4);
	\path[fill=green] (-1.5,3) to [out=0,in=225] (-.5,3.5) to [out=270,in=135] (.37,2) to 
		[out=225,in=90] (-.5,.25) to [out=180,in=0] (-1.5,0);
	\draw[very thick, directed=.65] (2,1) to [out=180,in=0] (1,.75);
	\draw[very thick, directed=.55] (1,.75) to [out=135,in=0] (-2,1);
	\draw[very thick, directed=.55] (2.5,0) to [out=180,in=315] (-.5,.25);
	\draw[very thick, directed=.55] (-.5,.25) to [out=180,in=0] (-1.5,0);
	\draw[very thick, directed=.55] (1,.75) to [out=225,in=45] (-.5,.25);
	\draw[very thick, red, directed=.65] (1,.75) to [out=90,in=270] (-.5,3.5);
	\draw[very thick, red, directed=.65] (1,3.5) to [out=270,in=90]  (-.5,.25);
	\draw[very thick] (2,1) to (2,4);
	\draw[very thick] (2.5,0) to (2.5,3);
	\draw[very thick] (-1.5,0) to (-1.5,3);
	\draw[very thick] (-2,1) to (-2,4);	
	\draw[very thick, directed=.55] (2,4) to [out=180,in=45] (1,3.5);
	\draw[very thick,directed=.55] (2.5,3) to [out=180,in=315] (1,3.5);
	\draw[very thick, directed=.55] (1,3.5) to (-.5,3.5);
	\draw[very thick, directed=.55] (-.5,3.5) to [out=225,in=0] (-1.5,3);
	\draw[very thick, directed=.55] (-.5,3.5) to [out=135,in=0] (-2,4);
	\node[opacity=1] at (1.75,3.75) {\tiny$b$};
	\node[opacity=1] at (2.25,2.75) {\tiny$a$};
	\node[opacity=1] at (-1.75,3.75) {\tiny$c$};
	\node[opacity=1] at (.25,.875) {\tiny$_{b-c}$};
	\node[opacity=1] at (.25,3.125) {\tiny$_{a+b}$};
	\node[opacity=1] at (-1,2.75) {\tiny$_{a+b}$};
	\node[opacity=1] at (-.75,2.5) {\tiny$_{-c}$};
\end{tikzpicture}
};
\endxy
\mapsto
\left(
\xymatrix{
\frac{\Sym(\W|\M|\Y)}{\langle \W = \M \cup \Y \rangle} \otimes \frac{\Sym(\V|\M|\X)}{\langle \V \cup \M =\X \rangle}
\ar[r]
\POS p-(0,9)*+{\overline{1} \otimes \overline{1}}="a"
& 
\frac{\Sym(\V | \W | \L)}{\langle \V \cup \W  =\L \rangle} \otimes \frac{\Sym(\L | \X | \Y)}{\langle \L = \X \cup \Y \rangle}
\POS p-(0,9)*+{\displaystyle \sum_{\alpha \in P(a,c)} (-1)^{|\widehat{\alpha}|} \; \overline{\pi_{\widehat{\alpha}}^{\V}} \otimes \overline{\pi_{\alpha}^{\Y}}}="b"
\ar@{|->} "a";"b"
}
\right) .
\]
The first is clear, and the second follows, for example, since one way of constructing the sideways crossing is given by the composition:
\[
\xy
(0,0)*{
\begin{tikzpicture} [scale=.5,fill opacity=0.2]
	\draw[very thick, directed=.65] (2,1) to [out=180,in=0] (1,.75);
	\draw[very thick, directed=.55] (1,.75) to [out=135,in=0] (-2,1);
	\draw[very thick, directed=.55] (2.5,0) to [out=180,in=315] (-.5,.25);
	\draw[very thick, directed=.55] (-.5,.25) to [out=180,in=0] (-1.5,0);
	\draw[very thick, directed=.55] (1,.75) to [out=225,in=45] (-.5,.25);
\end{tikzpicture}
};
\endxy
\xrightarrow{}
\xy
(0,0)*{
\begin{tikzpicture} [scale=.5,fill opacity=0.2]
	\draw[very thick, directed=.65] (2,1) to [out=180,in=0] (1,.75);
	\draw[very thick, directed=.55] (1,.75) to [out=135,in=45] (-1.5,.5);
	\draw[very thick, directed=.55] (2.5,0) to [out=180,in=315] (-.5,.25);
	\draw[very thick, directed=.55] (-.5,.25) to [out=180,in=315] (-1.5,.5);
	\draw[very thick, directed=.55] (1,.75) to [out=225,in=45] (-.5,.25);
	\draw[very thick, directed=.55] (-1.5,.5) to (-2.5,.5);
	\draw[very thick, directed=.55] (-2.5,.5) to [out=135,in=0] (-4,1);
	\draw[very thick, directed=.55] (-2.5,.5) to [out=225,in=0] (-3.5,0);
\end{tikzpicture}
};
\endxy
\xrightarrow{\cong}
\xy
(0,0)*{
\begin{tikzpicture} [scale=.5,fill opacity=0.2]
	\draw[very thick, directed=.55] (4,4) to [out=180,in=0] (3,4.25);
	\draw[very thick] (3,4.25) to [out=135,in=45] (2,4.25);
	\draw[very thick] (3,4.25) to [out=225,in=315] (2,4.25);
	\draw[very thick, directed=.55] (2,4.25) to [out=180,in=45] (1,3.5);
	\draw[very thick,directed=.55] (4.5,3) to [out=180,in=315] (1,3.5);
	\draw[very thick, directed=.55] (1,3.5) to (-.5,3.5);
	\draw[very thick, directed=.55] (-.5,3.5) to [out=225,in=0] (-1.5,3);
	\draw[very thick, directed=.55] (-.5,3.5) to [out=135,in=0] (-2,4);
\end{tikzpicture}
};
\endxy
\xrightarrow{}
\xy
(0,0)*{
\begin{tikzpicture} [scale=.5,fill opacity=0.2]
	\draw[very thick, directed=.55] (2,4) to [out=180,in=45] (1,3.5);
	\draw[very thick,directed=.55] (2.5,3) to [out=180,in=315] (1,3.5);
	\draw[very thick, directed=.55] (1,3.5) to (-.5,3.5);
	\draw[very thick, directed=.55] (-.5,3.5) to [out=225,in=0] (-1.5,3);
	\draw[very thick, directed=.55] (-.5,3.5) to [out=135,in=0] (-2,4);
\end{tikzpicture}
};
\endxy
\quad .
\]
The corresponding morphism of linear factorizations is the composition
\[
\xymatrix{
\frac{\Sym(\W|\M|\Y)}{\langle \W = \M \cup \Y \rangle} \otimes \frac{\Sym(\V|\M|\X)}{\langle \V \cup \M=\X \rangle}
\ar[d]
& \overline{1} \otimes \overline{1} 
\ar@{|->}[d]
\\
\frac{\Sym(\V | \W | \L)}{\langle \V \cup \W = \L \rangle} \otimes \frac{\Sym(\L | \S | \T)}{\langle \L=\S \cup \T \rangle}
\otimes
\frac{\Sym(\T|\M|\Y)}{\langle \T = \M \cup \Y \rangle} \otimes \frac{\Sym(\S|\M|\X)}{\langle \S \cup \M = \X \rangle}
\ar[d]
& \displaystyle \sum_{\alpha \in P(a+b-c,c)} (-1)^{|\widehat{\alpha}|} \;
\overline{\pi_{\widehat{\alpha}}^{\V}} \otimes \overline{1} \otimes \overline{1} \otimes \overline{\pi_{\alpha}^{\M}}
\ar@{|->}[d]
\\
\frac{\Sym(\V | \W | \L)}{\langle \V \cup \W =\L \rangle} \otimes \frac{\Sym(\L | \P | \Y)}{\langle \L= \P \cup \Y \rangle}
\otimes \frac{\Sym(\P | \S | \M)}{\langle \P = \S \cup \M \rangle} \otimes \frac{\Sym(\S | \M | \X)}{\langle \S \cup \M = \X \rangle}
\ar[d]
&
\displaystyle \sum_{\begin{smallmatrix}\alpha \in P(a+b-c,c) \\ \beta \in P(b-c) \\ \gamma \in P(a) \end{smallmatrix}} 
c_{\beta,\gamma}^{\alpha} (-1)^{|\widehat{\alpha}|} \;
\overline{\pi_{\widehat{\alpha}}^{\V}} \otimes \overline{\pi_{\gamma}^{\Y}} \otimes \overline{\pi_{\beta}^{\M}} \otimes \overline{1}
\ar@{|->}[d]
\\
\frac{\Sym(\V | \W | \L)}{\langle \V \cup \W = \L \rangle} \otimes \frac{\Sym(\L | \X | \Y)}{\langle \L = \X \cup \Y \rangle}
&
\displaystyle \sum_{\gamma \in P(a,c)} (-1)^{|\widehat{\gamma}|} \; \overline{\pi_{\widehat{\gamma}}^{\V}} \otimes \overline{\pi_{\gamma}^{\Y}}
}
\]
\comm{\[
\xymatrix{
&
\frac{\Sym(\W|\M|\Y)}{\langle \W = \M \cup \Y \rangle} \otimes \frac{\Sym(\V|\M|\X)}{\langle \V \cup \M=\X \rangle}
\ar[d]
\\
& 
\frac{\Sym(\V | \W | \L)}{\langle \V \cup \W = \L \rangle} \otimes \frac{\Sym(\L | \S | \T)}{\langle \L=\S \cup \T \rangle}
\otimes
\frac{\Sym(\T|\M|\Y)}{\langle \T = \M \cup \Y \rangle} \otimes \frac{\Sym(\S|\M|\X)}{\langle \S \cup \M = \X \rangle}
\ar[d]
\\
& 
\frac{\Sym(\V | \W | \L)}{\langle \V \cup \W =\L \rangle} \otimes \frac{\Sym(\L | \P | \Y)}{\langle \L= \P \cup \Y \rangle}
\otimes \frac{\Sym(\P | \S | \M)}{\langle \P = \S \cup \M \rangle} \otimes \frac{\Sym(\S | \M | \X)}{\langle \S \cup \M = \X \rangle}
\ar[d]
\\
&
\frac{\Sym(\V | \W | \L)}{\langle \V \cup \W = \L \rangle} \otimes \frac{\Sym(\L | \X | \Y)}{\langle \L = \X \cup \Y \rangle}
}
\]
given by
\[
\xymatrix{
&
\overline{1} \otimes \overline{1} 
\ar@{|->}[d]\\
& \displaystyle \sum_{\alpha \in P(a+b-c,c)} (-1)^{|\widehat{\alpha}|} \;
\overline{\pi_{\widehat{\alpha}}^{\V}} \otimes \overline{1} \otimes \overline{1} \otimes \overline{\pi_{\alpha}^{\M}} 
\ar@{|->}[d] \\ 
& \displaystyle \sum_{\begin{smallmatrix}\alpha \in P(a+b-c,c) \\ \beta \in P(b-c) \\ \gamma \in P(a) \end{smallmatrix}} 
c_{\beta,\gamma}^{\alpha} (-1)^{|\widehat{\alpha}|} \;
\overline{\pi_{\widehat{\alpha}}^{\V}} \otimes \overline{\pi_{\gamma}^{\Y}} \otimes \overline{\pi_{\beta}^{\M}} \otimes \overline{1}
\ar@{|->}[d] \\ 
& 
\displaystyle \sum_{\gamma \in P(a,c)} (-1)^{|\widehat{\gamma}|} \; \overline{\pi_{\widehat{\gamma}}^{\V}} \otimes \overline{\pi_{\gamma}^{\Y}} }
\]}
where we use the fact that $\overline{\pi_{\beta}^{\M}} \otimes \overline{1} \mapsto 0$ under the last map if $| \beta | \leq c(b-c)$. Given this, 
the only time the Littlewood-Richardson coefficient $c_{\beta,\gamma}^{\alpha}$ is non-zero is when $\beta = c^{b-c}$ 
(so $\gamma \in P(a,c)$ and $\widehat{\alpha} = \widehat{\gamma}$) in which case it equals one.
Both of the relations then follow from the descriptions of these maps.
\newline

\noindent \underline{Equations (QR 3.15) and (QR 3.16)}:
The linear factorization stabilized by the matrix factorization corresponding to the top and bottom webs 
in equation (QR 3.15) is 
\[
\left(
\frac{\Sym(\P|l|w)}
{\left\langle \P=\{ l, w\} \right\rangle}
\otimes
\frac{\Sym(w|\W|\L)}
{\langle  w  \cup \W = \L \rangle}
\right)
\otimes
\left(
\frac{\Sym(\L|\M|z)}
{\langle \L = \M \cup z \rangle}
\otimes
\frac{\Sym(l|\M|\X)}{\langle l\cup \M  = \X\rangle }
\right)
\]
where all of the tensor products are over polynomials partially symmetric in the common variables. 
The map between linear factorizations corresponding to the first term on the left-hand side of
equation (QR 3.15) is determined by the fact that it sends
\[
\overline{1} \otimes \overline{1} \otimes \overline{1} \otimes \overline{1} \mapsto 
\overline{1} \otimes \overline{1} \otimes \overline{1} \otimes \overline{1}
\quad \text{and} \quad
\overline{w} \otimes \overline{1} \otimes \overline{1} \otimes \overline{1} \mapsto 
\overline{1} \otimes \overline{1} \otimes \overline{z} \otimes \overline{1}
\]
and the second term is  determined by
\[
\overline{1} \otimes \overline{1} \otimes \overline{1} \otimes \overline{1} \mapsto 0
\quad \text{and} \quad
\overline{w} \otimes \overline{1} \otimes \overline{1} \otimes \overline{1} \mapsto 
\overline{1} \otimes \overline{1} \otimes \overline{z} \otimes \overline{1} - 
\overline{w} \otimes \overline{1} \otimes \overline{1} \otimes \overline{1} .
\]
The difference between these two maps is thus the identity, confirming the relation. 
The check of equation (3.16) is completely analogous.
\newline

\noindent \underline{Equations (QR 3.17) -- (QR 3.20)}:
The left-hand side of equation (QR 3.17) corresponds to the morphism of linear factorizations
\[
\xymatrix{
\frac{\Sym(\V|\L|\M)}{\langle \V = \L \cup \M \rangle} \otimes \frac{\Sym(\P|\L|\W)}{\langle \P \cup \L = \W \rangle} 
\otimes \frac{\Sym(\M|\X|\Y)}{\langle \M = \X \cup \Y \rangle} \ar[r]
& \frac{\Sym(\P|\V|\S)}{\langle \P \cup \V = \S \rangle} \otimes \frac{\Sym(\S|\T|\Y)}{\langle \S = \T \cup \Y \rangle} 
\otimes \frac{\Sym(\T|\W|\X)}{\langle \T = \W \cup \X \rangle}
}
\]
given by 
$\overline{1} \otimes \overline{1} \otimes \overline{1} \mapsto
\displaystyle \sum_{\begin{smallmatrix} \alpha \in P(b,d) \\ \beta \in P(a,d) \end{smallmatrix}} (-1)^{|\widehat{\alpha}| + |\widehat{\beta}|}
\overline{\pi_{\widehat{\alpha}}^{\P} \pi_{\widehat{\beta}}^{\P}} \otimes \overline{\pi_{\beta}^{\Y}} \otimes \overline{\pi_{\alpha}^{\X}}$,
while the right-hand side is given by
$\overline{1} \otimes \overline{1} \otimes \overline{1} \mapsto
\displaystyle \sum_{\gamma \in P(a+b,d)} c_{\alpha,\beta}^{\gamma} (-1)^{|\widehat{\gamma}|}
\overline{\pi_{\widehat{\gamma}}^{\P}} \otimes \overline{\pi_{\beta}^{\Y}} \otimes \overline{\pi_{\alpha}^{\X}}$. 
The relation then holds since
\[
\pi_{\widehat{\alpha}}^{\P} \pi_{\widehat{\beta}}^{\P} = \sum_{\widehat{\gamma}} 
c_{\widehat{\alpha}\widehat{\beta}}^{\widehat{\gamma}} \pi_{\widehat{\gamma}}^{\P} \quad \text{ and } \quad c_{\widehat{\alpha}\widehat{\beta}}^{\widehat{\gamma}} = c_{\alpha,\beta}^{\gamma}.
\]

Relation (QR 3.18) holds since both sides are given by the map
\[
\xymatrix{
\frac{\Sym(\P|\V|\S)}{\langle \P \cup \V = \S \rangle} \otimes \frac{\Sym(\S|\T|\Y)}{\langle \S= \T \cup \Y \rangle} 
\otimes \frac{\Sym(\T|\W|\X)}{\langle \T = \W \cup \X \rangle} \ar[r]
& \frac{\Sym(\V|\L|\M)}{\langle \V = \L \cup \M \rangle} \otimes \frac{\Sym(\P|\L|\W)}{\langle  \P \cup \L = \W \rangle} 
\otimes \frac{\Sym(\M|\X|\Y)}{\langle  \M = \X \cup \Y \rangle} 
}
\]
sending $\overline{1} \otimes \overline{1} \otimes \overline{1} \mapsto \overline{1} \otimes \overline{1} \otimes \overline{1}$. 
The final two relations follow via similar computations.
\newline

\noindent \underline{Isotopy relations}:
All isotopy relations follow from the fact that both way to construct the ``sideways crossings'' give the same map 
in $\HMF$, and the fact that the foam relation
\[
\xy
(0,0)*{
\begin{tikzpicture} [scale=.6,fill opacity=0.2]
	\path[fill=green] (3.75,0) to (2.25,0) to (2.25,1.5) to [out=90,in=0] (1.875,2) to [out=180,in=90] (1.5,1.5) to [out=270,in=0] (1.125,1) to
		[out=180,in=270] (.75,1.5) to (.75,3) to (3.75,3);
	\path[fill=green] (-.5,-.5) to [out=0,in=225] (2.25,0) to (2.25,1.5) to [out=90,in=0] (1.875,2) to [out=180,in=90] (1.5,1.5) to [out=270,in=0] (1.125,1) to
		[out=180,in=270] (.75,1.5) to (.75,3) to [out=225,in=0] (-.5,2.5);
	\path[fill=green] (-1,.5) to [out=0,in=135] (2.25,0) to (2.25,1.5) to [out=90,in=0] (1.875,2) to [out=180,in=90] (1.5,1.5) to [out=270,in=0] (1.125,1) to
		[out=180,in=270] (.75,1.5) to (.75,3) to [out=135,in=0] (-1,3.5);
	\draw [very thick,directed=.55] (3.75,0) to (2.25,0);
	\draw [very thick,directed=.55] (2.25,0) to [out=135,in=0] (-1,.5);
	\draw [very thick,directed=.55] (2.25,0) to [out=225,in=0] (-.5,-.5);
	\draw[very thick, red, directed=.55] (2.25,0) to (2.25,1.5) to [out=90,in=0] (1.875,2) to [out=180,in=90] (1.5,1.5) to [out=270,in=0] (1.125,1) to
		[out=180,in=270] (.75,1.5) to (.75,3);
	\draw [very thick] (3.75,3) to (3.75,0);
	\draw [very thick] (-1,3.5) to (-1,.5);
	\draw [very thick] (-.5,2.5) to (-.5,-.5);
	\draw [very thick,directed=.55] (3.75,3) to (.75,3);
	\draw [very thick,directed=.55] (.75,3) to [out=135,in=0] (-1,3.5);
	\draw [very thick,directed=.55] (.75,3) to [out=225,in=0] (-.5,2.5);
\end{tikzpicture}
};
\endxy
\sim
\xy
(0,0)*{
\begin{tikzpicture} [scale=.6,fill opacity=0.2]
	\path[fill=green] (2.25,3) to (.75,3) to (.75,0) to (2.25,0);
	\path[fill=green] (.75,3) to [out=225,in=0] (-.5,2.5) to (-.5,-.5) to [out=0,in=225] (.75,0);
	\path[fill=green] (.75,3) to [out=135,in=0] (-1,3.5) to (-1,.5) to [out=0,in=135] (.75,0);	
	\draw [very thick,directed=.55] (2.25,0) to (.75,0);
	\draw [very thick,directed=.55] (.75,0) to [out=135,in=0] (-1,.5);
	\draw [very thick,directed=.55] (.75,0) to [out=225,in=0] (-.5,-.5);
	\draw[very thick, red, directed=.55] (.75,0) to (.75,3);
	\draw [very thick] (2.25,3) to (2.25,0);
	\draw [very thick] (-1,3.5) to (-1,.5);
	\draw [very thick] (-.5,2.5) to (-.5,-.5);
	\draw [very thick,directed=.55] (2.25,3) to (.75,3);
	\draw [very thick,directed=.55] (.75,3) to [out=135,in=0] (-1,3.5);
	\draw [very thick,directed=.55] (.75,3) to [out=225,in=0] (-.5,2.5);
\end{tikzpicture}
};
\endxy
\]
and its analogs are satisfied in $\HMF$. Both are direct computations.
\newline

\noindent \underline{Dot relation}:
Finally, the foam relation
\[
P\left( \;
\xy
(0,0)*{
\begin{tikzpicture} [scale=.35,fill opacity=0.2]
	\path[fill=green] (2,2) to (2,-2) to (-2,-2) to (-2,2); 
	\draw[very thick] (2,-2) to (-2,-2);
	\draw[very thick] (2,-2) to (2,2);
	\draw[very thick] (-2,-2) to (-2,2);
	\draw[very thick] (2,2) to (-2,2);
	\node[opacity=1] at (1.4,1.65) {\tiny$_{1}$};
	\node[opacity=1] at (0,0) {$\bullet$};
\end{tikzpicture}
};
\endxy \;
\right)
= 0
\]
holds via a direct computation that multiplication by $P(X)$ is null-homotopic 
in the endomorphism algebra of the Kozsul factorization
$\{ Q(X) - Q(Y) , X-Y \}$ over $\C[X,Y]$. 
\end{proof}

\section{The link invariant}
\label{section-decomp} 

In this section, we assign a complex $\LC{\tau}{\Sigma}{}$ of webs and foams to certain 
labeled\footnote{In the study of quantum invariants, links and tangles are usually referred to as ``colored'' by representations of a Lie algebra (or more precisely, a quantum group). 
Since we reserve the word colored for webs and foams colored by idempotents, recall that we instead use the non-standard terminology ``labeled'', which agrees with our use of this word for webs.} 
tangle diagrams $\tau$, 
which, up to homotopy equivalence, is an invariant of the corresponding labeled tangle. 
We then show how to obtain a link homology isomorphic to that defined by Wu \cite{Wu2} from this invariant, 
proving Theorem \ref{mainthm2}. 
Finally, we use the foam technology to prove Theorem \ref{mainthm}.

The most precise setting for this invariant is in a certain limiting version of $\Foam{N}^\Sigma$. 
Note that $\Foam{N}^\Sigma$ is the direct sum of foam categories $\Foam{N}^\Sigma(K)$, where $K=\sum_{i=1}^m a_i$ 
is the sum of the entries in an object $(a_1,\ldots,a_m)$. 
We have a 2-functor $\Foam{N}^\Sigma(K) \to \Foam{N}^\Sigma(K+N)$ given by taking disjoint union with an $N$-labeled edge/facet. 
The natural setting for the tangle invariant\footnote{Here $k$ depends on the boundary and labeling of the tangle.} 
is the direct limit 
\begin{equation*}\displaystyle \Foam{N}(k+N\infty)^\Sigma := \lim_{\stackrel{\longrightarrow}{s}} \Foam{N}^\Sigma(k+Ns); 
\end{equation*}
however, the invariant can be viewed in $\Foam{N}^\Sigma(k+Ns)$ for $s$ sufficiently large. 

We begin by defining $\LC{\tau}{\Sigma}{}$ on generating tangles, and then explain how to define the invariant for general tangles. 
Given a labeled, oriented tangle diagram $\tau$, let $c_1,\ldots,c_r$ be the labels of the right endpoints and $d_1,\ldots,d_l$ be the labels of the 
left endpoints. Set 
\[
\mathcal{O}_R(c_i) = 
\begin{cases}
c_i & \text{if $\tau$ is directed out from the $i^{\mathrm{th}}$ endpoint,} \\
N-c_i & \text{if $\tau$ is directed into the $i^{\mathrm{th}}$ endpoint.}
\end{cases}
\]
\[
\mathcal{O}_L(d_i) = 
\begin{cases}
d_i & \text{if $\tau$ is directed into the $i^{\mathrm{th}}$ endpoint,} \\
N-d_i & \text{if $\tau$ is directed out from the $i^{\mathrm{th}}$ endpoint.}
\end{cases}
\]
then $\llbracket \tau \rrbracket$ is defined to be a complex in the $\Hom$-category 
\[
\Hom \big( \left(N,\ldots,N,\mathcal{O}_R(c_1),\ldots,\mathcal{O}_R(c_r)\right), \left(N,\ldots,N,\mathcal{O}_L(d_1),\ldots,\mathcal{O}_L(d_l)\right) \big)
\]
of $N\mathbf{Foam}\left({\displaystyle \sum_{i=1}^r \cal{O}_R(c_i)+Ns}\right)^\Sigma$.

For labeled cap and cup tangles we set
\begin{align*}
\left \llbracket
\xy
(0,0)*{\begin{tikzpicture} [scale=.5]
\draw[very thick, directed=.99] (0,1) to [out=180,in=90] (-1.25,.5) to [out=270,in=180] (0,0);
\node at (-1.5,1) {$_a$};
\end{tikzpicture}};
\endxy 
\;\; \right \rrbracket^\Sigma
\;\; = \;\;
\xy
(0,0)*{\begin{tikzpicture} [scale=.5]
 \draw [double] (-1.25,.5) to (-2.25,.5);
 \draw [very thick, directed=.55] (0,0) to [out=180,in=300] (-1.25,.5);
\draw [very thick, directed=.55] (0,1) to [out=180,in=60] (-1.25,.5);
 \node at (.875,0) {$_{N-a}$};
 \node at (.5,1) {$_a$};
 \node at (-2.75,.5) {$_N$}; 
\end{tikzpicture}};
\endxy
\;\; &, \;\;
\left \llbracket \;\;
\xy
(0,0)*{\begin{tikzpicture} [scale=.5]
\draw[very thick, directed=.99] (0,0) to [out=0,in=270] (1.25,.5) to [out=90,in=0] (0,1);
\node at (1.5,1) {$_a$};
\end{tikzpicture}};
\endxy
\right \rrbracket^\Sigma
\;\; = \;\;
\xy
(0,0)*{\begin{tikzpicture} [scale=.5]
 \draw [double] (2.25,.5) -- (1.25,.5);
 \draw [very thick, directed=.55] (1.25,.5) to [out=240,in=0] (0,0);
\draw [very thick, directed=.55] (1.25,.5) to [out=120,in=0] (0,1);
 \node at (-.875,0) {$_{N-a}$};
 \node at (-.5,1) {$_a$};
 \node at (2.75,.5) {$_N$}; 
\end{tikzpicture}};
\endxy \\
\left \llbracket
\xy
(0,0)*{\begin{tikzpicture} [scale=.5]
\draw[very thick, rdirected=.05] (0,1) to [out=180,in=90] (-1.25,.5) to [out=270,in=180] (0,0);
\node at (-1.5,1) {$_a$};
\end{tikzpicture}};
\endxy
\;\; \right \rrbracket^\Sigma
\;\; = \;\;
\xy
(0,0)*{\begin{tikzpicture} [scale=.5]
 \draw [double] (-1.25,.5) to (-2.25,.5);
 \draw [very thick, directed=.55] (0,0) to [out=180,in=300] (-1.25,.5);
\draw [very thick, directed=.55] (0,1) to [out=180,in=60] (-1.25,.5);
 \node at (.875,1) {$_{N-a}$};
 \node at (.5,0) {$_a$};
 \node at (-2.75,.5) {$_N$}; 
\end{tikzpicture}};
\endxy
\;\; &, \;\;
\left \llbracket \;\;
\xy
(0,0)*{\begin{tikzpicture} [scale=.5]
\draw[very thick, rdirected=.05] (0,0) to [out=0,in=270] (1.25,.5) to [out=90,in=0] (0,1);
\node at (1.5,1) {$_a$};
\end{tikzpicture}};
\endxy
\right \rrbracket^\Sigma
\;\; = \;\;
\xy
(0,0)*{\begin{tikzpicture} [scale=.5]
 \draw [double] (2.25,.5) -- (1.25,.5);
 \draw [very thick, directed=.55] (1.25,.5) to [out=240,in=0] (0,0);
\draw [very thick, directed=.55] (1.25,.5) to [out=120,in=0] (0,1);
 \node at (-.875,1) {$_{N-a}$};
 \node at (-.5,0) {$_a$};
 \node at (2.75,.5) {$_N$}; 
\end{tikzpicture}};
\endxy
\end{align*}
and for labeled, left-directed crossings we use homological shifts of the Rickard complexes $\cal{T}\mathbf{1}_{(a,b)}$ from \eqref{Rickardp} and \eqref{Rickardn} to set
\begin{align*}
\left \llbracket
\xy (0,0)*{
\tikzcrossingp[1]{a}{b}{}{}{}{}
};
(12,-3)*{_b};(12,3)*{_a};
\endxy \right \rrbracket^\Sigma \!\! &= \;\;
\Phi_\Sigma(\cal{T}\mathbf{1}_{(a,b)}) [\min(a,b)]\\
\left \llbracket
\xy (0,0)*{
\tikzcrossing[1]{a}{b}{}{}{}{}
};
(12,-3)*{_b};(12,3)*{_a};
\endxy \right \rrbracket^\Sigma \!\! &= \;\;
\Phi_\Sigma(\cal{T}^{-1} \mathbf{1}_{(a,b)}) [-\min(a,b)]
\end{align*}
where here $[-]$ here denotes a shift\footnote{In \cite{QR}, the crossing also involved a shift in the quantum grading; 
however, we omit it from this definition since this grading is broken in $\foam{N}{}^\Sigma$.}
in homological degree.

\begin{exa} The complex assigned to a negative crossing with $a\geq b$, compare with \eqref{RickardInvn}, is 
\begin{eqnarray*}
\eqnM
\end{eqnarray*}
where the \uwave{underlined} term is in homological degree zero and the differential is given by:
\begin{equation*}\label{diff} d_k\quad := \quad
\diff .
\end{equation*}
\end{exa}

Every labeled tangle admits a diagram given as the horizontal composition $\otimes$ of tangles which are the disjoint union $\sqcup$ of labeled, directed identity tangles with one of the tangles on which we've already defined the invariant. 
We define $\LC{\tau}{\Sigma}{}$ on the disjoint union of identity tangles and a crossing by first taking the disjoint union of $\Phi_\Sigma(\tilde{\cal{T}}\mathbf{1}_{(a,b)})$ 
or $\Phi_\Sigma(\tilde{\cal{T}}^{-1}\mathbf{1}_{(a,b)})$ with the identity webs (resp. foams) corresponding to the identity tangle, then taking the disjoint union with $N$-labeled strands (resp. facets).
Finally, we define the invariant on the disjoint union of identity tangles with a cap or cup by taking the disjoint union of the relevant webs with the corresponding identity webs, 
then repeatedly horizontally composing $\otimes$ with webs 
$
\xy
(0,0)*{\begin{tikzpicture} [scale=.5]
 \draw [double] (1.25,0) -- (0,0);
\draw [very thick, directed=.55] (1.25,1) -- (-.5,1);
\draw[very thick, directed=.55] (0,0) to (-1.75,0);
\draw[double] (-.5,1) to (-1.75,1);
 \draw[very thick, directed=.55] (0,0) to (-.5,1); 
 \node at (1.75,0) {\small$N$};
 \node at (1.75,1) {\small$a$}; 
\end{tikzpicture}};
\endxy
$
or 
$
\xy
(0,0)*{\begin{tikzpicture} [scale=.5]
 \draw [double] (3,1) -- (1.75,1);
 \draw[very thick, directed=.55] (3,0) to (1.25,0);
 \draw [double] (1.25,0) -- (0,0);
\draw [very thick, directed=.55] (1.75,1) -- (0,1);
 \draw [very thick, directed=.55] (1.75,1) -- (1.25,0);
 \node at (3.5,0) {\small$a$};
 \node at (3.5,1) {\small$N$}; 
\end{tikzpicture}};
\endxy
$
to obtain a web mapping between objects where the top-most label is $N$ in both 
the domain and codomain, and then taking the disjoint union with $N$-labeled strands. Note that the action of $\sqcup$ as well as $\otimes$ on 
complexes is modelled on the tensor product of chain complexes, exactly as in Bar-Natan's canopolis formalism \cite{BN}. 

\begin{exa} For the Hopf link we use the following ladder-type link diagram:
\begin{equation*}
\xy
(0,0)*{\begin{tikzpicture} [scale=.5]
 \draw[double] (2.5,1) to (2,1) to  (1.5,2) to (0.5,2);
  \draw[double] (3,3) to (0.5,3);
 \draw[very thick, directed=.55, directed=.13, directed=.92] (7.5,1) to (7,0) to (3,0) to (2.5,1);
 \draw[very thick, directed=.55, directed=.90] (5.8,1.3) to (5.5,1) to (5,1) to (4,2) to (3.5,2) to (3,3);
 \draw[very thick, directed=.32,directed=.95] (8,3) to (7.5,2) to (6.5,2) to (6.2,1.7) ;   
  \draw[very thick, directed=.55] (8,3) to (3,3); 
   \draw[double] (9.5,3) to  (8,3);
   \draw[double] (9.5,2) to (8.5,2) to (8,1) to  (7.5,1);
    \draw[very thick, directed=.15] (4.3,1.3) to (4,1) to (2.5,1);
        \draw[very thick, directed=.4] (7.5,1) to (6.5,1) to (5.5,2) to (5,2) to (4.7,1.7) ;
 \node at (0,2) {\small$N$};
 \node at (0,3) {\small$N$}; 
 \node at (10,2) {\small$N$};
 \node at (10,3) {\small$N$}; 
 \node at (8,0.4) {\small$_{N-i}$}; 
 \node at (8.2,2.4) {\small$_{j}$}; 
 \node at (2.8,2.4) {\small$_{j}$}; 
 \node at (3.3,1.3) {\small$_{i}$};
 \node at (7,1.3) {\small$_{i}$}; 
 \node at (2,0.4) {\small$_{N-i}$}; 
\end{tikzpicture}};
\endxy.
\end{equation*} 
\end{exa}

\begin{prop}
Given a oriented, framed, labeled tangle $\tau$, the complex $\LC{\tau}{\Sigma}{}$ is independent, up to homotopy, of the diagram used.
\end{prop}
\begin{proof}
Exactly the same as in \cite[Theorem 4.5]{QR}.
\end{proof}

In the case that the tangle is actually a labeled link $\cal{L}$, all of the boundary points in the complex $\LC{\mathcal{L}}{\Sigma}{}$ are 
$N$-labeled and all webs in it are endomorphisms of a highest weight object of the form $\mathbf{o}^{top}:=(N,\dots, N)$. 
Hence we can apply the representable functor 
\begin{equation*}
\taut(-):= \Hom(\mathbf{o}^{top}, - )
\end{equation*} 
to $\LC{\mathcal{L}}{\Sigma}{}$ to obtain a complex of vector spaces. 
Moreover, we claim that each term in this complex is finite-dimensional. 
Indeed, every web in $\End(\mathbf{o}^{top})$ is isomorphic to a (finite) direct sum of identity webs $1_{\mathbf{o}^{top}}$. 
Foam facets with label $N$ are additively indecomposable, since the only admissible coloring by idempotents is given by the full multiset $\Sigma$. It follows that endomorphisms of $1_{\mathbf{o}^{top}}$ are all given by the images of closed diagrams in 
$\Ucatc_Q(\glm)$, which act by scalars in $\foam{N}{}^\Sigma$, confirming our claim.

Denote by $\LH{\mathcal{L}}{\Sigma}$ the homology of this complex. 

\begin{thm}
Up to shifts in homological degree, $\LH{\mathcal{L}}{\Sigma}$ is isomorphic to Wu's colored, deformed Khovanov-Rozansky homology of the mirror link $\cal{L}'$.
\end{thm}

\begin{proof}
This result follows in the spirit of the proof of \cite[Theorem 4.12]{QR}.
We cannot directly apply the methods there, however, since the 2-functor $\Ucatc_Q(\glm) \to \Foam{N}^\Sigma$ is not 
a 2-representation in the strict sense, as it doesn't preserve the grading.

Nevertheless, we can consider the 2-category $\Ucatc_Q^{0\leq N}(\glm)^\Sigma$ where we've imposed relation 
\eqref{eq-doteq} in each weight. 
This implies the specifications of fake bubble parameters in highest weight to elementary symmetric functions evaluated at $\Sigma$.
Given a labeled link $\cal{L}$, we can pull the complex $\LC{\mathcal{L}}{\Sigma}{}$ back to $\Ucatc_Q^{0\leq N}(\glm)^\Sigma$ and 
simplify until each term the complex only consists of direct sums of the identity 1-morphism on the highest weight $(N,\ldots,N,0,\ldots,0)$ in $\Ucatc_Q^{0\leq N}(\glm)^\Sigma$ 
(which maps to the object $\mathbf{o}^{top}$ under $\Phi_\Sigma$).

The homology of the link can be computed entirely in the context of $\Ucatc_Q^{0\leq N}(\glm)^\Sigma$. 
Moreover, similarly to the case discussed in \cite{QR}, any link homologies defined using the images of 
the Rickard complexes in a ``skew Howe'' 2-representation factoring through $\Ucatc_Q^{0\leq N}(\glm)^\Sigma$ must agree
(see the work of Cautis \cite{Cau2} for the first appearance of this idea). 

Deformed foams give such a 2-representation, as do deformed matrix factorizations, via the 2-functors $\Gamma_m$ from 
the proof of Theorem \ref{thm-nondegen}. 
Note that the link homology theory defined using $\Gamma_m$ is not defined in exactly the same way as in Wu's work. 
Indeed, there are the following differences: Wu's assignment of complexes to link crossings is opposite to ours, 
he has shifts in homological degree for some crossings in order to obtain invariance under the first Reidemeister move, 
and he does not use matrix factorizations associated to $N$-labeled web edges. 

Nevertheless it is easy to see the relation between our invariant and Wu's. 
Given a labeled braid, the 2-functor $\Gamma_m$ assigns a to it a complex of matrix factorizations which, 
up to shifts in homological degree, 
agrees with the complex of matrix factorizations Wu assigns to the mirror image of the braid. 
It hence suffices to show that the vector spaces and differentials in this complex after closing 
the braid agrees with the those obtained by closing using $N$-labeled edges and taking 
$\Hom$ from $\mathbf{o}^{top}$. 
This follows exactly as in \cite[Theorem 4.12]{QR}.
\end{proof}

\begin{rem}
\label{flipRickard}
As a variation of equation \eqref{RickardInvp}, where $a_i\geq a_{i+1}$, consider the complex
\begin{equation*}
\onea  \cal{T}'^{-1}_i =
\xymatrix{ \cdots \ar[r]^-{d'_{s+1}}  & \onea \cal{E}_i^{(a_i-a_{i+1}+s)} \cal{F}_i^{(s)} \{-s\} \ar[r]^-{d'_s} & \cdots
\ar[r]^-{d'_2} & \onea \cal{E}_i^{(a_i-a_{i+1}+1)}  \cal{F}_i \{-1\} \ar[r]^-{d'_1}  & \uwave{ \onea \cal{E}_i^{(a_i-a_{i+1})} } }
\end{equation*}
with the \uwave{underlined} term (as usual) in homological degree zero and differentials given by compositions of 
splitters and thickness $1$ cap 2-morphisms.

It is easy to check that $\onea  \cal{T}'^{-1}_i$ is isomorphic to $\onea \cal{T}^{-1}_i $ via the chain map given on objects by 
\[
\pm 
\xy
(0,0)*{
\begin{tikzpicture}[scale=1]
\draw [ultra thick, green, ->] (0,0) to [out=90,in=270] (1,2);
\draw [ultra thick, green, ->] (0,2) to [out=270,in=90] (1,0);
\node [green] at (1,2.5) {\footnotesize $-\lambda+k$}; 
\node [green] at (0,2.5) {\footnotesize $k$}; 
\node at (1.7,1) {$\lambda$};
\end{tikzpicture}
};
\endxy
\]
for a suitable choice of signs. Analogously, the Rickard complexes \eqref{Rickardp} are isomorphic to complexes with objects $\cal{E}_i^{(s)} \cal{F}_i^{(a_i-a_{i+1}+ s)}$ 
and in general we may assume that the complexes $\llbracket - \rrbracket^\Sigma$ associated to crossings consist of webs of shape $\squareweb[1]{}{}{}{}{}{}{}$.
\end{rem}


We now proceed with the decomposition of the invariant.
Consider an oriented, labeled tangle diagram $\tau$, or more specifically, 
an oriented, labeled link diagram $\cal{L}$. 
Our goal is to understand the dependence of $\LC{\tau}{\Sigma}{}$ and $\LH{\cal{L}}{\Sigma}= \mathrm{H}_*(\taut(\LC{\cal{L}}{\Sigma}{}))$ on $\Sigma$. 
This is done in four steps:

\begin{enumerate}
\item In Section \ref{section-sum} we show that $\LC{\tau}{\Sigma}{}$, regarded as a complex over $\hat{\foam{N}{}^\Sigma}$, decomposes into a direct sum of complexes 
$\LC{\tau}{\Sigma}{f}$ indexed by colorings $f$ of the tangle components by multisubsets of $\Sigma$.
\item In Section \ref{section-tensor} we show that the summands $\LC{\tau}{\Sigma}{f}$ from the first step correspond under the splitting functor $\phi$ from Section \ref{2hom} to a tensor product with 
one tensorand $\LC{\tau}{\Sigma}{\lambda\in f}$ for every different root $\lambda \in \Sigma$.
\item In Section \ref{section-unicoloredfoam} we show that foams colored with only one root $\lambda$ behave like $\slnn{N_\lambda}$ foams.
\item In Section \ref{section-assembly} we assemble the previous results for $\tau = \cal{L}$ and track them through relatives of the functor $\taut$ to prove Theorem \ref{mainthm}.
\end{enumerate}

\subsection{The direct sum decomposition of the invariant}
\label{section-sum}
We already know that if we work in $\hat{\foam{N}{}^\Sigma}$, all webs in the complex $\LC{\tau}{\Sigma}{}$ split into direct sums under coloring web edges with multisubsets of $\Sigma$. 
The goal of this section is to show in Lemma \ref{sumdecomp} that the colorings that contribute to $\LC{\tau}{\Sigma}{}$ are the ones that are consistent along tangle components. 
This follows from the orthogonality of idempotents coming from inconsistent colorings, see Corollary \ref{admissiblecrossinglabel}, after observing in Proposition \ref{decorationslidecrossing} that decorations ``slide through crossings''.

\begin{defi} Let $p$, $q$, $r$ and $s$ be symmetric polynomials of the appropriate number of variables, then we define endomorphisms of the chain complexes for negative crossings 
\begin{equation*}
\left \llbracket \xy (0,0)*{
\tikzcrossing[1]{j}{i}{r}{s}{p}{q}};
\endxy \right \rrbracket^\Sigma \in \End \left ( \left \llbracket \xy (0,0)*{
\tikzcrossing[1]{j}{i}{}{}{}{}
};
(12,-3)*{_b};(12,3)*{_a};
\endxy\right \rrbracket^\Sigma
 \right )   \text{ given on webs by } \eqnO 
\end{equation*}
where we have assumed $a \geq b$. For the cases of $a\leq b$ and for the positive crossings we make analogous definitions.
\end{defi}

\begin{prop}
\label{decorationslidecrossing}
Let $p$ and $q$ be symmetric polynomials in the appropriate number of variables. Then the following chain maps are homotopic:

\begin{equation*}
\left \llbracket \xy
(0,0)*{\tikzcrossing[1]{}{}{}{}{p}{}
};
\endxy\right \rrbracket^\Sigma
\sim\;\;
\left \llbracket\xy
(0,0)*{\tikzcrossing[1]{}{}{}{p}{}{}
};
\endxy\right \rrbracket^\Sigma
,\;\;
\left \llbracket\xy
(0,0)*{\tikzcrossing[1]{}{}{}{}{}{q}
};
\endxy\right \rrbracket^\Sigma
\sim\;\;
\left \llbracket\xy
(0,0)*{\tikzcrossing[1]{}{}{q}{}{}{}
};
\endxy\right \rrbracket^\Sigma
\end{equation*}
Analogous statements also hold for the positive crossing. 
\end{prop}
\begin{proof}
From the foam description of these chain maps, it is easy to see that composing such chain maps is equivalent to multiplying decorations on the foam facets. 
Since null-homotopic chain maps form an ideal in the ring of endomorphisms of the crossing complex, it suffices to find homotopies in the cases where $p$ (or $q$) is a complete symmetric polynomial $h_i$. We'll only prove the first homotopy, as the other case is analogous. 
Denote $h_i$ acting on the left as $h_i^l$ and on the right as $h_i^r$. 
Recall that the differential for negative crossing complexes is given using $1$-labeled cap foams. 
We now prove by induction on $i\geq 1$ that the foams
\begin{equation*} \eta_i:=
\eqnQ
\end{equation*}
constructed using $(i-1)$-dotted
 $1$-labeled cup foams, assemble to a chain homotopy from $h_i^l$ to $h_i^r$. 
We start the computation with an equation from \cite[Lemma 4.16]{KLMS}, which under the foamation functor $\Phi_\Sigma$ gives 
\begin{equation*}
(-1)^{b-a} \Phi_\Sigma \left (
  \xy
 (0,0)*{\includegraphics[scale=0.35]{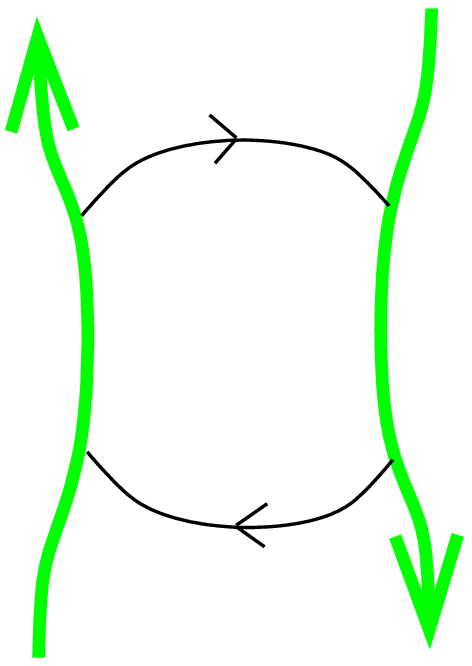}};
 (-7,-14)*{_{k}};(7,-14)*{_{a-b+k}};(-7,13)*{_{k}};(7,13)*{_{a-b+k}};
  (3,-7)*{\bullet}+(0,-2)*{_{i-1}}; (12,-7)*{_{a-b}};
  \endxy  + \; \xy
 (0,0)*{\includegraphics[scale=0.35]{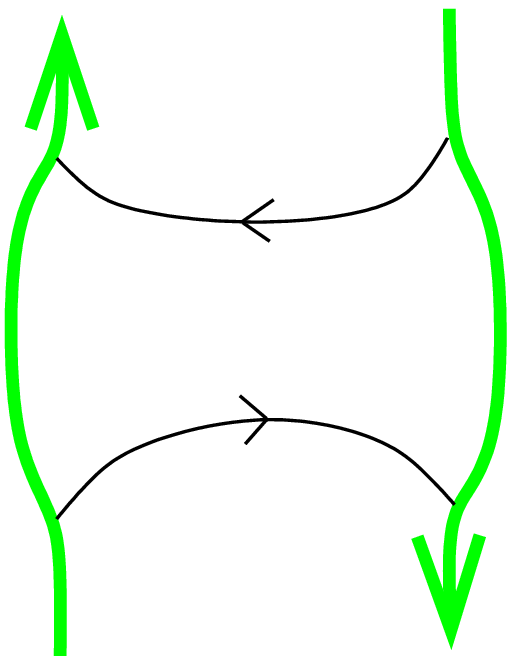}};
 (-7,-14)*{_{k}};(7,-14)*{_{a-b+k}};(-7,13)*{_{k}};(7,13)*{_{a-b+k}}; (12,-7)*{_{a-b}};
  (3,4)*{\bullet}+(0,2)*{_{i-1}};
  \endxy \right ) = \sum_{ \xy
  (0,-1)*{\scs p+q+r =};
  (0,-4)*{\scs i};
  \endxy}\hspace{0in}  \Phi_\Sigma \left ({ 
  \xy
 (-10,0)*{\includegraphics[scale=0.42]{figs/tlong-up.eps}};
 (-10.5,-14)*{_{k}};(-10,-2)*{\bigb{h_p}};
  (10,0)*{\includegraphics[angle=180, scale=0.42]{figs/tlong-up.eps}};
 (10.5,-14)*{_{a-b+k}};(10,-2)*{\bigb{h_q}};
  (0,-2)*{\cbub{\spadesuit +r}{}}; (15,-7)*{_{a-b}};
  \endxy \nn }\right ) 
\end{equation*}
which is a foam identity where the LHS is precisely $\eta_i d + d \eta_i$. We continue the computation, but since from the next step onwards all foams are identity foams with decorations, we only draw the underlying webs and write decorations next to the corresponding web edges. Using \cite[Equation (3.32)]{QR} to resolve the ``bubble'' in the previous step, we get:

\begin{align*}
 \eqnR .
\end{align*}

In the case where $i=1$, the homotopy $\sim$ at the beginning of the second line is an equality. 
This constitutes the start of the induction. For the induction step we use the homotopy of $h_p^l$ and $h_p^r$ for $p<i$ to proceed to the second line. 
This is possible because of the fact, which can easily be checked via the decoration migration relations on foams (see the right-hand side of equation \eqref{FoamRel2}), 
that $\eqnS$ is a chain map.
\end{proof}
\comm{ 
\begin{equation*}
\eqnT
\end{equation*}
\begin{equation*}
\eqnU
\end{equation*}}

\begin{cor} 
\label{admissiblecrossinglabel}
Let $A$, $B$, $C$ and $D$ be multisubsets of $\Sigma$ of the appropriate size and, by abuse of notation, we denote the associated idempotents with the same letter. Then 
\begin{equation*}
\left \llbracket \crossing[1]{}{}{A}{B}{C}{D}  \right\rrbracket^\Sigma
\end{equation*} 
is an idempotent chain map. Furthermore, if $A\neq D$ or $B \neq C$, then it is null-homotopic. Analogous statements hold for positive crossings.

\end{cor}
\begin{proof}
We have already noted that composition of such chain maps corresponds to multiplication of decorations. Thus, the chain map is clearly idempotent. Now suppose that $A\neq D$ or $B \neq C$. Then, using Proposition \ref{decorationslidecrossing}, we have:
\begin{equation*}
 \left \llbracket\crossing[1]{}{}{A}{B}{C}{D} \right\rrbracket^\Sigma  \sim \;\; \left \llbracket\crossing[1]{}{}{A D}{B C}{}{}  \right\rrbracket^\Sigma \sim \;\;0
\end{equation*} 
since $A$ and $D$ or $B$ and $C$ are orthogonal idempotents.
\end{proof}

\begin{lem} \label{adddecompcross} There is a homotopy equivalence of complexes over $\hat{\foam{N}{}^\Sigma}$:
\begin{equation*}
\left \llbracket
\xy (0,0)*{
\tikzcrossing[1]{a}{b}{}{}{}{}
};
(12,-3)*{_b};(12,3)*{_a};
\endxy \right \rrbracket^\Sigma \!\!\sim \;\; \bigoplus_{A,B} \;\left \llbracket\crossingw[1]{a}{b}{A}{B}{B}{A} \right\rrbracket^\Sigma
\end{equation*} 
where the summands on the right-hand side denote the subcomplexes of the complex on the left-hand side obtained by coloring webs and foams by idempotents $A$ and $B$ at the indicated positions. In the direct sum $A$ and $B$ range over all multisubsets of $\Sigma$ of the correct size. The analogous statements hold for positive crossings.
\end{lem}
\begin{proof}
The objects of the complex on the left-hand side, which are webs, split into direct sums according to the definition of $\hat{\foam{N}{}^\Sigma}$ when one colors all boundary edges by idempotents.  
The differential clearly respects this decomposition since it locally looks like an identity foam around the decoration by idempotents. 
Finally Corollary \ref{admissiblecrossinglabel} shows that summands are -
-homotopic if they do not come from a coloring that is consistent along the strands in the crossing, and working in the homotopy category of such complexes we immediately cancel null-homotopic summands. 
\end{proof}
 The following global version of this lemma follows directly:
 
\begin{lem}
\label{sumdecomp}
Let $\tau$ be a labeled, oriented tangle diagram. 
The complex $\LC{\tau}{\Sigma}{}$, regarded over $\hat{\foam{N}{}^\Sigma}$, splits into a direct sum of complexes $\LC{\tau}{\Sigma}{f}$, 
and there is one such piece for every coloring $f$ of tangle components by idempotents corresponding to multisubsets of $\Sigma$ of the appropriate size.
\end{lem}

\subsection{The tensor product decomposition of the summands}
\label{section-tensor}

In this section we show that the functor $\phi$ from Section \ref{2hom} can be used to split idempotent colored summands $\LC{\tau}{\Sigma}{f}$ 
of the chain complex associated to a tangle diagram -- as described at the end of the previous subsection -- 
into the tensor product complex $\bigotimes_\lambda \LC{\tau}{\Sigma}{\lambda \in f}$ of their $\lambda$-components. 
We now define these concepts:

\begin{defi} Let 
\begin{itemize}
\item $\LC{\tau}{\Sigma}{\lambda \in f}$, \emph{ the $\lambda$-component} of $\LC{\tau}{\Sigma}{f}$, 
be the sequence of colored webs and foams between them obtained by taking the $\lambda$-component of every web and foam appearing in $\LC{\tau}{\Sigma}{f}$. 
It is easy to check that $\LC{\tau}{\Sigma}{\lambda \in f}$ is itself a chain complex over $\hat{\foam{N}{}^\Sigma}$.

\item $\bigotimes_\lambda \LC{\tau}{\Sigma}{\lambda \in f}$ be the tensor product complex of $\LC{\tau}{\Sigma}{\lambda \in f}$ given on webs by taking disjoint union $\sqcup$. 
In particular, the webs in this chain complex are exactly the associated split webs $\bigsqcup_\lambda W_\lambda$ of webs $W$ in $\LC{\tau}{\Sigma}{f}$. 
The foams giving the components of the differential in $\bigotimes_\lambda \LC{\tau}{\Sigma}{\lambda \in f}$ are (up to a sign) the disjoint union of the 
$\lambda$-components of the differential foams in $\LC{\tau}{\Sigma}{f}$. 
The sign is the usual sign that is necessary to make the differential in the tensor product complex square to zero. 
\end{itemize}
\end{defi}

Applying the web splitting functor $\phi$ from Section \ref{2hom} to the chain complex $\LC{\tau}{\Sigma}{f}$ results in a chain complex  consisting of exactly the same split webs as 
$\bigotimes_\lambda \LC{\tau}{\Sigma}{\lambda \in f}$. 
Furthermore, there exists a natural choice of homological grading on $\LC{\tau}{\Sigma}{\lambda \in f}$ that makes the bijection between webs in 
$\phi(\LC{\tau}{\Sigma}{f})$ and $\bigotimes_\lambda \LC{\tau}{\Sigma}{\lambda \in f}$ grading preserving. 
This is explained for the local case of a single crossing in Remark \ref{tensorcxrem} and immediately generalizes to $\LC{\tau}{\Sigma}{f}$.  

The main task in this section is to prove Theorem \ref{splitdifferential}, which states that the isomorphism foams $T_W$ in the definition of 
$\phi$ can be chosen so that the differential of $\phi(\LC{\tau}{\Sigma}{f})$ equals the differential of $\bigotimes_\lambda \LC{\tau}{\Sigma}{\lambda \in f}$, 
and we have:
\begin{equation*}\phi(\LC{\tau}{\Sigma}{f}) = \bigotimes_\lambda \LC{\tau}{\Sigma}{\lambda \in f}.
\end{equation*} 

\begin{rem}
\label{tensorcxrem}
Consider the webs in the chain complex associated to a crossing, e.g. with cap-differentials for the sake of concreteness, 
where we have already placed idempotents on all boundary edges of the webs:
\begin{equation*}\eqnAA .
\end{equation*}
Without loss of generality, we assume that $|A|\geq |B|$. If such a web is not isomorphic to the zero web, it decomposes into a direct sum by coloring the interior edges of the web with various idempotents. The crossing complex starts with $W_{k_{max}} = W_{|B|}$ in homological degree zero -- which is isomorphic to the zero web if and only if $A\uplus B \not\subset \Sigma$, but which is indecomposable otherwise. Further, there exists a minimal $k_{min}=|B\setminus A|$ such that $W_{k_{min}}$ is non-zero and indecomposable:
\begin{equation*}\eqnAB .
\end{equation*}
Now consider the target $W_{k_{max}-1}$ of the differential on $W_{k_{max}}$:
 \begin{equation*}\eqnAC .
\end{equation*}
More generally, any non-zero web $W_k$ decomposes into a direct sum of webs which differ in labels and colorings from $W_{k_{max}}$ by a re-routing of a multisubset $C$ of $A\cap B$ around the square: 
\begin{equation*}\squarewebC .
\end{equation*}
Such an indecomposable web is non-zero if and only if $(A \uplus B)\setminus C \subset \Sigma$. 
Clearly $C=A\cap B$ satisfies this because we assume that $A,B\subset \Sigma$. 
Since this condition can be checked for every root individually, there is a minimal $C_{min}$ such that $W_{C_{min}}$ and every $W_C$ for $C_{min}\subset C \subset A\cap B$ is non-zero. 
We can think of the set of admissible $C$ as lying on the lattice $\Z^l$ with the $k^{th}$ coordinate indicating the multiplicity of the $k^{th}$ root in $C$. 
Then it is clear that the homological grading of a web is the sum of the coordinates and the support of the non-zero $W_C$ is an $l$-dimensional box. 
Components of the differential are caps colored by a single root $\lambda_k$ and hence map between summands in which the re-routing sets $C$ differ by $\lambda_k$, i.e. map between lattice points which differ by $1$ in the $k^{th}$ coordinate only. 
The differentials in this complex already appear as the ones coming from a tensor product of complexes -- one for each root $\lambda_k$ -- 
with homological grading the $k^{th}$ coordinate in the lattice and with differential corresponding to the $\lambda_k$ colored cap-differential.
\end{rem}

In the following Lemma we collect commutation relations needed in the proof of Theorem \ref{splitdifferential}. 
In the following graphics red facets are colored with a multisubset containing a single root $\lambda$, blue facets are colored with a multisubset not containing $\lambda$, 
and the coloring of the green facets is uniquely determined or arbitrary -- generically, they contain both $\lambda$ and other roots. 

\begin{lem} \label{commutation} Splitting off or merging a red facet commutes with arbitrary M-P foams, red-blue digon creation, digon removal, zip, and unzip foams, up to certain units. 
\end{lem}
The graphics in the following proof illustrate and make precise these statements.

\begin{proof} 
Throughout the proof of this lemma, the displayed graphics are to be interpreted as local foam pieces. First we consider the case of M-P foams, by which we mean the elementary foams between the two possible two-splitter (two-merger) webs. They are shown in green in the following graphics. 
\begin{equation*}
\eqnAE
\end{equation*}
The first commutation relation follows from a version of foam relation \cite[(3.12)]{QR} and repeated use of relation \eqref{FoamRel1}. The second commutation relation holds because it is an isotoped version of the pitchfork relation \cite[(3.19)]{QR}. 
There are analogous versions of these relations where the seam attaching the red facet to the rest of the foam is reoriented and inclined the other way, 
and another four relations hold for a red facet split off on the back-side of the green foam.  Clearly, red mergers and splitters then also commute with the inverse M-P foams. These 16 commutations relations describe all possible interactions of red splitters and mergers with a M-P foam between splitter-webs. The cases of M-P foams between merger webs is handled similarly.

While the commutation relations with M-P foams are independent of the coloring of foam facets with idempotents, 
this is in general no longer the case for digon creation, digon removal, zip, and unzip foams. 
Instead, we get commutation up to unit decorations, using the relations in Section \ref{relations}. 
If we denote the foams that split off or merge a red facet by $d$ and the foam across which we want to commute it by $X$ then the relations we get take the form:
\begin{equation*}X \circ d = u_1 \circ d \circ u_2 \circ X' 
\end{equation*} 
where $u_1$ and $u_2$ are identity foams with decorations that are invertible under composition $\circ$ in the $2$-morphism direction, 
and $X'$ is a foam that is equal to $X$ as a CW-complex, but might have different labels on facets\footnote{Facets which would have label $0$ have to be erased in $X'$.}. 
Practically this means that we can commute the red facet past the foams mentioned in the statement of the lemma at the expense of invertible decorations.
Furthermore, we'll see that we can keep the red facet clear of all such decorations.

First we look at the case of a digon creation. The following graphics represent the local piece around the digon creation. 
We suppress the precise description of the unit decorations, since they are not immediately relevant for the following discussion and can easily be reconstructed from the description here and the relations in Section \ref{relations}. 
We do, however, keep track of where the decorations are placed and of their type: we place a pair $\circ_1, \circ_2$ on faces with alphabets $\X$ and $\Y$ respectively for a decoration of the form 
$\sum_{\alpha \in P(-,-)} (-1)^{|\hat{\alpha}|} \pi_\alpha(\X)\pi_{\hat{\alpha}}(\Y)$ and $*_1,*_2$ for the corresponding inverse decoration.
\begin{equation*}
\eqnAF
\end{equation*}

The first commutation relation holds because it is a M-P foam with its inverse, see relation \eqref{FoamRel1}. 
For the second one we first introduce a red-blue blister via relation \eqref{iso4} below the seam of the red facet, next we slide this seam across the seam of the blister using relation \eqref{FoamRel1}, 
and finally we use equation \eqref{iso3} (with inverse units~$*$) to join the blister and the digon creation in the upper region of the foam\footnote{Here, we avoid the use of relation (3.13) in \cite{QR} which would put decorations on the split off red facet.}.
Analogous identities hold for sliding a facet past a digon creation on the other side. 

Next we consider the case of an unzip:
\begin{equation*}
\eqnAG
\end{equation*}
The first commutation relation again holds because it is a M-P foam with its inverse. 
For the second one, we first break the green strip in the lower half of the diagram on the left-hand side of the relation using relation \eqref{iso1}.
The seam bounding the upper green region can then be moved upward across the seam of the red facet using relation \eqref{FoamRel1}, 
and finally the whole upper green region can be removed via relation \eqref{iso2} at the expense of a unit $\circ$ acting on top. 
Similar commutation relations hold for digon removal and zip foams. 
\end{proof}

\begin{thm} 
\label{splitdifferential}
Let $\LC{\tau}{\Sigma}{f}$ be an idempotent colored summand of the chain complex associated to a tangle diagram, 
then there exists a choice of isomorphism foams $T_W$ used to define the functor $\phi$ (see Definition \ref{def-phifunct}) 
such that $\phi(\LC{\tau}{\Sigma}{f})=\bigotimes_\lambda \LC{\tau}{\Sigma}{\lambda \in f}$. 
\end{thm}

\begin{proof} Recall from Definition \ref{def-phi} and the proof of Proposition \ref{2homiso} that $\phi$ is the composition of functors $\phi_2$ and $\phi_1$. 
The latter acts on complexes by replacing colored webs $W$ by $L\otimes W\otimes R$ and foams $d$ by $\id_L\otimes d \otimes \id_R$. 
We prove this theorem by constructing splitter isomorphism foams:
\begin{equation*}T_W\colon \phi_1(W) = L\otimes W\otimes R \to \bigsqcup_\lambda W_\lambda=\phi(W)
\end{equation*} 
for each colored web $W$ in $\LC{\tau}{\Sigma}{f}$ that give an isomorphism of chain complexes $\phi_1(\LC{\tau}{\Sigma}{f}) \to \bigotimes_\lambda \LC{\tau}{\Sigma}{\lambda\in f}$. That is, we have to check that the $T_W$ assemble to a chain map with respect to the differential $d_1:=\id_L\otimes d \otimes \id_R$ on $\phi_1(\LC{\tau}{\Sigma}{f})$ 
and the differential $d_2$ on $\bigotimes_\lambda \LC{\tau}{\Sigma}{\lambda\in f}$. 
If $d_1\colon \phi_1(W_1) \to \phi_1(W_2)$, then we need
\begin{equation}
\label{eqn-chainmap}
T_{W_2} \circ d_1 = d_2 \circ T_{W_1}.
\end{equation}
Actually, it suffices to construct isomorphism foams $T''_W\colon \phi_1(W) \to \phi(W)$ such that for every web $W$ in $\LC{\tau}{\Sigma}{f}$ there is an identity foam with unit decoration $u_W\colon \phi(W) \to \phi(W)$ such that:
\begin{equation}
\label{eqn-chainmap2}
T''_{W_2} \circ d_1 = u_{W_2} \circ (\pm d_2) \circ u_{W_1}^{-1}  \circ T''_{W_1}.
\end{equation} 
Then setting $T'_{W}:= u_W^{-1}\circ T''_W$ gives isomorphism foams that satisfy
\begin{equation*}
T'_{W_2} \circ d_1 = (\pm d_2) \circ T'_{W_1} 
\end{equation*}
and with a suitable choice of signs $T_W:=\pm T'_W$ will satisfy equation \eqref{eqn-chainmap}. 
That such a sign-assignment always exists is well-known and can be proved along similar lines as the fact that Khovanov homology is independent of the numbering of the crossings in a link diagram.

It remains to construct web splitting isomorphism foams $T''_W$ that satisfy equation \eqref{eqn-chainmap2}. They are systematically built in three steps:
\begin{enumerate}
\item The resolutions of a crossing in the tangle diagram are ladder webs (see \cite{CKM} or \cite{QR} for this terminology) with two rungs. 
The first step splits the rungs in every crossing ladder web and sorts them into groups according to their root coloring.
\item The second step splits the uprights in every crossing ladder. The result is a semi-split web.
\item The third step is of a global nature; it completely separates the colored components, as in the proof of Proposition \ref{webisoprop}.
\end{enumerate}
Every step corresponds to a foam that splits the web further and $T''$ is then defined as their composition. 
In the following we show that the foams in every step satisfy an equation of type \eqref{eqn-chainmap2}. 
That is, the cap- (or cup-) differential can be moved through the splitting foam at the expense of signs and unit decorations which only depend on the identity foam on which they are placed. 
In this case we say that the differential \emph{commutes} with the splitting foam \emph{up to canonical units}. 
If every step satisfies this then so does the composite $T''$ since unit decorations slide through such isomorphism foams via relations \ref{dotmigration1}. 

In each of the three steps we only treat the case of colorings by two orthogonal idempotents, which are indicated by red and blue colorings. 
The same argument implies that we can split off one root at a time from the rest, and induction on the number of distinct roots in $\Sigma$ then proves the theorem. 
Furthermore, we only consider the case of cap-differentials, as the cup-differential case is completely analogous.

\textbf{Step 1:} For every crossing, we consider the corresonding ladder web. 
First we split the rungs of the ladder into components:
\begin{equation*}
\eqnAH .
\end{equation*}
We choose this foam to be the image under the foamation functor $\Phi_\Sigma$ of certain categorified quantum group 2-morphisms in the quotient $\Ucatc_Q^{0\leq N}(\glm)^\Sigma$. 
For this, we use thick calculus, but we omit the weights and thicknesses of strands. 
Here, we use colors blue and red to indicate decorations by idempotents corresponding to disjoint multisubsets of $\Sigma$, whereas green is the generic color which is used for mixed colorings. 
Let the reader be warned again that the following graphics show categorified quantum group 2-morphisms and \textbf{not} webs. 
The foam above is given by 
\begin{equation*}
\eqnAI .
\end{equation*}
Using the $\Ucat_Q(\glnn{2})$ relations, it is not hard to see that this 2-morphism is invertible via the vertically flipped 2-morphism with some unit decorations. 
The only non-trivial observation is that the oppositely oriented Reidemeister II type move can be undone at the expense of a sign, because all error terms are killed by orthogonal idempotents. 

In the following we investigate how this 2-morphism commutes with red and blue thickness $1$ cap 2-morphisms respectively. 
This computation immediately transfers to the corresponding foams via $\Phi_\Sigma$ 
\begin{eqnarray*}
\eqnAJ .
\end{eqnarray*}
Here and in the following, $r$ (respectively $b$) is the thickness of the right red (resp. left blue) component in the bottom green strands. An analogous computation shows:
\begin{equation*}
\eqnAK
\end{equation*}
The last equation holds because we can swap the positions $\circ_1$ and $\circ_2$ on strands of thickness $r$ and $b$ at the expense of multiplying by $(-1)^{r b}$.
This is immediate from:
\begin{equation*}
\sum_{\alpha \in P(r,b)} (-1)^{|\hat{\alpha}|} \pi_\alpha(\X)\pi_{\hat{\alpha}}(\Y) =\sum_{\hat{\alpha} \in P(b,r)} (-1)^{|\hat{\alpha}|} \pi_{\hat{\alpha}}(\Y)\pi_{\hat{\hat{\alpha}}}(\X) = (-1)^{r b}\sum_{\beta \in P(b,r)} (-1)^{|\hat{\beta|}} \pi_{\beta}(\Y)\pi_{\hat{\beta}}(\X).
\end{equation*}
and the analogous statement holds for the inverse decorations on positions $*_1$ and $*_2$ similarly.

We conclude that the foams from the first step commute with the differential up to canonical units.

\textbf{Step 2:} We further split the crossing resolutions into semi-split webs:

\begin{equation*}
\eqnAL .
\end{equation*}
A foam that peels off the outer strands can be constructed from the building blocks studied in Lemma \ref{commutation}. 
According to it, a differential with this target web commutes with the peeling foam up to canonical unit decorations. 
All these local foams glue together and can be further composed with unzips (if necessary) to have as target semi-split webs. These unzips are placed far away from crossing sites and thus don't change the commutation behavior.

\textbf{Step 3:} Finally, we construct a foam from the semi-split webs of step 2 to the completely split webs $\phi(W)$. 
This can be done as in the proof of Proposition \ref{webisoprop}, but we further assume that the ``squares'' in which the differentials are supported only interact with edges far away from other colored crossing sites 
during the homotopy. In other words, we assume that the vertices and edges from other colored squares never cross each other during the homotopy.
We thus check that the local move of isotoping a red square through a blue edge commutes with the red-cap differential up to canonical units.

Each of the isomorphisms
\begin{equation}
\label{squarethrough}
\eqnAN 
\end{equation}
except the third is a composite of red-blue zip, unzip, digon creation, digon removal, and M-P foams as in Lemma \ref{commutation}, and hence they commute with the differential up to canonical units. 
The third isomorphism can be realized as a blue cap in thick calculus, which commutes with a red cap on the same square, up to a sign. 
The red-cap differential then also commutes, up to sign and canonical units, with the inverse of the above isomorphism, and with pulling a blue facet across the square in the opposite direction:
\begin{equation*}
\eqnAO .
\end{equation*}
\end{proof}

\subsection{Identifying the tensorands}
\label{section-unicoloredfoam}
\begin{defi}
Let $\hat{\foam{N}{}^{\lambda\in \Sigma}}$ be the 2-subcategory of $\hat{\foam{N}{}^\Sigma}$ consisting of only those 1-morphisms and 2-morphisms colored by idempotents 
$\idem_\lambda$ corresponding to multisubsets of $\Sigma$ which only contain the root $\lambda$.
\end{defi}

\begin{lem} $\hat{\foam{N}{}^{\lambda\in \Sigma}}$ is generated as a 2-category by the same elementary foams as $\foam{N}{}$, 
but with idempotent decorations $\idem_\lambda$ on each web edge and foam facet. It satisfies the same relations as $\foam{N}{}$ and additionally:
\begin{equation}
\label{dotrellambda}
\eqnAP \quad .
\end{equation}

\end{lem}
\begin{proof}
All relations except equation \eqref{dotrellambda} are directly inherited from $\foam{N}{}$ via its quotient $\foam{N}{}^{\Sigma}$. The decorations $\idem_\lambda$ are idempotent and can be moved around freely. 
For relation \eqref{dotrellambda}, we write the action of a dot as $\xi$ and will show the equivalent formulation $\idem_\lambda (\xi-\lambda)^{N_\lambda}=0$. 
To see this, we consider the algebra of decorations of a $1$-labeled facet in $\Foam{N}^\Sigma$, which is given by $\C[\xi]/\langle P(\xi)\rangle$. 
Under the algebra isomorphism
\begin{eqnarray*}
\C[\xi]/\langle P(\xi)\rangle \to & \bigoplus_{k=1}^l \C[\xi]/\langle (\xi-\lambda_k)^{N_{\lambda_k}}\rangle \\
p(\xi) + \langle P(\xi) \rangle \mapsto & (p(\xi) + \langle (\xi-\lambda_1)^{N_{\lambda_1}} \rangle,\dots, p(\xi) + \langle (\xi-\lambda_l)^{N_{\lambda_l}} \rangle )
\end{eqnarray*}
$\idem_\lambda$ is sent to the vector having a single entry $1+\langle (\xi-\lambda)^{N_{\lambda}}\rangle$ and zero everywhere else, 
hence $\idem_\lambda (\xi-\lambda)^{N_\lambda}$ is sent to zero.
\end{proof}

\begin{prop}\label{prop-iso2cat} Let $N_\lambda$ be the multiplicity of $\lambda$ in $\Sigma$, then there is an isomorphism of 2-categories:
\begin{equation}
\foam{N_\lambda}{}^\bullet \cong \hat{\foam{N}{}^{\lambda\in \Sigma}} .
\end{equation}
\end{prop}
\begin{proof}
Let $\iota_\lambda\colon \foam{N_\lambda}{}^\bullet \to \hat{\foam{N}{}^{\lambda\in \Sigma}}$ be the 2-functor which is defined on:
\begin{itemize}
\item objects by sending a sequence $\mathbf{a}$ to itself,
\item 1-morphisms by sending webs to the same webs, but with additional coloring by multisets containing only $\lambda$ on the edges, and
\item 2-morphisms by sending a foam to the foam which is topologically identical but has a decoration by a $\lambda$-idempotent $\idem_\lambda$ added on every facet.
\begin{equation}
\smallfacet[1]{}{k}
\quad \mapsto \quad  ~ \smallfacet[1]{\idem_\lambda}{k}
\end{equation}
A decoration on a foam facet, interpreted as a symmetric polynomial in an alphabet $x_1,\dots, x_k$, is sent to the same symmetric polynomial, but in the alphabet $x_1-\lambda, \dots, x_k-\lambda$. 
Formally, it suffices to define:
\begin{equation}
\smalldotfacet[1]{}{1}
\quad \mapsto  \quad  ~ \smalldotfacet[1]{\idem_\lambda}{1} ~ - \lambda ~
\smallfacet[1]{\idem_\lambda}{1} \quad .
\end{equation}
\end{itemize}
We now check that $\iota_\lambda$ exactly maps the defining relations of $\foam{N_\lambda}{}^\bullet$, see \cite{QR}, to the set of relations that determine $\hat{\foam{N}{}^{\lambda\in \Sigma}}$, which was identified in Lemma \ref{dotrellambda}.

All relations that do not involve decorations are preserved by $\iota_\lambda$; these are (QR 3.8), (QR 3.12), (QR 3.15) - (QR 3.20). We examine the remaining relations:
\begin{itemize}
\item The following is a minimal version of relation (QR 3.9) 
\begin{equation*}
 \eqnAR
\end{equation*}
and the collection of all instances of this relation (for $1\leq s \leq a+b$) has the effect of identifying symmetric polynomials in the alphabet $\{x_1,\dots x_{a+b} \}$ on the $a+b$ facet with those in the alphabet $\{y_1, \dots, y_{a+b} \}$ which is the union 
of the alphabets on the other two facets. 
The 2-functor $\iota_\lambda$ maps these relations to relations that identify symmetric polynomials in $\{x_1-\lambda,\dots, x_{a+b} - \lambda \}$ with symmetric polynomials in $\{y_1-\lambda,\dots, y_{a+b} - \lambda \}$. 
They generate the same ideal, and hence are equivalent sets of relations.
\item The following is a minimal version of relation (QR 3.10)
\begin{equation*}
\eqnAS
\end{equation*}
and under $\iota_\lambda$ it is sent to
\begin{equation*}
\eqnAT
\end{equation*}
where the final equality holds
since all terms in the middle except the one with $k$ dots is zero.
\item Relations (QR 3.11)
\begin{equation*}
\eqnAU
\end{equation*}
are sent by $\iota_\lambda$ to relations of the same form, where $\pi_\alpha$ and $\pi_{\hat{\alpha}}$ are now interpreted as symmetric polynomials in the new alphabet which is shifted by $\lambda$. 
Denote the alphabets on the facets in the original relation on which the decorations are placed by $\X$ and $\Y$, then we can write the decoration on the right-hand side of the original relation as: 
\[
\sum_{\alpha \in P(a,b)} (-1)^{|\hat{\alpha}|} \pi_\alpha(\X)\pi_{\hat{\alpha}}(\Y) = \prod_{x\in \X}\prod_{y\in \Y}(y-x).
\]
Under $\iota_\lambda$ this is sent to 
\[\prod_{x\in \X}\prod_{y\in \Y}((y-\lambda)-(x-\lambda))= \prod_{x\in \X}\prod_{y\in \Y}(y-x) = \sum_{\alpha \in P(a,b)} (-1)^{|\hat{\alpha}|} \pi_\alpha(\X)\pi_{\hat{\alpha}}(\Y).\]
so $\iota_\lambda$ preserves the relation.

\item Relations (QR 3.13) and (QR 3.14) are also preserved by $\iota_\lambda$; the proofs are completely analogous to the case of (QR 3.11).  
\item The relation 
\begin{equation*}
\eqnAV
\end{equation*}
is mapped by $\iota_\lambda$ to 
\begin{equation*}
\eqnAW
\end{equation*}
which is precisely relation \eqref{dotrellambda}.
\end{itemize}

Furthermore, the functor $\iota_\lambda$ is clearly invertible (forget idempotents, shift decorations back) and similar arguments as above show that all relations in 
$\hat{\foam{N}{}^{\lambda\in \Sigma}}$ are sent by the inverse to relations of $\foam{N_\lambda}{}^\bullet$. 
\end{proof}

\subsection{Proof of the decomposition theorem}
\label{section-assembly}
For this section let $\cal{L}$ be an oriented, labeled link diagram.
Recall that the complex $\LC{\cal{L}}{\Sigma}{}$ over $\foam{N}{}^{\Sigma}$ is, up to homotopy equivalence, an invariant of the corresponding oriented, framed, labeled link and each web appearing in $\LC{\cal{L}}{\Sigma}{}$ has endomorphisms of $\mathbf{o}^{top}:=(N,\dots,N)$ as objects. 
We can view $\LC{\cal{L}}{\Sigma}{}$ as a complex in $\hat{\foam{N}{}^{\Sigma}}$ (since $\LC{\cal{L}}{\Sigma}{}$ embeds as a full 2-subcategory) 
where it splits into a direct sum of complexes $\LC{\cal{L}}{\Sigma}{f}$, one for each coloring $f$ of link components with a multisubset of $\Sigma$ of the correct size. 
The objects of a summand $\LC{\cal{L}}{\Sigma}{f}$ are again endomorphisms of $\mathbf{o}^{top}$ in $\hat{\foam{N}{}^{\Sigma}}$, and
there are natural isomorphisms of chain complexes of vector spaces:
\begin{equation}
\taut(\LC{\cal{L}}{\Sigma}{})=\underbrace{\Hom(1_{\mathbf{o}^{top}},\LC{\cal{L}}{\Sigma}{})}_{\text{over } \foam{N}{}^{\Sigma}} \cong 
\underbrace{\bigoplus_f \Hom(1_{\mathbf{o}^{top}}, \LC{\cal{L}}{\Sigma}{f})}_{\text{over } \hat{\foam{N}{}^{\Sigma}}} = \bigoplus_f  \taut(\LC{\cal{L}}{\Sigma}{f}).
\end{equation}
It follows that $\taut(-)$ respects the direct sum decomposition, and it remains to describe the summands $\taut(\LC{\cal{L}}{\Sigma}{f})$.

The only non-zero coloring of the identity web $1_{\mathbf{o}^{top}}$ is the one where every $N$-labeled strand is colored by the full multiset $\Sigma$. 
With respect to this coloring, we will use the object $\bigsqcup_\lambda \mathbf{o}^{top}_\lambda$, 
the (co-)domain of the associated split web $\bigsqcup_\lambda (1_{\mathbf{o}^{top}})_\lambda$, 
and the object $\mathbf{o}^{top}_\lambda$, the (co-)domain of the $\lambda$-component $(1_{\mathbf{o}^{top}})_\lambda$ of $1_{\mathbf{o}^{top}}$. 
For endomorphism webs on these objects and foams between, them we define the representable functors 
$\taut_{split}(-) := \Hom(\bigsqcup_\lambda (1_{\mathbf{o}^{top}})_\lambda, -)$ and 
$\taut_{\lambda}(-) := \Hom((1_{\mathbf{o}^{top}})_\lambda, -)$  respectively. 
Further, we now need the webs $L$ and $R$ for the object $\mathbf{o}^{top}$ with the only possible incidence condition, as given in Definition \ref{LRdef}.

Recall that $\LC{\cal{L}}{\Sigma}{f}$ is a complex of colored webs $W$ and foams $d$ between them. 
Then $\phi_1(\LC{\cal{L}}{\Sigma}{f})$ is the complex consisting of webs $L\otimes W \otimes R$ and foams $\id_L\otimes d \otimes \id_R$ between them and 
$\phi(\LC{\cal{L}}{\Sigma}{f}) =  \phi_2\phi_1(\LC{\cal{L}}{\Sigma}{f})$ is the complex consisting of webs $\bigsqcup_\lambda W_\lambda$ 
and foams $T_*\circ (\id_L\otimes d \otimes \id_R) \circ B_{*}$. We first show:

\begin{lem} There are isomorphisms of chain complexes of vector spaces:
\begin{equation}
\label{eqn-isocx}
\taut(\LC{\cal{L}}{\Sigma}{f}) \cong \taut_{split}(\phi_1(\LC{\cal{L}}{\Sigma}{f})) \cong \taut_{split}(\phi(\LC{\cal{L}}{\Sigma}{f})).
\end{equation}
\end{lem}
\begin{proof}
Proposition \ref{2homiso} provides the following isomorphisms between the objects of these chain complexes:
\begin{align*}
\taut(W) &= \Hom(1_{\mathbf{o}^{top}}, W) \\
& \overset{\phi_1}{\cong} \Hom(L\otimes 1_{\mathbf{o}^{top}} \otimes R, L\otimes W \otimes R) \\
& \cong \Hom(\bigsqcup_\lambda (1_{\mathbf{o}^{top}})_\lambda,L\otimes W \otimes R) = \taut_{split}(\phi_1(W)) \\
& \cong \Hom(\bigsqcup_\lambda (1_{\mathbf{o}^{top}})_\lambda,\bigsqcup_\lambda W_\lambda) = \taut_{split}(\phi(W)),
\end{align*}
where the last two isomorphisms are given by composition with $B_{1_{\mathbf{o}^{top}}}$ and $T_W$. 
Under these isomorphisms, the differentials transform as required for equation \eqref{eqn-isocx}:
\begin{flalign*}
\taut(d)= \; & (\Hom(1_{\mathbf{o}^{top}}, W) \xrightarrow{d\circ} \Hom(1_{\mathbf{o}^{top}}, W')) \\
\mapsto \;\; &(\Hom(L\otimes 1_{\mathbf{o}^{top}} \otimes R, L\otimes W \otimes R) \xrightarrow{(\id_L\otimes d \otimes \id_R) \circ}
\Hom(L\otimes 1_{\mathbf{o}^{top}}\otimes R, L\otimes W'\otimes R))\\
\mapsto \;\; &(\Hom(\bigsqcup_\lambda (1_{\mathbf{o}^{top}})_\lambda, L\otimes W \otimes R) \xrightarrow{(\id_L\otimes d \otimes \id_R) \circ} 
\Hom(\bigsqcup_\lambda (1_{\mathbf{o}^{top}})_\lambda, L\otimes W'\otimes R)) = \taut_{split}(\phi_1(d))\\
\mapsto \;\; &(\Hom(\bigsqcup_\lambda(1_{\mathbf{o}^{top}})_\lambda, \bigsqcup_\lambda W_\lambda) \xrightarrow{T_*\circ (\id_L\otimes d \otimes \id_R) \circ B_*\circ}
\Hom(\bigsqcup_\lambda (1_{\mathbf{o}^{top}})_\lambda, \bigsqcup_\lambda W'_\lambda)) = \taut_{split}(\phi(d)).
\end{flalign*}
\end{proof}
Theorem \ref{splitdifferential} shows that there is a consistent choice of isomorphism foams $T_*$ and $B_*$ in the definition of $\phi$, such that 
$\phi(\LC{\cal{L}}{\Sigma}{f})=\bigotimes_{\lambda} \LC{\cal{L}}{\Sigma}{\lambda\in f}$. Now we have:
\begin{equation}
\label{eq-tensordeco} \taut(\LC{\cal{L}}{\Sigma}{f})\cong \taut_{split}(\phi(\LC{\cal{L}}{\Sigma}{f})) = \taut_{split} (\bigotimes_{\lambda} \LC{\cal{L}}{\Sigma}{\lambda\in f})\cong 
\bigotimes_{\lambda} \taut_\lambda (\LC{\cal{L}}{\Sigma}{\lambda \in f}).
\end{equation}
The last isomorphism is clear from the definition of the two versions of $\taut$ and the tensor product structure given by disjoint union of webs and foams. 

It remains to identify the tensorands  $\taut_\lambda (\LC{\cal{L}}{\Sigma}{\lambda \in f})$. 
To this end, recall the notation $\cal{L}(a_1,\dots,a_k)$ introduced in the statement of Theorem \ref{mainthm}, which makes explicit that we consider 
$\cal{L}$ with the $i^{\mathrm{th}}$ component labeled by the fundamental $\slnn{N}$ representation $\bV^{a_i} \C^N$. 
Let $b_{i,j}$ be the multiplicity of the root $\lambda_j$ in the multisubset of $\Sigma$ that the coloring $f$ assigns to the $i^{\mathrm{th}}$ component of $\cal{L}$. 
Further, recall that $N_j$ denotes the multiplicity of $\lambda_j$ in $\Sigma$.

The complex $\LC{\cal{L}}{\Sigma}{\lambda_j \in f}$ is a complex over the 2-subcategory $\hat{\foam{N}{}^{\lambda_j\in \Sigma}}$. 
Under the isomorphism to the 2-category $\foam{N_j}{}^\bullet$, which was established in Proposition \ref{prop-iso2cat}, this complex corresponds to the undeformed 
$\slnn{N_j}$ complex $\LC{\cal{L}(b_{1,j},\dots, b_{k,j})}{\{0,\dots,0\}}{}$ of the re-labeled sublink $\cal{L}(b_{1,j},\dots, b_{k,j})$. 
As we have seen in Remark \ref{tensorcxrem}, this correspondence preserves the homological grading. 
Clearly, $\taut_{\lambda_j}(\LC{\cal{L}}{\Sigma}{\lambda_j \in f})$ is isomorphic to the image of $\LC{\cal{L}(b_{1,j},\dots, b_{k,j})}{\{0,\dots,0\}}{}$ under the appropriate representable functor, 
and the homology of this complex is $\LH{\cal{L}(b_{1,j},\dots, b_{k,j})}{\slnn{N_{j}}}$. 
Finally, by equation \eqref{eq-tensordeco}, $\taut(\LC{\cal{L}}{\Sigma}{f})$ is isomorphic to the tensor product of these complexes, 
and since we are working over $\C$, 
the K{\"u}nneth theorem gives that the homology of this tensor product complex is isomorphic to the tensor product of the respective homologies. 
This completes the proof of Theorem \ref{mainthm}.

\end{document}